\documentclass[10pt,a4paper]{amsart}
\usepackage{verbatim}

\usepackage[toc]{appendix}
\usepackage[T1]{fontenc}

\usepackage{graphicx}
\usepackage{tikz}
\usetikzlibrary{
  arrows.meta,
  calc,
  decorations.pathreplacing
}
\usepackage{enumerate}
\usepackage{mdframed}
\newcommand{\deris}{\frac{\textnormal{d}}{\textnormal{d}s}}
\usepackage{amsmath,amsfonts,amssymb}
\usepackage{color}
\def\loc{\operatorname{loc}}
\usepackage{cite}
\usepackage{ latexsym }
\definecolor{citation}{rgb}{0.11,0.67,0.84}
\definecolor{formula}{rgb}{0.1,0.2,0.6}
\definecolor{url}{rgb}{0.11,0.67,0.84}
\usepackage{pgf,tikz}
\usepackage{mathrsfs}
\usepackage{fancyhdr}
\usepackage{dutchcal}

\newcommand{\medint}{-\kern -,375cm\int}

\newcommand{\medintinrigo}{-\kern -,315cm\int}
\makeatletter
\newcommand{\linethrough}{\mathpalette\@thickbar}
\newcommand{\@thickbar}[2]{{#1\mkern0mu\vbox{
    \sbox\z@{$#1#2\mkern-0.5mu$}%
    \dimen@=\dimexpr\ht\tw@-\ht\z@+2\p@\relax 
    \hrule\@height0.5\p@ 
    \vskip\dimen@
    \box\z@}}
}
\makeatother

\newcommand{\vv}{\mathcal{v}}
\newcommand{\kk}{\kappa}
\newcommand{\bof}{\mathcal h_u}
\newcommand{\DD}{\mathcal D^{1,2}_{\#}(\er^n)}
\newcommand{\DDDD}{\mathcal D^{1,2}_{\#}(\er^2)}
\newcommand{\DDd}{\mathcal D^{1,2}(\er^n)}
\newcommand{\DDD}{\mathcal D^{1,2}_{\#}(\er^n)^*}

\newtheorem{theorem}{Theorem}[section]
\newtheorem{lemma}[theorem]{Lemma}
\newtheorem{proposition}[theorem]{Proposition}

\newtheorem{corollary}[theorem]{Corollary}
\newtheorem{definition}[theorem]{Definition}
\newtheorem{remark}[theorem]{Remark}
\numberwithin{equation}{section}

\usepackage{dsfont}

\newcommand{\reqnomode}{\tagsleft@false}

\usepackage{hyperref}

\def\dx{\,{\rm d}x}
\def\dz{\,{\rm d}z}
\def\dtt{\,{\rm d}t}
\def\drr{\,{\rm d}r}

\def\ccap{\textnormal{Cap}_{p}}
\def\ds{\,{\rm d}s}
\def\dt{\,{\rm d}t}
\def\dy{\,{\rm d}y}

\def \d{\,{\rm d}}
\def \diver{\,{\rm div}}
\def\dist{\,{\rm dist}}

\def\supp{\,{\rm supp}}
\def\diam{\,{\rm diam}}

\allowdisplaybreaks
\makeatletter
\DeclareRobustCommand*{\bfseries}{%
  \not@math@alphabet\bfseries\mathbf
  \fontseries\bfdefault\selectfont
  \boldmath
}

\makeatother

\newlength{\defbaselineskip}
\newcommand{\mint}{\mathop{\int\hskip -1,05em -\, \!\!\!}\nolimits}

\def \diver{\,{\rm div}}

\newcommand{\ett}{\textnormal T}

\newcommand{\er}{\mathbb R}
\newcommand{\el}{\mathbb L}

\newcommand{\logs}{\sqrt{\log}}

\newcommand{\eps}{\varepsilon}

\newcommand{\ti}[1]{\tilde{#1}}

\newcommand{\Ui}{\mathfrak{m}_\mu}

\newcommand{\dd}{\mathrm{d}}

\newcommand{\rr}{\varrho}
\newcommand{\snr}[1]{\lvert #1\rvert}
\newcommand{\nr}[1]{\lVert #1 \rVert}
\newcommand{\rif}[1]{(\ref{#1})}
\newcommand{\tx}[1]{\textnormal{\texttt{#1}}}
\newcommand{\stackleq}[1]{\stackrel{\rif{#1}}{ \leq}}

\def\loc{\operatorname{loc}}

\def\eqn#1$$#2$${\begin{equation}\label#1#2\end{equation}}

\def\supp{\,{\rm supp }}

\def\XXint#1#2#3{{\setbox0=\hbox{$#1{#2#3}{\int}$}
     \vcenter{\hbox{$#2#3$}}\kern-.5\wd0}}

\title{Light rays of Lorentzian graphs and the Born--Infeld model
}

\author[De Filippis]{Cristiana De Filippis}  \address{Cristiana De Filippis\\Dipartimento SMFI, Universit\`a di Parma\\ Parco Area delle Scienze 53/A, 43124 Parma, Italy} \email{\url{cristiana.defilippis@unipr.it}}
\author[Mingione]{Giuseppe Mingione}  \address{Giuseppe Mingione\\Dipartimento SMFI, Universit\`a di Parma, Parco Area delle Scienze 53/a, Campus, 43124 Parma, Italy} \email{\url{giuseppe.mingione@unipr.it}}

\begin{document}

\subjclass[2020]{Primary 35J93, 35B65;
Secondary 31C45, 53C50 \vspace{1mm}}
\renewcommand{\subjclassname}{2020 Mathematics Subject Classification}


\title{Light rays of Lorentzian graphs and the Born--Infeld model}

\vspace{1mm}

\begin{abstract}
Following a program outlined by Bartnik, and also employing nonlinear potential techniques we developed in the last years, we study lightlike singularities
of variational solutions to the prescribed Lorentzian mean curvature
equation with possibly unbounded (rough) curvature. This corresponds to allowing for unbounded charges in the Born--Infeld model. We prove a sharp
extension of the anti-peeling theorem of Bartnik and Simon beyond the
bounded-curvature regime. Anti-peeling persists under the
sole local assumption
$
\mu\in L(n-1,s)$, $s<\infty,$
thereby proving the conjectured Lebesgue-scale statement and refining it
to the exact threshold in the Lorentz scale. This is optimal: at the weak
endpoint \(L(n-1,\infty)\) we construct a filamentary charge whose
variational solution contains a light segment. As a consequence, we obtain
sharp absence of light segments in the critical finite-index Lorentz regime,
settling both the corresponding conjecture and the critical question of
Byeon, Ikoma, Malchiodi and Mari (2021). Below the anti-peeling threshold, we
develop a capacitary approach to the set of light segments and derive
quantitative Hausdorff-dimension estimates, providing an answer
to another conjecture of theirs and answers to a further question of Bartnik (1987) on non-isolated singularities. Relatedly, we also establish borderline regularity estimates
for the Lorentzian boost function
$
1/\sqrt{1-|Du|^2}
$
and for the Hessian. In particular, we obtain an endpoint
\(L^n\)--\(L(n,1)\) theory yielding uniform spacelikeness and continuity
of the gradient, extending previous regularity results for the
Born--Infeld equation up to the optimal endpoint. Notably, using nonlinear potential theoretic methods, we obtain a rigidity form of the classical Calabi--Bernstein and Cheng--Yau theorems.  Finally, we give a natural connection to nonlinear potential theory proving some intrinsic nonlinear potential estimates for the boost function. 
\end{abstract}
\maketitle

\setcounter{tocdepth}{1}
{\small \tableofcontents}

\section{Rough anti-peeling, a question of Bartnik, and two conjectures}
In this paper we study a few basic questions regarding the existence and the regularity properties of spacelike hypersurfaces with prescribed Lorentzian mean curvature, when the curvature is allowed to be unbounded. 
The regularity theory of such hypersurfaces is governed by a fundamental degeneracy: 
when the tangent plane approaches the light cone the nonuniform ellipticity character of the related equation  \footnote{In this paper, in most of the cases, $\Omega$ will always be either the whole $\er^n$ or a bounded domain of $\er^n$. Additional assumptions will be made according to the context.}
\eqn{bi}
$$
-\mathcal M u:= -\diver
\left(
\frac{Du}{\sqrt{1-\snr{Du}^2}}
\right)
=
\mu
\qquad
\mbox{in }\Omega
$$
becomes prominent and the surface ``goes null'' \cite{bs82}. This means that the so-called ellipticity ratio, i.e., the ratio between the highest and the lowest eigenvalue of the linearization of \rif{bi}, diverges when $|Du|\uparrow 1$. The relevant ellipticity parameter here is the boost function\footnote{See Section \ref{recapsec} for a recap on Lorentzian geometry and \eqref{boosty}.}
\eqn{boost}
$$
\mathcal h_u
:=
\frac{1}{\sqrt{1-\snr{Du}^2}}.
$$
Geometrically, $\mathcal h_u$ is the last component of the
future-directed Lorentzian unit normal to the graph of $u$; see Section \ref{recapsec}. From the analytical viewpoint, $\mathcal h_u$ is the square root of the ellipticity ratio of equation \rif{bi}; see Section \ref{ner} and \rif{ratione}. Therefore spacelikeness of solutions and nonuniform ellipticity are two aspects of the same phenomenon. Besides its geometric meaning, equation \rif{bi} appears
in the electrostatic Born--Infeld model, whose origins lie
in the works of Born \cite{bor33,bor34} and Born and Infeld
\cite{bi33,bi34}. Here $u$ represents the electric potential
generated by the charge density $\mu$. For the classical theory of entire spacelike hypersurfaces
with constant mean curvature, we refer to Treibergs \cite{tre82}. For prescribed mean curvature hypersurfaces in Lorentzian manifolds, see Gerhardt \cite{ger83} and, for their construction
by parabolic methods, Ecker--Huisken \cite{eh91}.
Isolated singularities of area maximising hypersurfaces
were studied by Ecker \cite{eck86}. See \cite{tho12,li22} for the Dirichlet problem for maximal spacelike graphs in higher codimension. Related developments of the Born--Infeld model include
electrostatic problems with an assigned magnetic field and
magnetostatic problems for a family of Born--Infeld-type
models \cite{dm26}. For autonomous nonlinear source terms,
existence and multiplicity results for radial and nonradial
solutions are obtained through nonsmooth critical point
theory in \cite{bimm26}, while ground states and further
nonradial solutions are studied in \cite{bim26}.

Using variational methods to find solutions to \rif{bi} and to solve related Dirichlet problems is particularly convenient, as first shown in the fundamental contribution \cite{bs82}. With 
$
\Omega\subset\mathbb R^n
$
and a weakly spacelike map
$
w:\Omega\to\mathbb R,
$ 
the Born--Infeld functional is
\eqn{bi.en}
$$
\mathcal E_\mu(w;\Omega):=
\int_\Omega\left(1-\sqrt{1-\snr{Dw}^2}\right)\dx-\langle\mu,w\rangle,
$$
where the duality pairing agrees with
$
\int_\Omega\mu w\dx
$
whenever the latter is well defined\footnote{In fact, for duality, when $\Omega=\er^n$, we essentially require that $\mu$ belongs to the dual of the classical Deny-Lions space $\DDd$, while in the case $\Omega$ is a bounded domain we simply consider $\mu \in L^1(\Omega)$.}. Its formal Euler--Lagrange
equation is precisely \eqref{bi}. In the rest of the paper we
abbreviate
$
\mathcal E_\mu(w)\equiv\mathcal E_\mu(w;\er^n).
$

It is then relatively easy to prove existence of  minimizers via Direct Methods 
under rather general assumptions on the charge $\mu$. The main difficulty, and the crucial problem at this stage, 
is instead to give answers to
\eqn{regimplica}
$$
\mbox{minimizer to $\mathcal E_\mu \Longrightarrow $ distributional solution to \eqref{bi}?}
$$
This issue lies at the heart of the classical program initiated by
Bartnik and Simon in \cite{bs82}, and outlined by Bartnik in his survey \cite{bar87}. The two authors  dealt with the case 
\eqn{limitata}
$$\mu\in L^{\infty},$$ that is, with surfaces with bounded Lorentzian mean curvature.  The aim of this paper is to study the scenario when \rif{limitata} does not hold and, indeed, very low degrees of integrability are prescribed on $\mu$, thereby allowing for prescribed mean curvature with large blow-ups, and for very rough charges when looking at the Born--Infeld model. This is a fundamental issue that has been the object of intensive investigation in the last years, see for instance the recent, interesting works \cite{bdp16, bi19, bi23, bimm24}. We provide several sharp conditions, in different
settings, on the prescribed mean curvature $\mu$ --the charge in the
Born--Infeld terminology -- which guarantee the solvability of
\eqref{bi}, thereby addressing \eqref{regimplica}. We also establish
regularity properties of the solutions and of the boost function
$\mathcal h_u$ defined in \rif{boost}. These, in turn, translate into geometric properties
of the corresponding hypersurfaces. Moreover, we provide a number of
basic results aimed at describing the possible singularities of
variational minimizers and solutions.

We shall mainly be interested in two settings: the global case
$
\Omega=\er^n,
$
and the one of Dirichlet problems on bounded domains, as in \cite{bs82}. Therefore, unless otherwise specified, by $\Omega$ we shall always mean a bounded domain of $\er^n$. In the
rest of this section we give a qualitative description of our main
results. For the notation and the basic definitions used throughout
the paper, we refer the reader to Sections \ref{notazioni},
\ref{recapsec}, and \ref{luce}. 

\subsection{Light segments, rough anti-peeling and filamentary charges}\label{llsec}

The analysis in \cite{bs82}  reveals that the natural geometric
obstruction to regularity is represented by so-called light segments. These are the supervillains of this theory.  A weakly spacelike variational extremal may contain a {\em light segment}, i.e., a nontrivial
segment
$
\overline{x_0x_1}
$
such that
$
\snr{u(x_1)-u(x_0)}
=
\snr{x_1-x_0}.
$ This implies that the same identity holds on all its subsegments, as a consequence of $|Du|\leq 1$\footnote{See Remark \ref{raggire}.}. Along such a segment the graph becomes lightlike. 
Every light segment is contained in a unique maximally extended light
segment, which we call a {\em light ray}; see Definition \ref{defilight}. 
A fundamental
result of Bartnik and Simon \cite{bs82} shows that this phenomenon
cannot simply terminate at an interior point: once a portion of the
graph lies in a null direction, the corresponding light segment
propagates along its whole admissible continuation. This rigidity
mechanism is usually referred to as \emph{anti-peeling}. This may be viewed as a geometric counterpart of a strong
maximum principle. Unlike a standard maximum principle,
however, this is a directional rigidity statement rather than an
order principle.

The structural role of this phenomenon becomes particularly clear in
Bartnik's later formulation of the variational regularity problem.
Once a weakly spacelike extremal has been constructed, one of the
remaining steps towards regularity is precisely to exclude entire
null geodesics from the singular set; see \cite[p.~41]{bar87}. Most importantly, excluding the presence of light segments has the fundamental consequence that minimizers also solve the Euler-Lagrange equation \eqref{bi} under fairly general assumptions \cite{bimm24}. Thus
anti-peeling is not merely a geometric feature of singular solutions,
but one of the mechanisms through which regularity of variational
extremals is recovered. One of our main results extends this rigidity to unbounded
prescribed mean curvature and identifies its sharp integrability
threshold, which, in terms of Lebesgue spaces, is
$\mu\in L_{\loc}^{n-1}(\Omega)$.
We actually obtain an even sharper result and determine its exact
threshold within the Lorentz scale\footnote{We point out a coincidence in terminology which may otherwise
cause some confusion: Lorentz transformations are named after the Dutch
physicist Hendrik Antoon Lorentz (1853--1928), whereas Lorentz spaces
are named after the Russian-born American mathematician George Gunter
Lorentz (1910--2006). The two notions happen to
meet in the present paper.}
\begin{theorem}[The rough anti-peeling]\label{al.t}
Let $u\in W^{1,\infty}_{\loc}(\Omega)$ be a weakly spacelike local minimizer of \eqref{bi.en}, where $\Omega\subset \er^n$ is an open subset of $\er^n$. Assume that
\eqn{mmmmm}
$$
\begin{cases}
\displaystyle
\ \mu\in L_{\loc}(n-1,s)(\Omega), \ \ 0<s<\infty\quad &\mbox{if} \ \ n\ge 3\vspace{1.5mm}\\ 
\displaystyle
\ \mu\in L^{1}_{\loc}(\Omega)\quad &\mbox{if} \ \ n=2.
\end{cases}
$$
If there is a segment $\overline{x_{1}x_{0}}\Subset \Omega$ such that
\eqn{a.0}
 $$
 u(x_{t})=u(x_{0})+t\snr{x_{1}-x_{0}}\qquad \mbox{for all} \ \ t\in [0,1],
 $$
with $x_{t}:=x_{0}+t(x_{1}-x_{0})$, then \eqref{a.0} holds for all $t\in \mathbb{R}$ such that $x_{t}\in \Omega$ and $\overline{x_{t}x_{0}}\subset \Omega$.
\end{theorem}
The above result is rather general, as it holds for any open subset and involves only the concept of local minimality, encoded in\footnote{We also consider global minimizers for Dirichlet
problems on bounded domains, as in Section \ref{dirisec},
and on $\mathbb R^n$, as in Definition \ref{minimiglobali}.
Dirichlet minimizers are local minimizers in the sense of
Definition \ref{minimilocali}; the same holds for global
minimizers on $\mathbb R^n$ under the assumptions of
Remark \ref{globale.locale}.}
\begin{definition}[Local minimizers]\label{minimilocali}
Let \(\Omega\subset\mathbb R^n\) be an open subset and let
\(\mu\in L^1_{\loc}(\Omega)\).
A weakly spacelike function
\(u\in W^{1,\infty}_{\loc}(\Omega)\)
is called a \emph{local minimizer} of \(\mathcal E_\mu\) if, for every bounded open set
\(U\Subset\Omega\) and every weakly spacelike
\(v\in W^{1,\infty}(U)\) such that
$
\supp(v-u)\Subset U,
$
one has
$
\mathcal E_\mu(u;U)
\le
\mathcal E_\mu(v;U).
$
\end{definition}
 As a consequence of Theorem \ref{al.t}, we obtain
\begin{theorem}[Absence of light segments and solvability]\label{light.t}
Let $n\ge2$, let $\mu\in\DDd^*$ satisfy\footnote{The space $L^1_0(\mathbb R^2)$ consists of  $L^1(\mathbb R^2)$-functions with zero global average, while $\DDd$ is the classical Deny-Lions space; see \rif{zeromean}. See Section \ref{spazifunzionali} below for the details.}
\[
\begin{cases}
\mu\in L_{\loc}(n-1,s)(\mathbb R^n),\ \ 0<s<\infty \quad 
&\mbox{if} \ \ n\ge 3\vspace{1.5mm}\\ 
\mu\in L^1_0(\mathbb R^2)
&\mbox{if}\ \  n=2,
\end{cases}
\]
and let $u\in\mathbb X(\mathbb R^n)$\footnote{The space $\mathbb X(\mathbb R^n)$ is defined in \rif{XXX}.} be the unique
minimizer of $\mathcal E_\mu$.
Then $u$ has no light segments, and, if $\mu\in L^2_{\loc}(\mathbb R^n)$, then
$u$ weakly solves \eqref{bi}.
In particular, this additional assumption is automatically
satisfied when $n\ge4$,  and  when $n=3$ with $s\le2$.
\end{theorem}

Recall that
$L^{n-1}=L(n-1,n-1)$ and this case is always covered for $n\geq 3$. 
Hence the critical Lebesgue space is only one member of the
admissible scale. In particular, our result also applies to the
strictly larger Lorentz spaces
$
L(n-1,s)$, $n-1 < s <\infty$. 
This is sharp, as shown in the following:
\begin{theorem}[Light segments shine at the weak-endpoint]
\label{ex.t}
Let $n\ge3$. There exists $\mu\in\mathcal D^{1,2}(\mathbb R^n)^*$ satisfying
\[
\mu\in L(n-1,\infty)(\mathbb R^n)
\setminus L(n-1,s)(\mathbb R^n)
\qquad\mbox{for every }0<s<\infty,
\]
such that the unique minimizer
$u\in\mathbb X(\mathbb R^n)$ of $\mathcal E_\mu$
admits a light segment.
Moreover, $u$ is a weak solution to \eqref{bi}
in $\mathbb R^n$.
\end{theorem}
Note that the occurrence of light segments below the $L^{n-1}$ level was shown in \cite{bimm24} and Theorem \ref{ex.t} exploits this phenomenology up to the natural threshold of Marcinkiewicz spaces $L(n-1,\infty)$, that are the natural ones to describe the regularity properties of inverse of distance functions.  
We shall in fact see in a few lines how the analysis in Lorentz spaces is actually essential in order to give a fully satisfying answer to the problems posed by Byeon, Ikoma, Malchiodi and Mari \cite{bimm24}.

The same anti-peeling mechanism applies to the Dirichlet problem in  bounded domains of the type 
\eqn{pd2}
$$
u\to \min_{w\in \mathcal{D}_{0}}\mathcal{E}_{\mu}(w;\Omega),
$$
where the boundary datum $u_0$ is continuous and spacelike, and $\mathcal{D}_{0}$ is the natural Dirichlet class of competitors - these are described in Section \ref{dirisec}. 
We then have 
\begin{theorem}[No light segments for the Dirichlet problem]
\label{b.th}
Let $n\ge 2$ and let $\Omega\subset\mathbb R^n$ be a bounded
domain and assume \eqref{u0u0.bd00}. Assume in addition
that
\[
\mu\in L_{\mathrm{loc}}(n-1,s)(\Omega)
\qquad\text{for some }0<s<\infty
\]
when $n\geq 3$. 
Let $u\in\mathcal D_0(\Omega)$ be the unique minimizer
of the Dirichlet problem \eqref{pd2}.
Then $u$ has no light segments in $\Omega$. Moreover, if $\mu\in L^2_{\mathrm{loc}}(\Omega)$,
then $u$ weakly solves \eqref{bi} in $\Omega$. In particular, this additional assumption is automatically
satisfied when $n\ge4$, or when $n=3$ and $s\le2$.
\end{theorem}
\begin{corollary}\label{cor.t}
Conjecture 1 in {\normalfont\cite{bimm24}} holds true, while the answer
to Question 2 in {\normalfont\cite{bimm24}} is negative (but, at the same time, almost positive).
\end{corollary}
Conjecture 1 from \cite{bimm24} predicts the absence of light segments in Dirichlet problems  when $\mu \in L^q_{\loc}$ for $q>n-1$. Theorem \ref{b.th} goes actually further as absence of light segments is obtained also in the borderline case $q=n-1$ as a consequence of the Lorentz space analysis, which in fact allows us to exclude light segments throughout the larger finite-index Lorentz regime $L_{\loc}(n-1,s)$, $s<\infty$. This is precisely the negative answer to \cite[Question 2]{bimm24}, asking whether or not light segments occur in the case $\mu\in L^{n-1}_{\loc}$. The answer is negative, as follows by taking $s=n-1$, but in a sense almost positive by Theorem \ref{ex.t} as light segments appear in the endpoint case $L(n-1, \infty)$. This is also in accordance with the accompanying comment to Conjecture 1 in \cite[page 18]{bimm24}, according to which {\em ``the case $L^{n-1}$, which includes $\mu \in L^{2}_{\loc}$ when $n=3$, is particularly subtle''}. Indeed, the subtlety calls for the use of a second integrability index, as done here.

The proof of the anti-peeling is genuinely different from that of Bartnik and Simon. 
No maximum or comparison principle is used and this appears somehow surprising. Indeed, extending anti-peeling beyond the bounded curvature regime
was a classical open issue. Indeed, as the authors comment in \cite[Page 7]{bimm24} about the classical anti-peeling of Bartnik and Simon: {\em ``The proof depends on a comparison argument that is not
applicable to more general sources $\mu$, a case for which the relation between singularities
of $\mu$ and properties of light segments, including their existence, is currently unknown}''. Our results give in fact an essentially complete answer to this issue. In our proof, we combine
direct variational competitors, quantitative convexity properties of
the Born--Infeld integrand and borderline interpolation inequalities
in Lorentz spaces. The exponent
$
n-1
$
arises from the geometry transverse to a light segment, whereas the
second Lorentz index captures the endpoint mechanism: finite Lorentz
index provides the absolute continuity of the norm which enters the
anti-peeling argument, while this property fails at the weak endpoint
$
L(n-1,\infty).
$

In dimension three, the weak endpoint singled out by our sharpness
result is
$
L(2,\infty).
$
This has a natural geometric interpretation. If
$
\Gamma\subset\mathbb R^3$
is a smooth curve, then a density with the filamentary singularity
$$
\mu(x)\simeq\frac{1}{\dist(x,\Gamma)}
$$
belongs locally to $L(2,\infty)$ 
but to no $L(2,s)$ 
with $s<\infty.$
Thus codimension-two singular charge densities naturally occur
precisely at the Lorentz borderline case where our anti-peeling rigidity
may fail and the occurrence of light segments is no longer excluded.
Filamentary geometries have a natural precedent in Born--Infeld theory.
Two-dimensional electrostatic configurations, which may equivalently
be interpreted as three-dimensional configurations invariant along
one spatial direction, were already studied in the early development
of the theory by Pryce \cite{pryce35}; see also \cite{ferraro04}
for a later treatment of two-dimensional Born--Infeld electrostatics,
including finite-energy configurations per unit length.
The theory introduces a maximal admissible electric-field
strength \cite{bi34} and, in the normalization used here,
the corresponding saturation regime is
$
|Du|=1.
$
Interestingly, the sharpness construction of Theorem \ref{ex.t}
itself realizes this geometry: the charge is of filamentary type,
while the corresponding solution saturates the Born--Infeld
field bound along a light segment.

\subsection{The regularity path to weak solvability}
In this section we explore a different way to study solvability, which is complementary to that seen above and does not rely on the use of the anti-peeling mechanism. This approach also covers regimes in which the integrability
of $\mu$ is insufficient to exclude light segments. For this, we prove regularity estimates for variational minimizers, and in particular, for the boost function $\mathcal h_u$. Such estimates in turn imply a positive answer to \eqref{regimplica}. This path is more traditional and has actually been pursued before. The global variational framework was developed in \cite{bdp16}, who established weak solvability for radial or locally bounded charges and obtained $C^1$-regularity at the critical exponent $n$ under radial symmetry. For charges without symmetry assumptions, \cite{bi19} establishes local $W^{2,2}$-regularity when $\mu\in L^q(\mathbb R^n)$, $q>2n$, under an additional global integrability assumption; weak solvability and regularity of the gradient were obtained under a further smallness condition. Their subsequent work \cite{bi23} lowered the threshold to $q>n$, removed the smallness requirement, and established uniform spacelikeness and local $W^{2,q}$ regularity. Here we are finally able to treat the missing range $2\le q\le n$, which is particularly relevant as it is not covered by the anti-peeling based strategy, at least when $q<n-1$ and light rays might still appear. The integral estimates for the Hessian and the Sobolev estimates for suitable functions of the boost here provide the compactness and integrability needed to pass to the limit in the nonlinear flux in a suitable approximation scheme and finally show solvability.

\begin{theorem}[Boost function regularity]\label{t1}
Let $n\ge2$, let $2\le q\le n$, and let
$u\in\mathbb X(\mathbb R^n)$ be the minimizer of
\eqref{bi.en}, with
$
\mu\in\mathbb Y(\mathbb R^n)\cap L^q(\mathbb R^n).
$
Then $u\in W^{2,2}_{\loc}(\mathbb R^n)$, it weakly solves
\eqref{bi} and it holds that 
\eqn{unaw}
$$
\mathcal h_u\in L^q_{\loc}(\mathbb R^n),
\qquad
\mathcal h_u^{-1},\log\mathcal h_u
\in W^{1,2}_{\loc}(\mathbb R^n),
$$
with
\eqn{duaw}
$$
u\in L^\infty(\mathbb R^n) \ \mbox{if}\  n\ge3 \qquad \mbox{and} \qquad  \dfrac{u}{\log(e+|\cdot|)}
\in L^\infty(\mathbb R^2) \ \mbox{if}\  n=2
$$
and 
\begin{itemize}
\item[(I)]
If $2<q<n$, then
$$
\mathcal h_u^{(q-2)/2}
\in W^{1,2}_{\loc}(\mathbb R^n).
$$

\item[(II)]
If $q=n\ge3$, then
$$
\mathcal h_u^{p/2}
\in W^{1,2}_{\loc}(\mathbb R^n)
\qquad\mbox{for every }1\le p<\infty.
$$
\item[(III)]
If $n=2$, then
$$
\mathcal h_u\in L^p_{\loc}(\mathbb R^2)
\qquad\mbox{for every }1\le p<\infty.
$$
\end{itemize}
\end{theorem}
Assertions (II)--(III) cover the critical case
$\mu\in L^n(\mathbb R^n)$ in every dimension $n\ge2$,
yielding local integrability of $\mathcal h_u$ to every
finite exponent. In dimensions $n\ge3$, we also obtain
Sobolev regularity for its positive powers.
The higher-dimensional borderline estimate uses a technique
developed in connection with stochastic homogenization,
combining Caccioppoli inequalities with an optimized choice
of cutoff functions \cite{bs20,bs24,dkk24}.
These estimates also provide the analytic input for the
capacitary study of singular sets in the next subsection.

\subsection{Singular sets, a question of Bartnik and another conjecture}
As already mentioned, the regularity estimates of Theorem \ref{t1} yield
crucial quantitative information and provide a way to
investigate, in the present variational setting, the problem of
regularity beyond isolated singularities as suggested by Bartnik
\cite{bar87}. We consider the geometric degeneracy set
$$
\mathcal S_u:=\bigcup \left\{\overline{xy}:x\neq y,\quad |u(x)-u(y)|=|x-y|\right\},
$$
namely the union of all nontrivial light segments, including
their endpoints - see \eqref{language1} below for the precise definition. We also consider
$$
\mathcal I_u
:=
\left\{
x\in\mathbb R^n: |Du(x)|=1
\right\},
$$
where $|Du|$ is understood as the precise representative of
the scalar function $|Du|$, as specified in \eqref{esilimite}.
 We refer to Section \ref{luce} for the
definitions and their relation to the terminology of optimal
transport. Our approach is capacitary. In Section \ref{capacity}, we establish
criteria for general weakly spacelike functions which imply
$
\operatorname{Cap}_p(\mathcal S_u)
=
\operatorname{Cap}_p(\mathcal I_u)
=0$ 
for  $1\le p<n$. 
The assumptions involve local $W^{1,p}$ regularity of $Du$
and of either a positive power of $\mathcal h_u$ or
$\log\mathcal h_u$.
Theorem \ref{t1} provides precisely the required Sobolev
control with $p=2$, since
$Du,\log\mathcal h_u\in W^{1,2}_{\loc}(\mathbb R^n)$.
\begin{theorem}[Capacitary bounds for the degeneracy sets]
\label{sm.t2}
Let $n\ge4$, let
$\mu\in
\mathcal D^{1,2}(\mathbb R^n)^*
\cap L^q(\mathbb R^n)$ and $2\le q<n-1$, 
and let $u\in\mathbb X(\mathbb R^n)$ be the unique minimizer
of $\mathcal E_\mu$.
Then $u$ weakly solves \eqref{bi} and
$
\operatorname{Cap}_2(\mathcal S_u)
=
\operatorname{Cap}_2(\mathcal I_u)
=0.
$
Consequently,
$\dim_{\mathcal H}\mathcal S_u\le n-2$, 
$\dim_{\mathcal H}\mathcal I_u\le n-2$.
\end{theorem}
Of course $\dim_{\mathcal H}$ here denotes the usual Hausdorff dimension. To compare these conclusions with results formulated in terms
of the closed light set, as in \cite{bimm24}, we also study
$$
\mathcal K_u:=\overline{\mathcal S_u}.
$$
Passing to the closure may add points which do not belong to any
nontrivial light segment. Our structural results identify the
geometric mechanism responsible for these additional points.
Let $\mathcal L_u$ denote the set of points
$z\in\mathbb R^n\setminus\mathcal S_u$ for which there exists
a sequence of light rays $\{\ell_k\}$ satisfying
$
\operatorname{dist}(z,\ell_k)\to0$ and 
$
\operatorname{diam}\, (\ell_k)\to0.
$
We prove the disjoint decomposition
\eqn{disjoint}
$$
\mathcal K_u =\mathcal S_u\dot\cup\mathcal L_u
$$ (see Theorem \ref{cap.t.1}). 
Thus every point added by taking the closure is an accumulation
point of light rays whose lengths tend to zero. Decomposition \ref{disjoint} is the conceptual key to understand the nature of the set $\mathcal K_u$. 
In particular, a uniform positive lower bound on the lengths
rules out this phenomenon and allows the capacitary estimate
to pass to $\mathcal K_u$. This phenomenon can occur even for global minimizers that weakly
solve \eqref{bi}. Indeed, in Section \ref{esempione}, for every $n\ge3$,
we construct such a minimizer with compact support and charge
$\mu\in L(n-1,\infty)(\mathbb R^n)$, and therefore outside the range covered by our rough anti-peeling result, whose light rays form a
sequence of disjoint segments with infinitesimal lengths and accumulating at the origin, i.e., $
\mathcal K_u=\mathcal S_u\dot\cup\{0\}$ and $
\mathcal L_u=\{0\}.
$
After constructing a one point singularity we can then approach more general compact sets with positive measure  replicating the previous construction via methods from Real Analysis that are standardly used in PDE. This allows to give negative answers to conjectures and questions posed in \cite{bimm24}. 
\begin{corollary}\label{brutto}
Conjecture 4 in {\normalfont\cite{bimm24}}, asserting that $\dim_{\mathcal H} \mathcal K_u\leq n-q$ when 
$\mu \in L^q$, with $2\le q\le n-1$, is false
for $n\ge4$ and $2\le q<n-1$. Specifically, there exists a compactly supported $ \mu \in L(n-1,\infty)$ 
while $|\mathcal K_u|>0$. Simultaneously, Question 5 in {\normalfont\cite{bimm24}}, asserting that $\mathcal H^{n-1}(\mathcal K_u)=0$, has a
negative answer for every $n\ge3$, even for absolutely
continuous data $\mu$. 
\end{corollary}
Note that, in contrast, the endpoint case $\mu\in L^{n-1}$
of Conjecture 4 in \cite{bimm24} holds true, with the stronger
conclusion $\mathcal K_u=\varnothing$.
Indeed, since $L^{n-1}=L(n-1,n-1)$, Theorem \ref{al.t}
and the spacelike boundary condition exclude light segments,
as in the proof of Theorem \ref{b.th}. In particular, Conjecture 4 holds true in dimension $n=3$,
where its admissible range reduces to $q=2$.
For global minimizers under the assumptions of Theorem \ref{t1},
the same conclusion also follows without using anti-peeling:
taking $q=n-1$ gives
$
\mathcal h_u\in L^{n-1}_{\loc}(\mathbb R^n),
$
and Theorem \ref{intre} yields
$
\mathcal S_u=\mathcal K_u=\varnothing.
$

The construction for \rif{brutto} relies on the presence of shrinking light segments accumulating on the compact set with positive measure. When no shrinking takes place the situation changes and we have the following consequence of Theorem \ref{sm.t2}.

\begin{corollary}\label{sm.c}
Under the assumptions of Theorem \ref{sm.t2}, suppose that
there exists $\ell_0>0$ such that
$
\operatorname{diam}\, (\ell)\ge\ell_0$
for every light ray $\ell$.
Then
$
\mathcal K_u=\mathcal S_u$ and 
$
\mathcal L_u=\varnothing$. Moreover, 
$
\operatorname{Cap}_2(\mathcal K_u)=0$
and $
\dim_{\mathcal H}\mathcal K_u\le n-2.
$
\end{corollary}

As mentioned before Theorem \ref{sm.t2}, in this line, but developing this time a more general direction, we derive a few general statements linking the integrability of the boost function and the Hausdorff dimension of the set of light segments. The peculiarity of these results is that they hold for general weakly spacelike functions, rather than for solutions or minima. An example is
\begin{theorem}[Boost integrability and light segments]
\label{intre}
Let
\(u\in W^{1,\infty}_{\loc}(\mathbb R^n)\) be weakly spacelike and let $\mathcal S_{u}$ be the union of all light segments as defined in \eqref{language1}.
Assume that
$
\mathcal h_u^q\in L^1_{\loc}(\mathbb R^n)
$
for some \(1\le q<\infty\). Then
\eqn{nicchia4}
$$
\begin{cases}
\displaystyle
\mathcal H^{n- q/2}
\ (\mathcal S_{u})=0
&\ \mbox{if}\quad  1\le q<n-1\\[2mm]
\displaystyle\mathcal S_{u}=\mathcal{K}_{u}=\varnothing
&\ \mbox{if}\quad  q\ge n-1.
\end{cases}
$$
In particular, when $n=2$  and  $\mathcal{h}_{u}\in L^{1}_{\loc}(\mathbb{R}^{2})$, it follows that  $\mathcal{S}_{\texttt{i},u}=\mathcal{S}_{u}=\mathcal{K}_{u}=\varnothing$.
\end{theorem}
Further examples of statements of this type are Theorems \ref{sobre1} and \ref{sobre2} in Section \ref{capacity}. 

\subsection{Borderline regularity, quantitative spacelikeness, Calabi--Bernstein and Cheng--Yau}\label{finerreg}
We now address the finer Lorentz endpoint $\mu\in L(n,1)(\mathbb R^n)$, at which we obtain a global bound for the boost, hence uniform spacelikeness, together with continuity of the gradient. A corresponding two-dimensional result holds under the stated global admissibility conditions and an $L^2(\log L)^\alpha$ assumption, $\alpha>2$. Both results are in perfect accordance with what is known in the general theory of nonuniformly elliptic problems developed in \cite{bm20}; see Remark \ref{plapre} below. Apart from the regularity viewpoint, proving the boundedness of the boost function $\mathcal h_u$ has a fundamental  geometric importance, see Remark \ref{geore} below. 
\begin{theorem}[Borderline boost function regularity and rigidity]\label{t2}
Let \(u \in \mathbb{X}(\mathbb{R}^{n})\), \(n \geq 3\), be the minimizer
of \eqref{bi.en} with $\mu\in \mathbb{Y}(\mathbb{R}^{n})\cap L(n,1).$ 
Then \(u\) weakly solves \eqref{bi} and\footnote{See \rif{normalo} and \rif{normalo2} for the definition of the Lorentz norm $[\mu]_{n,1}.$}
\eqn{boost.boundi}
$$
\left\lVert \mathcal{h}_u \right\rVert_{L^\infty(\mathbb{R}^n)} \leq 
\exp\left\{c_n[\mu]_{n,1}\right\}.
$$
Furthermore, \(Du\) is continuous in \(\mathbb{R}^{n}\).
\end{theorem}
The proof of Theorem \ref{t2} builds on a rather delicate mixture of two ingredients. First, geometrically based ideas - monotonicity formulas and their potential theoretic formulations as in Proposition \ref{p31}, with asymptotic consequences as in Proposition \ref{com.sup}. Second: some purely PDE, subtle nonlinear potential theoretic facts that find their origins in recent work of the authors \cite{bm20, dm23, dm25,ddp26}. We indeed introduce some novel global, potential theoretic methods via a calibrated version of the classical De Giorgi's iteration - see Lemma \ref{abstract.iteration}.  These are efficiently able to transform soft asymptotic information coming from monotonicity inequalities as in Proposition \ref{com.sup}, into rigidity statements as in \rif{boost.boundi}. 
\begin{remark}\label{geore}
\emph{Let us outline a few geometric consequences of estimate \rif{boost.boundi}. 
\begin{itemize}
\item The exponential type estimate in \rif{boost.boundi} -- as well as the later ones in \rif{boost.2di} and \rif{boost.bound.dir}-\rif{boost.2d.dir} -- calls for a parallel with the classical gradient estimate of Bombieri, De Giorgi and Miranda for the minimal surfaces \cite{bdgm69}. Note that the Born--Infeld and minimal surface integrands are Legendre
dual, up to an additive constant. We refer to the recent paper \cite{df26} for more in this direction. 
\item The classical theorems of Calabi--Bernstein \cite{cal70} and Cheng--Yau \cite{cy76} assert that every entire
spacelike maximal graph in Lorentz--Minkowski space is a
hyperplane. 
As the function class we are working with implies, 
$Du\in L^2(\mathbb R^n)$, we are a priori excluding tilted hyperplanes.
Nevertheless, estimate \eqref{boost.boundi} provides a
quantitative stability statement for such results within the present 
finite-energy class. Indeed, estimate \rif{boost.boundi} readily implies
$$
\nr{Du}_{L^\infty(\mathbb R^n)}^2
\le
1-\exp\left\{-2c_n[\mu]_{n,1;\mathbb R^n}\right\}.
$$
Geometrically, if $\theta_u$ denotes the hyperbolic angle
between the future-directed unit normal to the graph
and the vertical direction, then
$\mathcal h_u=\cosh\theta_u$, and therefore
$$
\nr{\theta_u}_{L^\infty(\mathbb R^n)}
\le
\operatorname{arcosh}\left(
\exp\left\{c_n[\mu]_{n,1;\mathbb R^n}\right\}
\right).
$$
Consequently, for assigned mean curvatures $\mu$ satisfying the hypotheses of
Theorem \ref{t2}, convergence to zero in $L(n,1)$ forces
the normals to converge uniformly to the vertical normal. Ultimately, Theorem \ref{t2} reaches the sharp Lorentz
threshold of elliptic gradient regularity, connecting it to
quantitative stability in the geometric setting.
\item The exclusion of light segments is a qualitative conclusion:
it does not by itself provide a uniform separation of $|Du|$
from the lightlike threshold. A bound for the boost function
gives precisely this quantitative information, since
$
\|\mathcal h_u\|_{L^\infty}\le M$ implies 
$|Du|\le\sqrt{1-M^{-2}}<1$, that is, the surface does not ``go null'', quantitatively.
Equivalently, since the induced metric is
$
g=\mathds I-Du\otimes Du,
$
a uniform bound for $\mathcal h_u$ prevents the metric from degenerating and keeps the tangent spaces uniformly separated from the light cone. 
\end{itemize}}
\end{remark}
\begin{remark}\label{plapre}
{\em The space $L(n,1)$ provides the natural setting for the boundedness of the boost function once this result is read as a stress tensor intrinsic regularity result. The best comparison is with the $p$-Laplacean equation for which $-\diver (|Du|^{p-2}Du) \in L(n,1)$ implies the continuity of $Du$ with optimal local a priori estimates on the vector field $|Du|^{p-2}Du$; see \cite{km14b} for a survey of such facts. In the present setting Theorem \ref{t2} reads exactly in the same way: $-\mathcal M u=-\diver(\mathcal h_uDu)\in L(n,1)$ implies the continuity of $Du$ and a bound on the stress $\mathcal h_uDu$ (recall $|Du|\leq 1$). Note that Theorem \ref{t2} is in a sense a geometric version of a classical theorem of Stein claiming that $\Delta u\in L(n,1)$ implies the continuity of $Du$. }
\end{remark}
The two-dimensional result is, instead
\begin{theorem}[Borderline boost function regularity II]\label{t222}
Let $u\in \mathbb{X}(\mathbb{R}^{2})$ be the minimizer of \eqref{bi.en} with 
$\mu\in \mathbb{Y}(\mathbb{R}^{2})\cap L^2(\log L)^\alpha(\er^2)$ where  $\alpha>2.$ 
Then $u$ solves \eqref{bi} and\footnote{See \rif{luxi} and \rif{lllog} for the definition of $\nr{\mu}_{L^2(\log L)^\alpha(\er^2)}$.}
\eqn{boost.2di}
$$
\nr{\mathcal h_u}_{L^\infty(\mathbb R^2)}
\lesssim_n
\left(1+  \left\lVert \mu \right\rVert_
    {\mathcal{D}^{1,2}(\mathbb{R}^{2})^{*}}^{2}\right)\exp\left\{
c_{\alpha}
\nr{\mu}_{L^2(\log L)^\alpha(\er^2)}
\right\}.
$$
Furthermore, $Du$ is continuous in $\mathbb{R}^{2}$.
\end{theorem}
The reader might now wonder whether it is possible to obtain an estimate of the type in \rif{boost.boundi} also in the two dimensional case, i.e., replacing $[\mu]_{n,1}$ by $\|\mu\|_{L^2(\log L)^\alpha(\er^2)}$ in \rif{boost.boundi}. The answer is negative, as described in Remark \ref{controre} below. 
\subsection{Borderline regularity for Dirichlet problems}
Theorems \ref{t2}-\ref{t222} have a parallel for solutions to Dirichlet problems \rif{pd2} in bounded and suitably regular domains (see also \rif{pd2in} below for the precise setting). We extend the classical results of \cite{bs82} to unbounded prescribed mean curvature under a smallness condition
confined to a boundary layer. The main assumption is a decomposition of the type
\eqn{data1}
$$
\mu=f+\mu_*,
$$
where \begin{equation}\label{data2}
\|f\|_{L^\infty(\Omega)}\le F_0,
\qquad
[\mu_*]_{n,1;\Omega}\le\Lambda_0,
\qquad
[\mu_*]_{n,1;\Omega_{r_0}}
\le\varepsilon_\partial,
\end{equation}
where 
$
\Omega_{r_0}
:=
\{x\in\Omega:\dist(x,\partial\Omega)<r_0\}. 
$
In other words we are considering $L(n,1)$-perturbations of $L^\infty$-data, where the perturbation is small, in the corresponding $L(n,1)$-norm, but only on {\em an arbitrarily small} boundary layer $\Omega_{r_0}$. We impose the following assumptions on the domain
and the boundary datum. 
\eqn{assuntidir0}
$$
\begin{cases}
\, \mbox{$\Omega$ is a $C^{2,\beta}$-regular, bounded domain of $\er^n$, $n\geq 2$, $\beta\in(0,1)$}\\[2pt]
\, \mbox{$\tx{u}_0 \in C^{2,\beta}(\overline{\Omega})$ such that $\nr{D\tx u_0}_{C^0(\overline\Omega)}
\leq
1-\sigma_0$ for some $\sigma_0\in(0,1)$}.
\end{cases}
$$
\begin{theorem}[Borderline boost regularity for Dirichlet problems]
\label{dir.borderline}
Assume \eqref{assuntidir0}.
For every choice of the constants $\Lambda_0,F_0$ in \eqref{data2}, and of the boundary layer width $r_0>0$, there exist constants
$
\varepsilon_\partial\in(0,1)$ and $
\mathcal h_{\partial\Omega}\ge2$, 
depending only on
$
n,\ r_0,\ \sigma_0,\ \Lambda_0,\ F_0,\
\|\tx u_0\|_{C^2(\overline\Omega)},\ \Omega,
$
with the following property.
Let $u$ be the minimizer of the Dirichlet problem
\eqref{pd2in}, with boundary datum
$u_0=\tx u_0|_{\partial\Omega}$ and  $\mu$
satisfying \eqref{data1}--\eqref{data2}.
\begin{itemize}
\item If $n\ge3$, then
\begin{equation}\label{boost.bound.dir}
\|\mathcal h_u\|_{L^\infty(\Omega)}
\le
\mathcal h_{\partial\Omega}
\exp\left\{c_n[\mu]_{n,1;\Omega}\right\}.
\end{equation}

\item If $n=2$ and, additionally,
$\mu\in L^2(\log L)^\alpha(\Omega)$ for some $\alpha>2$,
then
\begin{equation}\label{boost.2d.dir}
\|\mathcal h_u\|_{L^\infty(\Omega)}
\le
\mathcal h_{\partial\Omega}
\exp\left\{
c_{\alpha,|\Omega|}
\|\mu\|_{L^2(\log L)^\alpha(\Omega)}
\right\}.
\end{equation}
\end{itemize}
Here $c_n$ depends only on $n$, and
$c_{\alpha,|\Omega|}$ depends only on $\alpha$ and $|\Omega|$.
In both cases, $u$ weakly solves \eqref{bi} in $\Omega$, and 
satisfies $\|Du\|_{L^\infty(\Omega)}<1$ with 
$
u\in C^1(\Omega)\cap W^{2,2}_{\loc}(\Omega).
$
\end{theorem}
In \eqref{boost.bound.dir} and \eqref{boost.2d.dir} we cannot replace $\mathcal h_{\partial\Omega}$ by $1$. See Remark \ref{prefactor}. 
\subsection{Geometric potential estimates}\label{geosec}
Nonlinear potential theory is a field of elliptic PDE theory aiming, among other things, at reproducing, for solutions to nonlinear equations, the estimates valid for solutions to the Poisson equation $-\Delta u=\mu$. The story starts with the fundamental paper of Havin \& Maz'ya \cite{hm72}, who introduced classes of nonlinear potentials to study fine properties of solutions. These find, as special cases, the usual linear Riesz potentials and those today known as Wolff potentials. Specifically, let $B_r(x_0)\subset \mathbb{R}^n$, let $\sigma>0$ and $\vartheta\geq 0$ be fixed parameters, and let $w\in L^1(B_r(x_0))$. The nonlinear Havin--Maz'ya type potential ${\bf P}_{\sigma}^{\vartheta}(w;\cdot)$ is defined by
\eqn{defi-P} 
$$
{\bf P}_{\sigma}^{\vartheta}(w;x_0,r)
:= \int_0^r \varrho^{\sigma}
\left( \mint_{B_{\varrho}(x_0)} \snr{w}\, dx \right)^{\vartheta}
\frac{\d\varrho}{\varrho}.
$$
These potentials allow one to pointwise estimate both solutions of nonlinear equations and their gradients in the same way Riesz potentials do in the linear case. For instance, in the case of the classical $p$-Laplacean equation $-\diver\, (|Du|^{p-2}Du )=\mu$, $p\geq 2$, it holds that 
$$
|Du(x_0)|^{p-1}\lesssim_{n,p}
\mathbf P_1^1(\mu;x_0,r)
+ 
\mint_{B_r(x_0)}|Du|^{p-1}\, \dd x
$$
and 
$$
|u(x_0)|
\lesssim_{n,p}\mathbf P_{\frac{p}{p-1}}^{\frac1{p-1}}(\mu;x_0,r)+
\mint_{B_r(x_0)}|u|\, \dd x.
$$
We refer to \cite{km14b} for details. In analogy with such classical theory, we introduce
\begin{definition}[Intrinsic Lorentzian potentials]
\label{defiintrinsic.potentials}
Let $u:\mathbb R^n\to\mathbb R$ be weakly spacelike, and write
$K_\rho^u(x_0)$ for the Lorentzian balls associated with its graph.
For $\sigma,\vartheta>0$ and $w\in L^1_{\loc}(\mathbb R^n)$,
we define the intrinsic Lorentzian potential
\begin{equation}\label{defi-PL}
\mathbf{PL}_{\sigma,u}^{\vartheta}(w;x_0,r)
:=
\int_0^r
\varrho^\sigma
\left(
\frac{1}{\omega_n\varrho^n}
\int_{K_\varrho^u(x_0)}|w|\,dx
\right)^\vartheta
\frac{d\varrho}{\varrho},
\qquad x_0\in\mathbb R^n,\quad r>0,
\end{equation}
with values in $[0,+\infty]$. The set $K_\varrho^u(x_0)$ is the projection onto
$\mathbb R^n$ of the Lorentzian ball of radius $\varrho$
on the graph of $u$, centered at $(x_0,u(x_0))$.
\end{definition}
When $u$ is constant, this reduces to
$\mathbf P_\sigma^\vartheta$ in \eqref{defi-P}.
If $u\in C^1(\mathbb R^n)$ is strictly spacelike,
Lemma \ref{balls} gives, for each $x_0\in\mathbb R^n$,
a constant $c_0=c_0(x_0,u)\ge1$ such that\footnote{Indeed, Lemma \ref{balls} gives
$
B_\varrho(x_0)\subset K_\varrho^u(x_0)
\subset B_{c_0\varrho}(x_0)$, 
for $ 0<\varrho\le r\le1.$
The first inclusion yields the lower bound. For the upper
bound, the second inclusion and the change of variables
$t=c_0\varrho$ give
$$
\mathbf{PL}_{\sigma,u}^{\vartheta}(w;x_0,r)
\le
c_0^{n\vartheta}
\int_0^r
\varrho^\sigma
\left(
\mint_{B_{c_0\varrho}(x_0)}|w|\,dx
\right)^\vartheta
\frac{d\varrho}{\varrho}=
c_0^{n\vartheta-\sigma}
\mathbf P_\sigma^\vartheta(w;x_0,c_0r).
$$}
\eqn{javier}
$$
\mathbf P_\sigma^\vartheta(w;x_0,r)
\le
\mathbf{PL}_{\sigma,u}^{\vartheta}(w;x_0,r)
\le
c_0^{n\vartheta-\sigma}
\mathbf P_\sigma^\vartheta(w;x_0,c_0r),
\qquad 0<r\le1.
$$

If, in addition,
$\|\mathcal h_u\|_{L^\infty(\mathbb R^n)}<\infty$,
the same comparison holds for every $x_0\in\mathbb R^n$
and every $r>0$, with
$c_0=\|\mathcal h_u\|_{L^\infty(\mathbb R^n)}$.

\begin{theorem}[Intrinsic potential estimate]\label{t3}
Let $n\ge2$ and let
$u\in\mathbb X(\mathbb R^n)$ be the minimizer of
\eqref{bi.en} with  $\mu\in\mathbb Y(\mathbb R^n)$. Moreover, assume that 
$\mu\in L(n,1)(\mathbb R^n)$ if $n\ge3$ and $\mu\in L^2(\log L)^\alpha(\mathbb R^2)$
for some $\alpha>2$ when $ n=2$.
Then the intrinsic potential estimate
\begin{align}\label{int.pot}
\mathcal h_u(x_0)^{-\gamma}
&\ge
\mint_{K_r^u(x_0)}
\mathcal h_u^{-(\gamma+1)}\dx
-2\mathbf{PL}_{2,u}^{1}(\mu^2;x_0,r)
-\mathbf{PL}_{1,u}^{1}(\mu;x_0,r)
\end{align}
holds for some $\gamma\in (0, 1/n)$\footnote{The constant $\gamma\equiv\gamma(n)$ originates in
Proposition \ref{p31}, and the same choice is used throughout
the paper. A possible choice is $\gamma=\frac1{4n}$.}, every $x_0\in\mathbb R^n$ and every $r>0$,
where the intrinsic Lorentzian potentials are defined
in \eqref{defi-PL}.
\end{theorem}
This result provides an intrinsic potential formulation of the
estimates of Bonheure \& Iacopetti \cite{bi23},
extending them to dimension two and to the borderline classes of
rough data considered here. We refer to Corollary \ref{cor.pot} for a similar estimate using the traditional nonlinear potentials in \rif{defi-P}. 
\subsection{Plan of the paper}
We now give a brief overview of the technical content of the paper, breaking down the material according to sections.

\begin{itemize}

\item In Section \ref{notazioni} we describe the main notation and clarify the main function space settings. In particular, we describe the Dirichlet problem, the function space setting for the global problem, i.e., when $\Omega \equiv \er^n$, and a few preliminaries from nonlinear potential theory, adopting an approach already described in \cite{bm20}. Moreover, we also clarify how the notion of nonuniform ellipticity enters the present setting.

\item In Section \ref{babysec} we derive a very preliminary bound for the size of $u$ in the global case.

\item In Section \ref{recapsec} we gather the basic notions and results on Lorentz-Minkowski spaces we shall employ in the following. These include, among other things, Stokes' theorem on graphs (Lemma \ref{ibp}) and the monotonicity formula in Proposition \ref{prop.mon}. Most importantly, in Lemma \ref{balls} we establish some geometric conditions ruling the relations between Euclidean balls $B(x_0,r)$ and Lorentzian balls $K(x_0,r)$.

\item Section \ref{luce} is instead devoted to clarifying the notions of light segment and light ray, their connection to Optimal Transport theory, and a few basic differentiability properties of functions along light segments.

\item Section \ref{lr.s} contains the proof of the rough anti-peeling theorem, Theorem \ref{al.t}, for local minimizers on arbitrary open sets.

\item Section \ref{capacity} deals with general capacitary criteria to estimate the set of light rays of a general, given weakly spacelike function. The key idea is to establish estimates on the capacity and Hausdorff dimension of the set of light rays defined in \rif{language1} in terms of integrability and differentiability of the boost function $\mathcal h_u$. Such criteria hold whether or not the function in question solves \rif{bi}. It contains the proof of Theorem \ref{intre}. 

\item In Section \ref{potest} we derive estimate \rif{int.pot} for more regular solutions. We also point out an interesting connection with intrinsic potentials developed in the setting of degenerate evolution problems \cite{KM13,km14}.

\item Section \ref{boostinf} establishes an exterior upper bound for the boost function. This property is essential for the subsequent global estimates. We also prove in Proposition \ref{connessione} that classical solutions in the global energy space are minimizers.

\item Section \ref{diridiri} develops the preparatory material for dealing with Dirichlet problems involving rough charges. Using nonlinear correctors and the anti-peeling theorem, we obtain the quantitative boundary spacelikeness estimate of Proposition \ref{corr}. The resulting boundary control and regularity in Corollary \ref{app.cor} provide the tools needed for the subsequent estimates and approximation arguments. An interesting point is that considering unbounded charges leads us to depart from traditional barrier methods \cite{bs82} and to merge them with nonlinear fixed point arguments. This is a technical and heavy section that we recommend skipping on a first reading.

\item In Section \ref{sec.4}, again in the form of a priori estimates for more regular solutions, we derive the necessary arguments to prove Theorem \ref{t1}. The most delicate point is the borderline result displayed in \eqref{hiiigh}. For this, as already mentioned before, we inject a technique developed for nonuniformly elliptic problems in the setting of stochastic homogenization and originally due to Bella \& Schäffner \cite{bs20,bs24}; see Lemma \ref{bslem} and \cite{dkk24}.

\item Section \ref{calisec} contains material of general interest that might find applications in several other places. We indeed give a ``calibrated'' version of the classical global iterations used in elliptic theory and going back to Stampacchia \cite{sta58} and Maz'ya \cite{maz61}. The outcome is a global version of De Giorgi iteration, which, unlike the usual global methods, is able to catch borderline spaces unachievable otherwise; see Lemma \ref{abstract.iteration}. The idea of the proof borrows from nonlinear potential theoretic methods, and, in particular, it gives a global version of some of the ideas explained in \cite{bm20, km94, km12}.

\item In the following Section \ref{caliapp}, we apply this abstract result to prove the global boost bounds of Proposition \ref{boostdopo}, under borderline Lorentz assumptions in dimensions $n\ge3$ and $L^2(\log L)^\alpha$ assumptions in dimension two. We stress that the flexibility of the calibrated iteration in Section \ref{calisec} leads to the neat estimate in \rif{boost.boundi}. 

\item Section \ref{apsec} constructs the approximation scheme. We approximate the charge by smooth, compactly supported data, preserving the required norm bounds and the zero-mean condition in dimension two. We then introduce the corresponding smooth minimizers and identify their limit with the minimizer of the original problem.

\item In Section \ref{apsec2} we apply the a priori estimates to the approximate minimizers and pass to the limit, proving Theorems \ref{t1}, \ref{t2}, and \ref{t222}. This yields weak solvability, Sobolev regularity, and the borderline boost and gradient estimates.

\item Following the previous one, in Section \ref{diriproof} we give the proof of Theorem \ref{dir.borderline}, using the content of Section \ref{diridiri}. 

\item In Section \ref{passa}, Theorem \ref{t3} is proved by passing to the limit in the intrinsic potential estimates of Section \ref{potest}, including the convergence of the Lorentzian integration sets.

\item Section \ref{s7.7} proves Theorems \ref{light.t} and \ref{b.th}. The anti-peeling theorem is used to exclude light segments in the global and bounded-domain settings, with the corresponding consequences for weak solvability.

\item In Section \ref{shine} we then construct the example in Theorem \ref{ex.t}, which exhibits a light segment at the weak Lorentz endpoint and establishes the sharpness of the anti-peeling threshold. 

\item In Section \ref{esempione} we give examples of fat closures of light rays and prove Corollary \ref{brutto}, thereby disproving the conjectures in \cite{bimm24}. 

\item Finally, Appendix \ref{appe} revisits and extends the basic existence, Euler--Lagrange, and Calder\'on--Zygmund results used in the paper, with particular attention to the two-dimensional case.

\end{itemize}

\vspace{.46cm}

{\bf Acknowledgements and use of AI}. 
During the preparation of this work, the authors used LLMs to correct typographical and grammatical errors, check the presentation, double-check the proofs, and generate the code for Figures \ref{fig:onepoint}-\ref{fig.1}. The mathematical ideas and strategy of proofs, belong to the authors and were first described in the first named author's proposal of the ERC project NEW, submitted in October 2024. We indeed acknowledge the support of the European Research Council, through the ERC StG project NEW, nr.~101220121. 

\vspace{.7cm}

\section{Preliminaries, Function spaces setting}\label{notazioni} 
\subsection{Notation}  
In this paper we denote by $c$ a general, finite constant such that $c\geq 1$, which, as usual, may change from line to line. The same will happen with constants playing the same role like $c_*,  \tilde c$ and so on. Relevant dependencies on parameters will be as usual emphasized by putting them in parentheses. By  $\texttt{x} \lesssim \texttt{y}$, with $\texttt{x},\texttt{y}$ being two non-negative real numbers, we mean that  $\texttt{x}\leq c\,  \texttt{y}$ holds for a universal constant $c$, i.e., a constant at most depending on the ambient dimension $n$. In less frequent cases, the constant will depend on a fixed set of parameters that will be clear from the context.  In case we want to emphasize such parameters, when for instance the constant $c$ depends on, say, $\gamma, \sigma$, we shall denote  $\texttt{x} \lesssim_{\gamma, \sigma} \texttt{y}$. We shall write $\texttt{x} \approx_{\gamma, \sigma} \texttt{y}$ when both $\texttt{x} \lesssim_{\gamma, \sigma} \texttt{y}$ and $\texttt{y} \lesssim_{\gamma, \sigma} \texttt{x}$ occur. In particular, we shall write $\texttt{x} \lesssim 1$ provided there exists an absolute constant $c$, i.e., depending at most on $n$, such that $\texttt{x} \leq c$, and $\texttt{x} \approx 1$ when $1/c \leq \texttt{x} \leq c$. A similar meaning occurs when using the notation $\texttt{x} \approx_{\gamma, \sigma}  1$.
When a constant $c$ depends essentially on one parameter, like for instance the ambient dimension $n$, we shall denote $c_n$ rather than $c(n)$.  We shall always abbreviate $0_{\er^n}\equiv 0$. With $x_{0} \in \er^n$ and $r>0$, we denote 
$
B_r(x_{0}):= \{x \in \er^n  :   |x-x_{0}|< r\}\,.
$
We shall omit denoting the center, i.e., abbreviating $B_r \equiv B_r(x_{0})$ when no ambiguity will arise; this will often be the case when various balls in the same context share the same center. When no need to specify center and radius will occur we shall denote a generic  ball of $\er^n$ by $B$. Moreover, with $B$ being a given ball with radius $r$ and $\gamma$ being a positive number, we denote by $\gamma B$ the concentric ball with radius $\gamma r$ and write $B/\gamma \equiv (1/\gamma){B}$.  We shall often denote by $\omega_n:=|B_1(0)|$ the Lebesgue measure of the $n$-dimensional ball in $\er^n$. We shall also denote 
$\mathbb S^{n-1}:= \{x\in \er^n\colon |x|=1\}\subset \er^n$. 
In this paper 
$
(\cdot)_+
$
denotes the positive part of a function, i.e. $(f)_+:=\max\{f,0\}$. Moreover, we shall abbreviate $\logs\, t:= \sqrt{\log t}$ for $t\geq 1$. Let $x,y\in\mathbb R^n$; we recall that the \emph{closed segment} joining $x$ and
$y$ is
$
\overline{xy}
:=
\{(1-t)x+ty\colon t\in[0,1]\}.
$
The segment is called \emph{nondegenerate} if $x\neq y$. Its relative interior is the  open segment
$
(xy)
:=
\{(1-t)x+ty\colon t\in(0,1)\}. 
$ In the rest of the paper the symbol $\langle\cdot, \cdot \rangle $ will be used both for the Euclidean scalar product in $\er^n$ and to express duality pairing in function spaces as already done in \rif{bi.en}. 
With $\mathcal  A \subset \er^{n}$ being a measurable subset such that  $0<|\mathcal A|<\infty$, and $w \colon \mathcal  A \to \er^{k}$, $k\geq 1$, being an integrable map, we denote  its integral average  by
$$(w)_{\mathcal  A}:=\frac{1}{\snr{\mathcal  A}}\int_{\mathcal  A}w\dx:= \mint_{\mathcal  A}w\dx\,.$$
\subsection{Nonuniform ellipticity}\label{ner} Here we discuss the main growth/ellipticity features of the integrand in \eqref{bi.en}. For $z\in \mathbb{R}^{n}$, $\snr{z}\leq  1$, set
$
H(z):=1-\sqrt{1-\snr{z}^{2}}$
so that, when $\snr{z}<  1$ it holds that
\eqn{recalla}
$$
\partial H(z)= \frac{z}{\sqrt{1-\snr{z}^{2}}}, \qquad \quad \mathcal{H}(z):= (1-|z|^2)\mathds I+z\otimes z=(1-\snr{z}^{2})^{3/2}\partial^{2}H(z).
$$
A direct computation grants
\eqn{0.1.1}
$$
\langle\partial^{2}H(z)\xi,\xi\rangle\ge\frac{\snr{\xi}^{2}}{\sqrt{1-\snr{z}^{2}}}\qquad \mbox{and}\qquad \snr{\partial^{2}H(z)}\le \frac{\sqrt{n}}{(1-\snr{z}^{2})^{\frac{3}{2}}},
$$
and 
\eqn{0.1}
$$
\begin{cases}
\displaystyle
\ \langle\mathcal{H}(z)\xi,\xi\rangle\ge (1-\snr{z}^{2})\snr{\xi}^{2}\vspace{1.5mm}\\
\ \snr{\mathcal{H}(z)}\le \sqrt{n},\qquad \quad  \mathcal{H}(z)z=z,
\end{cases}
$$
for any $z\in \mathbb{R}^{n}$ with $\snr{z}<1$, and all $\xi\in \mathbb{R}^{n}$. Bounds \eqref{0.1.1} highlight the nonuniform ellipticity of $H$. In fact, the ellipticity ratio $\mathcal{R}_{\partial H}$ (i.e., the quotient of the highest and lowest eigenvalues) blows up when $|z|\uparrow 1$
\eqn{ratione}
$$
\mathcal{R}_{\partial H}(z)=\frac{1}{1-\snr{z}^{2}}\stackrel{\snr{z}\to 1}{\longrightarrow}\infty.
$$
\begin{remark} {\em In the following we agree to set
\eqn{light}
$$
\frac{1}{\sqrt{1-|z|^2}}=\infty, \qquad \mbox{whenever $|z|=1$.}
$$
In particular, we set $\mathcal h_u(x)=\infty$ whenever $|Du(x)|=1$. Such case will be anyway very seldom considered.}
\end{remark}
\begin{remark}[Notation on derivatives]{\em We shall denote by $Dv=(D_i v)$ the usual gradient operator in $\er^n$, while, given a unit vector $\nu\in \er^n$, we denote by $\partial_\nu v$ the partial derivative of $v$ in direction $\nu$, so that, with  $\{\tx{e}_i\}$ being the standard basis of $\er^n$, $D_i v = \partial_{\tx{e}_i} v\equiv \partial_i v$. We shall sometimes depart from this notation using $\partial$ when considering an integrand as in \rif{recalla}-\rif{0.1}, so that, for instance, we have a notation like $D(H(Dv))= D^2v\partial H(Dv)$ allowing to better use the structural properties of the considered integrand when acting on gradients.}
\end{remark}
\subsection{Weakly spacelike functions and notions of solutions}
We recall the following, basic terminology \cite{bs82}. Let
$
\Omega\subset\mathbb R^n
$
be an open subset  and let
$
w\in W^{1,\infty}_{\loc}(\Omega)$, which we identify with its locally Lipschitz representative. The function $w$ is called
\begin{itemize}
\item \emph{weakly spacelike} if
$ \snr{Dw(x)} \leq 1 $
holds for a.e.\,$x\in\Omega$.
\item \emph{spacelike} if
$
\snr{w(x)-w(y)}
<
\snr{x-y}
$
holds for every distinct $x,y\in\Omega$ such that $\overline{xy}\subset\Omega$.
\item
\emph{strictly spacelike} if
$
w\in C^1(\Omega)
$
and
$
\snr{Dw}<1$ in $\Omega$.
\end{itemize}

\begin{definition}
Let $\Omega\subset\mathbb R^n$ be an open subset and let
$
\mu\in L^1_{\loc}(\Omega).
$
A weakly spacelike function
$
u\in W^{1,\infty}_{\loc}(\Omega)
$
is called a distributional or a weak solution to \eqref{bi} in $\Omega$ if
$
\mathcal h_u  Du
\in L^1_{\loc}(\Omega;\mathbb R^n)
$
and\footnote{This implies, in particular, that $|Du(x)|<1$ for almost every $x\in \Omega$ in accordance with 
\rif{light}.}
\eqn{weshall}
$$
 \int_\Omega
\langle\mathcal h_u  Du,D\varphi\rangle
\dx
=\int_\Omega \mu\varphi\dx
$$
for every
$
\varphi\in C^\infty_0(\Omega).
$
Equivalently,
$
-\mathcal Mu=\mu
$
in $\mathcal D'(\Omega)$.
\end{definition}
In the following we shall often write, instead of \eqref{weshall}
\eqn{weshall2}
$$
\int_\Omega
\langle
\partial H(Du),D\varphi\rangle\dx
=\int_\Omega \mu\varphi\dx. 
$$
\begin{definition}
Let $\Omega\subset\mathbb R^n$ be open and let
$\mu:\Omega\to\mathbb R$ be given. We say that
$
u\in C^2(\Omega)
$
is a \emph{classical solution} to \eqref{bi} if
$|Du|<1$
in $\Omega$ and
$
-\mathcal M u
=
\mu
$
holds pointwise in $\Omega$.
\end{definition}
\subsection{The Dirichlet problem}\label{dirisec} We shall consider two cases: the first is the global case $\Omega=\er^n$. The second is when $\Omega$ is an arbitrary, bounded domain of $\er^n$. When considering this last case we of course have to fix boundary data $u_0$ and to make precise the way the boundary values are attained. For this, following \cite{bs82} we consider 
\begin{definition}\label{defidatobordo}
Let $\Omega\subset\mathbb R^n$ be a bounded domain, let
$
u_0:\partial\Omega\to\mathbb R,
$ 
and let
$
u:\Omega\to\mathbb R
$
be weakly spacelike. We write
$
u=u_0
$
on $\partial\Omega$ if, for every
$
x_0\in\partial\Omega
$
and every open straight line segment
$
\ell\subset\Omega
$
having $x_0$ as an endpoint, one has
\eqn{fbi}
$$
\lim_{\substack{x\to x_0\\ x\in\ell}}u(x)
=
u_0(x_0).
$$
\end{definition}
\begin{remark}
{\em
Assume that $\Omega$ is Lipschitz regular and
$
u\in W^{1,\infty}(\Omega).
$
Then $u$ admits a Lipschitz continuous representative on
$
\bar\Omega.
$
Consequently, if
$
u_0\in C^0(\partial\Omega),
$
condition \rif{fbi} holds at every point
$
x_0\in\partial\Omega
$
if and only if
$
u_0
$
coincides with the Sobolev trace of $u$ on $\partial\Omega$.}
\end{remark}
\noindent For the formulation of the Dirichlet problem we consider a setting essentially equivalent to those in \cite{bs82} and \cite{bimm24}, adapted to the present situation in which the charge may be unbounded. More precisely, we consider boundary data \(u_0\) and charges \(\mu\) satisfying
\eqn{u0u0.bd00}
$$
\begin{cases}
u_{0}\in C^{0}(\partial\Omega)\mbox{ such that }
\snr{u_{0}(x)-u_{0}(y)}<\snr{x-y}
\mbox{ whenever }x\neq y
\mbox{ and }\overline{xy}\subset\overline\Omega,
\\[1mm]
\mbox{$\exists\,\tx{u}_0\in W^{1,\infty}(\Omega)$ with
$\|D\tx{u}_0\|_{L^{\infty}(\Omega)}\leq1$
and $\tx{u}_0=u_0$ on $\partial\Omega$
in the sense of Definition \ref{defidatobordo},}
\\[1mm]
\mu\in L^1(\Omega).
\end{cases}
$$
The corresponding class of competitors is
$$
\mathcal D_0(\Omega)
:=
\left\{
w\in W^{1,\infty}(\Omega):
\|Dw\|_{L^\infty(\Omega)}\le1
\mbox{ and }
w=u_0\mbox{ on }\partial\Omega
\mbox{ in the sense of Definition \ref{defidatobordo}}
\right\}.
$$
By \eqref{u0u0.bd00}, the class \(\mathcal D_0(\Omega)\) is non-empty.
The Dirichlet problem
\eqn{pd2in}
$$
\min_{w\in\mathcal D_0(\Omega)}
\mathcal E_\mu(w;\Omega)
$$
admits a unique minimizer by the classical Direct Methods. The Dirichlet minimizer is also a local minimizer in the sense of Definition \ref{minimilocali}. Indeed, if \(u\in\mathcal D_0(\Omega)\) is the Dirichlet minimizer and
\(v\) is as in the above definition, then the function which agrees with
\(v\) in \(U\) and with \(u\) in \(\Omega\setminus U\) belongs to
\(\mathcal D_0(\Omega)\); hence the local minimality inequality follows
from the global minimality in \(\mathcal D_0(\Omega)\).

\subsection{Function spaces and minimality for problems in $\er^n$}\label{spazifunzionali}
We describe the functional-analytic framework for the global
Born--Infeld functional
\eqn{globale}
$$
\mathcal E_\mu(w)
\equiv
\mathcal E_\mu(w;\mathbb R^n)
:=
\int_{\mathbb R^n}
H(Dw)\dx
-
\langle\mu,w\rangle.
$$
Our purpose is to specify both the admissible class for $w$ and the
class of data for which the duality term is well defined. For $n\ge2$, we denote by
$
\mathcal D^{1,2}(\mathbb R^n)
$
the completion of
$
C^\infty_0(\mathbb R^n)
$
with respect to the norm of the gradient; this is usually called Deny-Lions space \cite{lad64,ff21}. In dimension two, this
completion is naturally understood modulo additive constants. By the
corresponding Hardy inequalities, it can be identified with the
following concrete realization:
\eqn{24.3}
$$
\mathcal D^{1,2}(\mathbb R^n)
:=
\begin{cases}
\displaystyle\left\{
w\in W^{1,2}_{\loc}(\mathbb R^n):\frac{w}{1+|\cdot|}\in L^2(\mathbb R^n),
\quad
Dw\in L^2(\mathbb R^n)
\right\}
& n\ge3
\\[3mm]
\displaystyle
\left\{
w\in W^{1,2}_{\loc}(\mathbb R^2):
\frac{w}
{(1+|\cdot|)\log(e+|\cdot|)}\in L^2(\mathbb R^2),
\quad
Dw\in L^2(\er^2)
\right\}
\big/\er
& n=2.
\end{cases}
$$
The last symbol means that elements of $\mathcal D^{1,2}(\mathbb R^2)$ are identified modulo constants. Note that every equivalence class in
$
\mathcal D^{1,2}(\mathbb R^2)
$
has a unique representative whose average over $B_1(0)$ vanishes.
Accordingly, we introduce the normalized realization
$$
\DD
:=
\begin{cases}
\displaystyle
\mathcal D^{1,2}(\mathbb R^n)
& \mbox{if} \ \ n\ge3
\\[2mm]
\displaystyle
\left\{w\in W^{1,2}_{\loc}(\mathbb R^2):Dw\in L^2(\mathbb R^2),\ 
(w)_{B_1(0)}=0\right\}
& \mbox{if} \ \ n=2.
\end{cases}
$$
The space $\DD$ is a Hilbert space with scalar product and norm given by 
$$
(w_1,w_2)_{\DD}
:=\int_{\mathbb R^n}
\langle Dw_1,Dw_2\rangle\dx, \qquad 
\|w\|_{\DD}
:=
\|Dw\|_{L^2(\mathbb R^n)},
$$
respectively.  In dimension two, this is implied by the normalization and Poincaré's inequality. The natural admissible class for the Born--Infeld energy is now defined as 
\eqn{XXX}
$$
\mathbb X(\mathbb R^n)
:=
\left\{
w\in\DD:
\|Dw\|_{L^\infty(\mathbb R^n)}
\le1
\right\}.
$$
Notice that $\mathbb X(\mathbb R^n)$ is a closed convex subset of
$\DD$, rather than a vector space. Under the constraint
$
|Dw|\le1,
$
one has
\eqn{elema}
$$
\frac12|Dw|^2
\le
1-\sqrt{1-|Dw|^2}
\le
|Dw|^2.
$$
Consequently, on $\mathbb X(\mathbb R^n)$, the Born--Infeld energy
term is finite if and only if
$
Dw\in L^2(\mathbb R^n).
$
Thus, in every dimension, a natural sufficient assumption
for the global variational problem is $\mu\in\DD^*$.
For the global two-dimensional theory developed below, we shall
additionally require that the datum is integrable and has zero mean.
We introduce the space
\eqn{zeromean}
$$
L_0^1(\mathbb R^2):=\left\{\mu\in L^1(\mathbb R^2):\int_{\mathbb R^2}\mu\dx=0\right\}.
$$
We then define the data space
\eqn{yyy}
$$
\mathbb Y(\mathbb R^n)
:=
\begin{cases}
\displaystyle
\DD^*
=
\mathcal D^{1,2}(\mathbb R^n)^*
& \mbox{if} \ \ n\ge3
\\[2mm]
\displaystyle
\DD^*
\cap L_0^1(\mathbb R^2)& \mbox{if} \ \ n=2,
\end{cases}
$$
which is endowed with the norm
\eqn{yyy.norm}
$$
\|\mu\|_{\mathbb Y(\mathbb R^n)}
:=
\begin{cases}
\displaystyle
\|\mu\|_{\DD^*}
& \mbox{if} \ \ n\ge3
\\[2mm]
\displaystyle
\|\mu\|_{\DD^*}
+
\|\mu\|_{L^1(\mathbb R^2)}
& \mbox{if} \ \ n=2.
\end{cases}
$$
\begin{remark}[Concrete representation]\label{concreto}
{\em
Let
$
\mu\in L^1_{\loc}(\mathbb R^n),
$
and, when $n=2$, assume in addition that
$
\mu\in L^1_0(\mathbb R^2).
$
We say that $\mu$ induces an element of $\DD^*$ if
\eqn{dual.function.norm}
$$
\|\mu\|_{\DD^*}
:=\sup_{\substack{
\varphi\in C^\infty_0(\mathbb R^n)\\
\|D\varphi\|_{L^2(\mathbb R^n)}\le1}}
\left|\int_{\mathbb R^n}
\mu\varphi\dx
\right|<\infty.
$$
When $n\ge3$, this is the usual concrete characterization of the
dual of
$
\DD=\mathcal D^{1,2}(\mathbb R^n),
$
since
$
C^\infty_0(\mathbb R^n)
$
is dense in $\DD$. 
When $n=2$, for every
$
\varphi\in C^\infty_0(\mathbb R^2),
$
set
$
P\varphi := \varphi-(\varphi)_{B_1(0)}.
$
Then
$
P\varphi\in\DDDD
$
and
$
D(P\varphi)=D\varphi.
$
Moreover, since
$
\mu\in L^1_0(\mathbb R^2),
$
one has
\eqn{dual.zero.mean}
$$
\int_{\mathbb R^2}\mu P\varphi\dx
=
\int_{\mathbb R^2}\mu\varphi\dx.
$$
Since
$
\{P\varphi:\varphi\in C^\infty_0(\mathbb R^2)\}
$
is dense in $\DDDD$, condition \eqref{dual.function.norm} is also
precisely the boundedness condition required for the distribution
induced by $\mu$ to extend continuously to the normalized space
$\DDDD$. Thus, in every dimension $n\ge2$, condition
\eqref{dual.function.norm} yields a unique continuous linear
functional on $\DD$, still denoted by $\mu$, satisfying
$$
|\langle\mu,w\rangle|
\le
\|\mu\|_{\DD^*}
\|w\|_{\DD}
\qquad
\mbox{for every }w\in\DD.
$$
The Riesz representation theorem
gives a unique
$
v_\mu\in\DD
$
such that
\eqn{dual.riesz}
$$
\begin{cases}
\displaystyle 
\langle\mu,w\rangle
=
\int_{\mathbb R^n}
\langle Dv_\mu,Dw\rangle\dx
\qquad
\mbox{for every }w\in\DD\\
\|\mu\|_{\DD^*}
=
\|Dv_\mu\|_{L^2(\mathbb R^n)}.
\end{cases}
$$
We call such a $v_{\mu}$ the Riesz representative of $\mu$. 
If $n\ge3$, this implies 
\eqn{dual.ide}
$$
\int_{\mathbb R^n}
\langle Dv_\mu,D\varphi\rangle\dx
=
\int_{\mathbb R^n}
\mu\varphi\dx, \qquad \mbox{for every $\varphi\in C^\infty_0(\mathbb R^n)$}. 
$$
If $n=2$, we instead take
$
w=P\varphi
$
and use \eqref{dual.zero.mean}, thereby 
obtaining the same identity. Consequently, in every dimension
$n\ge2$, it holds that  
$
-\Delta v_\mu=\mu
$
in $\mathcal D'(\mathbb R^n).
$
}
\end{remark}
\begin{definition}[Global minimizers on $\mathbb R^n$]
\label{minimiglobali}
Let $\mu\in\DD^*$.
A function $u\in\mathbb X(\mathbb R^n)$ is called a
\emph{global minimizer} of $\mathcal E_\mu$ if
$
\mathcal E_\mu(u)\le\mathcal E_\mu(v)
$
for every $v\in\mathbb X(\mathbb R^n)$.
\end{definition}

\begin{remark}[From global to local minimality]
\label{globale.locale}{\em 
Let
$\mu\in\DD^*\cap L^1_{\loc}(\mathbb R^n)$,
with the identification in Remark \ref{concreto}, and assume
in addition that $\mu\in L^1_0(\mathbb R^2)$ when $n=2$.
Let $u\in\mathbb X(\mathbb R^n)$ be a global minimizer
in the sense of Definition \ref{minimiglobali}.
For every bounded open subset $U\subset\mathbb R^n$, introduce
$$
\mathfrak X(u;U)
:=
\left\{
w\in W^{1,\infty}(U):
w-u\in W^{1,2}_0(U),
\quad
\|Dw\|_{L^\infty(U)}\le1
\right\}.
$$
Then $\mathcal E_\mu(u;U)\le\mathcal E_\mu(w;U)$ holds 
for every $w\in\mathfrak X(u;U)$.
In particular, $u$ is a local minimizer in the sense
of Definition \ref{minimilocali} with $\Omega =\er^n$. Indeed, given $w\in\mathfrak X(u;U)$, set
$$
w_u:=
\begin{cases}
w&\text{in }U\\
u&\text{in }\mathbb R^n\setminus U
\end{cases}
\qquad
z:=w_u-u.
$$
Define, as in Remark \ref{concreto}, 
$Pz:=z$ if $n\ge3$ and 
$Pz:=z-(z)_{B_1(0)}$ when $n=2$. 
Then $u+Pz\in\mathbb X(\mathbb R^n)$.
By Remark \ref{concreto}, approximating $z$ by smooth
compactly supported functions uniformly and in the
Dirichlet norm gives
$
\langle\mu,Pz\rangle
=
\int_U\mu(w-u)\,dx.
$
Global minimality therefore yields
$$
0
\le
\mathcal E_\mu(u+Pz)-\mathcal E_\mu(u)
=
\int_U
[H(Dw)-H(Du)-\mu(w-u)]\, \dd x=
\mathcal E_\mu(w;U)-\mathcal E_\mu(u;U).
$$}
\end{remark}

\subsection{Nonlinear potentials, Lorentz spaces}\label{tolli} In this section we recall a few function spaces of interest and connect them with the Havin-Maz'ya nonlinear potentials ${\bf P}_{\sigma}^{\vartheta}$ displayed in \rif{defi-P}. We shall rely in particular on the potential theoretic tools developed in \cite[Section 4]{dm23}; see also \cite{dm23,dm25,ddp26}. Let \(U\subset\mathbb R^n\) be an open set, and let
\(\mathfrak s\in(0,\infty)\) and \(\mathfrak{q}\in(0,\infty]\). Given a measurable
function \(w:U\to\mathbb R^k\), we denote its distribution function
by
\(
\mu_w(\lambda):=
|\{x\in U:|w(x)|>\lambda\}|
\)
and the non-increasing rearrangement of \(|w|\) by
\eqn{rearrangia}
$$
w^*(\varrho)
:=
\inf
\left\{
\lambda>0:
\mu_w(\lambda)\le\varrho
\right\}$$ for 
$\varrho\ge0. 
$
If \(\mathfrak{q}<\infty\), the Lorentz quasi-norm is defined by
\eqn{definorma}
$$
\|w\|_{\mathfrak s,\mathfrak{q};U}
:=
\left(
\int_0^\infty
\left(
\varrho^{1/\mathfrak s}w^*(\varrho)
\right)^\mathfrak{q}
\frac{d\varrho}{\varrho}
\right)^{1/\mathfrak{q}}=
\left(
\mathfrak s\int_0^\infty
\left(
\lambda^\mathfrak s
|\{x\in U:|w(x)|>\lambda\}|
\right)^{\mathfrak{q}/\mathfrak s}
\frac{d\lambda}{\lambda}
\right)^{1/\mathfrak{q}}.
$$
If \(\mathfrak{q}=\infty\), we set
$$
\|w\|_{\mathfrak s,\infty;U}
:=
\sup_{\varrho>0}
\varrho^{1/\mathfrak s}w^*(\varrho)
=
\sup_{\lambda>0}
\lambda\,
\mu_w(\lambda)^{1/\mathfrak s}.
$$
The Lorentz space \(L(\mathfrak s,\mathfrak{q})(U;\mathbb R^k)\) consists of all
measurable functions for which the corresponding quasi-norm is
finite; see \cite{o'n68,sw71}. The local version is defined in the usual way, i.e., \(v\in L_{\loc}(\mathfrak s,\mathfrak{q})(U;\mathbb R^k)\) whenever \(v\in L(\mathfrak s,\mathfrak{q})(\tilde U;\mathbb R^k)\) whenever $\tilde U \Subset U$ is an open subset. 
Define
$$
w^{**}(r)
:=
\frac1r\int_0^r w^*(\varrho)\,d\varrho,\quad  r>0
$$
and, for $0<\mathfrak{q} < \infty$
\eqn{normalo}
$$
[w]_{\mathfrak s,\mathfrak{q};U}
:=
\left(
\int_0^\infty
\left(
r^{1/\mathfrak s}w^{**}(r)
\right)^\mathfrak{q}
\frac{\drr}{r}
\right)^{1/\mathfrak{q}},
$$
whereas
$$
[w]_{\mathfrak s,\infty;U}
:=
\sup_{r>0}
r^{1/\mathfrak s}w^{**}(r).
$$
If \(\mathfrak s>1\) and \(0<\mathfrak{q}\le\infty\), then
\eqn{normalo2}
$$
\nr{w}_{\mathfrak s,\mathfrak{q};U}
\approx_{\mathfrak s,\mathfrak{q}}
[w]_{\mathfrak s,\mathfrak{q};U}.
$$
If, in addition, \(1\le\mathfrak{q}\le\infty\), the quantity
\([\cdot]_{\mathfrak s,\mathfrak{q};U}\) is a norm making
\(L(\mathfrak s,\mathfrak{q})(U;\mathbb R^k)\) a Banach space. For $m\in\mathbb N$, $\mathfrak s> 1$, and again 
$1\le\mathfrak{q}\le\infty$, we set
$
W^{m;\mathfrak s,\mathfrak{q}}(U)
:=
\left\{
w\in W^{m,\mathfrak s}(U):
|D^mw|\in L(\mathfrak s,\mathfrak{q})(U)
\right\},
$
endowed with the norm
$$
\|w\|_{W^{m;\mathfrak s,\mathfrak{q}}(U)}
:=\|w\|_{W^{m,\mathfrak s}(U)}+[D^mw]_{\mathfrak s,\mathfrak{q};U}.
$$
Its local
variant is defined in the usual way, and we set
$
W^{1;\mathfrak s,\mathfrak{q}}_0(U;\mathbb R^k)
:=
W^{1;\mathfrak s,\mathfrak{q}}(U;\mathbb R^k)
\cap
W^{1,\mathfrak s}_0(U;\mathbb R^k).
$

We shall also use Orlicz spaces, including Orlicz--Zygmund
spaces, which are useful in borderline situations.
Throughout this paper, we consider finite Young functions
$A:[0,\infty)\to[0,\infty)$, namely convex functions
satisfying
$
A(0)=0$ and $
\lim_{t\to\infty} A(t)/t=\infty.
$
For an open set $U\subset\mathbb R^n$, the Orlicz space
$L^A(U)$ consists of the measurable functions
$w:U\to\mathbb R$ for which the Luxemburg norm
\eqn{luxi}
$$
\nr{w}_{L^A(U)}
:=
\inf\left\{
\lambda>0:
\int_U
A\left(\frac{\snr{w}}{\lambda}\right)\dx
\le1
\right\}
$$
is finite. In particular, for $\alpha>0$, we set
\eqn{lllog}
$$
L^2(\log L)^\alpha(U):=L^{A_\alpha}(U),
\qquad
A_\alpha(t):=t^2\log^\alpha(e+t),
\quad t\ge0,
$$
with the Luxemburg norm defined in \rif{luxi}.
Since $A_\alpha(t)\ge t^2$, we have
$
\nr{w}_{L^2(U)}
\le
\nr{w}_{L^2(\log L)^\alpha(U)}.
$
Moreover, whenever $0<|U|<\infty$, the embedding
\eqn{immergilog}
$$
\nr{w}_{L(2,1)(U)}
\le
c\,\nr{w}_{L^2(\log L)^\alpha(U)},
\qquad \alpha>1,
$$
holds with $c$ depending only on $|U|$ and $\alpha$. We shall also use the rearrangement estimate
\cite[Theorem~8.8(i)]{opicpick99}
\eqn{orlicz.rearr}
$$
\int_0^\infty
(w^*(s))^2
\log^\alpha\left(e+\frac1s\right)\ds
\lesssim_\alpha
\nr{w}_{L^2(\log L)^\alpha(\mathbb R^n)}^2,
$$
valid for every $\alpha>0$ and
$w\in L^2(\log L)^\alpha(\mathbb R^n)$,
where $w^*$ is the nonincreasing rearrangement of $w$
defined in \rif{rearrangia}. We next include a lemma allowing to bound ${\bf P}_\sigma^\vartheta$
in terms of Lorentz and Orlicz norms. This is a slight variant of those in 
\cite[Lemma 4.1]{dm23} and
\cite[Section 2]{bm20}. 

\begin{lemma}\label{crit}
Let $n\ge2$ and $r\in(0,1)$.
If $\sigma,\vartheta>0$, $n\vartheta>\sigma$, and
$w\in L^1(\mathbb R^n)$, then
$$
\nr{{\bf P}_\sigma^\vartheta(w;\cdot,r)}
_{L^\infty(\mathbb R^n)}
\lesssim_{n,\vartheta,\sigma}
[w]_{n\vartheta/\sigma,\vartheta;\mathbb R^n}^{\vartheta}.
$$
If $n=2$, $\alpha>2$, and
$w\in L^2(\log L)^\alpha(\mathbb R^2)$, then
$$
\nr{{\bf P}_1^{1/2}(\snr{w}^2;\cdot,r)}
_{L^\infty(\mathbb R^2)}
\lesssim_\alpha
\nr{w}_{L^2(\log L)^\alpha(\mathbb R^2)}.
$$
\end{lemma}

\begin{proof}
The first estimate follows from
\cite[Lemma 4.1]{dm23}.
For completeness, we sketch the proof of the second estimate, a version of which already appears in \cite[Section 2]{bm20}.
Denote by $J$ the quantity in the left-hand side of \rif{orlicz.rearr}. 
By \eqref{orlicz.rearr},
$
J\lesssim_\alpha
\nr{w}_{L^2(\log L)^\alpha(\mathbb R^2)}^2.
$
The Hardy--Littlewood rearrangement inequality, 
the change of variables $s=\pi\varrho^2$ and the use of \rif{orlicz.rearr} give
$$
\begin{aligned}
{\bf P}_1^{1/2}(\snr{w}^2;x,r)
&\le c\int_0^{\pi r^2}
\left(\int_0^s(w^*(t))^2\dt\right)^{1/2}
\frac{\ds}{s}\\
&\le c\int_0^{\pi r^2}\left[
\frac{1}{\log^\alpha(e+1/s)}\int_0^s(w^*(t))^2\log^\alpha\left(e+\frac1t\right)\dt\right]^{1/2}\frac{\ds}{s}
\\
&\le c \int_0^\pi \frac{1}{\log^{\alpha/2}(e+1/s)}\frac{\ds}{s}\nr{w}_{L^2(\log L)^\alpha(\mathbb R^2)}
\le c_\alpha \nr{w}_{L^2(\log L)^\alpha(\mathbb R^2)}. 
\end{aligned}
$$
Note that the last integral appearing in the display is finite because $\alpha>2$.
Taking the supremum over $x\in\mathbb R^2$
completes the proof.
\end{proof}

Finally, we recall the following, classical 
\begin{lemma}[Moser--Trudinger inequality]\label{trudinger.lemma}
Let $n\ge2$. There exist constants
$
\mathfrak c_2\equiv\mathfrak c_2(n)\ge1
$
and
$
\mathfrak c_3\equiv\mathfrak c_3(n)>0
$
such that, for every bounded open set
$
U\subset\mathbb R^n
$
and every
$w\in W^{1,n}_0(U)$
with
$\nr{Dw}_{L^n(U)}>0,$
one has
\eqn{trudinger.scale}
$$
\int_U
\exp\left\{
\mathfrak c_3
\left(\frac{\snr{w}}{\nr{Dw}_{L^n(U)}}
\right)^{\frac{n}{n-1}}\right\}\dx
\le
\mathfrak c_2\snr U.
$$
Moreover, for every ball $B\subset\mathbb R^n$ and every
$w\in W^{1,n}(B)$ with $\nr{Dw}_{L^n(B)}>0$, one has
\eqn{trudinger.mean}
$$
\int_B
\exp\left\{\mathfrak c_4\left(\frac{\snr{w-(w)_B}}{\nr{Dw}_{L^n(B)}}\right)^{\frac{n}{n-1}}\right\}
\dx\le
\mathfrak c_5\snr B,
$$
for constants $\mathfrak c_4, \mathfrak c_5$ again depending on $n$. 
\end{lemma}
For the previous results we refer to \cite{c05,m71,tru67}. 
\section{A baby bound in the global case}\label{babysec}
Here we derive a simple global bound for minimizers of the functional
in \eqref{globale}.
\begin{proposition}\label{boun.p}
Let $u\in \mathbb{X}(\mathbb{R}^{n})$ be the minimizer of \eqref{bi.en} with $\mu\in \DD^{*}$. Then
\eqn{linf}
$$
\begin{cases}
\displaystyle
\ \nr{u}_{L^{\infty}(\mathbb{R}^{n})}\le c\Ui\quad &\mbox{if} \ \ n\ge 3\vspace{1.5mm}\\
\displaystyle
\ \left\|\frac{u}{\log(e+\snr{\ \cdot \ })}\right\|_{L^{\infty}(\mathbb{R}^{2})}\le c\Ui\quad &\mbox{if} \ \ n=2,
\end{cases}
$$
where $c\equiv c(n)$, and $\Ui$ is defined as 
\eqn{mmi}
$$
\Ui
:=
\begin{cases}
\displaystyle
\nr{\mu}_{\mathcal D^{1,2}(\mathbb R^n)^*}^{3/n}+\nr{\mu}_{\mathcal D^{1,2}(\mathbb R^n)^*}^{1/n}
&\mbox{if }n\ge3
\\[3mm]
\displaystyle
\nr{\mu}_{\DD^*}+1
&\mbox{if }n=2.
\end{cases}
$$
Moreover, 
\eqn{coerciva}
$$
\nr{Du}_{L^2(\er^n)}=\nr{u}_{\DD}\leq 2 \|\mu\|_{\DD^*}\,.
$$
\end{proposition}
The above estimate is a consequence of some general properties of the space $\mathbb X(\mathbb R^n)$. 
\begin{lemma}\label{lemmino}
Let $w\in\mathcal D^{1,2}(\mathbb R^n)$ be such that 
$\|Dw\|_{L^\infty(\mathbb R^n)}\le1$. The following holds. 
\begin{itemize}
    \item If $n\ge 3$, then $w\in L^{2^{*}}(\mathbb{R}^{n})$ with $\lim_{\snr{x}\to \infty}w(x)=0$, and 
    \eqn{infsup}
    $$
    \nr{w}_{L^{\infty}(\mathbb{R}^{n})}\le c \nr{Dw}_{L^{2}(\mathbb{R}^{n})}^{\frac{2(n+\mathcal{s})}{n\mathcal{s}}}+c\nr{Dw}_{L^{2}(\mathbb{R}^{n})}^{\frac{2}{\mathcal{s}}}
    $$
    for all $n<\mathcal{s}<\infty$, where $c\equiv c (n, \mathcal{s})$.
\item If $n=2$, let $w$ denote a locally Lipschitz
representative of its class in \eqref{24.3}$_2$. Then
    \eqn{logsup}
    $$
\sup_{x\in \mathbb{R}^{2}}\frac{\snr{w(x)}}{\log(e+\snr{x})} + \sup_{x\in \mathbb{R}^{2}}\frac{\snr{w(x)}}{1+\snr{x}} \le c\nr{Dw}_{L^{2}(\mathbb{R}^{2})} + c\snr{(w)_{B_{1}(0)}}+c,
    $$
    where $c\equiv c(n)$. Moreover, for any $\varepsilon>0$ there exists $r_{\varepsilon}\equiv r_{\varepsilon}(w,\varepsilon)>0$ such that
    \eqn{weightsmall}
    $$
    \sup_{x\in \mathbb{R}^{2}\setminus B_{r_{\varepsilon}}(0)}\frac{\snr{w(x)}}{1+\snr{x}}<\varepsilon.
    $$
\end{itemize}
\end{lemma}
\begin{proof}
For the content of the first bullet see \cite[Lemma 5.2]{bi23}. Let us focus on \eqref{logsup}-\eqref{weightsmall}, which are probably well-known under different assumptions, but we couldn't locate them in the literature. So we report the proof for completeness. 
In the following the denoted constant $c\equiv c_{n}$ depends only on the dimension $n$. Let 
$x\in\mathbb R^{2}$ and $r>0$, then we have
\begin{align*}
\snr{(w)_{B_r(x)}-(w)_{B_{2r}(x)}} &\leq \mint_{B_r(x)}
   \snr{w-(w)_{B_{2r}(x)}}\dy \leq  c
   \left(\mint_{B_{2r}(x)}
   \snr{w-(w)_{B_{2r}(x)}}^2\dy\right)^{1/2}
\end{align*}
and here we have used Jensen and  H\"older's inequalities. By further using Poincaré's inequality and exploiting that $n=2$ we conclude with 
\begin{equation}\label{eq:dyadic-means}
\snr{(w)_{B_r(x)}-(w)_{B_{2r}(x)}}
\leq c\nr{Dw}_{L^{2}(B_{2r}(x))}.
\end{equation}
Choose $k\in\mathbb N\cup\{0\}$ such that
$
2^{k-1}<1+\snr{x}\leq 2^{k}. 
$ With $R:=2^{k}$ we have 
$k+1\leq c\log(e+\snr{x})$
and $
B_R(x)\subset B_{2R}(0).
$
Iterating \eqref{eq:dyadic-means} yields 
\eqn{eq:chain-x}
$$
\begin{cases}
\snr{(w)_{B_1(x)}-(w)_{B_R(x)}}
\leq ck\nr{Dw}_{L^{2}(\mathbb R^{2})}
\\[4pt]
\snr{(w)_{B_{2R}(0)}-(w)_{B_1(0)}}
\leq
c(k+1)\nr{Dw}_{L^2(\mathbb R^2)}.
\end{cases}
$$
Moreover, since $B_R(x)\subset B_{2R}(0)$, arguing as for \rif{eq:dyadic-means} we find 
\eqn{eq:overlapping-balls}
$$
\snr{(w)_{B_R(x)}-(w)_{B_{2R}(0)}}
 \leq c\nr{Dw}_{L^{2}(\mathbb R^{2})}.
$$
Combining \eqref{eq:chain-x} and \eqref{eq:overlapping-balls}, we obtain
\eqn{eq:mean-log-growth}
$$
\snr{(w)_{B_1(x)}-(w)_{B_1(0)}}
\leq
c\log(e+\snr{x})\nr{Dw}_{L^2(\mathbb R^2)}.
$$
On the other hand, since \(w\) is \(1\)-Lipschitz,
$
\snr{w(x)-(w)_{B_1(x)}}\leq1.
$
This and \eqref{eq:mean-log-growth} imply that
\begin{equation}\label{eq:pointwise-log-growth}
\snr{w(x)}
\leq
1+\snr{(w)_{B_1(0)}}
+c\log(e+\snr{x})\nr{Dw}_{L^2(\mathbb R^2)}.
\end{equation}
Since \(\log(e+\snr{x})\geq\log2\), dividing
\eqref{eq:pointwise-log-growth} by \(\log(e+\snr{x})\), and recalling
that \(x\in\mathbb R^2\) is arbitrary, proves the estimate for the
first quantity in \eqref{logsup}. We next estimate the second, which follows by observing that the Lipschitz continuity 
of $w$ gives 
$\snr{w(x)}
\leq \snr{w(0)}+\snr{x}$ and $\snr{w(0)}
\leq
\snr{(w)_{B_1(0)}}+1
$. The property in \rif{weightsmall} now easily follows from \rif{eq:pointwise-log-growth}. 
\end{proof}
\begin{proof}[Proof of Proposition \ref{boun.p}] 
By \eqref{infsup}, with the choice $\mathcal s=2n$, and by
\eqref{logsup}, recalling that $(u)_{B_1(0)}=0$ when $n=2$, it is
enough to prove \eqref{coerciva}.
Using that
$ \mathcal E_{\mu}(u)\le\mathcal E_{\mu}(0)=0$, we find
$$
\begin{aligned}
\frac12 \|Du\|_{L^2(\mathbb R^n)}^2\le
\int_{\mathbb R^n}H(Du)\dx
&\le \langle\mu,u\rangle\\&\le
\|\mu\|_{\DD^*}\|u\|_{\DD}=\|\mu\|_{\DD^*}\|Du\|_{L^2(\mathbb R^n)},
\end{aligned}
$$
from which \rif{coerciva} follows immediately. 
\end{proof}

\section{Recap on Lorentzian geometry and auxiliary results}\label{recapsec}
In this section we recall some basic facts in Lorentzian geometry and state a few auxiliary results. Our main references in this respect are \cite{bs82,bi23,bimm24}, from which we adopt most of the notations and conventions. By $\mathbb{L}^{n+1}$ we denote the $(n+1)$-dimensional Lorentz-Minkowski space, that is $\mathbb{R}^{n+1}$ equipped with the symmetric bilinear form
\eqn{bil}
$$
\langle y,z\rangle_{\mathbb{L}^{n+1}}:=\sum_{i=1}^{n}y_{i}z_{i}-y_{n+1}z_{n+1}.
$$
The form in \eqref{bil} is nondegenerate with index one, and, naturally 
$$
\snr{y}_{\mathbb{L}^{n+1}}:=\sqrt{\snr{\langle y,y\rangle_{\mathbb{L}^{n+1}}}} \qquad \mbox{for every $y\in \mathbb{L}^{n+1}$}. 
$$
Notice that this quantity is not a norm, since it vanishes on nonzero
lightlike vectors.
A vector $y\in \mathbb{L}^{n+1}$ is spacelike if either $\langle y,y\rangle_{\mathbb{L}^{n+1}}>0$ or $y=0$, timelike if $\langle y,y\rangle_{\mathbb{L}^{n+1}}<0$ or lightlike if $\langle y,y\rangle_{\mathbb{L}^{n+1}}=0$ and $y\not =0$. The induced metric on vector subspaces is defined in the standard way. If $\{\tx{e}_{1},\cdots,\tx{e}_{n+1}\}$ denotes the standard basis of $\mathbb{L}^{n+1}$, we have 
$$
\begin{cases}
    \displaystyle
    \ \langle \tx{e}_{i},\tx{e}_{j}\rangle_{\mathbb{L}^{n+1}}=0\quad &\mbox{for all} \ \ i,j\in \{1,\cdots,n+1\}, \ \ i\not =j\vspace{1.5mm}\\
    \displaystyle
    \ \langle \tx{e}_{i},\tx{e}_{i}\rangle_{\mathbb{L}^{n+1}}=1\quad &\mbox{for all} \ \ i\in \{1,\cdots,n\} \vspace{1.5mm}\\
    \displaystyle
    \ \langle \tx{e}_{n+1},\tx{e}_{n+1}\rangle_{\mathbb{L}^{n+1}}=-1.
\end{cases}
$$
A causal vector $x\in \mathbb{L}^{n+1}$ is future directed if $\langle x,\tx{e}_{n+1}\rangle_{\mathbb{L}^{n+1}}<0$, and past directed if $\langle x,\tx{e}_{n+1}\rangle_{\mathbb{L}^{n+1}}>0$. 
\begin{definition}
A \(C^1\) hypersurface \(\mathfrak M\subset\el^{n+1}\) is said to be
spacelike if, for every \(z\in\mathfrak M\), the restriction of the
Lorentzian metric \(\langle\cdot,\cdot\rangle_{\el^{n+1}}\) to
\(\ett_z\mathfrak M\) is positive definite.
\end{definition}
\begin{definition}
If $u\in C^1(\Omega)$, then its vertical graph is a spacelike
hypersurface if and only if $u$ is strictly spacelike in the sense
introduced in Section \ref{notazioni}.
\end{definition}
Needless to say, a $C^{1}$-regular strictly spacelike function is such that for every compact $K\subset \Omega$ there exists $\delta \in (0,1)$ such that $\nr{Du(x)}<1-\delta$ for every $x\in K$.

Let $\Omega\subset \mathbb{R}^{n}$ be a domain, and $u\in C^{1}(\Omega)$ a function. The related vertical graph is defined as $\mathfrak{M}:=\{(x,u(x))\in \mathbb{L}^{n+1}\colon x\in \Omega\}$, so the natural parametrization is given by $\Phi\colon \Omega\to \mathbb{L}^{n+1}$ such that $\Phi(x)=(x,u(x))$. In this respect, through the parametrization $\Phi$, and with a slight abuse of notation, when no confusion shall arise we shall identify functions and vector fields on $\mathfrak M$ with their pull-backs to $\Omega$ and points $z\equiv (x, u(x))\in \mathfrak M$ with points $x\in \Omega$. For instance, if \(v\) is defined on \(\mathfrak M\), its pull-back is
$
\widetilde v:=v\circ\Phi.
$
With the above identification, we shall simply write \(v\equiv\widetilde v\),
namely \(v(z)=\widetilde v(x)\) whenever \(z=\Phi(x)=(x,u(x))\). Conversely, for a function $f$ defined on $\Omega$,
we denote its vertical extension to $\Omega\times\mathbb R$
by the same symbol, setting $f(x,t):=f(x)$. Given a point $z=(y_1,\ldots,y_{n+1})\in\el^{n+1}$,
we denote its projection onto $\er^n$ by
$\pi(z):=(y_1,\ldots,y_n)$.  Accordingly, the vectors 
\eqn{recallX}
$$\mathbf{X}_{i}(z)\equiv \mathbf{X}_{i}(x):=D_{i}\Phi(x)=\partial_{y_i}+D_{i}u(x)\partial_{y_{n+1}}\qquad i\in \{1,\cdots,n\}$$ are a basis for the tangent space $\ett_{z}\mathfrak{M}$, for $z=(x,u(x))$\footnote{In the following, we shall use \(\{x_1,\dots,x_n\}\) to denote coordinates
in \(\mathbb R^n\), used for the parametrization \(\Phi\), and
\(\{y_1,\dots,y_{n+1}\}\) to denote the ambient coordinates in
\(\mathbb L^{n+1}\).}. In the case $u$ is strictly spacelike we shall denote 
\eqn{boosty}
$$
\mathcal{h}_u(z)\equiv \mathcal{h}_u(x):= \frac{1}{\sqrt{1-|Du(x)|^2}}
$$ for $z=(x,u(x))$, and the graphical measure will be indicated as 
\eqn{area}
$$\d A:=\frac{\dx}{\mathcal{h}_u}= \sqrt{1-|Du|^2}\dx, \qquad A(K):=\int_{K} \d A $$
for every $K\subset \mathfrak M$ such that $\Phi^{-1}(K)$ is measurable. 
The induced metric on $\mathfrak{M}$ is given by 
$$
g\equiv g_z=(g_{ij})_{i,j\in \{1,\cdots,n\}}, 
\qquad g_{ij}=\langle \mathbf{X}_{i}(z),\mathbf{X}_{j}(z)\rangle_{\mathbb{L}^{n+1}}=\delta_{ij}-D_{i}u(x)D_{j}u(x),
$$
that is, 
$
g_z= \mathds I - Du(x)\otimes Du(x)
$ when $z\equiv (x, u(x))$, so that 
$g_z(\xi, \xi)= |\xi|^2-\langle Du,\xi\rangle^2$ for every $\xi \in \er^n$. It follows that $\mathfrak{M}$ is spacelike if and only if $\snr{Du(x)}<1$ for all $x\in \Omega$. 
Let $u\in C^{2}(\Omega)$ be a strictly spacelike function and $\mathfrak{M}$ its vertical graph. The future-oriented Gauss map is $\vv=(\vv_{i})_{1\leq i\leq n+1}$ with 
$$
\vv_i:=
\begin{cases}
\mathcal h_u D_i u&\mbox{if } i\in\{1,\ldots,n\}
\\[1.5mm]
\mathcal h_u
&\mbox{if } i=n+1.
\end{cases}
$$
Keep in mind the notation and conventions adopted in Section \ref{ner}. 
Note that, obviously, 
$
\vv=\mathcal{h}_{u}(Du,1)$ and $\langle \vv,\vv\rangle_{\mathbb{L}^{n+1}}=-1$. 
The inverse of the metric matrix $g$ is 
$$g^{-1}\equiv g^{-1}_z=(g^{ij})_{i,j\in\{1,\cdots,n\}}, \qquad g^{ij}=\delta_{ij}+\vv_i\vv_j=\delta_{ij}+\frac{D_i u\,D_{j}u}{1-|Du|^2} .$$ 
For \(v\in C^1(\mathfrak M)\), we denote by
\(\nabla^{\mathfrak M}_g v\) its tangential gradient with respect to the
induced metric \(g\), i.e., $g_z ((\nabla_g^{\mathfrak M}v)(z),Y)=Yv$ 
for every $Y\in \ett_z\mathfrak M$. The operators \(\delta_\alpha\), \(\alpha=1,\dots,n+1\),
are defined as the ambient components of this vector field, namely
\[
\nabla^{\mathfrak M}_g v
=
\sum_{i=1}^{n+1}\delta_i v\,\partial_{y_i}.
\] 
Equivalently, if \(\widetilde v=v\circ\Phi\) and $z=(x, u(x))$, then
$$
\delta_i v(z):=
\begin{cases}
\displaystyle
\sum_{j=1}^n g^{ij}D_j\widetilde v(x)
&\mbox{if } i\in\{1,\ldots,n\}
\\[1.5mm]
\displaystyle
\mathcal h_u\sum_{j=1}^n\vv_jD_j\widetilde v(x)
&\mbox{if } i=n+1.
\end{cases}
$$
\noindent A related useful tool is an integration-by-parts formula, a consequence of Stokes' theorem.
\begin{lemma}\label{ibp}
Let \(u\in C^2(\Omega)\) be strictly spacelike and let \(\mathfrak M\)
be its vertical graph. Let \(v\in C_{0}^1(\mathfrak M)\) and
\(w\in C^1(\mathfrak M)\). Let $\tx H_u\in C(\Omega)$ be defined by
 $$
 \tx{H}_{u}:=-\mathcal{h}_{u}\sum_{i,j=1}^{n}g^{ij}D_{ij}u.
 $$
 It holds
 $$
 \begin{cases}
     \displaystyle
     \ \int_{\mathfrak{M}}v\delta_{n+1}w\d A=\int_{\mathfrak{M}}\mathcal{h}_{u}vw \tx{H}_{u}\d A-\int_{\mathfrak{M}}w\delta_{n+1}v\d A\vspace{1.5mm}\\
     \displaystyle
     \ \int_{\mathfrak{M}}v\delta_{i}w\d A=\int_{\mathfrak{M}}\vv_{i}vw \tx{H}_{u}\d A-\int_{\mathfrak{M}}w\delta_{i}v\d A,
 \end{cases}
 $$
 for all $i\in \{1,\cdots,n\}$.
\end{lemma}

In the following we are interested in functions $w$ defined in a neighbourhood of $\mathfrak M$ that are independent of the last variable $y_{n+1}$. These are naturally identified with functions \(w=w(x)\) on \(\Omega\) via the vertical
extension \(w(x,t)=w(x)\) to a neighbourhood of \(\mathfrak M\). Therefore, as \(\partial_{y_{n+1}}w=0\), in such cases
$
\mathbf X_iw=\partial_{y_i}w=D_iw.
$
\noindent Next, for a weakly spacelike function $u$, $x_{0},x_{1},x_{2}\in \mathbb{R}^{n}$, $r\in (0,\infty)$, the Lorentz distance is defined as 
$$\ell_u(x_{1},x_{2}):=\sqrt{\snr{x_{1}-x_{2}}^{2}-\snr{u(x_{1})-u(x_{2})}^{2}},$$
 the Lorentz ball of radius $r$ centered at $(x_{0},u(x_{0}))$ is 
 \eqn{loreball}
 $$L_{r}^u(x_{0}):=\{(x,u(x))\in \mathfrak{M}\colon \ell_u(x,x_{0})<r\},$$ 
 while its projection on $\mathbb{R}^{n}$ is given by 
 \eqn{loreballp}
 $$K_{r}^u(x_{0}):=\{x\in \mathbb{R}^{n}\colon \ell_u(x,x_{0})<r\}.$$ 
 Note that, when the function $u$ will be fixed or clear from the context, we shall abbreviate
 \eqn{abbreviation}
 $$
 \ell(x_{1},x_{2})\equiv \ell_u(x_{1},x_{2}), \quad L_{r}(x_{0})\equiv L_{r}^u(x_{0}), \quad K_{r}(x_{0})\equiv K_{r}^u(x_{0}). 
 $$
 Given $z_0=(x_0,u(x_0))$ and $z=(x,u(x))$, we set
$\ell_{z_0}(z)\equiv \ell_{z_0,u}(z):=\ell_u(x,x_0)$.
When no confusion arises, we also write
$\ell_{z_0}(z)\equiv\ell_{x_0}(x)\equiv \ell_{x_0,u}(x)$, or simply $\ell(x)$. Given two nonempty sets $\Omega_{1},\Omega_{2}\subset \mathbb{R}^{n}$, we indicate by 
 $$\begin{cases}
 \dist(\Omega_{1},\Omega_{2}):=\inf_{x_{1}\in \Omega_{1}, x_{2}\in \Omega_{2}}\snr{x_{1}-x_{2}}\\
 \dist_{L}(\Omega_{1},\Omega_{2}):=\inf_{x_{1}\in \Omega_{1},x_{2}\in \Omega_{2}}\ell(x_{1},x_{2})
 \end{cases}
 $$
the Euclidean and the Lorentzian set-distance between $\Omega_{1}$ and $\Omega_{2}$, respectively. The next lemma reports a Lorentzian version of Federer's coarea formula, \cite[(2.14)]{bs82}.
\begin{lemma}\label{coa}
Let $u\in C^{2}(\mathbb{R}^{n})$ be a strictly spacelike function,\footnote{For strictly spacelike functions, $K_{s}(x_{0})$ is always bounded for all $s\in (0,\infty)$, cf. Lemma \ref{balls} \textnormal{(h$_1$)}.} $\mathfrak{M}:=\textnormal{graph}(u)$, $x_{0}\in \mathbb{R}^{n}$, $w\in C^1(\mathfrak{M})$, 
$$
f(t):= \int_{L_{t}^u(x_{0})}w\d A \quad \mbox{for every $t>0$.}
$$
Then
$$
\deris f(s)=\int_{\partial L_{s}^u(x_{0})}\frac{w}{\snr{\delta \ell_{x_0,u}}_{\mathbb{L}^{n+1}}}\d\sigma
$$
holds whenever $s>0$, where $\d\sigma$ denotes the surface measure induced by $g$
on $\partial L_s^u(x_0)$.
\end{lemma}
\noindent We record below an integral inequality following from \cite[(3.5)--(3.8)]{bi23}, in the spirit of the monotonicity formula in \cite[Section~2]{bs82}. Although \cite[Theorem~1.1]{bi23} is stated for \(n\ge3\), the computations leading to these formulas remain valid for \(n=2\). In particular, the constants in the following proposition can be chosen as \(\gamma=1/(4n)\) and \(c_\star=\gamma/2\). For this, see Proposition \ref{p31} later on. 

\begin{proposition}\label{prop.mon}
Let $u\in C^{3}(\mathbb{R}^{n})$ be a classical solution to \eqref{bi} (and therefore strictly spacelike) with a smooth charge $\mu\in C^{\infty}(\mathbb{R}^{n})$ with compact support, $r\in (0,\infty)$ a number, and assume that $(x_{0},u(x_{0}))=(0,0)$. With $\gamma, s>0$ we define the quantities 
$$
\mathbb A(s) :=\frac{1}{s^{n}}\int_{L_{s}}\mathcal{h}_{u}^{-\gamma}\d A , \qquad \mathbb B(s):= \int_{L_{s}}\mathcal{h}_{u}^{-\gamma}\ell^{-n-2}\snr{\langle\Phi,\vv\rangle_{\mathbb{L}^{n+1}}}^{2}\d A, 
$$
$$
\mathbb E(s):= \frac{1}{s^{n+1}}\int_{L_{s}}\frac{1}{2}(s^{2}-\ell^{2})\mathcal{h}_{u}^{2-\gamma}\left[\sum_{i,j=1}^{n}(D_{ij}u)^{2}+\sum_{j=1}^{n}\left(\sum_{i=1}^{n}\vv_{i}D_{ij}u\right)^{2}\right]\d A,
$$
$$
\mathbb E_\mu(s):=\frac{1}{s^{n+1}}\int_{L_{s}}\frac{1}{2}(s^{2}-\ell^{2})\left(\frac{1}{4}\mathcal{h}_{u}^{-\gamma}\mu^{2}+\gamma\delta_{n+1}(\mathcal{h}_{u}^{-(\gamma+1)}\mu)\right)\d A
$$
and 
$$
\mathbb F_\mu(s):=
\frac{1}{s^{n+1}}\int_{L_{s}}\mathcal{h}_{u}^{-\gamma}\mu\langle\Phi,\vv\rangle_{\mathbb{L}^{n+1}}\d A, 
$$
where $\ell(\Phi(x)):=\ell_u(x,0)$ and $L_{s}:=\{(x,u(x))\in \mathfrak{M}\colon \ell(x,0)<s\}$. 
There exists $\gamma\equiv \gamma(n)\in (0,1/n)$ such that
\eqn{monotonia}
$$
\deris [\mathbb  A(s)+\mathbb  B(s)] \leq - c_{\star}\mathbb  E(s) + \mathbb  E_\mu(s) - \mathbb  F_\mu(s)
$$holds for all $s\in (0,r)$, with $c_{\star}\equiv c_{\star}(n)>0$. In particular, when $\mu\equiv 0$ the function $s \mapsto \mathbb  A(s)+\mathbb  B(s)$ is non-increasing.
\end{proposition}
\begin{remark}\label{rem0}
    \emph{If $u\in C^{3}(\mathbb{R}^{n})$ is a classical solution to \eqref{bi} with charge $\mu\in C^{\infty}_{0}(\mathbb{R}^{n})$ and $x_{0}\in \mathbb{R}^{n}$ is a point, the translated map $u_{0}(x):=u(x_{0}+x)-u(x_{0})$ is still a solution to \eqref{bi} with datum $\mu_{0}(x):=\mu(x_{0}+x)$, so there is no loss of generality in assuming that $x_{0}=0$ and $u(0)=0$. In Proposition \ref{prop.mon} note that 
    $$
    \langle\Phi,\vv\rangle_{\mathbb{L}^{n+1}}= - \mathcal h_u (u-\langle Du, x\rangle).
    $$}
\end{remark}
\noindent Next, we highlight a basic relation between the projection on $\mathbb{R}^{n}$ of Lorentz balls and the standard Euclidean balls.
\begin{lemma}\label{balls} Let $u\in C^{1}(\mathbb{R}^{n})$ be strictly spacelike, $s\in (0,\infty)$ be a number and $x_{0}\in \mathbb{R}^{n}$ be any point. Then the following holds.
\begin{itemize}
    \item[(\textnormal{h}$_1$)] For all balls $B_{s}(x_{0})\subset \mathbb{R}^{n}$, it is $B_{s}(x_{0})\subset K_{s}(x_{0})$. Moreover there exists $r_{0}\equiv r_{0}(x_{0},u,s)>0$ such that $K_{s}(x_{0})\subset B_{r_{0}}(x_{0})$. If in addition $s\in (0,1]$, then $K_{s}(x_{0})\subset B_{c_{0}s}(x_{0})$ for some positive constant $c_{0}\equiv c_{0}(x_{0},u)$. Finally, if 
    \eqn{h-}
    $$\mathcal{h}_{-}:=\inf_{\mathbb{R}^{n}}\mathcal{h}_{u}^{-1}=\inf_{\mathbb{R}^{n}}\sqrt{1-|Du(x)|^2}>0,$$ then $K_{s}(x_{0})\subset B_{s/\mathcal{h}_{-}}(x_{0})$.
    \item[(\textnormal{h}$_2$)] If $u\in L^{\infty}(\mathbb{R}^{n})\cap C^{1}(\mathbb{R}^{n})$, then $K_{s}(x_{0})\subset B_{\ti{s}}(x_{0})$ with $\ti{s}:=(s^{2}+4\nr{u}_{L^{\infty}(\mathbb{R}^{n})}^{2})^{1/2}$.
    \item[(\textnormal{h}$_3$)] In dimension $n=2$, if $u\in C^{1}(\mathbb{R}^{2})\cap \mathbb{X}$ and $K\subset \mathbb{R}^{2}$ is a compact set and $s\in (0,\infty)$ a number, there exists a radius $r_{*}\equiv r_{*}(u,K,s)>1$ such that $$K\Subset B_{r_{*}}(0) \quad \mbox{and}\quad \dist_{L}(K, \mathbb{R}^{2}\setminus B_{r_{*}}(0))\ge 10\max\{s,1\}\,.$$ In other terms, for any compact set $K\subset \mathbb{R}^{2}$, radius $s\in (0,\infty)$, we can find a ball $B_{r_{*}}(0)\subset \mathbb{R}^{2}$ with radius $r_{*}\equiv r_{*}(K,u,s)>1$ such that $K\cap K_{s}(y)=\varnothing$ for all $y\in \mathbb{R}^{2}\setminus B_{r_{*}}(0)$. 
\end{itemize}
\end{lemma}
\begin{proof} We start with the proof of (\textnormal{h}$_1$). Inclusion $B_{s}(x_{0})\subset K_{s}(x_{0})$ is trivial by the very definition of Lorentz distance. Since $|Du(x_0)|<1$ and $Du$ is continuous, there exist
$\delta_0\in(0,1)$ and $\rr_0>0$, depending on
$1-|Du(x_0)|$ and on the local modulus of continuity of $Du$,
such that
$\|Du\|_{L^\infty(B_{\rr_0}(x_0))}\le\delta_0$. We then bound
\begin{eqnarray*}
\snr{x-x_{0}}^{2}&\le&\ell(x,x_{0})^{2}+\snr{u(x)-u(x_{0})}^{2}\nonumber \\
&\stackrel{x\in K_{s}(x_{0})}{\le}& s^{2} +\snr{x-x_{0}}^{2}\int_{0}^{\min\left\{1,\frac{\rr_{0}}{\snr{x-x_{0}}}\right\}}\snr{Du(x_{0}+t(x-x_{0}))}^{2}\dt\nonumber \\
&&+\snr{x-x_{0}}^{2}\int^{1}_{\min\left\{1,\frac{\rr_{0}}{\snr{x-x_{0}}}\right\}}\snr{Du(x_{0}+t(x-x_{0}))}^{2}\dt\nonumber\\
&\le&s^{2}+\snr{x-x_{0}}^{2}\left(\delta_{0}^{2}\min\left\{1,\frac{\rr_{0}}{\snr{x-x_{0}}}\right\}+1-\min\left\{1,\frac{\rr_{0}}{\snr{x-x_{0}}}\right\}\right),
\end{eqnarray*}
so, reabsorbing terms, we obtain that 
$$\snr{x-x_{0}}^{2}(1-\delta_{0}^{2})\min \left\{1,\frac{\rr_{0}}{\snr{x-x_{0}}}\right\}\le s^{2}.$$ We set 
$$
r_{0}:=\max\left\{\frac{s}{\sqrt{1-\delta_{0}^{2}}},\frac{s^{2}}{\rr_{0}(1-\delta_{0}^{2})}\right\},
$$
and deduce from the above displays that $K_{s}(x_{0})\subset B_{r_{0}}(x_{0})$. In particular, if $s\in (0,1]$ by definition we can further estimate 
$$r_0\le
s\max\left\{
(1-\delta_0^2)^{-1/2},
\frac{1}{\rr_0(1-\delta_0^2)}
\right\}
=:c_0(x_0,u)s.$$
Concerning the final claim in (\textnormal{h}$_1$), we notice that $\mathcal{h}_{-}>0$ implies that $\nr{Du}_{L^{\infty}(\mathbb{R}^{n})}^{2}\le 1-\mathcal{h}_{-}^{2}$, thus 
\begin{flalign*}
\snr{x-x_{0}}^{2}&\le \ell(x,x_{0})^{2}+\snr{u(x)-u(x_{0})}^{2}\nonumber \\
&\le s^{2}+\nr{Du}_{L^{\infty}(\mathbb{R}^{n})}^{2}\snr{x-x_{0}}^{2}\le s^{2}+(1-\mathcal{h}_{-}^{2})\snr{x-x_{0}}^{2},
\end{flalign*}
and $\snr{x-x_{0}}\le s/\mathcal{h}_{-}$.
For assertion \textnormal{(h$_2$)}, see
\cite[Lemma 2.8(i)]{bi23}. We now turn to the proof of
\textnormal{(h$_3$)}. The idea is that Lorentz balls
centered far away from a given compact set do not
``grow'' fast enough to reach it. Being $K$ compact, there exists $\bar{r}\equiv \bar{r}(K)$ such that $K\Subset B_{\bar{r}}(0)$. Moreover by \eqref{weightsmall} applied with $\varepsilon=1/2^{4}$, we can find a radius $r_{2^{-4}}>0$ such that
\eqn{24.6}
$$
\sup_{x\in \mathbb{R}^{2}\setminus B_{r_{2^{-4}}}(0)}\frac{\snr{u(x)}}{1+\snr{x}}<\frac{1}{2^{4}}.
$$
Now set
\eqn{24.7}
$$
r_{*}:=2^{10}\max\{s,1\}\max\left\{\bar{r}, r_{2^{-4}},2\max\{\bar{r},1\}\left\|\frac{u}{1+\snr{\ \cdot \ }}\right\|_{L^{\infty}(\mathbb{R}^{2})},1\right\}.
$$
For $x\in K$ and
$y\in\mathbb R^2\setminus B_{r_*}(0)$, we have $$
\snr{x-y}\ge \snr{y}-\snr{x}\ge 2^{-1}\snr{y}+2^{-1}r_{*}(1-2^{-9}\max\{s,1\}^{-1})\ge 2^{-1}\snr{y}+2^{-2}r_{*},
$$
and
\begin{eqnarray*}
\snr{u(x)-u(y)}&\le&  \left\|\frac{u}{1+\snr{\ \cdot \ }}\right\|_{L^{\infty}(\mathbb{R}^{2})}(1+\snr{x})+\left(\sup_{x\in \mathbb{R}^{2}\setminus B_{r_{*}}(0)}\frac{\snr{u(x)}}{1+\snr{x}}\right)(1+\snr{y})\nonumber \\
&\stackrel{\eqref{24.6}}{\le}&2\max\{\bar{r},1\}\left\|\frac{u}{1+\snr{\ \cdot \ }}\right\|_{L^{\infty}(\mathbb{R}^{2})}+\frac{\snr{y}}{2^{3}}\stackrel{\eqref{24.7}}{\le}\frac{r_{*}}{2^{10}}+\frac{\snr{y}}{2^{3}}.
\end{eqnarray*}
We then control
\begin{flalign*}
\ell(x,y)^{2}&\ge\snr{y}^{2}(2^{-2}-2^{-6})+r_{*}^{2}(2^{-4}-2^{-20})+\snr{y}r_{*}(2^{-2}-2^{-12})\nonumber \\
&\ge\frac{\snr{y}^{2}}{5}+\frac{r_{*}^{2}}{20}+\frac{\snr{y}r_{*}}{5}\ge \frac{2r_{*}^{2}}{5}\stackrel{\eqref{24.7}}{>}100\max\{s,1\}^{2},
\end{flalign*}
so $\dist(K,\mathbb{R}^{2}\setminus B_{r_{*}}(0))\ge \dist_{L}(K,\mathbb{R}^{2}\setminus B_{r_{*}}(0))\ge 10\max\{s,1\}$ and $K\cap K_{s}(y)=\varnothing$ for all $y\in \mathbb{R}^{2}\setminus B_{r_{*}}(0)$.
\end{proof}
\noindent We then report an intrinsic variant of the Lebesgue differentiation theorem, whose proof is a  consequence of classical Lebesgue theory and of the definition of the area element in \rif{area}.
\begin{lemma}\label{limlem}
    Let $u\in C^{1}(\mathbb{R}^{n})$ be strictly spacelike, $g\in L^{1}_{\loc}(\mathbb{R}^{n})$, and $x_{0}$ be a Lebesgue point for $g$. Then
    \eqn{0.8}
    $$
    \lim_{s\to 0}\frac{1}{s^{n}}\int_{L_{s}(x_{0})}g(x)\d A=\omega_{n} g(x_{0}).
    $$
    In particular, if $g\equiv 1$, it is 
    \eqn{itis}
    $$\lim_{s\to 0}\frac{A(L_{s}(x_{0}))}{\snr{B_{s}(x_{0})}}=1.$$
    \end{lemma}
    \begin{proof} We report the proof for completeness. 
Set $
J(x):=\sqrt{1-|Du(x)|^2}$
 and $Q:=I-Du(x_0)\otimes Du(x_0)$, 
so that 
$
\det Q=1-|Du(x_0)|^2=J(x_0)^2.
$ We first prove that \rif{itis}. 
By Lemma \ref{balls} \textnormal{(h$_1$)}, there exists a constant
$c_0>0$, independent of $s$, such that
$ K_s(x_0)\subset B_{c_0s}(x_0)$
for $0<s\le1$. 
Define the rescaled sets
\[
E_s:=\frac{K_s(x_0)-x_0}{s}\subset B_{c_0}(0),
\qquad
E:=\{y\in\mathbb R^n:\langle Qy, y\rangle<1\}.
\]
Increasing $c_0$ if necessary, we may also assume that $E\subset B_{c_0}(0)$.
By the differentiability of $u$ at $x_0$,
\[
\frac{\ell_u(x_0+sy,x_0)^2}{s^2}
=
|y|^2-
\left(\frac{u(x_0+sy)-u(x_0)}{s}\right)^2
\stackrel{s\to 0}{\longrightarrow}
|y|^2-\langle Du(x_0), y\rangle^2
=
\langle Qy, y\rangle
\]
uniformly for $y\in B_{c_0}(0)$. Since $\partial E$ has zero
Lebesgue measure, it follows that
$
\mathbf 1_{E_s}(y)\stackrel{s\to 0}{\longrightarrow}\mathbf 1_E(y)$
for almost every $y\in B_{c_0}(0)$. 
Using $dA=J(x)\,dx$, the continuity of $J$, and dominated
convergence, we obtain
\[
\frac{A(L_s(x_0))}{s^n}=\int_{B_{c_0}(0)}\mathbf 1_{E_s}(y)J(x_0+sy)\,dy\stackrel{s\to 0}{\longrightarrow} J(x_0)|E|
=J(x_0)\frac{\omega_n}{\sqrt{\det Q}}
=\omega_n.
\]
Next, since $0<J\le1$ and $x_0$ is a Lebesgue point of $g$,
\[
\begin{aligned}
&\left|
\frac1{s^n}\int_{K_s(x_0)}g(x)J(x)\,dx
-g(x_0)\frac{A(L_s(x_0))}{s^n}
\right|\\
&\qquad\le
\frac1{s^n}\int_{B_{c_0s}(x_0)}|g(x)-g(x_0)|\,dx =
\frac{\omega_n c_0^n}{|B_{c_0s}(x_0)|}
\int_{B_{c_0s}(x_0)}|g(x)-g(x_0)|\,dx
\stackrel{s\to 0}{\longrightarrow} 0.
\end{aligned}
\]
Combining the preceding limits yields \rif{0.8}. 
\end{proof}

\section{Light segments in the optimal transport terminology}\label{luce}

 We first establish some basic terminology aimed at describing the fine structure of the so-called ``singular set'' of local minima $u\in \mathbb{X}(\mathbb{R}^{n})$ of \eqref{bi.en}, that is the set of light segments \cite{bs82,bar87}. We shall adopt some terminology from optimal transport theory, specifically regarding transport rays \cite{fm02}. 
\begin{definition}\label{defilight}
Let $\Omega\subset\mathbb R^n$ be open and let
$u:\Omega\to\mathbb R$ be weakly spacelike.
\begin{itemize}

\item A nondegenerate segment
$\overline{xy}\subset\Omega$, $x\neq y$, is called a
\emph{light segment} if
$
|u(x)-u(y)|=|x-y|.
$

\item A \emph{light ray} is a maximally extended light segment.
More precisely, given a light segment $\overline{xy}$, let
$\mathsf r_{x,y}$ denote the line through $x$ and $y$. The light ray
containing $\overline{xy}$ is the set
\[
\ell_{xy}
:=
\bigcup
\left\{
\overline{zw}\subset\Omega:
\begin{array}{l}
\overline{zw}\text{ is a light segment},\\
\overline{xy}\subset\overline{zw}\subset\mathsf r_{x,y}
\end{array}
\right\}.
\]
\end{itemize}
\end{definition}

Note that every light segment is contained in a unique light ray and
that every nondegenerate closed subsegment of a light ray is a light
segment. We call
$
\diam(\ell)\in(0,\infty]
$
the length of a light ray $\ell$.
While every light segment is closed and compactly contained in $\Omega$,
a light ray need not have either property.
Indeed, its closure may meet $\partial\Omega$, and, when $\Omega$ is
unbounded, it may also be a half-line or an entire line. Light segments remain central in our analysis. 
\begin{remark}\label{raggire}
{\em
Let $u:\Omega\to\mathbb R$ be weakly spacelike and let
$
\overline{xy}\subset\Omega
$
with $x\neq y$. If
$
|u(x)-u(y)|=|x-y|,
$
then it is easy to see that $u$ is affine on $\overline{xy}$. More precisely, assuming for
definiteness that $u(y)>u(x)$ and setting
$
z_t:=(1-t)x+ty
$
for $t\in[0,1]$, one has
$
u(z_t)
=
u(x)+t|y-x|
=
(1-t)u(x)+tu(y)
$
for every $t\in[0,1]$. Moreover, if
$
u\in\mathbb X(\mathbb R^n)
$
and $n\ge3$, all light segments, and therefore all light rays, have
uniformly bounded length. Indeed, by \eqref{infsup},
\eqn{aff.0.1}
$$
\sup
\left\{
|x-y|:
\overline{xy}\subset\mathbb R^n
\mbox{ is a light segment}
\right\}
\le
2\|u\|_{L^\infty(\mathbb R^n)}.
$$
In dimension $n=2$, \eqref{logsup} rules out unbounded light rays.}
\end{remark}
With $u,\Omega$ fixed and understood as in Definition \ref{defilight},
we set
\eqn{language1}
$$
\mathcal S_{\texttt{i},u}:=
\bigcup\left\{(xy):\overline{xy}\text{ is a light segment}\right\},
\qquad \mathcal S_u :=\bigcup
\left\{\overline{xy}:\overline{xy}\text{ is a light segment}\right\},
$$
and
\eqn{language2}
$$
\mathcal I_u
:=\left\{ x\in\Omega:|Du(x)|=1 \right\}.
$$
Thus, $\mathcal S_{\texttt{i},u}$ and $\mathcal S_u$ are the unions
of the relative interiors and of the closed light segments,
respectively, while $\mathcal I_u$ is the $1$-level set of $|Du|$. Note that $\mathcal S_u$ can equivalently be regarded as the union of all light rays.

In analogy with the zero-length rays arising in optimal transport
theory, cf. \cite[Definition 6]{fm02}, we define the zero-length
accumulation set
\eqn{language3}
$$
\mathcal L_u
:=
\left\{
z\in\Omega\setminus\mathcal S_u:
\begin{array}{l}
\text{there exists a sequence of light rays }\{\ell_k\}
\text{ such that}\\[1mm]
\dist(z,\ell_k)\to0
\quad\text{and}\quad
\diam(\ell_k)\to0
\end{array}
\right\}.
$$
When the light rays are compact, say
$\ell_k=\overline{x_ky_k}$, the above condition is equivalently
expressed by
$
x_k\to z$ and $
y_k\to z.
$
By definition,
\eqn{vuoto}
$$
\mathcal S_u\cap\mathcal L_u=\varnothing.
$$
Finally, following \cite{bimm24}, we set
\eqn{deffiK}
$$
\mathcal K_u:=\overline{\mathcal S_u}\cap\Omega.
$$
The set $\mathcal K_u$ may contain points which do not belong to any
nondegenerate light segment. The set $\mathcal L_u$ introduced above
is designed to describe accumulation points arising from light rays
whose diameters tend to zero. The precise relation between
$\mathcal S_u$, $\mathcal L_u$, and $\mathcal K_u$ will be established
in Section \ref{capacity}.
\begin{remark}[Precise representatives]{\em We clarify a few conventions we shall adopt for the rest of the paper. 
\begin{itemize}
\item Here, as in all the rest of the paper we identify $u$ with its locally Lipschitz continuous representative. This allows for no ambiguity occurring in the definitions of the light segments and in \rif{defilight} and in \rif{language1} and \rif{language3}. \item The situation is slightly different for \rif{language2}. In the definition of $\mathcal I_u$ in \rif{language2},
the notation $|Du(z)|$ refers to the precise representative
of the scalar function $|Du|$, defined by 
\eqn{esilimite}
$$
|Du(z)|:=  
\begin{cases} 
\lim_{r\to 0} \, (|Du|)_{B_r(z)} & \mbox{if this limit exists}\\
0 & \mbox{otherwise}. 
\end{cases}
$$
In this sense the precise representative is defined everywhere and the definition of $\mathcal{I}_{u}$ automatically discards the points where the limit in \rif{esilimite}$_1$ does not exist. Later on, we shall assume higher regularity properties on $Du$ that will allow us to better clarify the structure of $\mathcal{I}_{u}$.
\item If $z$ is a Lebesgue point of the weak gradient $Du$,
we denote its precise value by $Du(z)$. Then $z$ is
also a Lebesgue point of $|Du|$, and
\[
\lim_{r\to 0}(|Du|)_{B_r(z)}
=
|\lim_{r\to 0}(Du)_{B_r(z)}|.
\]
Thus the scalar representative in \rif{esilimite}
agrees with the modulus of the vector representative
at such points. The converse implication between
the two Lebesgue-point properties is generally false.
 \item Moreover, we shall denote by $D_{\mathrm{cl}}u(z)$ the classical, pointwise differential of $u$ at $z$, whenever it exists. We recall that whenever $z$ is both a Lebesgue point for $Du$ and a point of classical differentiability, then the precise representative of $Du$ at $z$ and the classical differential do coincide, i.e., $Du(z)=D_{\mathrm{cl}}u(z)$. 
 \end{itemize}}
\end{remark}
\subsection{Differentiability on light segments}\label{bls.s} Here we collect a few basic facts on light segments that are partially folklore when seen in connection to the Optimal Transport theory.
\begin{lemma}\label{lemmaraggi}
Let \(\overline{xy}\) be a light segment, oriented so that
\(u(x)>u(y)\). Then every point
$
z\in\overline{xy}\setminus\{x,y\}
$
is both a classical differentiability point of \(u\) and a Lebesgue
point of the weak gradient \(Du\), and moreover 
$$
Du(z)=D_{\mathrm{cl}}u(z)=\frac{x-y}{|x-y|}.
$$
If either endpoint \(x\) or \(y\) is a Lebesgue point of \(Du\), then
\(u\) is classically differentiable at that endpoint and the classical
differential agrees with the precise representative of \(Du\).
In particular,
\eqn{aff.3}
$$
\begin{cases}
\displaystyle |Du(z)|=1
&\mbox{for every } z\in\overline{xy}\setminus\{x,y\}
\\[1.5mm]
\displaystyle |Du(z)|=1
&\mbox{if \(z \in \{x,y\}\) is a Lebesgue point of \(Du\).}
\end{cases}
$$
\end{lemma}

\begin{proof} All balls appearing below are taken sufficiently small
to be compactly contained in $\Omega$. By the argument in \cite[Lemma 10]{fm02}, $u$ is
classically differentiable at every interior point $z$
of $\overline{xy}$ and, since $u(x)>u(y)$,
\eqn{valore}
$$
D_{\mathrm{cl}}u(z)=\frac{x-y}{|x-y|}.
$$
Moreover, we have 
\eqn{aff.0}
$$
u(x+t_2(y-x))-u(x+t_1(y-x))=-|x-y|(t_2-t_1),\qquad 0\le t_1\le t_2\le1.
$$
We next show that every such \(z\) is also a Lebesgue point of the
weak gradient \(Du\) (that we identify via its precise representative) and that $Du(z)=D_{\mathrm{cl}}u(z)$. Fix
$
z\in\overline{xy}\setminus\{x,y\}
$
and set
$
q:=D_{\mathrm{cl}}u(z)$ so that $ |q|=1$,
and consider the affine map 
$
\ell(w):=\langle q,w-z\rangle+u(z)
$
for every $w \in \er^n$. 
Integration by parts gives 
\begin{flalign*}
|(Du-q)_{B_r(z)}| & =\frac{1}{|B_r(z)|} \left|\int_{\partial B_r(z)}
(u-\ell)\nu_r\,\textnormal{d}\mathcal H^{n-1} \right|\le\frac{\mathcal H^{n-1}(\partial B_r(z))}
{|B_r(z)|}\sup_{\partial B_r(z)}|u-\ell | =\frac{\tx{o}(r)}{r},
\end{flalign*}
where $\nu_r(w):=(w-z)/r$ is the outward Euclidean unit
normal to $\partial B_r(z)$, and therefore
$
(Du)_{B_r(z)}\to q
$
as $r\to 0$. 
Using $\snr{Du}\le1$ and $\snr q=1$, we obtain
$$
\mint_{B_r(z)}|Du-q|^2\,\textnormal{d}w=\mint_{B_r(z)}\left(
|Du|^2+|q|^2-2\langle Du,q\rangle\right)\,dw\leq 2\left(1-\langle (Du)_{B_r(z)},q\rangle\right)\to 0.
$$
Thus \(z\) is a Lebesgue point of \(Du\), and its precise value is the one displayed in \rif{valore}. 
To prove \rif{aff.3}$_2$ and differentiability at the endpoints, we now consider an endpoint, say $x$, which is a Lebesgue
point of $Du$. Since $u$ is locally Lipschitz, the argument
in \cite[proof of Theorem 6.2]{eg15} yields
\eqn{contr}
$$
\lim_{\rr \to 0} \frac{1}{\rr}\mint_{B_{\rr}(x)}\snr{u(z)-u(x)-\langle Du(x),z-x\rangle}\dz=0.
$$
In our setting this implies classical differentiability at $x$. By contradiction, we assume that there is $\delta\in(0,1)$ and a sequence of vectors $\{h_{k}\}_{k\in \mathbb{N}}\subset\mathbb{R}^{n}$ with $\rr_{k}:=\snr{h_{k}}\to 0$ such that
\eqn{aff.1}
$$
\snr{u(x+h_{k})-u(x)-\langle Du(x),h_{k}\rangle}>\delta\snr{h_{k}}\qquad \mbox{for all} \ \ k\in \mathbb{N}.
$$
Set $x_{k}:=x+h_{k}$. For any $z\in B_{\delta\rr_{k}/4}(x_{k})$ we bound
\begin{eqnarray}\label{aff.2}
\snr{u(z)-u(x)-\langle Du(x),z-x\rangle}&\ge&\snr{u(x_{k})-u(x)-\langle Du(x),x_{k}-x\rangle}\nonumber \\
&&-\snr{u(z)-u(x_{k})}-\snr{\langle Du(x),x_{k}-z\rangle}\nonumber \\
&\stackrel{\eqref{aff.1}}{\ge}&\delta\snr{h_{k}}-2\snr{x_{k}-z}\ge  \frac{\delta\rr_{k}}{2}.
\end{eqnarray}
As for all $k\in \mathbb{N}$ it is $B_{\delta\rr_{k}/4}(x_{k})\subset B_{2\rr_{k}}(x)$, we find
\begin{flalign*}
&\frac{1}{2\rr_{k}}\mint_{B_{2\rr_{k}}(x)}\snr{u(z)-u(x)-\langle Du(x),z-x\rangle}\dz\nonumber \\
&\qquad \quad \ge\frac{1}{\rr_{k}}\left(\frac{\delta}{2^{4}}\right)^{n}\mint_{B_{\frac{\delta\rr_{k}}{4}}(x_{k})}\snr{u(z)-u(x)-\langle Du(x),z-x\rangle}\dz\stackrel{\eqref{aff.2}}{\ge}\frac{\delta^{n+1}}{2^{5n}}>0,
\end{flalign*}
which contradicts \rif{contr}. Hence $u$ is classically differentiable at every Lebesgue point of $Du$, and the classical differential agrees there with the precise representative of $Du$. We look back at \eqref{aff.0}, and notice that
$$
\frac{u(x+t(y-x))-u(x)}{t\snr{x-y}}=-1\qquad \mbox{for all} \ \ t\in (0,1),
$$
which, combined with the classical differentiability of $u$ at $x$, implies that $(\langle Du(x),(y-x)\rangle)/\snr{x-y}=-1$. Similarly, if also $y$ is a Lebesgue point of $Du$, we gain that $(\langle Du(y),(y-x)\rangle)/\snr{x-y}=-1$. Recalling that $\|Du\|_{L^\infty(\Omega)}\le1$ we deduce that \rif{aff.3}$_2$ holds, and the proof is complete. 
\end{proof}

\section{The $L(n-1,s)$-anti-peeling, and Theorem \ref{al.t}}\label{lr.s}

If $n\ge3$, there is no loss of generality in assuming $s>1$.
Let $\Omega\subseteq\mathbb R^n$, $n\ge2$, be open, and let
$u\in W^{1,\infty}_{\loc}(\Omega)$ be a weakly spacelike local
minimizer of \eqref{bi.en} according to Definition
\ref{minimilocali}.
By contradiction, assume that \eqref{a.0} fails for some
admissible $t<0$; an analogous argument works if it fails for
an admissible $t>1$.
By weak spacelikeness, this means that
\eqn{a.1}
$$
u(x_t)>u(x_0)+t\snr{x_1-x_0}
\qquad \mbox{for some } t<0.
$$
Fix such a number $t_*<0$ with
$\overline{x_{t_*}x_0}\subset\Omega$.
Since $\overline{x_{t_*}x_1}\Subset\Omega$, we can choose a bounded
domain $\widetilde\Omega$ such that
$
\overline{x_{t_*}x_1}\Subset\widetilde\Omega\Subset\Omega.
$
Introduce also the backward light cones $\Gamma_0$ and $\Gamma_1$
with vertices at $(x_0,u(x_0))$ and $(x_1,u(x_1))$, respectively,
that is
$$
\Gamma_0(x):=u(x_0)-\snr{x-x_0}
\qquad \mbox{and}\qquad
\Gamma_1(x):=u(x_1)-\snr{x-x_1},
$$
and the balls
$$
\tx{B}_0:=B_{\snr{x_1-x_0}/2}(x_0)
\subset
\tx{B}_1:=B_{2\snr{x_1-x_0}}(x_1).
$$
The rest of the proof proceeds in five steps.

\subsection*{Step 1: Basic reductions.} Here we follow \cite[Proof of Theorem 3.2]{bs82}. By choosing different $x_{0},x_{1}$ (on the extended light segment), we can assume that \eqref{a.0} holds on $[-1/4,1]$, i.e.,
\eqn{a.3}
$$
u(x_{t})=u(x_{0})+t\snr{x_{1}-x_{0}}\qquad \mbox{for all}\ \ t\in [-1/4,1],
$$
while \eqref{a.1} is satisfied at $t=-1/2$,
\eqn{a.4}
$$
u(x_{-1/2})>u(x_{0})- \frac{ \snr{x_{1}-x_{0}}}{2},
$$
and $\tx{B}_{1}\Subset \ti{\Omega}$. Next, set for simplicity
\eqn{a.7}
$$
\tx{l}:=\snr{x_{1}-x_{0}},\qquad \quad \tx{e}:=\frac{x_{1}-x_{0}}{\tx{l}}.
$$
After applying a translation in the \(x\)-variables, an orthogonal transformation, and a vertical translation of \(u\), and relabelling the transformed objects again as \(u\), \(\mu\), and \(\widetilde{\Omega}\), we may assume that
$
x_{0}=0 $, $u(0)=0$, $\tx{e}=\tx{e}_{n}$.
Specifically, we rescale as follows
$$
u_{\tx{l}}(x):=\tx{l}^{-1}u(\tx{l} x),\qquad
\mu_{\tx{l}}(x):=\tx{l}\,\mu(\tx{l} x),\qquad
\widetilde{\Omega}_{\tx{l}}:=\{x\in\mathbb{R}^{n}:\tx{l} x\in\widetilde{\Omega}\}.
$$
These transformations preserve local minimality and
\eqref{a.3}--\eqref{a.4}; hence, again we may assume that 
\eqn{u0u0}
$$
\begin{cases}
\displaystyle
\ x_{0}=0, \qquad u_{\tx{l}}(0)=0,\qquad \tx{e}=\tx{e}_{n}\vspace{1.5mm}\\ \displaystyle
\ \overline{x_{1}x_{0}}=\{t\tx{e}_{n},\ t\in [0,1]\}\vspace{1.5mm} \\ \displaystyle
\ \Gamma_{0}(x)=-\snr{x},\qquad \quad \Gamma_{1}(x)=1-\snr{x-\tx{e}_{n}}\vspace{1.5mm}\\ \displaystyle
\ \tx{B}_{0}=B_{\tx{1}/2}(0)\subset \tx{B}_{1}=B_{2}(\tx{e}_{n})\Subset \ti{\Omega}_{\tx{l}},
\end{cases}
$$
where $\tx{e}_{n}:=(0,\cdots,0,1)$, and
\eqn{a.11}
$$
\begin{cases}
    \displaystyle
    \ u_{\tx{l}}(t\tx{e}_{n})=t\quad &\mbox{for all} \ \ t\in [-1/4,1]\vspace{1.5mm}\\
    \displaystyle
    \ u_{\tx{l}}(-\tx{e}_{n}/2)>-1/2\,.\quad &
\end{cases}
$$
\noindent After such preliminaries, we deliver the main construction.
\subsection*{Step 2: Construction of radial barriers.} The idea is to derive a quantitative version of the barrier construction by Bartnik \& Simon, starting from the family of radial solutions identified in \cite[Section 3]{bs82}. Let $\varepsilon_{0},\ti{\varepsilon}_{0}\in (0,1)$ be small thresholds to be fixed, assume $0<\varepsilon_{0}\le \ti{\varepsilon}_{0}$ and pick $\varepsilon\in (0,\varepsilon_{0})$. For $\rr\in [0,1/2]$, set
\eqn{bn}
$$
\begin{array}{c}
\displaystyle
a(\rr):=N-\frac{\rr^{n}}{n},\qquad \quad N:=2>\frac{1}{n2^{n}}+1,
\end{array}
$$
introduce the parameter
\eqn{lamlam}
$$
\Lambda_{\varepsilon}:=\frac{\varepsilon}{\mathds{1}_{\{n\ge 3\}}\nr{\mu_{\tx{l}}}_{L(n-1,s)(\tx{B}_{0})}+\mathds{1}_{\{n=2\}}\nr{\mu_{\tx{l}}}_{L^{1}(\tx{B}_{0})}+1},
$$
and notice that $a(\rr)\ge 1$ for all $\rr\in [0,1/2]$. Then introduce
\eqn{fw}
$$
\mathcal{f}_{\varepsilon}^{-}(\rr):=-\int_{0}^{\rr}\frac{a(t)\dtt}{\sqrt{\Lambda_{\varepsilon}^{2}t^{2(n-1)}+a(t)^{2}}},\qquad \quad \mathcal{w}_{\varepsilon}^{-}(x):=\mathcal{f}_{\varepsilon}^{-}(\snr{x}),
$$
thus 
\eqn{fw.1}
$$
\snr{D\mathcal{w}_{\varepsilon}^{-}(x)}=\snr{(\mathcal{f}^{-}_{\varepsilon})'(\snr{x})}<1\mbox{ for all }x\in \tx{B}_{0}\setminus \{0\}\mbox{ and }\snr{D\mathcal{w}_{\varepsilon}^{-}}\le 1\mbox{ a.e. on }\tx{B}_{0},
$$
meaning that $\mathcal{w}_{\varepsilon}^{-}\in W^{1,\infty}(\tx{B}_{0})$ is strictly spacelike in $\tx{B}_{0}\setminus \{0\}$. A direct computation shows that
\eqn{fw.2}
$$
\diver\, \frac{D\mathcal{w}_{\varepsilon}^{-}}{\sqrt{1-\snr{D\mathcal{w}^{-}_{\varepsilon}}^{2}}} \stackrel{\eqref{fw}}{=}\rr^{1-n}\left(\frac{\rr^{n-1}(\mathcal{f}_{\varepsilon}^{-})'(\rr)}{\sqrt{1-\snr{(\mathcal{f}_{\varepsilon}^{-})'(\rr)}^{2}}}\right)'
$$
for all $\rr=\snr{x}>0$. As $$1-(\mathcal{f}_{\varepsilon}^{-})'(\rr)^{2}=\frac{\Lambda_{\varepsilon}^{2}\rr^{2(n-1)}}{\Lambda_{\varepsilon}^{2}\rr^{2(n-1)}+a(\rr)^{2}},
$$
we have
\eqn{a.8}
$$
\diver\, \frac{D\mathcal{w}_{\varepsilon}^{-}}{\sqrt{1-\snr{D\mathcal{w}_{\varepsilon}^{-}}^{2}}} \stackrel{\eqref{fw.2}}{=} \frac{1}{\Lambda_{\varepsilon}}\qquad \mbox{pointwise in} \ \ \tx{B}_{0}\setminus \{0\}.
$$
\begin{remark}{\em 
 For later use, observe that the vector field
$$
\partial H(D\mathcal w_\varepsilon^-)=\frac{D\mathcal w_\varepsilon^-}{\sqrt{1-|D\mathcal w_\varepsilon^-|^2}}
$$
is smooth in $\tx{B}_{0}\setminus\{0\}$; hence \eqref{a.8} also holds
distributionally on every open set $U\Subset \tx{B}_{0}\setminus\{0\}$.}
\end{remark}

\subsection*{Step 3: Barrier estimates.} We first compare cones $\Gamma_{0}$, $\Gamma_{1}$. By triangle inequality, we have
$$
\Gamma_{1}(x)-\Gamma_{0}(x)=1-\snr{x-\tx{e}_{n}}+\snr{x}\ge 0,\qquad x\in \tx{B}_{0},
$$
with equality holding when $x$, $0$, and $\tx{e}_{n}$ are collinear - specifically, when $0$ stays between $x$ and $\tx{e}_{n}$. This implies that 
\eqn{a.9}
$$
\begin{cases}
\displaystyle
\ \Gamma_{1}(x)\ge \Gamma_{0}(x)\quad &\mbox{for all} \ \ x\in \tx{B}_{0}\vspace{1.5mm}\\
\displaystyle
\ \Gamma_{1}(x)>\Gamma_{0}(x)\quad &\mbox{for all} \ \ x\in \tx{B}_{0}\setminus \{-t\tx{e}_{n}\, \colon\, t\in [0,1/2]\}.
\end{cases}
$$
Now, since the light segment is $t\tx{e}_{n}$, $t\in [-1/4,1]$, it holds $u_{\tx{l}}(\tx{e}_{n})=1$ so by weak spacelikeness we get
\eqn{a.10}
$$
u_{\tx{l}}(x)-\Gamma_{1}(x)=u_{\tx{l}}(x)-1+\snr{x-\tx{e}_{n}}=u_{\tx{l}}(x)-u_{\tx{l}}(\tx{e}_{n})+\snr{x-\tx{e}_{n}}\ge 0\qquad \mbox{for all} \ \ x\in \tx{B}_{0},
$$
and \eqref{a.9} yields $u_{\tx{l}}(x)\ge \Gamma_{0}(x)$ for all $x\in \tx{B}_{0}$. Let us show that this inequality is strict on $\partial \tx{B}_{0}$. By $\eqref{a.9}_{2}$-\eqref{a.10} we have that $u_{\tx{l}}(x)>\Gamma_{0}(x)$ for all $x\in \tx{B}_{0}\setminus \{-t\tx{e}_{n}\, \colon\, t\in [0,1/2]\}$, so we only need to check what happens on $\partial\tx{B}_{0}\cap \{-t\tx{e}_{n}\, \colon\,  \ t\in [0,1/2]\}$. The only option is $\partial\tx{B}_{0}\cap \{-t\tx{e}_{n}\, \colon\,  \ t\in [0,1/2]\}=\{x_{-1/2}=-\tx{e}_{n}/2\}$, thus
$$
u_{\tx{l}}(x)-\Gamma_{0}(x)\stackrel{\eqref{a.11}_{2}}{>}-\frac{1}{2}+\snr{x}\stackrel{x\in \partial\tx{B}_{0}}{=}0,
$$
therefore, being $\partial \tx{B}_{0}$ compact, it is 
\eqn{a.25.1}
$$\tx{c}_{0}:=\min_{x\in \partial\tx{B}_{0}}(u_{\tx{l}}(x)-\Gamma_{0}(x))>0,$$ $\tx{c}_{0}\equiv \tx{c}_{0}(u)$. On the other hand, \eqref{a.11}$_{1}$ forces $u_{\tx{l}}=\Gamma_{0}$ on the segment $t\tx{e}_{n}$ with $t\in [-1/4,0]$, so equality holds on $\mathcal{S}:=\{t\tx{e}_{n}, \ t\in [-1/4,0]\}\subseteq \overline{(1/2)\tx{B}}_{0}$. Overall, we have just proven that
\eqn{a.13}
$$
\begin{cases}
    \displaystyle
    \ u_{\tx{l}}\ge \Gamma_{0}\quad &\mbox{in} \ \ \tx{B}_{0}\vspace{1.5mm}\\
    \displaystyle
    \ u_{\tx{l}}=\Gamma_{0}\quad &\mbox{on} \ \ \mathcal{S}\vspace{1.5mm}\\
    \displaystyle
    \ u_{\tx{l}}>\Gamma_{0}\quad &\mbox{on} \ \ \partial \tx{B}_{0}.
\end{cases}
$$
At this stage, we identify the small parameter $\varepsilon_{0}\in (0,\ti{\varepsilon}_{0})$ introduced in \emph{Step 2}, $\ti{\varepsilon}_{0}>0$ being fixed later on. We set $$M_{\tx{l}}:=\mathds{1}_{\{n\ge 3\}}\nr{\mu_{\tx{l}}}_{L(n-1,s)(\tx{B}_{0})}+\mathds{1}_{\{n=2\}}\nr{\mu_{\tx{l}}}_{L^{1}(\tx{B}_{0})}+1,$$ and
\eqn{a.14}
$$
\varepsilon_{0}:=\min\left\{\frac{\ti{\varepsilon}_{0}}{2},2^{n-1}M_{\tx{l}}\sqrt{\frac{\tx{c}_{0}}{(2n-1)}},M_{\tx{l}}\right\},
$$
and permanently work under the assumption $\varepsilon\in (0,\varepsilon_{0})$. Next, observe that on $\bar{\tx{B}}_{0}$ it holds 
\begin{eqnarray}\label{a.25.2}
0&\stackrel{\snr{(\mathcal{f}_{\varepsilon}^{-})'}\le 1}{\le}&\int_{0}^{\snr{x}}\left(1-\frac{a(t)}{\sqrt{\Lambda_{\varepsilon}^{2}t^{2(n-1)}+a(t)^{2}}}\right)\dtt\nonumber \\ &=&\mathcal{w}^{-}_{\varepsilon}(x)-\Gamma_{0}(x)\nonumber \\
&=&\int_{0}^{\snr{x}}\frac{\Lambda_{\varepsilon}^{2}t^{2(n-1)}}{\sqrt{\Lambda_{\varepsilon}^{2}t^{2(n-1)}+a(t)^{2}}(\sqrt{\Lambda_{\varepsilon}^{2}t^{2(n-1)}+a(t)^{2}}+a(t))}\dtt\nonumber \\
&\le&\int_{0}^{\snr{x}}\frac{\Lambda_{\varepsilon}^{2}t^{2(n-1)}}{a(t)^{2}}\dtt\nonumber \\
&\stackrel{a\ge 1}{\le} &\frac{\Lambda_{\varepsilon}^{2}\snr{x}^{2n-1}}{2n-1}\notag \\ 
&\le & \frac{\Lambda_{\varepsilon}^{2}}{(2n-1)2^{2n-1}}=:\tx{c}_{1}\Lambda_{\varepsilon}^{2},
\end{eqnarray} 
so, via \eqref{a.13}-\eqref{a.14} we can conclude that
\eqn{a.25}
$$
\begin{cases}
    \displaystyle
    \ \min_{x\in \partial \tx{B}_{0}}(u_{\tx{l}}(x)-\mathcal{w}_{\varepsilon}^{-}(x))\ge \frac{\tx{c}_{0}}{2}\vspace{1.5mm}\\
    \displaystyle
    \ (\mathcal{w}_{\varepsilon}^{-}-u_{\tx{l}})_{+}\le \tx{c}_{1}\Lambda_{\varepsilon}^{2}\mbox{ for all }x\in \tx{B}_{0}.
    \end{cases}
    $$
In fact, $\eqref{a.25}_{1}$ follows from
$$
u_{\tx{l}}(x)-\mathcal{w}_{\varepsilon}^{-}(x)=(u_{\tx{l}}(x)-\Gamma_{0}(x))-(\mathcal{w}_{\varepsilon}^{-}(x)-\Gamma_{0}(x))\stackrel{\eqref{a.25.1},\eqref{a.25.2}}{\ge}\tx{c}_{0}-\tx{c}_{1}\Lambda_{\varepsilon}^{2}\stackrel{\eqref{a.14}}{\ge}\frac{\tx{c}_{0}}{2},
$$
for all $x\in \partial \tx{B}_{0}$, while $\eqref{a.25}_{2}$ comes from the simple observation that for any $x\in \tx{B}_{0}$ such that $\mathcal{w}_{\varepsilon}^{-}(x)\le u_{\tx{l}}(x)$, it is $(\mathcal{w}_{\varepsilon}^{-}-u_{\tx{l}})_{+}=0$, while if $\mathcal{w}_{\varepsilon}^{-}(x)> u_{\tx{l}}(x)$, we have
\begin{eqnarray*}
(\mathcal{w}_{\varepsilon}^{-}(x)-u_{\tx{l}}(x))_{+}&\stackrel{\mathcal{w}_{\varepsilon}^{-}(x)> u_{\tx{l}}(x)}{=}&\mathcal{w}_{\varepsilon}^{-}(x)-u_{\tx{l}}(x)=(\mathcal{w}_{\varepsilon}^{-}(x)-\Gamma_{0}(x))-(u_{\tx{l}}(x)-\Gamma_{0}(x))\nonumber \\
&\stackrel{\eqref{a.13}_{1}}{\le}& (\mathcal{w}_{\varepsilon}^{-}(x)-\Gamma_{0}(x))\stackrel{\eqref{a.25.2}}{\le}\tx{c}_{1}\Lambda_{\varepsilon}^{2}.
\end{eqnarray*}
Moreover, by continuity we deduce that there exists $r_{0}\equiv r_{0}(u,\varepsilon)\in (0,1/2)$ so close to $1/2$ that 
\eqn{a.15}
$$\min_{\bar{\tx{B}}_{0}\setminus B_{r_{0}}(0)}(u_{\tx{l}}(x)-\mathcal{w}_{\varepsilon}^{-}(x))>0.$$
Next, we shorten segment $\mathcal{S}$ by introducing $\mathcal{S}':=\{t\tx{e}_{n}\, \colon \, \ t\in [-1/6,-1/8]\}\Subset \mathcal{S}$, and observe that on $\mathcal{D}_{\mathcal{S}'}:=\tx{B}_{0}\cap\{\dist(x,\mathcal{S}')\le 2^{-4}\}$ it is
\eqn{a.16}
$$
x\in \mathcal{D}_{\mathcal{S}'} \ \Longrightarrow \ \frac{1}{16}\le \snr{x}\le \frac{11}{48}.
$$
For all $x\in \mathcal{D}_{\mathcal{S}'}$, we bound below
\begin{eqnarray}\label{a.17}
\mathcal{w}^{-}_{\varepsilon}(x)-\Gamma_{0}(x)&\stackrel{\eqref{a.16}}{\ge}&\int_{0}^{\frac{1}{16}}\left(1-\frac{a(t)}{\sqrt{\Lambda_{\varepsilon}^{2}t^{2(n-1)}+a(t)^{2}}}\right)\dtt\nonumber \\
&=&\int_{0}^{\frac{1}{16}}\frac{\Lambda_{\varepsilon}^{2}t^{2(n-1)}}{\sqrt{\Lambda_{\varepsilon}^{2}t^{2(n-1)}+a(t)^{2}}(\sqrt{\Lambda_{\varepsilon}^{2}t^{2(n-1)}+a(t)^{2}}+a(t))}\dtt\nonumber \\
&\stackrel{\eqref{a.14},a\ge 1}{\ge}&\frac{\Lambda_{\varepsilon}^{2}}{4}\int_{0}^{\frac{1}{16}}\frac{t^{2(n-1)}}{a(t)^{2}}\dtt\notag \\ &\stackrel{\eqref{bn}}{\ge}&\frac{\Lambda_{\varepsilon}^{2}}{2^{16n}N^{2}(2n-1)}=:\tx{c}_{2}\Lambda_{\varepsilon}^{2},
\end{eqnarray}
where we also used that, by \eqref{lamlam}
and \eqref{a.14}, we have $\Lambda_{\varepsilon}^{2}\le1$ and $\tx{c}_{2}\equiv \tx{c}_{2}(n)$. Finally, we decompose $x\in \tx{B}_{0}\cap \{\langle x,\tx{e}_{n}\rangle\in [-1/6,-1/8]\}$ as $x=-t\tx{e}_{n}+x_{\perp}$, with $t\in [1/8,1/6]$ and $x_{\perp}$ orthogonal to $\tx{e}_{n}$. By $\eqref{a.13}_{2}$ it is $u_{\tx{l}}(-\tx{e}_{n}/4)=-1/4$, so weak spacelikeness gives
\begin{eqnarray}\label{a.19}
0&\stackrel{\eqref{a.13}_{1}}{\le}&u_{\tx{l}}(x)-\Gamma_{0}(x)=u_{\tx{l}}(x)-u_{\tx{l}}\left(-\frac{\tx{e}_{n}}{4}\right)-\frac{1}{4}+\snr{x}\nonumber \\
&\le& -\frac{1}{4}+\left|x+\frac{\tx{e}_{n}}{4}\right|+\snr{x}\nonumber \\
&=&-\frac{1}{4}+\sqrt{\left(\frac{1}{4}-t\right)^{2}+\snr{x_{\perp}}^{2}}+\sqrt{t^{2}+\snr{x_{\perp}}^{2}}\nonumber \\
&\le& -\frac{1}{4}+\left(\frac{1}{4}-t\right)+\frac{2\snr{x_{\perp}}^{2}}{1-4t}+t+\frac{\snr{x_{\perp}}^{2}}{2t}\nonumber \\
&\stackrel{t\in [1/8,1/6]}{\le}&10\snr{x_{\perp}}^{2}\le \tx{c}_{3}\dist(x,\mathcal{S}')^{2},
\end{eqnarray}
where we also used the elementary inequality $$\sqrt{b_{1}^{2}+b_{2}^{2}}\le b_{1}+\frac{b_{2}^{2}}{2b_{1}}\, \quad \mbox{for all $b_{1}>0$, $b_{2}\in \mathbb{R}$}, 
$$ 
and it is $\tx{c}_{3}:=12$. Fix $\delta\equiv \delta(n)>0$ as
\eqn{a.20}
$$\delta:= \min\left\{\sqrt{\frac{\tx{c}_{2}}{2\tx{c}_{3}}},\frac{1}{2^{4}}\right\},$$ introduce the cylinder
\eqn{a.18}
$$
\mathcal{C}_{\varepsilon}:=\tx{B}_{0}\cap\left\{x=-t\tx{e}_{n}+x_{\perp}, \ t\in [1/8,1/6], \ \snr{x_{\perp}}\le \Lambda_{\varepsilon}\delta\right\}\stackrel{\eqref{a.14},\eqref{a.20}}{\subset} \mathcal{D}_{\mathcal{S}'},
$$
and record that, by \eqref{a.17}-\eqref{a.18}, it holds
\begin{flalign}\label{a.21}
\mathcal{w}_{\varepsilon}^{-}(x)-u_{\tx{l}}(x)&= (\mathcal{w}_{\varepsilon}^{-}(x)-\Gamma_{0}(x))-\left(u_{\tx{l}}(x)-\Gamma_{0}(x)\right)\nonumber \\
&\ge \tx{c}_{2}\Lambda_{\varepsilon}^{2}-\tx{c}_{3}\dist(x,\mathcal{S}')^{2}\ge (\tx{c}_{2}-\tx{c}_{3}\delta^{2} ) \Lambda_{\varepsilon}^{2}\ge  \tx{c}_{2}\Lambda_{\varepsilon}^{2}/2
\end{flalign}
for all $x\in \mathcal{C}_{\varepsilon}$. In particular,
\eqn{a.26}
$$
\nr{(\mathcal{w}_{\varepsilon}^{-}-u_{\tx{l}})_{+}}_{L^{1}(\tx{B}_{0})}\ge \nr{(\mathcal{w}_{\varepsilon}^{-}-u_{\tx{l}})_{+}}_{L^{1}(\mathcal{C}_{\varepsilon})}\ge \tx{c}_{2}\Lambda_{\varepsilon}^{2}\snr{\mathcal{C}_{\varepsilon}}/2=:\tx{c}_{4}(n)\Lambda_{\varepsilon}^{n+1}.
$$
\subsection*{Step 4: Spacelike competitors.} Here we use $\mathcal{w}_{\varepsilon}^{-}$ to build suitable competitors for problem \eqref{bi.en}. Let $\sigma>0$, set $$
v_{\varepsilon;\sigma}:=\begin{cases}(\mathcal{w}_{\varepsilon}^{-}-u_{\tx{l}}-\sigma)_{+} & \mbox{in $ \tx{B}_{0}$}\\
0 & \mbox{outside $ \tx{B}_{0}$}
\end{cases}
$$ and introduce the open (by continuity) set $\tx{E}_{\sigma}^{-}:=\tx{B}_{0}\cap \{\mathcal{w}_{\varepsilon}^{-}-\sigma>u_{\tx{l}}\}$. By construction, $u_{\tx{l}}+v_{\varepsilon;\sigma},v_{\varepsilon;\sigma}\in W^{1,\infty}(\tx{B}_{0})$. In particular, by \eqref{fw.1}, we see that $u_{\tx{l}}+v_{\varepsilon;\sigma}=\max\{\mathcal{w}_{\varepsilon}^{-}-\sigma,u_{\tx{l}}\}$ is weakly spacelike with $\nr{Du_{\tx{l}}+Dv_{\varepsilon;\sigma}}_{L^{\infty}(\tx{B}_{0})}\le 1$. Let us show that there exists a radius $r_{\varepsilon;\sigma}\equiv r_{\varepsilon;\sigma}(\varepsilon,\sigma,\mathcal{w}_{\varepsilon}^{-},u)\in (0,1/4)$ such that 
$\supp(v_{\varepsilon;\sigma})\Subset \tx{B}_{0}\setminus B_{r_{\varepsilon;\sigma}}(0)$. By \eqref{a.15} we have $\supp(v_{\varepsilon;\sigma})\Subset \tx{B}_{0}$. Furthermore, as $\mathcal{w}_{\varepsilon}^{-}(0)=u_{\tx{l}}(0)=0$, cf. \eqref{u0u0}, \eqref{fw}, for $\varepsilon,\sigma$ fixed, by continuity we can find $r_{\varepsilon;\sigma}\equiv r_{\varepsilon;\sigma}(\varepsilon,\sigma,\mathcal{w}_{\varepsilon}^{-},u)$ as above  so small that $\snr{\mathcal{w}_{\varepsilon}^{-}(x)-u_{\tx{l}}(x)}\le \sigma/2$ holds for all $x\in B_{r_{\varepsilon;\sigma}}(0)$, so $v_{\varepsilon;\sigma}=0$ on $B_{r_{\varepsilon;\sigma}}(0)$. All in all, we have just shown that
\eqn{a.22}
$$
\begin{cases}
    \displaystyle
    \ u_{\tx{l}}+v_{\varepsilon;\sigma}\in W^{1,\infty}(\tx{B}_{0}), \qquad \quad &\nr{Du_{\tx{l}}+Dv_{\varepsilon;\sigma}}_{L^{\infty}(\tx{B}_{0})}\le 1\vspace{1.5mm}\\
    \displaystyle
    \ v_{\varepsilon;\sigma}\in W^{1,\infty}(\tx{B}_{0}),\qquad \quad &\supp(v_{\varepsilon;\sigma})\Subset \tx{B}_{0}\setminus B_{r_{\varepsilon;\sigma}}(0)\vspace{1.5mm}\\
    \displaystyle
    \ (Du_{\tx{l}}+Dv_{\varepsilon;\sigma})\mathds{1}_{\tx{E}^{-}_{\sigma}}=D\mathcal{w}_{\varepsilon}^{-}\mathds{1}_{\tx{E}^{-}_{\sigma}}.
\end{cases}
$$

\subsection*{Step 5: $L_{\loc}(n-1,s)$-anti-peeling.} The local minimality of $u_{\tx{l}}$, \eqref{a.22} and the strict convexity of $H$, cf. \eqref{0.1.1}, yield
\begin{eqnarray*}
0&\stackrel{\eqref{a.22}_{2}}{\ge}&\mathcal{E}_{\mu_{\tx{l}}}(u_{\tx{l}};\tx{B}_{0})-\mathcal{E}_{\mu_{\tx{l}}}(u_{\tx{l}}+v_{\varepsilon;\sigma};\tx{B}_{0})\\ &=&\mathcal{E}_{\mu_{\tx{l}}}(u_{\tx{l}};\tx{E}_{\sigma}^{-})-\mathcal{E}_{\mu_{\tx{l}}}(u_{\tx{l}}+v_{\varepsilon;\sigma};\tx{E}_{\sigma}^{-})\nonumber \\
&=&\int_{\tx{E}_{\sigma}^{-}}\mu_{\tx{l}} v_{\varepsilon;\sigma}\dx+\int_{\tx{E}_{\sigma}^{-}}H(Du_{\tx{l}})-H(Du_{\tx{l}}+Dv_{\varepsilon;\sigma})\dx\nonumber \\
&\stackrel{\eqref{0.1.1}}{\ge}&\int_{\tx{E}_{\sigma}^{-}}\mu_{\tx{l}} v_{\varepsilon;\sigma}\dx-\int_{\tx{E}_{\sigma}^{-}}\langle \partial H(Du_{\tx{l}}+Dv_{\varepsilon;\sigma}),Dv_{\varepsilon;\sigma}\rangle\dx\nonumber \\
&\stackrel{\eqref{a.22}_{2,3}}{=}&\int_{\tx{E}_{\sigma}^{-}}\mu_{\tx{l}} v_{\varepsilon;\sigma}\dx-\int_{\tx{B}_{0}}\langle \partial H(D\mathcal{w}_{\varepsilon}^{-}),Dv_{\varepsilon;\sigma}\rangle\dx\nonumber \\
&=&\int_{\tx{E}_{\sigma}^{-}}\mu_{\tx{l}} v_{\varepsilon;\sigma}\dx+\int_{\tx{B}_{0}}\diver (\partial H(D\mathcal{w}_{\varepsilon}^{-}) )v_{\varepsilon;\sigma}\dx\nonumber \\
&\stackrel{\eqref{a.22},\eqref{a.8}}{=}&\int_{\tx{E}_{\sigma}^{-}}\mu_{\tx{l}} v_{\varepsilon;\sigma}\dx+\frac{1}{\Lambda_{\varepsilon}}\int_{\tx{B}_{0}}v_{\varepsilon;\sigma}\dx.
\end{eqnarray*}
Notice that above we used that $\mathcal{w}_{\varepsilon}^{-}$ is strictly spacelike outside zero, cf. \eqref{fw.1}, and that $\supp(v_{\varepsilon;\sigma})\Subset \tx{B}_{0}\setminus \{0\}$ so that $\diver(\partial H(D\mathcal{w}_{\varepsilon}^{-}))$ makes sense pointwise, see \eqref{a.8} and \eqref{a.22}$_{2}$. The content of the previous display, Fatou's lemma, and the monotone convergence theorem imply
\begin{flalign}\label{a.30}
\frac{1}{\Lambda_{\varepsilon}}\nr{(\mathcal{w}_{\varepsilon}^{-}-u_{\tx{l}})_{+}}_{L^{1}(\tx{B}_{0})}&\le \frac{1}{\Lambda_{\varepsilon}}\liminf_{\sigma\to 0} \int_{\tx{B}_{0}}v_{\varepsilon;\sigma}\dx\nonumber \\
&\le \liminf_{\sigma\to 0} \int_{\tx{E}_{\sigma}^{-}}\snr{\mu_{\tx{l}}} v_{\varepsilon;\sigma}\dx\le \int_{\tx{B}_{0}}\snr{\mu_{\tx{l}}}(\mathcal{w}_{\varepsilon}^{-}-u_{\tx{l}})_{+}\dx,
\end{flalign}
so to conclude we need to control the term on the right-hand side of \eqref{a.30}. For $L\ge1$ to be fixed, define
$$
\tx B_{0;L}
:=
\tx B_0\cap\{|\mu_{\tx{l}}|>L\}.
$$
When $n\ge3$, recalling that $1<s<\infty$, we use
the following interpolation inequality in Lorentz spaces:
\eqn{inter}
$$
\nr{(\mathcal w_\varepsilon^--u_{\tx{l}})_+}
_{L\left(\frac{n-1}{n-2},\frac{s}{s-1}\right)(\tx B_0)}
\le c(n,s)
\nr{(\mathcal w_\varepsilon^--u_{\tx{l}})_+}
_{L^1(\tx B_0)}^{\frac{n-2}{n-1}}
\nr{(\mathcal w_\varepsilon^--u_{\tx{l}})_+}
_{L^\infty(\tx B_0)}^{\frac1{n-1}}.
$$
By H\"older's inequality in Lorentz spaces, \eqref{inter}, and
\eqref{a.25}$_2$, it follows that
\begin{flalign*}
\int_{\tx B_{0;L}}
|\mu_{\tx{l}}|
(\mathcal w_\varepsilon^--u_{\tx{l}})_+\dx
&\le
c(n,s)
\nr{\mu_{\tx{l}}}_{L(n-1,s)(\tx B_{0;L})}
\nr{(\mathcal w_\varepsilon^--u_{\tx{l}})_+}
_{L\left(\frac{n-1}{n-2},\frac{s}{s-1}\right)(\tx B_0)}
\\
&\le
\tx c_5
\Lambda_\varepsilon^{\frac2{n-1}}
\nr{\mu_{\tx{l}}}_{L(n-1,s)(\tx B_{0;L})}
\nr{(\mathcal w_\varepsilon^--u_{\tx{l}})_+}
_{L^1(\tx B_0)}^{\frac{n-2}{n-1}},
\end{flalign*}
where $\tx c_5=\tx c_5(n,s)$.
When $n=2$, 
\eqref{a.25}$_2$ plainly gives 
$$
\int_{\tx B_{0;L}}
|\mu_{\tx{l}}|
(\mathcal w_\varepsilon^--u_{\tx{l}})_+\dx
\le
\nr{\mu_{\tx{l}}}_{L^1(\tx B_{0;L})}
\nr{(\mathcal w_\varepsilon^--u_{\tx{l}})_+}
_{L^\infty(\tx B_0)}
\le
\tx c_6\Lambda_\varepsilon^2
\nr{\mu_{\tx{l}}}_{L^1(\tx B_{0;L})},
$$
where $\tx c_6>0$ is a dimensional constant.
On $\tx B_0\setminus\tx B_{0;L}$, we simply use
$|\mu_{\tx{l}}|\le L$.
Combining these estimates with \eqref{a.30}, we obtain
\eqn{a.50}
$$
\begin{aligned}
&
(1/\Lambda_\varepsilon-L)
\nr{(\mathcal w_\varepsilon^--u_{\tx{l}})_+}
_{L^1(\tx B_0)}
\\
&\qquad\le
\tx c_5\mathds{1}_{\{n\ge3\}}
\Lambda_\varepsilon^{\frac2{n-1}}
\nr{\mu_{\tx{l}}}_{L(n-1,s)(\tx B_{0;L})}
\nr{(\mathcal w_\varepsilon^--u_{\tx{l}})_+}
_{L^1(\tx B_0)}^{\frac{n-2}{n-1}}
\\
&\qquad\quad+
\tx c_6\mathds{1}_{\{n=2\}}
\Lambda_\varepsilon^2
\nr{\mu_{\tx{l}}}_{L^1(\tx B_{0;L})}.
\end{aligned}
$$
with $\tx{c}_{5},\tx{c}_{6}\equiv \tx{c}_{5},\tx{c}_{6}(n,s)$. By the absolute continuity of the Lorentz norm (keep in mind that $s<\infty$), we choose $L\equiv L(n,\tx{l},\mu,s)$ large enough to have
\eqn{l.05}
$$
\frac{2\tx{c}_{5}}{\tx{c}_{4}^{1/(n-1)}} \mathds{1}_{\{n\ge 3\}}\nr{\mu_{\tx{l}}}_{L(n-1,s)(\tx{B}_{0;L})}+ \frac{2\tx{c}_{6}}{\tx{c}_{4}} \mathds{1}_{\{n= 2\}}\nr{\mu_{\tx{l}}}_{L^{1}(\tx{B}_{0;L})}\le \frac{1}{4},
$$
and then fix the threshold $\ti{\varepsilon}_{0}$ as
\eqn{tie0.4}
$$
\ti{\varepsilon}_{0}:=\min\left\{\frac12,\frac{M_{\tx{l}}}{4L}\right\}.
$$
Estimate \eqref{a.50} yields
\begin{eqnarray*}
\frac{\tx{c}_{4}^{1/(n-1)}}{2} \Lambda_{\varepsilon}^{\frac{2}{n-1}}&\stackrel{\eqref{a.26}}{\le}&\frac{1}{2\Lambda_{\varepsilon}}\nr{(\mathcal{w}_{\varepsilon}^{-}-u_{\tx{l}})_{+}}_{L^{1}(\tx{B}_{0})}^{\frac{1}{n-1}}\nonumber\\
&\stackrel{\eqref{a.50}}{\le}&\tx{c}_{5}\mathds{1}_{\{n\ge 3\}}\Lambda_{\varepsilon}^{\frac{2}{n-1}}\nr{\mu_{\tx{l}}}_{L(n-1,s)(\tx{B}_{0;L})}+\tx{c}_{6}\mathds{1}_{\{n=2\}}\Lambda_{\varepsilon}^{2}\nr{\mu_{\tx{l}}}_{L^{1}(\tx{B}_{0;L})}.
\end{eqnarray*}
Simplifying above and recalling \eqref{l.05} and \eqref{tie0.4} we obtain
$$
1\le \frac{2\tx{c}_{5}}{\tx{c}_{4}^{1/(n-1)}} \mathds{1}_{\{n\ge 3\}}\nr{\mu_{\tx{l}}}_{L(n-1,s)(\tx{B}_{0;L})}+ \frac{2\tx{c}_{6}}{\tx{c}_{4}}\mathds{1}_{\{n= 2\}}\nr{\mu_{\tx{l}}}_{L^{1}(\tx{B}_{0;L})} \stackrel{\eqref{l.05}}{\le}\frac{1}{4},
$$
a contradiction. The proof is complete.
\begin{remark}
\emph{In \eqref{mmmmm}, the Lorentz condition is in force only in dimension $n\ge 3$, while for $n=2$ we assume the (slightly stronger) $L^{1}_{\loc}(\Omega)\equiv L_{\loc}(1,1)(\Omega)$. This is due to the fact that for $s>1$, and any fixed ball $B\Subset \Omega$, the H\"older-type bound $\nr{fg}_{L^{1}(B)}\lesssim_{n,s}\nr{f}_{L(1,s)(B)}\nr{g}_{L^{\infty}(B)}$ fails already for $f\in L(1,s)(B)$ and $g=1$. }   
\end{remark}

\section{Hausdorff dimension of light segments and Theorem \ref{intre}}\label{capacity} 

Here we consider a weakly spacelike function $u$, not necessarily being a minimizer of the Born--Infeld functional, and show how the integrability and the differentiability of certain nonlinear functions of $\mathcal{h}_{u}$, such as $\mathcal h_u^q\in L^1_{\loc}(\mathbb R^n)$, imply Hausdorff dimension estimates of the sets of light segments as defined in Section \ref{luce}. Eventually, we derive further qualitative properties of such sets. Note that a condition of the type
$\mathcal h_u^q\in L^1_{\loc}(\mathbb R^n)$ for some $q>0$ automatically implies that $|Du|<1$ almost everywhere.  We start with the 
\begin{proof}[Proof of Theorem \ref{intre}]
If \(\mathcal S_{\texttt{i},u}=\varnothing\), there is nothing to
prove. We therefore fix 
$z\in\mathcal S_{\texttt{i},u}.$

{\em Step 1: A close-to-lightness estimate}.
We may choose a light segment
$\overline{y_zx_z}$
such that \(z\) is its midpoint, and label its endpoints so that
$u(y_z)>u(x_z)$.
Set
$\tx{l}_z:=|x_z-y_z|>0$,
$\tx{e}_z:=(y_z-x_z)/\tx{l}_z$,
so that
$y_z=z+\tx{l}_z\tx{e}_z/2$,
$x_z=z-\tx{l}_z\tx{e}_z/2$,
and
$u(y_z)=u(x_z)+\tx{l}_z$.
Fix
$
0<\rr<\tx{l}_z/2^{3}.
$
We now decompose
$$
x=z+t\tx{e}_z+x_\perp, \qquad |t|<\rr, \qquad \langle x_\perp,\tx{e}_z\rangle=0, \qquad |x_\perp|<\rr,
$$
and define
$$
\begin{cases}
B_\rr^{n-1}:=
\{x_\perp\in\tx{e}_z^\perp:|x_\perp|<\rr\}\\[1mm]
\mathcal C_\rr(z):=\left\{z+t\tx{e}_z+x_\perp:|t|<\rr,\ x_\perp\in B_\rr^{n-1}\right\}
\Subset B_{2\rr}(z).
\end{cases}
$$
The notation here is obvious: $\tx{e}_z^\perp$ denotes the space orthogonal to $\tx{e}_z$. 
The \(1\)-Lipschitz regularity of \(u\), together with the elementary
inequality
$\sqrt{a^2+b^2}\le a+b^2/(2a)$ valid whenever \(a>0\),
gives
\eqn{cc.0}
$$\left|u(x)-\left(u(x_z)+\frac{\tx{l}_z}{2}+t\right)\right|
\le\frac{4|x_\perp|^2}{3\tx{l}_z}
\qquad
\mbox{for every }x\in\mathcal C_\rr(z).
$$
Indeed,
\begin{flalign*}
u(x)
&\le
u(x_z)+|x-x_z|=u(x_z)+
\sqrt{\left(t+\frac{\tx{l}_z}{2}\right)^2
+|x_\perp|^2}
\\
&\le
u(x_z)+t+\frac{\tx{l}_z}{2}+\frac{|x_\perp|^2}{2t+\tx{l}_z}\le
u(x_z)+t+\frac{\tx{l}_z}{2}+\frac{4|x_\perp|^2}{3\tx{l}_z},
\end{flalign*}
and, similarly,
\begin{flalign*}
u(x)&\ge u(y_z)-|y_z-x|=u(x_z)+\tx{l}_z-
\sqrt{\left(\frac{\tx{l}_z}{2}-t\right)^2+
|x_\perp|^2}
\\
&\ge
u(x_z)+t+\frac{\tx{l}_z}{2}-
\frac{|x_\perp|^2}{\tx{l}_z-2t}
\ge u(x_z)+t+\frac{\tx{l}_z}{2}-\frac{4|x_\perp|^2}{3\tx{l}_z}.
\end{flalign*}
Merging the inequalities in the last two displays leads to
\eqref{cc.0}. Next, for every \(x_\perp\in B_\rr^{n-1}\), define
$
\mathcal f_{x_\perp}(t)
:=
u(z+t\tx{e}_z+x_\perp) $ for 
$t\in[-\rr,\rr].$
Since \(u\) is \(1\)-Lipschitz, \(\mathcal f_{x_\perp}\) is
\(1\)-Lipschitz, and hence absolutely continuous, for every
\(x_\perp\in B_\rr^{n-1}\). Moreover, by the slicing properties of
Sobolev functions, for a.e. \(x_\perp\in B_\rr^{n-1}\),
$
\mathcal f_{x_\perp}'(t)=\left\langle Du(z+t\tx{e}_z+x_\perp),\tx{e}_z \right\rangle
$
for a.e. \(t\in(-\rr,\rr)\). By \eqref{cc.0},
$$
\int_{-\rr}^{\rr}\left(1-\mathcal f_{x_\perp}'(t)\right)\dt
=
2\rr-\left(\mathcal f_{x_\perp}(\rr)
-
\mathcal f_{x_\perp}(-\rr)\right)\le\frac{8|x_\perp|^2}{3\tx{l}_z}.
$$
Since $1-|Du|^2\le2(1-|Du|)\le2(1-\mathcal f_{x_\perp}'),$
we obtain the following {\em close-to-lightness} estimate:
\eqn{cc.2}
$$
\int_{-\rr}^{\rr}
\left(1-|Du(z+t\tx{e}_z+x_\perp)|^2\right)\dt
\le\frac{16|x_\perp|^2}{3\tx{l}_z}.
$$
\medskip

{\em Step 2: Consequences of higher integrability}.
Define
\eqn{choice}
$$
\mathcal h_{u;\rr}^{q}(x_\perp)
:=\mint_{-\rr}^{\ \ \rr}\mathcal h_u(z+t\tx{e}_z+x_\perp)^q\dt.
$$
Since
$
s\longmapsto s^{-q/2}
$
is convex on \((0,\infty)\), Jensen's inequality and \eqref{cc.2}
give
\eqn{higher.light}
$$
\mathcal h_{u;\rr}^{q}(x_\perp)
\ge
\left(\mint_{-\rr}^{\ \ \rr}\left(1-|Du(z+t\tx{e}_z+x_\perp)|^2\right)\dt\right)^{-q/2}
\ge \left(\frac{3\rr\tx l_z}{8|x_\perp|^2}\right)^{q/2}
$$
for a.e.
\(x_\perp\in B_\rr^{n-1}\setminus\{0\}\). Consequently, for every \(0<\delta<\rr\), Fubini's theorem and
the inclusion $\mathcal C_\rr(z)\Subset B_{2\rr}(z)$ yield
\begin{align}\label{higher.int.lower}
\int_{B_{2\rr}(z)}\mathcal h_u^q\dx
&\ge
\int_{\mathcal C_\rr(z)}\mathcal h_u^q\dx
\nonumber =2\rr\int_{B_\rr^{n-1}}\mathcal h_{u;\rr}^{q}(x_\perp)\dx_\perp
\nonumber\\
&\ge \frac{\tx l_z^{q/2}\rr^{1+q/2}}{c}\int_{\{\delta<|x_\perp|<\rr\}}\frac{dx_\perp}{|x_\perp|^q}=c(n,q)\tx l_z^{q/2}\rr^{1+q/2}\int_\delta^\rr s^{n-2-q}\,\dd s.
\end{align}
Assume first that
$
q\ge n-1.
$
Letting \(\delta\downarrow0\) in \eqref{higher.int.lower}, we obtain
$$
\int_{B_{2\rr}(z)}\mathcal h_u^q\dx=\infty,
$$
which is a contradiction. 
Thus
$
\mathcal S_{\texttt{i},u}=\varnothing
$ and 
hence
$
\mathcal S_{u}=\varnothing=\mathcal{K}_{u}=\overline{\mathcal S_{u}}=\varnothing. 
$ This proves the second conclusion in \rif{nicchia4}. 
Assume now that
$
1\le q<n-1.
$
Letting \(\delta\downarrow0\) in \eqref{higher.int.lower}, we find
\eqn{findy}
$$
\int_{B_{2\rr}(z)}\mathcal h_u^q\dx \ge
\frac{\tx l_z^{q/2}\rr^{1+q/2}}{c(n,q)}\int_0^\rr s^{n-2-q}\,\dd s=\frac{\tx l_z^{q/2}\rr^{n-q/2}}{c(n,q)}.
$$
Define
$$
E_q
:=\left\{
x\in\mathbb R^n: \limsup_{r\to0}\, r^{q/2-n}\int_{B_r(x)}\mathcal h_u^q\dx>0\right\}.
$$
Estimate \rif{findy} shows that
$
\mathcal S_{\texttt{i},u}\subset E_q.
$
Since \(\mathcal h_u^q\in L^1_{\loc}(\mathbb R^n)\) we have 
$
\mathcal H^{n-q/2}(E_q)=0 
$ by \cite[Theorem 2.10]{eg15}, 
and therefore 
$
\mathcal H^{n-q/2}(\mathcal S_{\texttt{i},u})=0.
$ This implies, via Lemma \ref{endpoints} below, that $
\mathcal H^{n-q/2}
(\mathcal S_{u})
=0
$ and the proof is complete. 
\end{proof}
\begin{lemma}
\label{endpoints}
Let $U\subset\mathbb R^n$ be open and let
$u\in W^{1,\infty}_{\loc}(U)$ be weakly spacelike.
For every $s>0$,
$
\mathcal H^s(\mathcal S_u)=0$ if and only if 
$\mathcal H^s(\mathcal S_{\texttt{i},u})=0$.
In particular,
$
\dim_{\mathcal H}\mathcal S_u=\dim_{\mathcal H}\mathcal S_{\texttt{i},u}.
$
\end{lemma}
\begin{proof} Since $\mathcal S_{\texttt{i},u}\subset\mathcal S_u$,
it is sufficient to prove that
$\mathcal H^s(\mathcal S_{\texttt{i},u})=0$ implies
$
\mathcal H^s(\mathcal S_u)=0$, and therefore we assume that 
\eqn{walter}
$$\mathcal H^s(\mathcal S_{\texttt{i},u})=0.$$
The idea of the proof is that every midpoint of a light segment
determines its two endpoints once the length is fixed.
Since the directions depend Lipschitz continuously on the midpoints,
this association has Lipschitz character too.
This implies that the endpoints do not spread too much,
so that the information in \rif{walter} extends to the endpoints too.
Choose a countable covering
$
U=\bigcup_{j\ge1}B_j
$
by balls $B_j\Subset U$.
For $\delta>0$, let $M_{\delta,j}$ be the set of midpoints
of light segments of length $2\delta$ contained in $B_j$. 
Clearly,
$
M_{\delta,j}\subset\mathcal S_{\texttt{i},u}.
$
For $y\in M_{\delta,j}$, choose the corresponding increasing
unit direction $\tx{e}_y$, i.e.,
$
u(y\pm\delta \tx{e}_y)=u(y)\pm\delta$.
This direction is unique by Lemma \ref{lemmaraggi}.
Next, consider the endpoint maps
$
F_{\delta,j}^\pm:M_{\delta,j}\to\mathcal S_u$
defined by
$
F_{\delta,j}^\pm(y):=y\pm\delta \tx{e}_y.
$
Upon extending the restriction $u|_{B_j}$ to 
$1$-Lipschitz extension to $\mathbb R^n$, we can use 
 \cite[Lemma 16, Remark 17 and the proof of Lemma 22]{cfm02}, to deduce that the dependence 
of $\tx{e}_y$ 
on $y\in M_{\delta,j}$ is Lipschitz continuously
on $y$, for every fixed $\delta>0$. It follows that the maps $F_{\delta,j}^\pm$ are Lipschitz too. Now, every point $p\in \mathcal S_u$ is an endpoint of a
nondegenerate light subsegment $\ell$ contained in some $B_j$ and, up to shorten $\ell$, we can assume $\diam(\ell)=2/k$ for some integer $k$. The midpoint of $\ell$ belongs to $M_{1/k,j}$.
As $p$ was generic, we conclude with 
$$
\mathcal S_u
\subset
\bigcup_{j,k\ge1}
\left(
F_{1/k,j}^+(M_{1/k,j})
\cup
F_{1/k,j}^-(M_{1/k,j})
\right).
$$
By \rif{walter}, since $M_{1/k,j}\subset\mathcal S_{\texttt{i},u}$,
the Lipschitz continuity of the maps $F_{1/k,j}^{\pm}$ yields
$$
\mathcal H^s\left(F_{1/k,j}^\pm(M_{1/k,j})\right)=0
\qquad
\mbox{for every }j,k\ge1.
$$
Countable subadditivity of Hausdorff measure then yields
$
\mathcal H^s(\mathcal S_u)=0.
$
\end{proof}
\begin{theorem}[Sobolev exclusion of light segments - higher dimensions]\label{sobre1}
Let \(u\in W^{1,\infty}_{\loc}(\mathbb R^n)\) be weakly spacelike. Let $p\geq 1$  and $m>0$ such that $
\mathcal h_u^m\in W^{1,p}_{\loc}(\mathbb R^n)
$ and 
\eqn{notache}
$$
\frac{n-1}{1+m}\le p<n-1.
$$
Then
$
\mathcal S_{\texttt{i},u}=\mathcal S_{u}=\mathcal{K}_{u}=\varnothing.
$
\end{theorem}
\begin{proof} Note that assumption \rif{notache} automatically implies $n\geq 3$. 
Assume by contradiction that
\(z\in\mathcal S_{\texttt{i},u}\). Using the notation and the
close-to-lightness estimate from the proof of Theorem
\ref{intre}, restart from \rif{higher.light} (there we replace $q$ by $m$ and observe that the whole argument remains unchanged). 
Then $|x_\perp|^{-m} \lesssim \mathcal h_{u;\rr}^{m}$
for a.e.
\(x_\perp\in B_\rr^{n-1}\setminus\{0\}\); integrating then yields
\eqn{lefty}
$$
\int_{B_\rr^{n-1}}
(\mathcal h_{u;\rr}^{m})^{p^*}\dx_\perp
\ge
\frac1c
\int_0^\rr s^{n-2-mp^*}\,\dd s,
\qquad
p^*:=\frac{(n-1)p}{n-1-p}.
$$
By the assumption on \(\mathcal h_u^m\) and Fubini's theorem,
$
\mathcal h_{u;\rr}^{m}\in W^{1,p}(B_\rr^{n-1}).
$
Sobolev embedding in dimension \(n-1\) gives
$
\mathcal h_{u;\rr}^{m}\in L^{p^*}(B_\rr^{n-1}),
$ so that the last integral in \rif{lefty} is finite. 
On the other hand, the assumption
$
(1+m)p\ge n-1
$
is equivalent to
$
mp^*\ge n-1,
$
and hence the last integral diverges, a contradiction. Therefore
$
\mathcal S_{\texttt{i},u}=\varnothing,
$
and consequently
$
\mathcal S_{u}=\mathcal{K}_{u}=\varnothing.
$
\end{proof}
\begin{theorem}[Sobolev exclusion of light segments - lower dimensions]
\label{sobre2}
Let
\(u\in W^{1,\infty}_{\loc}(\mathbb R^n)\) be weakly spacelike.
Assume that
$
\log\mathcal h_u\in W^{1,p}_{\loc}(\mathbb R^n)$ for some 
$p\ge n-1.$
Then
$
\mathcal S_{\texttt{i},u}=\mathcal S_{u}=
\mathcal{K}_{u}=\varnothing.
$
\end{theorem}
\begin{proof}
Assume by contradiction there exists 
$
z\in\mathcal S_{\texttt{i},u}.
$
This time it is natural to define
$$
\ell\mathcal h_{u;\rr}(x_\perp)
:=\mint_{-\rr}^{\ \ \rr}\log\mathcal h_u(z+t\tx{e}_z+x_\perp)\dt,
$$
which replaces the choice in \rif{choice}. 
Again Fubini's theorem yields $
\ell\mathcal h_{u;\rr}
\in W^{1,p}(B_\rr^{n-1}).$
Moreover, Jensen's inequality and \eqref{cc.2} give
\eqn{zz.0}
$$
\ell\mathcal h_{u;\rr}(x_\perp)
\ge
\log\left(
\frac{1}{|x_\perp|}
\sqrt{\frac{3\rr\tx{l}_z}{8}}
\right)
$$
for a.e. \(x_\perp\in B_\rr^{n-1}\setminus\{0\}\). When  \(p>n-1\), Morrey's embedding theorem gives a locally bounded
continuous representative of
\(\ell\mathcal h_{u;\rr}\), contradicting \eqref{zz.0} as
\(x_\perp\to0\). Therefore it remains to consider the borderline case \(p=n-1\). 
If \(n=2\), then \(p=1\), and $
\ell\mathcal h_{u;\rr}\in W^{1,1}(-\rr,\rr) $
has an absolutely continuous, bounded representative.
This again contradicts \eqref{zz.0}. We may therefore assume that \(n\ge3\), so that \(p=n-1>1\). If
$
\left\|
\partial_{x_\perp}\ell\mathcal h_{u;\rr}
\right\|_{L^p(B_\rr^{n-1})}=0,
$
then $\ell\mathcal h_{u;\rr}$ is constant almost everywhere,
contradicting \eqref{zz.0}.
Otherwise we apply 
\eqref{trudinger.mean} of Lemma \ref{trudinger.lemma},
in dimension $n-1=p$, to obtain
\eqn{zz.1}
$$
\int_{B_\rr^{n-1}}\exp\left\{c_*\left|\ell\mathcal h_{u;\rr}-(\ell\mathcal h_{u;\rr})_{B_\rr^{n-1}}
\right|^{\frac{p}{p-1}}
\right\}\dx_\perp
<\infty,
$$
where 
$c_*>0$ depends on $
n,p,
\left\|
\partial_{x_\perp}
\ell\mathcal h_{u;\rr}
\right\|_{L^p(B_\rr^{n-1})}$. 
Choose \(\rr_\perp\in(0,\rr)\) so small that
\[
\frac12
\log\left(
\frac1{\rr_\perp}
\sqrt{\frac{3\rr\tx{l}_z}{8}}
\right)
\ge
(\ell\mathcal h_{u;\rr})_{B_\rr^{n-1}}.
\]
Then \rif{zz.0} implies that
$$
\left|\ell\mathcal h_{u;\rr}(x_\perp)-(\ell\mathcal h_{u;\rr})_{B_\rr^{n-1}}\right|
\ge
\frac12\log\left(\frac1{|x_\perp|}\sqrt{\frac{3\rr\tx{l}_z}{8}}\right)
$$
holds for a.e. \(x_\perp\in B_{\rr_\perp}^{n-1}\). 
Set
$ c_1:=c_*2^{-p/(p-1)}$ and $c_2:=\sqrt{3\rr\tx{l}_z/8}.$
Using polar coordinates and then changing variables via 
$
t=\log (c_2/s),
$
we obtain
\begin{flalign*}
\int_{B_\rr^{n-1}}\exp\left\{c_*\left|\ell\mathcal h_{u;\rr}-(\ell\mathcal h_{u;\rr})_{B_\rr^{n-1}}\right|^{\frac{p}{p-1}}
\right\}\dx_\perp
&\ge
\frac1c
\int_0^{\rr_\perp}
s^{n-2}
\exp\left\{
c_1
\left[
\log\left(\frac{c_2}{s}\right)
\right]^{\frac{p}{p-1}}
\right\}\,\dd s
\\
&=
\frac{c_2^{n-1}}{c}
\int_{\log(c_2/\rr_\perp)}^\infty
\exp\left\{
c_1t^{\frac{p}{p-1}}-(n-1)t
\right\}\,\textnormal{d}t.
\end{flalign*}
Since $
p/(p-1)>1,$
the last integral is infinite, contradicting \eqref{zz.1}.\end{proof}
We now switch to capacitary estimates and structure results. Note that some of the conclusions of the following Theorem might be obtained as corollaries of Theorems \ref{intre} and \ref{sobre2}, when the parameter $p$ is large enough to meet their assumptions. It is otherwise more general as it works under greater generality. 
\begin{theorem}\label{cap.t.1}
 Let $u\in W^{1,\infty}_{\loc}(\mathbb{R}^{n})$ be a weakly spacelike function such that $Du\in W^{1,p}_{\loc}(\er^n;\er^n)$, with $1\le p<n$. It holds that
 \begin{itemize}
 \item  $\mathcal{K}_{u}
=
\mathcal S_{u}\,\dot\cup\,\mathcal {L}_{u}.$
 \item Up to zero $p$-capacity sets, every point of $\mathcal {L}_{u}$ with $\snr{Du}<1$ is a discontinuity point of the precise representative of $\snr{Du}$.
\item If either 
$
\mathcal h_{u}^{m}\in W^{1,p}_{\loc}(\mathbb R^n)
$ for some $m >0$, 
or
$
\log \mathcal h_u\in W^{1,p}_{\loc}(\mathbb R^n) 
$ holds,  then 
\eqn{zerocap}
$$
\begin{cases}
\displaystyle
\ccap(\mathcal I_{u})=0 \quad  \mbox{and}  \quad  \dim_{\mathcal H}(\mathcal I_{u})\leq n-p
\\[1mm]
\displaystyle
\ccap(\mathcal S_{u})=0 \quad \mbox{and}  \quad  \dim_{\mathcal H}(\mathcal S_{u})\leq n-p. 
\end{cases}
$$
If, in addition,
\eqn{ndeg}
$$
\inf\left\{\diam(\ell):\ell\mbox{ is a light ray}\right\}>0,
$$
then
\eqn{nicchia5}
$$
\mathcal{K}_{u}=\mathcal S_{u},
\qquad
\mathcal {L}_{u}=\varnothing,
\qquad
\ccap(\mathcal{K}_{u})=0,
\qquad
\dim_{\mathcal H} \mathcal{K}_{u}\le n-p.
$$
\end{itemize}
\end{theorem}
The proof of Theorem \ref{cap.t.1} - put at the end of this section - will be obtained via a series of intermediate propositions, some of which are of their own interest and contain partial and additional results under weaker assumptions. 

\begin{remark}[Canonical exceptional set]\label{loc.rem}
\emph{In the following we shall extensively make use of the usual $p$-Capacity in $\er^n$ as described in \cite[Section 4]{eg15}, and therefore with $1\leq p<n$. The basic relevant property \cite[Theorem 4.19]{eg15} is of course that if a function $v\in W^{1,p}(A)$, where $A\subset \er^n$ is any open subset, then there exists an exceptional Borel set $E\equiv E(v)\subset A$ such that $\ccap(E)=0$ and $z$ is a Lebesgue point for $v$ whenever $z\not \in E$. When proving that a certain set $\mathcal U$ has zero $p$-Capacity by standard properties of the $p$-Capacity \cite[Theorem 4.15 (viii)]{eg15} it will be sufficient to prove that $\textnormal{Cap}_{p}(\mathcal U\cap B)=0$ for every ball $B\subset \er^n$. We shall furthermore denote $\mathcal {U}^B:= \mathcal U \cap B$ and $\mathcal {U}^{\bar B}:= \mathcal U \cap \bar B$. }
\end{remark}  
We start with a very simple lemma, whose proof is reported for completeness. 
\begin{lemma}\label{blow-up}
Let $u\in W^{1,\infty}_{\loc}(\mathbb R^n)$ be weakly
spacelike and assume that
$\log\mathcal h_u\in L^1_{\loc}(\mathbb R^n)$.
For every $z\in\mathbb R^n$, one has
\eqn{unalog}
$$
\lim_{r\to0}(\mathcal h_u^{-1})_{B_r(z)}=0
\quad\Longrightarrow\quad
\lim_{r\to0}(\log\mathcal h_u)_{B_r(z)}=\infty.
$$
\end{lemma}

\begin{proof}
The integrability assumption implies that
$0<\mathcal h_u^{-1}\le1$ almost everywhere.
By Jensen's inequality applied to the convex function
$t\mapsto-\log t$,
\[
(\log\mathcal h_u)_{B_r(z)}
=
(-\log\mathcal h_u^{-1})_{B_r(z)}
\ge
-\log\bigl((\mathcal h_u^{-1})_{B_r(z)}\bigr).
\]
The right-hand side tends to $+\infty$ as $r\to0$,
which proves the assertion. 
\end{proof}
\begin{proposition}[$p$-capacity of $\mathcal I_u$ and $\mathcal S_u$]
\label{cap.t}
Let $1\le p<n$ and let
$u\in W^{1,\infty}_{\loc}(\mathbb R^n)$ be weakly spacelike.
Assume that
$
Du\in W^{1,p}_{\loc}(\mathbb R^n;\mathbb R^n)$ and $
\log\mathcal h_u\in W^{1,p}_{\loc}(\mathbb R^n).
$
Then \eqref{zerocap} holds.
In particular, the conclusion holds if the assumption on
$\log\mathcal h_u$ is replaced by
$\mathcal h_u^m\in W^{1,p}_{\loc}(\mathbb R^n)$
for some $m>0$.
\end{proposition}
\begin{proof}
It suffices to prove the assertions about capacity.
By Remark \ref{loc.rem}, fix an arbitrary ball
$B\subset\mathbb R^n$ and prove that
$
\ccap(\mathcal I_u^B)=\ccap(\mathcal S_u^B)=0.
$ Let $E_0:=E_0(Du)$ and $E_1:=E_1(\log\mathcal h_u)$
be the exceptional sets of non-Lebesgue points from
Remark \ref{loc.rem}. Then
$
\ccap(E_0^B)=\ccap(E_1^B)=0.
$
If $z\in\mathcal I_u^B\setminus E_0^B$, then $z$ is
a Lebesgue point of $Du$, hence also of $|Du|$, and
$|Du(z)|=1$.
Since $t\mapsto\sqrt{1-t^2}$ is H\"older continuous
on $[0,1]$, it follows that
\[
\lim_{r\to0}(\mathcal h_u^{-1})_{B_r(z)}
=\sqrt{1-|Du(z)|^2}=0.
\]
By \eqref{unalog} we infer 
$$
\lim_{r\to0}(\log\mathcal h_u)_{B_r(z)}=\infty,
$$
so $z\in E_1^B$. Therefore, $
\mathcal I_u^B\subset E_0^B\cup E_1^B.
$
By \eqref{aff.3}$_{1,2}$, every point of a light segment
which is a Lebesgue point of $Du$ satisfies $|Du|=1$,
including the endpoints. Hence
$
\mathcal S_u^B\setminus E_0^B\subset\mathcal I_u^B,
$
and consequently
$
\mathcal I_u^B\cup\mathcal S_u^B
\subset E_0^B\cup E_1^B.
$
Monotonicity and subadditivity of the $p$-capacity give
the desired conclusion on $B$.
Since $B$ is arbitrary, Remark \ref{loc.rem}
yields \eqref{zerocap}. Finally, if
$\mathcal h_u^m\in W^{1,p}_{\loc}(\mathbb R^n)$
for some $m>0$, then $\mathcal h_u^m\ge1$ almost everywhere.
Composition with the Lipschitz function
$t\mapsto m^{-1}\log\max\{1,t\}$ gives
$\log\mathcal h_u\in W^{1,p}_{\loc}(\mathbb R^n)$,
so the preceding argument applies.
\end{proof}
\begin{proposition}[Structure of $\mathcal K_u$]
\label{lem:structure-Kmu}
Let
$
u\in W^{1,\infty}_{\loc}(\mathbb R^n)
$
be weakly spacelike. Then
\eqn{structure.K}
$$
\mathcal K_u
=
\mathcal S_u\,\dot\cup\,\mathcal L_u.
$$
In particular, if $1\le p<n$ and
$
\ccap(\mathcal S_u)=0,
$
then
$
\ccap\left(
\mathcal K_u\mathbin{\triangle}\mathcal L_u
\right)
=0.
$
Assume, in addition, that
$
Du\in W^{1,p}_{\loc}(\mathbb R^n;\mathbb R^n),
$
and let $E_0=E_0(Du)$ be the exceptional set from
Remark \ref{loc.rem}. Every point
$
z\in\mathcal L_u\setminus E_0
$
such that
$
|Du(z)|<1
$
is a discontinuity point of the precise representative of $|Du|$.
\end{proposition}
\begin{proof}
Since $\mathcal S_u\subset\mathcal K_u$ by definition and
$\mathcal S_u\cap\mathcal L_u=\varnothing$ by \eqref{vuoto},
it remains to prove the inclusions
$
\mathcal L_u\subset\mathcal K_u$ and 
$\mathcal K_u\setminus\mathcal S_u\subset\mathcal L_u.$ Let first $z\in\mathcal L_u$. There exist light rays $\ell_k$ such that
$
\dist(z,\ell_k)\to0$
and 
$\diam(\ell_k)\to0.
$
Choose $z_k\in\ell_k$ such that
$
|z-z_k|
\le
\dist(z,\ell_k)+1/k.
$
It follows
that $z_k\in\mathcal S_u$. Hence
$
z\in\overline{\mathcal S_u}=\mathcal K_u.
$
Conversely, let
$
z\in\mathcal K_u\setminus\mathcal S_u.
$
Choose $z_k\in\mathcal S_u$ such that $z_k\to z$, and let $\ell_k$ be
the light ray containing a light segment through $z_k$. Then
$\dist(z,\ell_k)\le|z-z_k|\to0.$
We claim that
$\diam(\ell_k)\to0.$
Otherwise, after passing to a subsequence, there exists $\rho>0$ such
that
$
\diam(\ell_k)\ge4\rho
$
for every $k$. Since $\ell_k$ is an interval of a straight line and
$z_k\in\ell_k$, we may choose $w_k\in\ell_k$ such that
$|w_k-z_k|=\rho.$
The segment $\overline{z_kw_k}$ is obviously a light segment. After passing to a
further subsequence,
$\rho^{-1}(w_k-z_k)\to \tx{e}$
for some $\tx{e}\in \er^n$, $|\tx{e}|=1$, and therefore
$w_k\to w:=z+\rho \tx{e}.$
By continuity of $u$,
$|u(w)-u(z)|
=|w-z|=\rho.$
Thus $\overline{zw}$ is a nondegenerate light segment, contradicting
$
z\notin\mathcal S_u.
$
Consequently,
$
\diam(\ell_k)\to0,
$
and hence $z\in\mathcal L_u$. This proves \eqref{structure.K}. Since
$
\mathcal K_u\mathbin{\triangle}\mathcal L_u=\mathcal S_u,
$
the capacity assertion follows immediately. Finally, let
$
z\in\mathcal L_u\setminus E_0
$
with $|Du(z)|<1$, and let $\ell_k$ be as in
\eqref{language3}. We can choose
$
\tilde z_k\in\mathcal S_{\texttt{i},u}\cap\ell_k
$
such that
$
|\tilde z_k-z|
\le
\dist(z,\ell_k)+\diam(\ell_k)+1/k.
$
Then $\tilde z_k\to z$, and Lemma \ref{lemmaraggi} gives
$
|Du(\tilde z_k)|=1.
$
Since $z\notin E_0$ and $|Du(z)|<1$, the precise representative of
$|Du|$ is discontinuous at $z$.
\end{proof}

\begin{proposition}[Nondegenerate light-ray lengths]
\label{lem:nondegenerate-Kmu}
Let
$
u\in W^{1,\infty}_{\loc}(\mathbb R^n)
$
be weakly spacelike and assume that \eqref{ndeg} holds. Then
$
\mathcal K_u=\mathcal S_u
$ and $
\mathcal L_u=\varnothing.
$
In particular, if $1\le p<n$ and
$
\ccap(\mathcal S_u)=0,
$
then
$
\ccap(\mathcal K_u)=0
$ and 
$\dim_{\mathcal H}\mathcal K_u\le n-p.$
\end{proposition}
\begin{proof}
Condition \eqref{ndeg} and the definition in \rif{language3} immediately
give
$\mathcal L_u=\varnothing.$
The identity
$\mathcal K_u=\mathcal S_u$
then follows from Proposition \ref{lem:structure-Kmu}. The remaining
assertions are immediate.
\end{proof}\begin{proof}[Proof of Theorem \ref{cap.t.1}] The proof follows by combining the above propositions. The first two assertions follow from Proposition
\ref{lem:structure-Kmu}, which gives the disjoint decomposition
$
\mathcal{K}_{u}
=
\mathcal S_{u}\,\dot\cup\,\mathcal {L}_{u}
$
and the stated discontinuity property outside an exceptional set of
zero \(p\)-capacity. Finally, Propositions \ref{cap.t} and \ref{lem:nondegenerate-Kmu} imply the third point. \end{proof}


\section{Intrinsic potential estimates and potentials based on Lorentzian balls}\label{potest}
Here we build on the Michael \& Simon \cite{ms73} type monotonicity/mean value estimates in \cite{bs82}, see also \cite{bi19,bi23}. For the rest of the section we recommend the reader to keep in mind the notation and the definitions given in Section \ref{recapsec}, and the abbreviations in \rif{abbreviation}. 

\begin{proposition}\label{p31}
Let $u\in \mathbb{X}(\er^n)\cap C^{\infty}(\mathbb{R}^{n})$ be a classical solution of \eqref{bi} with $\mu\in C^{\infty}_{0}(\mathbb{R}^{n})$. There exists $\gamma\equiv \gamma(n)\in (0,1/n)$ such that the intrinsic potential estimate
\begin{flalign}\label{0.10}
\mathcal h_u(x_0)^{-\gamma}
&\ge \mint_{K_r(x_0)}\mathcal h_u^{-(\gamma+1)}\dx
+\frac{c_n}{r^{n-2}}\int_{K_{r/2}(x_0)}\snr{D^2u}^2\dx
\nonumber\\
&\quad-2\mathbf{PL}_{2,u}^{1}(\mu^2;x_0,r)-\mathbf{PL}_{1,u}^{1}(\mu;x_0,r)
\end{flalign}
holds for all $x_0\in\mathbb R^n$ and every $r>0$,
with a constant $c_n>0$ depending only on $n$.
Moreover, the following choice is possible:
\eqn{sceltagamma}
$$
\gamma = \frac{1}{4n}\,.
$$
\end{proposition}
\begin{proof}
As pointed out in Remark \ref{rem0}, we may assume that $x_0=0$ and
$u(0)=0$, namely $\Phi(0)=(0,u(0))=0$. Throughout the proof, we write
$L_s:=L_s^u(0)$, $K_s:=K_s^u(0)$, and
$\ell(\Phi(x))=\ell_u(x,0)$. Proposition \ref{prop.mon} gives
\eqn{diffy}
$$
\deris \mathbb  A(s) \leq - \deris\mathbb  B(s) - c_{\star}\mathbb  E(s) + \mathbb  E_\mu(s) - \mathbb  F_\mu(s)
$$
for all $s\in(0,r)$, with $c_\star=c_\star(n)>0$. We leave the term $c_{\star}\mathbb  E(s)$ as it is for the moment. In the following we shall use the basic relations
 \eqn{basic}
$$
\begin{cases}
\displaystyle
\ \snr{\delta \ell}_{\mathbb{L}^{n+1}}^{2}=1+\ell^{-2}\langle\vv,\Phi\rangle_{\mathbb{L}^{n+1}}^{2}\vspace{1.5mm}\\ \displaystyle
\ \snr{\mathcal{h}_{u}^{-1}\delta_{n+1}\ell}\le \snr{\delta \ell}_{\mathbb{L}^{n+1}},
\end{cases}
$$
for which we refer to \cite[(2.9) and (3.14)]{bi23}. Set \(\varphi_s:=(s^2-\ell^2)/2\). By Lemma \ref{ibp}\footnote{Strictly speaking, we apply Lemma \ref{ibp} to a smooth approximation of
\(\varphi_s\mathds 1_{L_s}\); no boundary term appears since
\(\varphi_s=0\) on \(\partial L_s\).}, with
\(w=\mathcal h_u^{-(\gamma+1)}\mu\), and since \(\tx H_u=\mu\), 
$$
\int_{L_s}\varphi_s\,\delta_{n+1}w\,dA
=
\int_{L_s}\varphi_s\mathcal h_u^{-\gamma}\mu^2\,dA
+
\int_{L_s}\ell\,\mathcal h_u^{-(\gamma+1)}\mu\,\delta_{n+1}\ell\,dA.
$$
Hence, using \(0\le \varphi_s\le s^2/2\) on \(L_s\),
$$
\mathbb E_\mu(s)
\le
\frac{1+4\gamma}{8s^{n-1}}
\int_{L_s}\mathcal h_u^{-\gamma}\mu^2\,dA
+
\frac{\gamma}{s^{n+1}}
\int_{L_s}
\ell\mathcal h_u^{-\gamma}|\mu|\,
|\mathcal h_u^{-1}\delta_{n+1}\ell|\,dA .
$$
Again using \rif{basic} we gain
$$
\mathbb E_\mu(s)
\le\frac{1+4\gamma}{8s^{n-1}}\int_{L_{s}}\mathcal{h}_{u}^{-\gamma}\mu^{2}\d A+\frac{\gamma}{s^{n}}\int_{L_{s}}\mathcal{h}_{u}^{-\gamma}\snr{\mu} \d A  +\frac{\gamma}{s^{n}}\int_{L_{s}}\mathcal{h}_{u}^{-\gamma}\snr{\mu}\ell^{-1}\snr{\langle \vv,\Phi\rangle_{\mathbb{L}^{n+1}}}\d A. 
$$
By Young's inequality we continue to bound
\begin{flalign}
\notag \mathbb E_\mu(s)
&\le\frac{1+4\gamma}{8s^{n-1}}\int_{L_{s}}\mathcal{h}_{u}^{-\gamma}\mu^{2}\d A+\frac{\gamma}{s^{n}}\int_{L_{s}}\mathcal{h}_{u}^{-\gamma}\snr{\mu}\d A\nonumber \\
\notag&\quad +\frac{\gamma}{2s^{n}}\int_{L_{s}}\mathcal{h}_{u}^{-\gamma}\mu^{2}\ell \d A+\frac{\gamma}{2s^{n}}\int_{L_{s}}\mathcal{h}_{u}^{-\gamma}\ell^{-3}\snr{\langle\Phi,\vv\rangle_{\mathbb{L}^{n+1}}}^{2}\d A\nonumber \\
&\le\frac{1+8\gamma}{8s^{n-1}}\int_{L_{s}}\mathcal{h}_{u}^{-\gamma}\mu^{2}\d A+\frac{\gamma}{s^{n}}\int_{L_{s}}\mathcal{h}_{u}^{-\gamma}\snr{\mu}\d A+\frac{\gamma}{2s^{n}}\int_{L_{s}}\mathcal{h}_{u}^{-\gamma}\ell^{-3}\snr{\langle\Phi,\vv\rangle_{\mathbb{L}^{n+1}}}^{2}\d A. \label{you1}
\end{flalign}
Similarly, again by Young's inequality, we find
\eqn{you2}
$$
\snr{\mathbb F_\mu(s)}\le\frac{1}{2s^{n-1}}\int_{L_{s}}\mathcal{h}_{u}^{-\gamma}\mu^{2}\d A+\frac{1}{2s^{n+1}}\int_{L_{s}}\mathcal{h}_{u}^{-\gamma}\ell^{-2}\snr{\langle\Phi,\vv\rangle_{\mathbb{L}^{n+1}}}^{2}\d A.
$$
The last two integrals appearing in \rif{you1}-\rif{you2} can be estimated using the coarea formula (Lemma \ref{coa}) as follows:
\begin{flalign*}
\frac{1}{s^{n}}\int_{L_{s}}\mathcal{h}_{u}^{-\gamma}\ell^{-3}\snr{\langle\Phi,\vv\rangle_{\mathbb{L}^{n+1}}}^{2}\d A&=-\frac{1}{n-1}\deris\left(s^{1-n}\int_{L_{s}}\mathcal{h}_{u}^{-\gamma}\ell^{-3}\snr{\langle\Phi,\vv\rangle_{\mathbb{L}^{n+1}}}^{2}\d A\right)\nonumber \\
&\quad +\frac{1}{n-1}\int_{\partial L_{s}}\mathcal{h}_{u}^{-\gamma}\ell^{-2-n}\snr{\langle\Phi,\vv\rangle_{\mathbb{L}^{n+1}}}^{2}\snr{\delta \ell}_{\mathbb{L}^{n+1}}^{-1}\d\sigma\nonumber \\
&=-\frac{1}{n-1}\deris\left(s^{1-n}\int_{L_{s}}\mathcal{h}_{u}^{-\gamma}\ell^{-3}\snr{\langle\Phi,\vv\rangle_{\mathbb{L}^{n+1}}}^{2}\d A\right)\nonumber \\
&\quad +\frac{1}{n-1}\deris\left(\int_{L_{s}}\mathcal{h}_{u}^{-\gamma}\ell^{-2-n}\snr{\langle\Phi,\vv\rangle_{\mathbb{L}^{n+1}}}^{2}\d A\right),
\end{flalign*}
and 
\begin{flalign*}
\frac{1}{s^{n+1}}\int_{L_{s}}\mathcal{h}_{u}^{-\gamma}\ell^{-2}\snr{\langle\Phi,\vv\rangle_{\mathbb{L}^{n+1}}}^{2}\d A&= -\frac{1}{n}\deris\left(s^{-n}\int_{L_{s}}\mathcal{h}_{u}^{-\gamma}\ell^{-2}\snr{\langle\Phi,\vv\rangle_{\mathbb{L}^{n+1}}}^{2}\d A\right)\nonumber \\
&\quad +\frac{1}{n}\int_{\partial L_{s}}\mathcal{h}_{u}^{-\gamma}\ell^{-2-n}\snr{\langle\Phi,\vv\rangle_{\mathbb{L}^{n+1}}}^{2}\snr{\delta \ell}_{\mathbb{L}^{n+1}}^{-1}\d\sigma\nonumber \\
&=-\frac{1}{n}\deris\left(s^{-n}\int_{L_{s}}\mathcal{h}_{u}^{-\gamma}\ell^{-2}\snr{\langle\Phi,\vv\rangle_{\mathbb{L}^{n+1}}}^{2}\d A\right)\nonumber \\
&\quad +\frac{1}{n}\deris\left(\int_{ L_{s}}\mathcal{h}_{u}^{-\gamma}\ell^{-2-n}\snr{\langle\Phi,\vv\rangle_{\mathbb{L}^{n+1}}}^{2}\d A\right).
\end{flalign*}
Combining the last four displays with \rif{diffy} we obtain
\begin{flalign}\label{0.4}
\deris \mathbb  A(s)&\le - c_{\star} \mathbb  E(s)-\left(1-\frac{\gamma}{2(n-1)}-\frac{1}{2n}\right)\deris \mathbb  B(s)\nonumber \\
&\quad -\deris\left(\int_{L_{s}}\left(\frac{\gamma \mathcal{h}_{u}^{-\gamma}\ell^{-3}}{2(n-1)s^{n-1}}+\frac{ \mathcal{h}_{u}^{-\gamma}\ell^{-2}}{2ns^{n}}\right)\snr{\langle\Phi,\vv\rangle_{\mathbb{L}^{n+1}}}^{2}\d A\right)\nonumber \\
&\quad +\left(\frac{5+8\gamma}{8s^{n-1}}\right)\int_{L_{s}}\mathcal{h}_{u}^{-\gamma}\mu^{2}\d A+\frac{\gamma}{s^{n}}\int_{L_{s}}\mathcal{h}_{u}^{-\gamma}\snr{\mu}\d A.
\end{flalign}
Next, denote
\eqn{recapp}
$$
\mathcal{A}(s):=\int_{L_{s}}\left(\frac{\gamma \ell^{-1}}{2(n-1)s^{n-1}}+\frac{1}{2ns^{n}}+\left(1-\frac{n(\gamma+1)-1}{2n(n-1)}\right)\ell^{-n}\right)\frac{\snr{\langle\Phi,\vv\rangle_{\mathbb{L}^{n+1}}}^{2}}{\ell^{2}\mathcal{h}_{u}^{\gamma}}\d A \geq 0.
$$
With this notation, \eqref{0.4} implies
\eqn{0.9}
$$
 \deris \mathbb A(s)\le -c_{\star} \mathbb  E(s) -\deris \mathcal{A}(s)
+\frac{2\nr{\mathcal{h}_{u}^{-1}}_{L^{\infty}(L_{s})}^{\gamma}}{s^{n-1}}\int_{L_{s}}\mu^{2}\d A+\frac{\nr{\mathcal{h}_{u}^{-1}}_{L^{\infty}(L_{s})}^{\gamma}}{s^{n}}\int_{L_{s}}\snr{\mu}\d A.
$$
Since $u\in C^{\infty}(\mathbb R^n)$ is strictly spacelike, by Lemma \ref{limlem} we find 
\eqn{lim1}
$$\lim_{s\to0} \mathbb A(s)= \omega_{n}\mathcal h_u(0)^{-\gamma}$$
and, moreover, 
 $$\langle \Phi,\vv\rangle_{\mathbb{L}^{n+1}}=-\mathcal{h}_{u}\left(u(x)-\langle x,Du(x)\rangle\right)=O(\snr{x}^{2}), \quad \ell(x)\approx \snr{x}$$
 for $|x|$ small enough. 
 It follows that 
 $$
 \lim_{s\to 0}\frac{1}{s^n}\int_{L_{s}}\snr{\langle\Phi,\vv\rangle_{\mathbb{L}^{n+1}}}^{2}\ell^{-2}\d A=0
 $$
and 
$$
\lim_{s\to 0}\int_{L_{s}}\snr{\langle\Phi,\vv\rangle_{\mathbb{L}^{n+1}}}^{2}\ell^{-n-2}\d A=\lim_{s\to 0}\int_{L_{s}}\snr{\langle\Phi,\vv\rangle_{\mathbb{L}^{n+1}}}^{2}s^{-n+1}\ell^{-3}\d A=0.
$$
Recalling \rif{recapp} we conclude with 
\eqn{lim2}
$$
\lim_{s\to 0} \mathcal A(s)=0\,.
$$
Using \rif{lim1} and \rif{lim2}, and recalling that
$\mathcal h_u^{-1}\le1$, we integrate \eqref{0.9} over
$(\varepsilon,r)$ and let $\varepsilon\downarrow0$ to obtain
\begin{flalign*}
\mathbb A(r) &\le \omega_{n}\mathcal h_u(0)^{-\gamma}-\mathcal{A}(r) - c_{\star} \int_0^r \mathbb E(s) \ds \nonumber \\
&\quad +2\int_{0}^{r}s^{2-n}\int_{L_{s}}\mu^{2}\d A\frac{\d s}{s}+\int_{0}^{r}s^{1-n}\int_{L_{s}}\snr{\mu}\d A\frac{\ds}{s}.
\end{flalign*}
We now proceed as in \cite[p.~140]{bs82}. Note that for any function $w\in L^{1}(\mathfrak{M})$ it holds that
$$
\int_{L_{s}}w\d A=\int_{\mathfrak{M}}\mathds{1}_{[0,\infty)}(s-\ell)w\d A,\qquad \quad \mathds{1}_{[0,\infty)}(t):=\begin{cases}
\ 1\quad &\mbox{if} \ \ t\ge 0\vspace{1.5mm}\\
\displaystyle
\ 0\quad &\mbox{if} \ \ t<0.
\end{cases}
$$
Therefore, also using Fubini's theorem, we have 
\begin{flalign*}
 \int_0^r \mathbb E(s) \ds 
& =\int_{\mathfrak{M}}\mathcal{h}_{u}^{2-\gamma}\left[\sum_{i,j=1}^{n}(D_{ij}u)^{2}+\sum_{j=1}^{n}\left(\sum_{i=1}^{n}\vv_{i}D_{ij}u\right)^{2}\right]\left(\int_{0}^{r}\frac{\mathds{1}_{[0,\infty)}(s-\ell)(s^{2}-\ell^{2})}{2s^{n+1}}\ds\right)\d A\nonumber \\
& \ge\int_{\mathfrak{M}}\mathcal{h}_{u}^{2-\gamma}\left(\sum_{i,j=1}^{n}(D_{ij}u)^{2}\right)\mathcal{J}(\ell,r)\d A \ge\int_{L_{r}}\mathcal{h}_{u}^{2-\gamma}\left(\sum_{i,j=1}^{n}(D_{ij}u)^{2}\right)\mathcal{J}(\ell,r)\d A \\
& \ge c_{n}r^{2-n}\int_{L_{r/2}}\mathcal{h}_{u}^{2-\gamma}\snr{D^{2}u}^{2}\d A,
\end{flalign*}
where
$$
\mathcal{J}(\ell,r):=
\begin{cases}
\displaystyle
\frac{\ell^{2-n}}{n(n-2)}
-\frac{r^{2-n}}{2(n-2)}
+\frac{\ell^{2}r^{-n}}{2n}
&\mbox{if } n\ge3,\quad 0<\ell<r
\\[2mm]
\displaystyle
\frac12\log\left(\frac r\ell\right)
-\frac14\left(1-\frac{\ell^2}{r^2}\right)
&\mbox{if } n=2,\quad 0<\ell<r
\\[2mm]
0
&\mbox{if } \ell\ge r.
\end{cases}
$$
We also used that $\mathcal{J}(\ell,r)\ge c_{n}r^{2-n}$ whenever $\ell\in (0,r/2]$. Recalling that $\mathcal{A}(r)\ge 0$, combining the preceding estimates gives
\begin{flalign*}
\omega_{n}\mathcal h_u(0)^{-\gamma}&\ge \frac{1}{r^{n}}\int_{L_{r}}\mathcal{h}_{u}^{-\gamma}\d A+\frac{c_n}{r^{n-2}}\int_{L_{r/2}}\mathcal{h}_{u}^{2-\gamma}\snr{D^{2}u}^{2}\d A\nonumber \\
&\quad -2\int_{0}^{r}s^{2-n} \int_{L_{s}}\mu^{2}\d A \frac{\d s}{s}-\int_{0}^{r}s^{1-n} \int_{L_{s}}\snr{\mu}\d A \frac{\ds}{s}.
\end{flalign*}
Translating the previous estimate back to $(x_0,u(x_0))$, we obtain 
\begin{flalign*}
\omega_{n}\mathcal{h}_u(x_0)^{-\gamma}&\ge \frac{1}{r^{n}}\int_{L_{r}(x_{0})}\mathcal{h}_{u}^{-\gamma}\d A+\frac{c_n}{r^{n-2}}\int_{L_{r/2}(x_{0})}\mathcal{h}_{u}^{2-\gamma}\snr{D^{2}u}^{2}\d A\nonumber \\
&\quad -2\int_{0}^{r}s^{2-n} \int_{L_{s}(x_{0})}\mu^{2}\d A \frac{\d s}{s}-\int_{0}^{r}s^{1-n} \int_{L_{s}(x_{0})}\snr{\mu}\d A \frac{\ds}{s}.
\end{flalign*}
Recalling that $B_{r}(x_{0})\subset K_{r}(x_{0})$ and that $\gamma\in (0,1)$, we obtain 
\begin{flalign}
\mathcal{h}_u(x_0)^{-\gamma}&\ge\mint_{K_{r}(x_{0})}\mathcal{h}_{u}^{-(\gamma+1)}\d x+\frac{c_{n}}{r^{n-2}}\int_{K_{r/2}(x_{0})}\snr{D^{2}u}^{2}\dx\nonumber \\
&\quad -\frac{2}{\omega_{n}}\int_{0}^{r}s^{2-n} \int_{K_{s}(x_{0})}\mu^{2}\d x \frac{\d s}{s}-\frac{1}{\omega_{n}}\int_{0}^{r}s^{1-n} \int_{K_{s}(x_{0})}\snr{\mu}\dx \frac{\ds}{s},
\end{flalign}
which proves \eqref{0.10}. Finally, the choice in \eqref{sceltagamma} is admissible in the computation leading to
\cite[(3.8)]{bi23}: applying the trace inequality and
Young's inequality with parameter $1/4$ in
\cite[(3.6)]{bi23} yields $c_\star=1/(8n)$ and a coefficient
at most $1/4$ for the term $\mathcal h_u^{-\gamma}\mu^2$.
\end{proof}
\noindent The integral quantities appearing on the right-hand side of \eqref{0.10} feature the same homogeneity as the classical Havin--Maz'ya potentials \cite{hm72}. Let us show how these objects relate.
\begin{corollary}\label{cor.pot}
Within the same setting as Proposition \ref{p31},
for every $x_0\in\mathbb R^n$ there exists
$c_0=c_0(x_0,u)\ge1$ such that
\begin{flalign}\label{0.10.1}
\mathcal h_u(x_0)^{-\gamma}
&\ge
\mint_{K_r(x_0)}\mathcal h_u^{-(\gamma+1)}\dx
+
\frac{c_n}{r^{n-2}}
\int_{K_{r/2}(x_0)}\snr{D^2u}^2\dx
\nonumber\\
&\quad
-2c_0^{n-2}\mathbf P_2^1(\mu^2;x_0,c_0r)
-c_0^{n-1}\mathbf P_1^1(\mu;x_0,c_0r)
\end{flalign}
for every $r\in(0,1]$.
If additionally $\mathcal h_->0$, see \eqref{h-},
then \eqref{0.10.1} holds uniformly in $x_0\in\mathbb R^n$
and for every $r>0$, with $\mathcal h_-^{-1}$
replacing $c_0$. Moreover, let $0<r<c_0^{-1}$ and set
$B:=B_{c_0r}(x_0)$. Then
\begin{equation}\label{cor.pot.norms}
\begin{aligned}
&\mint_{K_r(x_0)}
\mathcal h_u^{-(\gamma+1)}\dx
+
\frac{c_n}{r^{n-2}}
\int_{K_{r/2}(x_0)}\snr{D^2u}^2\dx
\\
&\quad\le
\mathcal h_u(x_0)^{-\gamma}
+
\begin{cases}
c_n\left(
c_0^{n-2}[\mu]_{n,2;B}^2
+c_0^{n-1}[\mu]_{n,1;B}
\right),
& \mbox{if} \ \  n\ge3,
\\[2mm]
c_\alpha\left(
\|\mu\|_{L^2(\log L)^\alpha(B)}^2
+c_0\|\mu\|_{L^2(\log L)^\alpha(B)}
\right),
&\mbox{if} \ \  n=2,
\end{cases}
\end{aligned}
\end{equation}
where $\alpha>2$ in the two-dimensional case.
If $\mathcal h_->0$, estimate \eqref{cor.pot.norms}
holds uniformly in $x_0$ for $0<r<\mathcal h_-$,
with $c_0=\mathcal h_-^{-1}$.
\end{corollary}

\begin{proof}
Apply the comparisons in \rif{javier} with
$(\sigma,\vartheta,w)=(2,1,\mu^2)$ and $(1,1,\mu)$
to the two potential terms in \eqref{0.10}.
This gives \eqref{0.10.1}.
When $\mathcal h_->0$, the uniform comparison applies
with $c_0=\mathcal h_-^{-1}$ for every centre and radius,
yielding the asserted uniform version of \eqref{0.10.1}. In order to prove \eqref{cor.pot.norms}, set $R:=c_0r<1$ and
$m:=\mu\mathds1_B$, extended by zero outside $B$.
The potentials of $m$ and $\mu$ centred at $x_0$
coincide up to radius $R$.
If $n\ge3$, Lemma \ref{crit}, applied with
$(\sigma,\vartheta,w)=(2,1,m^2)$ and $(1,1,m)$,
together with \eqref{normalo2}, gives
$$
\mathbf P_2^1(\mu^2;x_0,R)
 \lesssim_n
[m^2]_{n/2,1;\mathbb R^n}
\lesssim_n
[\mu]_{n,2;B}^2, \qquad 
\mathbf P_1^1(\mu;x_0,R)
\lesssim_n
[\mu]_{n,1;B}.
$$
If $n=2$, Cauchy--Schwarz gives
$
\mathbf P_1^1(m;x_0,R)
\le
\mathbf P_1^{1/2}(m^2;x_0,R).
$
Moreover, by the definition of the potentials,
$$
\begin{aligned}
\mathbf P_2^1(m^2;x_0,R)
&=
\frac1{\omega_2}
\int_0^R
\|m\|_{L^2(B_\rho(x_0))}^2
\frac{d\rho}{\rho}
\\
&\le
\frac{\|m\|_{L^2(\mathbb R^2)}}{\sqrt{\omega_2}}
\mathbf P_1^{1/2}(m^2;x_0,R).
\end{aligned}
$$
The second part of Lemma \ref{crit} and
$
\|m\|_{L^2(\mathbb R^2)}
\le \|\mu\|_{L^2(\log L)^\alpha(B)}
$
therefore imply
$$
\begin{aligned}
\mathbf P_1^1(\mu;x_0,R)
&\lesssim_\alpha
\|\mu\|_{L^2(\log L)^\alpha(B)},
\\
\mathbf P_2^1(\mu^2;x_0,R)
&\lesssim_\alpha
\|\mu\|_{L^2(\log L)^\alpha(B)}^2.
\end{aligned}
$$
Substituting these bounds into \eqref{0.10.1}
proves \eqref{cor.pot.norms}.
Finally, when $\mathcal h_->0$, the choice
$c_0=\mathcal h_-^{-1}$ makes the condition
$c_0r<1$ equivalent to $r<\mathcal h_-$,
and yields the last assertion.
\end{proof}
\begin{remark}[Other intrinsic potentials]
\label{rem.intrinsic.potentials}
{\em Intrinsic potentials of a different type also occur in the nonlinear potential theory of degenerate parabolic equations of the type
\eqn{appearpar}
$$
u_t-\diver\, (|Du|^{p-2}Du)=\mu\,, \qquad p> 2-\frac{1}{n+1}
$$
in cylinders of the type $\Omega \times \er$, $\Omega \subset \er^n$, as pioneered in  \cite{KM13,km14}.
There, the intrinsic cylinders
\[
Q_\varrho^\lambda(x_0,t_0)
:=
B_\varrho(x_0)\times
(t_0-\lambda^{2-p}\varrho^2,t_0),
\qquad p\ge2,
\]
give rise to the potentials
$$
I_{\beta,\lambda}^{\mu}(x_0,t_0;r):=
\int_0^r
\frac{|\mu|(Q_\varrho^\lambda(x_0,t_0))}
{\varrho^{n+2-\beta}}
\frac{d\varrho}{\varrho},
\qquad 0<\beta\le n+2,
$$
where $\mu$ is the measure datum appearing in \rif{appearpar}.
The parameter $\lambda$ is linked to the solution by the
condition in \cite[Theorem 1.1]{km14}. More precisely, for $p\ge2$, \cite[Theorem 1.1]{km14}
gives the implication
\[
cI_{1,\lambda}^{\mu}(x_0,t_0;r)
+
c\left(
\mint_{Q_r^\lambda(x_0,t_0)}
|Du|^{p-1}\,dx\,\dd t
\right)^{1/(p-1)}
\le\lambda
\quad\Longrightarrow\quad
|Du(x_0,t_0)|\le\lambda,
\]
where $c=c(n,p)>1$, provided that
$Q_r^\lambda(x_0,t_0)$ is contained in the domain
and $(x_0,t_0)$ is a Lebesgue point of $Du$.
Thus $\lambda$ determines the integration cylinder
and simultaneously controls the potential and the
gradient average computed on it, tying the geometry
to the solution. Such estimates yield sharp regularity
criteria and demonstrate the intrinsic geometry approach
going back to DiBenedetto \cite{dib93}.
In the present setting, the corresponding dependence is
encoded directly in the sets $K_\varrho^u(x_0)$. In both constructions, the potential is therefore evaluated
on a geometry tied to the function being estimated.}
\end{remark}

\section{Decay at infinity}\label{boostinf}
A  significant consequence of \eqref{0.10} is a bound at
infinity for $\mathcal h_u$. This is stated in the following:
\begin{proposition}\label{com.sup}
Let
$
u\in\mathbb X(\mathbb R^n)\cap C^\infty(\mathbb R^n)
$
be the minimizer of \eqref{bi.en} with charge
$
\mu\in C^\infty_0(\mathbb R^n).
$
If $n=2$, assume in addition that $\mu\in L^1_0(\er^2)$ (so that $u$ is a classical solution to \eqref{bi} by Proposition \ref{exex}).
Then
\eqn{boost.at.infinity}
$$
\lim_{|x|\to\infty}\mathcal h_u(x)=1.
$$
Moreover, there exist a radius
$
\tx r_\mu
\equiv
\tx r_\mu(n,\mu)
\in(0,\infty)
$
and dimensional constants
$c_0(n)\ge4$
and
$\mathfrak s_n>0$
such that, setting
\eqn{soglie}
$$
\mathcal h_\infty(\mu)
:=
c_0(n)
\left(1+\nr{\mu}_{\mathcal D^{1,2}(\mathbb R^n)^*}^2\right)
\geq 4,
$$
one has
\eqn{0.17}
$$
\mathcal h_u
\le
\frac{\mathcal h_\infty(\mu)}{2}
\qquad
\mbox{in }
\mathbb R^n\setminus B_{\tx r_\mu}(0)
$$
and
\eqn{stimamis}
$$
\left| \left\{x\in\mathbb R^n:\mathcal h_u(x)>\mathcal h_\infty(\mu)\right\}\right|
\le\mathfrak s_n.
$$
In particular,
\eqn{external.compact.support}
$$
\supp \left(\mathcal h_u-\mathcal h_\infty(\mu) \right)_+\Subset B_{\tx r_\mu}(0).
$$
\end{proposition}
\begin{proof}
We strongly rely on \eqref{0.10}. For this, we first recall that
by Proposition \ref{exex}, $u$ is strictly spacelike
and solves \eqref{bi} in the classical sense.
Let $\gamma$ be the exponent in Proposition \ref{p31}.
Fix a dimensional constant $c_0(n)\ge4$.
We first prove \eqref{boost.at.infinity} using only the equation
and the fact that $Du\in L^2(\mathbb R^n)$.
Fix $\tau>1$ and choose $R>0$ such that
\eqn{barry}
$$
\supp\,\mu\Subset B_R(0),
\qquad
\int_{\mathbb R^n\setminus B_R(0)}|Du|^2\dx
\le\omega_n(1-\tau^{-\gamma}).
$$
By Lemma \ref{lemmino} and Lemma \ref{balls}
\textnormal{(h$_2$)} when $n\ge3$, and by Lemma
\ref{balls} \textnormal{(h$_3$)} when $n=2$,
there exists $R_1\equiv R_1(\tau)>R$ such that
$
K_1(x_0)\cap\overline{B_R(0)}=\varnothing
$
whenever $|x_0|\ge R_1$.
Indeed, in dimensions $n\ge3$ the function $u$ is bounded,
so the Euclidean radii containing $K_1(x_0)$ can be
chosen independently of $x_0$.
In dimension two, apply \textnormal{(h$_3$)} with
the compact set $\overline{B_R(0)}$ and $s=1$.
Since $K_s(x_0)\subset K_1(x_0)$ for $0<s\le1$,
both potentials involving $\mu$ in \eqref{0.10}
vanish at such centers when $r=1$.
Further discarding the nonnegative Hessian term, we obtain
$$
\begin{aligned}
\mathcal h_u(x_0)^{-\gamma}
&\ge
\mint_{K_1(x_0)} (1-|Du|^2)^{(\gamma+1)/2}\dx
\ge 1-\mint_{K_1(x_0)}|Du|^2\dx\\
&\ge 1-\frac1{\omega_n} \int_{\mathbb R^n\setminus B_R(0)}|Du|^2\dx \stackrel{\eqref{barry}}{\ge}
\tau^{-\gamma}.
\end{aligned}
$$
Here we used $0<(\gamma+1)/2<1$ and
$B_1(x_0)\subset K_1(x_0)$.
We deduce that
$
\mathcal h_u(x_0)\le\tau
$
whenever $|x_0|\ge R_1(\tau)$.
Since $\tau>1$ is arbitrary and $\mathcal h_u\ge1$,
this proves \eqref{boost.at.infinity}.
Taking $\tau=2$, we obtain
$
\mathcal h_u(x_0)\le2 \le\mathcal h_\infty(\mu)/2
$
whenever $|x_0|\ge R_1(2)$.
This proves \eqref{0.17} with
$\tx r_\mu:=2R_1(2)$.
Note that since $u$ is uniquely determined by $\mu$,
this radius depends only on $n$ and $\mu$.
Moreover,
$
\supp\,(\mathcal h_u-\mathcal h_\infty(\mu))_+
\subset\overline{B_{R_1(2)}(0)}
\Subset B_{\tx r_\mu}(0),
$
which proves \eqref{external.compact.support}.
It remains to establish \eqref{stimamis}. 
By \eqref{l1l1} we have 
\[
\|Du\|_{L^2(\mathbb R^n)}^2
\le
\int_{\mathbb R^n}\mathcal h_u|Du|^2\dx
\le
\langle\mu,u\rangle
\le
\|\mu\|_{\mathcal D^{1,2}(\mathbb R^n)^*}
\|Du\|_{L^2(\mathbb R^n)}.
\]
Thus
$
\|Du\|_{L^2(\mathbb R^n)}
\le
\|\mu\|_{\mathcal D^{1,2}(\mathbb R^n)^*},
$
and, finally
\eqn{jenny}
$$
\int_{\mathbb R^n}\mathcal h_u|Du|^2\dx
\le
\|\mu\|_{\mathcal D^{1,2}(\mathbb R^n)^*}^2.
$$
Set
$
E_0:=\{x\in\mathbb R^n:
\mathcal h_u(x)>\mathcal h_\infty(\mu)\}.
$
Since $\mathcal h_\infty(\mu)\ge4$, on $E_0$ we have
$|Du|^2=1-\mathcal h_u^{-2}\ge1/2$.
Therefore,
$$
\frac{\mathcal h_\infty(\mu)}2\,|E_0|
\le
\int_{E_0}\mathcal h_u|Du|^2\dx
\stackleq{jenny}
\|\mu\|_{\mathcal D^{1,2}(\mathbb R^n)^*}^2.
$$
Using \eqref{soglie}, we conclude
$$
|E_0|
\le
\frac{2}{c_0(n)}
\frac{\|\mu\|_{\mathcal D^{1,2}(\mathbb R^n)^*}^2}{1+\|\mu\|_{\mathcal D^{1,2}(\mathbb R^n)^*}^2}
\le\frac{2}{c_0(n)}
=:\mathfrak s_n.
$$
This proves \eqref{stimamis}.
\end{proof}

We finally close the section with a result that clarifies the classical connection between being a solution and being a minimizer in this setting. The material in the proof will also be useful later on, in Remark \ref{controre}.

\begin{proposition}[Classical solutions are minimizers]
\label{connessione}
Let
$u\in\mathbb X(\mathbb R^n)\cap C^\infty(\mathbb R^n)$
be a classical solution to \eqref{bi}, with
$\mu\in C^\infty_0(\mathbb R^n)$.
Then $\mu\in\DD^*$ and $u$ minimizes $\mathcal E_\mu$
over $\mathbb X(\mathbb R^n)$.
Moreover, if $n=2$, then $
\mu\in L^1_0(\er^2).
$
\end{proposition}

\begin{proof} The exterior bound established in the first part of the proof of Proposition \ref{com.sup} uses only the classical equation and $u\in\mathbb X(\mathbb R^n)$.
It therefore applies here and gives $\mathcal h_u\le2$ outside a sufficiently large ball.
Together with the continuity of $Du$ and strict spacelikeness on compact sets, this implies
$\mathcal h_u\in L^\infty(\mathbb R^n)$.
Since $Du\in L^2(\mathbb R^n)$, it follows that
$
F:=\mathcal h_uDu
\in L^2(\mathbb R^n;\mathbb R^n).
$ If $n=2$, choose $R>0$ such that $\supp\, \mu\Subset B_R(0)$.
To prove $\mu\in L^1_0(\mathbb R^2)$,
fix a nonincreasing function 
$\vartheta\in C^\infty(\mathbb R)$ such that
$0\le\vartheta\le1$, $\vartheta=1$ on $(-\infty,1]$,
and $\vartheta=0$ on $[2,\infty)$.
For $T>\max\{e^2,R+e\}$, define
\eqn{logatau}
$$
\eta_T(x)
:=\vartheta\left(\frac{\log(e+|x|)}{\log T}\right).
$$
Then $\eta_T\in C^\infty_0(\mathbb R^2)$,
$\eta_T=1$ in $B_{T-e}(0)$, and
$\eta_T=0$ outside $B_{T^2-e}(0)$.
In particular, $\eta_T=1$ on $\supp\, \mu$.
Moreover,
$$
|D\eta_T(x)|\le\frac{\|\vartheta'\|_{L^\infty}}{(e+|x|)\log T}\mathds1_{\{T-e<|x|<T^2-e\}}(x),
$$
for every $x\in \mathbb R^2$, so that
\eqn{cutoffa}
$$
\|D\eta_T\|_{L^2(\mathbb R^2)}^2\le
\frac{2\pi\|\vartheta'\|_{L^\infty}^2}{(\log T)^2}
\int_{T-e}^{T^2-e}\frac{\rho}{(e+\rho)^2}\,d\rho\le
\frac{2\pi\|\vartheta'\|_{L^\infty}^2}{\log T}.
$$
Testing $-\diver F=\mu$ directly with $\eta_T$, we obtain
$$
\left|\int_{\mathbb R^2}\mu\,dx\right|
=\left|\int_{\mathbb R^2}\langle F, D\eta_T\rangle\,dx\right|
\le
\frac{\sqrt{2\pi}\|\vartheta'\|_{L^\infty}}{\sqrt{\log T}}
\|F\|_{L^2(\mathbb R^2)}
\stackrel{T\to \infty}{\longrightarrow}0,
$$
which establishes $
\mu\in L^1_0(\er^2).
$ In every dimension, the equation gives
$$
\left|\int_{\mathbb R^n}\mu\varphi\,dx\right|
=\left|\int_{\mathbb R^n}\langle F, D\varphi\rangle \,dx\right|
\le
\|F\|_{L^2(\mathbb R^n)}
\|D\varphi\|_{L^2(\mathbb R^n)}
$$
for all $\varphi\in C^\infty_0(\mathbb R^n)$.
Hence $\mu\in\DD^*$, and a standard density argument, together with
Remark \ref{concreto}, gives
$
\langle\mu,w\rangle
=
\int_{\mathbb R^n}\langle F, Dw\rangle\,dx$
for every $w\in\DD$. 
Finally, another standard argument: the convexity of
$H(\xi)=1-\sqrt{1-|\xi|^2}$ implies that, for every
$v\in\mathbb X(\mathbb R^n)$,
$$
\mathcal E_\mu(v)-\mathcal E_\mu(u)\ge
\int_{\mathbb R^n}\langle F, D(v-u)\rangle\,dx
-\langle\mu,v-u\rangle =0.
$$
This proves the minimality of $u$.
\end{proof}


\section{Dirichlet problems with rough data}\label{diridiri}

\noindent In this section we turn our attention to the Dirichlet problem. Our
main goal is to derive a priori estimates for sufficiently regular
solutions under substantially weaker control of the charge. We build
on the strategy developed by Bartnik and Simon in \cite{bs82}, while
departing from it at several points. Their barrier arguments rely on a
pointwise bound for the prescribed mean curvature, whereas here we
consider equations of the form \(-\mathcal Mu=f+\mu\) in a bounded
domain \(\Omega\), where \(f\in L^\infty(\Omega)\) and the rough
perturbation \(\mu\) belongs to \(L(n,1)(\Omega)\) and satisfies a
smallness condition on arbitrarily small boundary collars, but not on the whole domain. This, among other things, 
requires a different treatment of the boundary barriers. In
particular, we combine the nonhomogeneous barriers of \cite{bs82} with
perturbative nonlinear correctors near the boundary, whose size is
controlled in Lorentz-Sobolev spaces. The endpoint space \(L(n,1)\)
is naturally suited to this construction, since the corresponding
Lorentz-Sobolev estimates allow one to preserve a quantitative
spacelike gap. The central tool is the new anti-peeling phenomenon for
rough charges established in Theorem \ref{al.t}. Throughout this section, we assume the following:
\eqn{assuntidir}
$$
\begin{cases}
\, \mbox{$\Omega$ is a $C^2$-regular, bounded domain of $\er^n$, $n\geq 2$}\\[2pt]
\, \mbox{$\tx{u}_0 \in C^2(\overline{\Omega})$ such that $\nr{D\tx u_0}_{C^0(\overline\Omega)}
\leq
1-\sigma_0$ for some $\sigma_0\in(0,1)$}. 
\end{cases}
$$
The number $\sigma_0$ will be established in the relevant statements, and additional assumptions will be put both on $\Omega$ and $\tx{u}_0$ when needed. 
\begin{lemma}[Spacelikeness via compactness]\label{ww.lem}
Under assumptions \eqref{assuntidir}, 
let
$
\Lambda_0>0$, 
$
M>0$, 
$\tx l>0.
$
Then there exists
$
\sigma
=
\sigma(n,\tx l,\Lambda_0,\sigma_0,M,\Omega)
\in(0,1)
$
with the following property. Let
$
u\in C^2(\Omega)\cap C^0(\overline\Omega)
$
satisfy
$$
\snr{Du(x)}<1 \mbox{ for every $x\in\Omega$  \quad $[-\mathcal Mu]_{n,1;\Omega}\le\Lambda_0$ \quad and $ u=\tx u_0$  
 on $\partial\Omega$}
 $$
and assume 
$
\nr{\tx u_0}_{C^2(\overline\Omega)}
\le M$.
Then, for every
$
x_0,x_1\in\overline\Omega
$
such that
$
\overline{x_1x_0}\subset\overline\Omega$ and 
$\tx l/4
\le
\snr{x_1-x_0}
\le
\tx l,
$
it holds that 
\eqn{proved}
$$
\snr{u(x_1)-u(x_0)}
\le
(1-\sigma)\snr{x_1-x_0}.
$$
\end{lemma}

\begin{proof}
In the following we shall repeatedly pass
to subsequences without relabeling them. We first observe that every function satisfying the assumptions of the
lemma is weakly spacelike along every segment contained in
$\overline\Omega$, namely
\eqn{equi}
$$
\snr{u(x)-u(y)}
\le
\snr{x-y}
\qquad
\mbox{whenever }
\overline{xy}\subset\overline\Omega.
$$
Indeed, set
$
\gamma(t):=x+t(y-x)
$
for $t\in[0,1]$. If
$
\gamma((0,1))\subset\Omega,
$
the claim follows from $\snr{Du}<1$ and continuity up to the
endpoints. Otherwise, setting
$
a:=\min\{t\in[0,1]:\gamma(t)\in\partial\Omega\},
$, $
b:=\max\{t\in[0,1]:\gamma(t)\in\partial\Omega\},
$
we have, by $\snr{Du}<1$ in $\Omega$ and continuity,
$
\snr{u(x)-u(\gamma(a))}
\le
\snr{x-\gamma(a)}$, $
\snr{u(\gamma(b))-u(y)}
\le
\snr{\gamma(b)-y}.
$
Moreover, since $u=\tx u_0$ on $\partial\Omega$ and
$\nr{D\tx u_0}_{C^0(\overline\Omega)}\le1-\sigma_0$,
$
\snr{u(\gamma(a))-u(\gamma(b))}
\le
(1-\sigma_0)\snr{\gamma(a)-\gamma(b)}.
$
Since the points are collinear and ordered along the segment, the
triangle inequality gives \eqref{equi}.

We argue by contradiction. Then there exist sequences of functions 
$
\{u_i\}\subset  C^2(\Omega)\cap C^0(\overline\Omega)$, $
\{\tx u_{0;i}\}\subset  C^2(\overline\Omega),
$
and points
$
\{x_{0;i} \},\{x_{1;i}\}\subset \overline\Omega
$
such that
$
\snr{Du_i}<1
$
 in $\Omega$, 
$
[-\mathcal Mu_i]_{n,1;\Omega}\le\Lambda_0$, 
$
u_i=\tx u_{0;i}
$ 
on $\partial\Omega,
$
with
\eqn{ww.unif.bd}
$$
\nr{\tx u_{0;i}}_{C^2(\overline\Omega)}
\le M,
\qquad
\nr{D\tx u_{0;i}}_{C^0(\overline\Omega)}
\le1-\sigma_0,
$$
and
$
\overline{x_{1;i}x_{0;i}}\subset\overline\Omega$, 
 $ \tx l/4 \le\snr{x_{1;i}-x_{0;i}}\le \tx l,
$
while
\eqn{pp.0}
$$
\snr{x_{1;i}-x_{0;i}}
\ge
\snr{u_i(x_{1;i})-u_i(x_{0;i})}
\ge
(1-1/i)\snr{x_{1;i}-x_{0;i}}.
$$

By compactness of $\overline\Omega$, we may assume that
$
x_{0;i}\to x_0,
$ and 
$x_{1;i}\to x_1.$
It follows that
$
\tx l/4
\le
\snr{x_1-x_0}
\le
\tx l,
$
and hence $x_0\ne x_1$. Moreover, since
$
\overline{x_{1;i}x_{0;i}}\subset\overline\Omega
$
for every $i$ and $\overline\Omega$ is closed, we have
$
\overline{x_1x_0}\subset\overline\Omega.
$ The uniform $C^2$-bound in \eqref{ww.unif.bd} and the compactness of
$\overline\Omega$ allow us, after passing to a further subsequence, to
find
$
\tx u_0\in C^1(\overline\Omega)
$
such that
\eqn{ww.bd.conv}
$$
\tx u_{0;i}\to\tx u_0
\qquad
\mbox{in }C^1(\overline\Omega).
$$
In particular,
$
\nr{D\tx u_0}_{C^0(\overline\Omega)}\le1-\sigma_0.
$
By defining 
$
u_0:=\left.\tx u_0\right|_{\partial\Omega},
$
we also have
$
\left.\tx u_{0;i}\right|_{\partial\Omega}
\to u_0$
uniformly in $\partial\Omega$. Furthermore, whenever
$y_0,y_1\in\partial\Omega$ and
$\overline{y_1y_0}\subset\overline\Omega$,
we have 
\eqn{ww.bd.strict}
$$
\snr{u_0(y_1)-u_0(y_0)}
\le
(1-\sigma_0)\snr{y_1-y_0}. 
$$
We next establish compactness of $\{u_i\}$ in
$C^0(\overline\Omega)$. For every $x\in\Omega$, choose
$y_x\in\partial\Omega$
such that
$\snr{x-y_x}=\dist(x,\partial\Omega).
$
The open segment joining $x$ to $y_x$ is contained in $\Omega$, and
therefore \eqref{equi} gives
$
\snr{u_i(x)-\tx u_{0;i}(y_x)}
\le
\dist(x,\partial\Omega).
$
Hence
$
\sup_i\nr{u_i}_{L^\infty(\Omega)}
\le
M+\diam(\Omega).
$
Moreover, the functions $u_i$ are locally uniformly $1$-Lipschitz in
$\Omega$. By the Ascoli-Arzel\`a theorem and a diagonal argument,
they converge, after passing to a subsequence, locally uniformly in
$\Omega$. The previous boundary estimate, together with
\eqref{ww.bd.conv}, upgrades the convergence to
\eqn{ww.unif.conv}
$$
u_i\to u
\qquad
\mbox{uniformly on }\overline\Omega
$$
for some
$
u\in C^0(\overline\Omega)
$
satisfying
$
u=u_0
$ on $\partial\Omega.
$
Passing to the limit in \eqref{equi}, we obtain
\eqn{ww.limit.sp}
$$
\snr{u(x)-u(y)}
\le
\snr{x-y}
\qquad
\mbox{whenever }
\overline{xy}\subset\overline\Omega.
$$
In particular,
$
u\in W^{1,\infty}(\Omega),
$ with $
\snr{Du}\le1$
a.e. in $\Omega.$ 
Set
$
\mu_i:=-\mathcal Mu_i.
$
Since
$
[\mu_i]_{n,1;\Omega}\le\Lambda_0
$
and
$
L(n,1)(\Omega)\hookrightarrow L^n(\Omega)
$,
the sequence $\{\mu_i\}$ is uniformly bounded in $L^n(\Omega)$.
Since $n\ge2$, after passing to a subsequence we may assume that
\eqn{ww.weak.mu}
$$
\mu_i\rightharpoonup\mu
\qquad
\mbox{weakly in }L^n(\Omega)
$$
for some $\mu\in L^n(\Omega)$. By the weak lower semicontinuity
properties of the Lorentz norm,
$
\mu\in L(n,1)(\Omega)$, 
with $
[\mu]_{n,1;\Omega}\le\Lambda_0.
$ We can now adapt the convergence argument of
\cite[Lemma 1.3]{bs82}\footnote{Indeed, each $u_i$ minimizes
$\mathcal E_{\mu_i}(\,\cdot\,;\Omega)$ with its own boundary datum.
To justify this without any control of the boost near the boundary,
let $v\in W^{1,\infty}(\Omega)$ be weakly spacelike with
$v=u_i$ on $\partial\Omega$, and set
$
v_\delta
:=u_i+(v-u_i-\delta)_+-(u_i-v-\delta)_+$, 
$ \delta>0.
$
Then $v_\delta$ is weakly spacelike and
$\supp(v_\delta-u_i)\Subset\Omega$.
Convexity of $H$ and the equation for $u_i$, tested with
$v_\delta-u_i$, give
$\mathcal E_{\mu_i}(u_i;\Omega)
\le\mathcal E_{\mu_i}(v_\delta;\Omega)$.
Since $v_\delta\to v$ uniformly and $Dv_\delta\to Dv$ a.e.,
letting $\delta\downarrow0$ yields the desired minimality
by dominated convergence, using $0\le H\le1$ on the unit ball
and $\mu_i\in L^1(\Omega)$. See also the argument at the end of the proof of Proposition \ref{connessione}.}. The uniform $L(n,1)$-bound gives uniform
$L^n$- and $L^1$-bounds for the charges, while
\eqref{ww.unif.conv} and \eqref{ww.weak.mu} allow passage to the limit
in the linear terms. In particular, the moving-set argument is
unchanged: whenever
$
E_i\to E
$
in measure and $v$ is bounded,
$$
\left|
\int_{E_i\triangle E}\mu_i v\dx
\right|
\le
c\nr{v}_{L^\infty(\Omega)}
\nr{\mu_i}_{L^n(\Omega)}
\snr{E_i\triangle E}^{1-1/n}
\to0.
$$
It follows that $u$ is the variational minimizer of
$
\mathcal E_\mu(\,\cdot\,;\Omega)
$
with boundary datum $u_0$.
By \eqref{ww.unif.conv}, letting \(i\to\infty\) in \eqref{pp.0}
gives
\eqn{ww.light.end}
$$
\snr{u(x_1)-u(x_0)}
=
\snr{x_1-x_0}.
$$
Set  
\(\gamma(t):=x_0+t(x_1-x_0)\), \(t\in[0,1]\) and  we may assume that
\(u(x_1)-u(x_0)=\snr{x_1-x_0}\). Then \eqref{ww.limit.sp} and \eqref{ww.light.end} imply
\eqn{ombra}
$$
u(\gamma(t))
=
u(x_0)+t\snr{x_1-x_0}
\qquad
\mbox{for every }t\in[0,1].
$$
If \(\overline{x_1x_0}\cap\Omega=\varnothing\), then
\(\overline{x_1x_0}\subset\partial\Omega\), and
\eqref{ww.light.end}, together with \(u=u_0\) on
\(\partial\Omega\), gives
\(\snr{u_0(x_1)-u_0(x_0)}=\snr{x_1-x_0}\), contradicting
\eqref{ww.bd.strict}. Assume therefore that
\(\overline{x_1x_0}\cap\Omega\ne\varnothing\). Then
\(I:=\{t\in(0,1):\gamma(t)\in\Omega\}\) is a nonempty open subset
of \((0,1)\), and hence contains a nonempty open interval. By
\eqref{ombra}, the graph of \(u\) contains a nontrivial light
segment entirely contained in \(\Omega\). Since \(u\) is the
variational minimizer with charge \(\mu\in L(n,1)(\Omega)\), the
results of Section \ref{dirisec} allow us to apply
Theorem \ref{al.t}. The light segment therefore extends along the
whole maximal connected component of the line through
\(x_0,x_1\) contained in \(\Omega\). Let \(y_0,y_1\in\partial\Omega\) be the endpoints of this maximal
component. By the definition of the component,
\(\overline{y_1y_0}\subset\overline\Omega\). By continuity, the light-ray identity extends to the endpoints and gives
$
\snr{u(y_1)-u(y_0)}
=
\snr{y_1-y_0}.
$
Since $u=u_0$ on $\partial\Omega$, we obtain
$
\snr{u_0(y_1)-u_0(y_0)}
=
\snr{y_1-y_0},
$
which contradicts \eqref{ww.bd.strict}.
\end{proof}

\begin{lemma}[Nonlinear correctors via contractions]
\label{correctors}
Let $\mathcal U\subset\mathbb R^n$ be a $C^2$-regular, bounded domain, 
let $\mathcal w\in C^2(\overline{\mathcal U})$, and assume that
\eqn{ww.background.f}
$$
-\mathcal M \mathcal w=\mathfrak f
\qquad
\mbox{in }\mathcal U
$$
for some $\mathfrak f\in L^\infty(\mathcal U)$. Suppose moreover that
\eqn{ww.}
$$
\nr{D\mathcal w}_{C^0(\overline{\mathcal U})}
\leq1-\mathcal s
\qquad
\mbox{for some }\mathcal s\in(0,1).
$$
Let $\mu_*\in L(n,1)(\mathcal U)$. There exists a constant
$
\ti c_* =\ti c_* (n,\mathcal s,\nr{D\mathcal w}_{C^1(\overline{\mathcal U})},\mathcal U)
\geq1
$
such that, if
\eqn{ww.17}
$$
\ti c_*[\mu_*]_{n,1;\mathcal U}\leq1,
$$
then there exists
$
\mathcal b \in
W^{2;n,1}(\mathcal U) \cap W^{1;n,1}_0(\mathcal U)
$
solving
\eqn{pd.ww}
$$
\begin{cases}
-\mathcal M(\mathcal w+\mathcal b)
=
\mathfrak f+\mu_*
&
\mbox{in }\mathcal U,
\\[1mm]
\mathcal b=0
&
\mbox{on }\partial\mathcal U.
\end{cases}
$$
Moreover,
\eqn{ww.15}
$$
\nr{\mathcal b}_{W^{2;n,1}(\mathcal U)}
\leq
c[\mu_*]_{n,1;\mathcal U},
\qquad
\nr{D\mathcal b}_{L^\infty(\mathcal U)}
\leq
\frac{\mathcal s}{4},
$$
where
$c$ depends only on 
$
n,\mathcal s,
\nr{D\mathcal w}_{C^1(\overline{\mathcal U})},
\mathcal U.
$
The solution is unique among the solutions satisfying the gradient
bound in \eqref{ww.15}. Finally,
\eqn{ww.18}
$$
\mu_*\geq0
\quad\Longrightarrow\quad
\mathcal b\geq0,
\qquad
\mu_*\leq0
\quad\Longrightarrow\quad
\mathcal b\leq0.
$$
\end{lemma}

\begin{proof}
The proof splits into three steps. 

{\em Step 1: Linearized operator and nonlinear remainder}.
Set
\(
\tx H(x):=\partial^2H(D\mathcal w(x)).
\)
By \eqref{ww.}, the matrix $\tx H(\cdot)$ is symmetric and uniformly
elliptic. More precisely, there exists
$c_{\mathcal s}=c_{\mathcal s}(n,\mathcal s)\geq1$ such that
\eqn{ww.2}
$$
|\xi|^2
\leq
\langle\tx H(x)\xi,\xi\rangle
\leq
c_{\mathcal s}|\xi|^2
\qquad
\mbox{for every }
x\in\overline{\mathcal U},
\quad
\xi\in\mathbb R^n.
$$
Moreover,
\eqn{ww.1}
$$
\nr{\tx H}_{C^0(\overline{\mathcal U})}
\leq c,
\qquad
\nr{D\tx H}_{C^0(\overline{\mathcal U})}
\leq
c(n,\mathcal s)
\nr{D^2\mathcal w}_{C^0(\overline{\mathcal U})},
$$
where $c=c(n,\mathcal s)$. Expanding
$-\diver(\tx H(x)D\varphi)$, define
\begin{flalign*}
\mathcal{L}_{\mathcal{w}}\varphi&:=-\sum_{i,j=1}^{n}\partial^{2}_{ij}H(D\mathcal{w})D_{ij}\varphi-\sum_{j=1}^{n}\left(\sum_{i,k=1}^{n}\partial^{3}_{ij,k}H(D\mathcal{w})D_{ik}\mathcal{w}\right)D_{j}\varphi\nonumber \\
&=:-\sum_{i,j=1}^{n}\tx{H}_{ij}(x)D_{ij}\varphi-\sum_{j=1}^{n}\tx{J}_{j}(x)D_{j}\varphi.
\end{flalign*}
Note that 
$
\nr{\tx{J}}_{C^0(\mathcal{\overline U})}\le c(n,\mathcal{s})\nr{D^{2}\mathcal{w}}_{C^0(\mathcal{\overline U})}.
$
Using \eqref{ww.2}, \eqref{ww.1}, and the above bound
on $\tx{J}$, we may apply
\cite[Theorem 9.15 and Lemma 9.17, pp.~241--242]{gt01}.
Thus, 
for every \(p\in(1,\infty)\) and every \(\tx{f}\in L^p(\mathcal U)\), we can uniquely solve $\mathcal {L}_{\mathcal w}v=\tx{f}$ in the Dirichlet class $W^{1,p}_0(\mathcal U)$ with  
\eqn{dirip}
$$
\|v\|_{W^{2,p}(\mathcal U)}\le c\|\tx{f}\|_{L^p(\mathcal U)},
$$
where $c$ depends only on $n,p,\nr{D\mathcal w}_{C^1(\mathcal{\overline U})}, \mathcal s, \mathcal U$. 
With abuse of notation we denote by $\mathcal {L}_{\mathcal w}^{-1}$ the linear operator such that 
$\mathcal {L}_{\mathcal w}^{-1}\tx{f} =v$ (the abuse is linked to the fact that, in principle, such an operator should depend on the exponent $p$, which is in fact not the case). By \eqref{dirip} and standard real interpolation,
we obtain
\eqn{ww.4}
$$
\nr{\mathcal {L}_{\mathcal w}^{-1}\tx{f}}_{W^{2;n,1}(\mathcal U)}
\le
c_{\rm inv}[\tx{f}]_{n,1;\mathcal U}
\quad
\mbox{for every }\tx{f}\in L(n,1)(\mathcal U),
$$
where $c_{\rm inv}\geq 1$ depends on $n,\nr{D\mathcal w}_{C^1(\mathcal{\overline U})}, \mathcal s, \mathcal U$ (see \rif{normalo} for the definition of the right-hand side). 
We then linearize around
\eqn{ww.5}
$$
v\in W^{2;n,1}(\mathcal{U})\cap W^{1;n,1}_{0}(\mathcal{U}) \ \ \mbox{such that} \ \ \nr{Dv}_{L^{\infty}(\mathcal{U})}\le \mathcal{s}/4.
$$
By \eqref{ww.} and \eqref{ww.5},
\(
|D\mathcal{w}+tDv|
\le 1-3\mathcal{s}/4
\)
for every \(t\in[0,1]\). Hence all the derivatives of \(H\) up to
order four are uniformly bounded on the corresponding compact
subset of the unit ball. 
Since $-\mathcal M \mathcal w=\mathfrak f$, we can write
\begin{eqnarray}\label{ww.8}
-\mathcal M(\mathcal w+v)-\mathfrak f
&=&
-\mathcal M(\mathcal w+v)+\mathcal M \mathcal w
=
\mathcal L_{\mathcal w}v
+
\left(
-\mathcal M(\mathcal w+v)
+\mathcal M \mathcal w
-\mathcal L_{\mathcal w}v
\right)
\nonumber\\
&=:&
\mathcal L_{\mathcal w}v
+
\mathcal R_{\mathcal w}(v).
\end{eqnarray}
Recalling that $\partial^{3}_{ij,k}H$ is symmetric in all indices, we split the nonlinear remainder $\mathcal{R}_{\mathcal{w}}(v)$ as 
\begin{flalign}\label{ww.7}
\mathcal{R}_{\mathcal{w}}(v)&=-\sum_{i,j=1}^{n}\big(\partial^{2}_{ij}H(D\mathcal{w}+Dv)-\partial^{2}_{ij}H(D\mathcal{w})\big)D_{ij}v\nonumber \\
&\quad -\sum_{i,j=1}^{n}\big(\partial^{2}_{ij}H(D\mathcal{w}+Dv)-\partial^{2}_{ij}H(D\mathcal{w})- \sum_{k=1}^{n}\partial^{3}_{ij,k}H(D\mathcal{w})D_{k}v\big)D_{ij}\mathcal{w}\nonumber \\
&=:\mbox{(I)}+\mbox{(II)}.
\end{flalign}
For the first term, the fundamental theorem of calculus gives
\[
\partial^2_{ij}H(D\mathcal{w}+Dv)
-
\partial^2_{ij}H(D\mathcal{w})
=
\sum_{k=1}^{n}
\int_0^1
\partial^3_{ij,k}H(D\mathcal{w}+tDv)
D_kv\dt.
\]
We conclude with 
\eqn{del1}
$$
|\mathrm{(I)}|
\le
c(n,\mathcal{s})|Dv||D^2v|.
$$
For the second term, Taylor's formula with integral remainder yields
\[
\begin{aligned}
&
\partial^2_{ij}H(D\mathcal{w}+Dv)
-
\partial^2_{ij}H(D\mathcal{w})
-
\sum_{k=1}^{n}
\partial^3_{ij,k}H(D\mathcal{w})D_kv
\\
&\qquad=
\sum_{k,\ell=1}^{n}
\int_0^1
(1-t)\,
\partial^4_{ij,k\ell}H(D\mathcal{w}+tDv)
D_kv\,D_\ell v\dt.
\end{aligned}
\]
Therefore,
\eqn{del2}
$$
|\mathrm{(II)}| \le c(n,\mathcal{s})|Dv|^2|D^2\mathcal{w}|.
$$
Combining the estimates in \eqref{del1} and \eqref{del2} with \rif{ww.7}, we obtain
\eqn{rem}
$$
|\mathcal R_{\mathcal w}(v)|
\le
c|Dv||D^2v|
+
c|Dv|^2|D^2\mathcal w|
$$
with $c\equiv c(n,\mathcal{s})$. By the endpoint Sobolev--Lorentz embedding \cite{cp98}, we have
\eqn{immersione}
$$
\nr{Dv}_{L^\infty(\mathcal U)}
\le
c_{\rm emb}\nr{v}_{W^{2;n,1}(\mathcal U)}, \quad c_{\rm emb} \equiv c_{\rm emb}(n,\mathcal U)\geq 1\,.
$$
Using Young's inequality in \eqref{rem}, we get
\begin{flalign*}
[\mathcal{R}_{\mathcal{w}}(v)]_{n,1;\mathcal{U}}&\le c\nr{Dv}_{L^{\infty}(\mathcal{U})}[D^{2}v]_{n,1;\mathcal{U}}+c\nr{Dv}_{L^{\infty}(\mathcal{U})}^{2}\nr{D^{2}\mathcal{w}}_{L^{\infty}(\mathcal{U})}[\mathds 1_{\mathcal U}]_{n,1;\mathcal U} \nonumber \\
&\le c \nr{Dv}_{L^{\infty}(\mathcal{U})}^{2}+c[D^{2}v]_{n,1;\mathcal{U}}^{2} 
\end{flalign*}
and therefore, by means of \rif{immersione}, we conclude with 
\eqn{ww.10}
$$
[\mathcal{R}_{\mathcal{w}}(v)]_{n,1;\mathcal{U}}\leq c_{\rm rem} \nr{v}_{W^{2;n,1}(\mathcal{U})}^{2},
$$
where  $c_{\rm rem}\geq 1$ depends on $n,\mathcal{s},\diam (\mathcal U)\nr{D^2\mathcal{w}}_{C^{0}(\bar{\mathcal{U}})},c_{\rm emb}$\footnote{For such a constant dependence note that
$
[\mathds 1_{\mathcal U}]_{n,1;\mathcal U}
\leq c(n)|\mathcal U|^{1/n}
\leq c(n)\diam(\mathcal U).
$}. A similar argument then yields
\eqn{ww.13}
$$
[\mathcal{R}_{\mathcal{w}}(v_{1})-\mathcal{R}_{\mathcal{w}}(v_{2})]_{n,1;\mathcal{U}}\le c_{\rm diff}\left(\nr{v_{1}}_{W^{2;n,1}(\mathcal{U})}+\nr{v_{2}}_{W^{2;n,1}(\mathcal{U})}\right)\nr{v_{1}-v_{2}}_{W^{2;n,1}(\mathcal{U})},
$$
which holds for all $v_{1},v_{2}\in W^{2;n,1}(\mathcal{U})\cap W^{1;n,1}_{0}(\mathcal{U})$ such that $\nr{Dv_{i}}_{L^{\infty}(\bar{\mathcal{U}})}\le \mathcal{s}/4,$ $i\in \{1,2\}$, where $c_{\rm diff}\geq 1$ again depends on $n,\mathcal{s},\diam (\mathcal U)\nr{D^2\mathcal{w}}_{C^{0}(\bar{\mathcal{U}})}$ and $c_{\rm emb}$.
\medskip

{\em Step 2: Fixed-point argument}. 
Recall the meaning of the constants $c_{\rm inv}, c_{\rm emb}, c_{\rm rem}, c_{\rm diff}$; these have been defined in \rif{ww.4}, \rif{immersione}, \rif{ww.10} and \rif{ww.13}, respectively, and altogether ultimately depend on $n,\mathcal{s},\nr{D\mathcal{w}}_{C^1(\mathcal{\overline U})}$ and $\mathcal{U}$.  
Set
\(
X:=
W^{2;n,1}(\mathcal U)\cap W^{1;n,1}_0(\mathcal U)
\)
and take a number $\rho$ such that 
\eqn{ww.rho}
$$
0<\rho
\le
\frac18
\min
\left\{
\frac{\mathcal s}{c_{\rm emb}},
\frac1{c_{\rm inv}c_{\rm rem}},
\frac1{c_{\rm inv}c_{\rm diff}}
\right\}=:\rho_{\rm uni}. 
$$
Consider the closed ball
$
\mathbb B_{\rho}^X
:=\{v\in X:\nr{v}_{W^{2;n,1}(\mathcal U)}\le\rho\}.
$
It is a complete metric  space  and, by \eqref{ww.rho},
$$
\nr{Dv}_{L^\infty(\mathcal U)}
\le
c_{\rm emb}\rho
\le
\mathcal s/4
\qquad
\mbox{for every }v\in\mathbb B_{\rho}^X.
$$
For \(v\in\mathbb B_{\rho}^X\), define
$$
\mathcal G_{\mathcal w}(v;\mu_{*})
:=
\mathcal {L}_{\mathcal w}^{-1}
(\mu_{*}-\mathcal R_{\mathcal w}(v)).
$$
We choose the constant $\widetilde c_*$ appearing
in \eqref{ww.17} as
$
\widetilde c_*:=2c_{\rm inv}/\rho.
$
With this choice, condition \eqref{ww.17} becomes
$
c_{\rm inv}[\mu_*]_{n,1;\mathcal U}\le \rho/2.
$
Then, using \eqref{ww.4}, \eqref{ww.10}, and \eqref{ww.rho}, we obtain
\eqn{ww.14}
$$
\begin{aligned}
\nr{\mathcal G_{\mathcal w}(v;\mu_{*})}
_{W^{2;n,1}(\mathcal U)}
&\le
c_{\rm inv}
[\mu_{*}]_{n,1;\mathcal U}
+c_{\rm inv}
[\mathcal R_{\mathcal w}(v)]_{n,1;\mathcal U}
\le
\frac{\rho}{2}
+
c_{\rm inv}c_{\rm rem}\rho^2
\le
\frac{5\rho}{8}
<
\rho.
\end{aligned}
$$
Thus \(\mathcal G_{\mathcal w}\) maps
\(\mathbb B_{\rho}^X\) into itself. Similarly, by \eqref{ww.4}, \eqref{ww.13}, and
\eqref{ww.rho},
$$
\begin{aligned}
&
\nr{
\mathcal G_{\mathcal w}(v_1;\mu_{*})
-
\mathcal G_{\mathcal w}(v_2;\mu_{*})
}_{W^{2;n,1}(\mathcal U)}
\le
2c_{\rm inv}c_{\rm diff}\rho
\nr{v_1-v_2}_{W^{2;n,1}(\mathcal U)}
\le
\frac14
\nr{v_1-v_2}_{W^{2;n,1}(\mathcal U)}
\end{aligned}
$$
holds whenever \(v_1,v_2\in\mathbb B_{\rho}^X\). 
Hence \(\mathcal G_{\mathcal w}\) is a contraction on
\(\mathbb B_{\rho}^X\). The Banach--Caccioppoli fixed-point theorem gives a unique
\(\mathcal b\in\mathbb B_{\rho}^X\) such that
\(
\mathcal b=\mathcal G_{\mathcal w}(\mathcal b;\mu_{*}).
\)
By \eqref{ww.8}, this is equivalent to \rif{pd.ww}. Proceeding as for \rif{ww.14}, the fixed-point property  and \eqref{ww.10} also give
$$
\nr{\mathcal b}_{W^{2;n,1}(\mathcal U)} = \nr{\mathcal G_{\mathcal w}(\mathcal b;\mu_{*})}
_{W^{2;n,1}(\mathcal U)}
\le
c_{\rm inv}[\mu_{*}]_{n,1;\mathcal U}
+
c_{\rm inv}c_{\rm rem}\nr{\mathcal b}_{W^{2;n,1}(\mathcal U)}^2.
$$
Since
\(
\nr{\mathcal b}_{W^{2;n,1}(\mathcal U)}\le\rho
\)
and \(c_{\rm inv}c_{\rm rem}\rho\le1/8\), the last term can be absorbed, yielding
\rif{ww.15}. The solution is unique in the class determined by the last
gradient bound. Indeed, if \(\mathcal b_1,\mathcal b_2\) are two
such solutions, subtracting their weak formulations and testing
with \(\mathcal b_1-\mathcal b_2\) gives
$$
\begin{aligned}
0
={}&
\int_{\mathcal U}
\left\langle
\partial H(D\mathcal w+D\mathcal b_1)
-
\partial H(D\mathcal w+D\mathcal b_2),
D\mathcal b_1-D\mathcal b_2
\right\rangle\dx.
\end{aligned}
$$
The strict monotonicity of \(\partial H\) implies
\(D\mathcal b_1=D\mathcal b_2\) a.e. in \(\mathcal U\), and the
zero boundary trace gives \(\mathcal b_1=\mathcal b_2\).

\medskip

{\em Step 3: Sign of the corrector}.
Finally, \eqref{ww.18} follows from the weak comparison principle.
Indeed, by \eqref{ww.} and \eqref{ww.15},
$
\nr{D\mathcal w+D\mathcal b}_{L^\infty(\mathcal U)}
\le
1-3\mathcal s/4,
$
and, since
$
\nr{D\mathcal w}_{L^\infty(\mathcal U)}\le1-\mathcal s,
$
we have
$
\snr{D\mathcal w+tD\mathcal b}
\le
1-3\mathcal s/4
$
for every $t\in[0,1]$.
Hence, setting
$$
\tx{H}_{\mathcal b}(x)
:=
\int_0^1
\partial^2H(D\mathcal w(x)+tD\mathcal b(x))\dt,
$$
the matrix \(\tx{H}_{\mathcal b}\) is uniformly elliptic in
\(\mathcal U\). Subtracting
\(-\mathcal M \mathcal w =\mathfrak f\) from
\(-\mathcal M(\mathcal w+\mathcal b)=\mathfrak f+\mu_*\), we obtain
$
-\diver \, (\tx{H}_{\mathcal b}D\mathcal b)
=
\mu_*$
in $\mathcal U$
with 
$\mathcal b=0$
on $\partial \mathcal U$, and \rif{ww.18} 
follows by the weak comparison principle. 
\end{proof}
We record the uniform version of the preceding construction needed below. 
\begin{corollary}[Uniform correctors on boundary caps]
\label{correctors.uniform}
Let \(\Omega\subset\mathbb R^n\) be a $C^2$-regular, bounded domain.
There exists \(r_\Omega>0\) with the following property. For every
\(r\in(0,r_\Omega]\), one can choose, for each
\(x_0\in\partial\Omega\), a \(C^2\)-regular boundary cap
\(\mathcal U_{x_0}\) such that
$$
\Omega\cap B_{r/2}(x_0)
\subset
\mathcal U_{x_0}
\subset
\Omega\cap B_r(x_0),
\qquad
\partial\mathcal U_{x_0}\cap B_{r/2}(x_0)
=
\partial\Omega\cap B_{r/2}(x_0),
$$
and such that the family
\(\{\mathcal U_{x_0}\}_{x_0\in\partial\Omega}\) has uniformly controlled
\(C^2\)-geometry. More precisely, the caps may be constructed from a
fixed reference cap in boundary coordinates, so that, after
translation, rotation, and dilation by \(1/r\), their \(C^2\)
boundary charts have uniform bounds only depending on \(r\) and
\(\Omega\).
Let \(\mathcal s\in(0,1)\) and \(K>0\). For every
\(x_0\in\partial\Omega\), let
\(\mathcal w_{x_0}\in C^2(\overline{\mathcal U}_{x_0})\) satisfy
\eqn{ww.uniform.background}
$$
-\mathcal M \mathcal w_{x_0}=\mathfrak f_{x_0}
\quad\mbox{in }\mathcal U_{x_0},
\qquad
\nr{D\mathcal w_{x_0}}_{L^\infty(\mathcal U_{x_0})}
\leq1-\mathcal s,
\qquad
r\nr{D^2\mathcal w_{x_0}}_{L^\infty(\mathcal U_{x_0})}
\leq K,
$$
where \(\mathfrak f_{x_0}\in L^\infty(\mathcal U_{x_0})\). Then there exist constants $
\eps_{\rm \partial} \in (0,1)$ and $c_{\rm uni}\geq 1$, 
only depending on $n,\mathcal s,K,r,\Omega$, 
such that the following holds. If
\(\mu_{x_0}\in L(n,1)(\mathcal U_{x_0})\) satisfies
$$
[\mu_{x_0}]_{n,1;\mathcal U_{x_0}}
\leq
\eps_{\rm \partial},
$$
then there exists
$
\mathcal b_{x_0}
\in
W^{2;n,1}(\mathcal U_{x_0})
\cap
W^{1;n,1}_0(\mathcal U_{x_0})
$
solving \eqref{pd.ww} with $\mathcal b\equiv  \mathcal b_{x_0}$, $\mathfrak f\equiv  \mathfrak f_{x_0}$ and $\mu_*\equiv  \mu_{x_0}$, 
and satisfying \eqref{ww.15}, with $c_{\rm uni}$ in place of $c$. Moreover, the sign conclusion \eqref{ww.18} holds.
\end{corollary}

\begin{proof}

Since \(\partial\Omega\) is compact and the caps
\(\mathcal U_{x_0}\) are constructed at the fixed scale \(r\), their
quantitative \(C^2\)-geometry is controlled uniformly with respect to
\(x_0\in\partial\Omega\). Consequently, the constant $c_{\rm emb}$ in the endpoint Sobolev-Lorentz
embedding \rif{immersione} can be chosen independently of \(x_0\in \partial \Omega\).
On the other hand, by \eqref{ww.uniform.background}, 
it follows that the coefficient matrices
$
\partial^2H(D\mathcal w_{x_0})
$
are uniformly elliptic, with ellipticity constants only depending on
\(n\) and \(\mathcal s\), and that their \(W^{1,\infty}\)-norms are bounded
uniformly in \(x_0\), in terms of \(n,\mathcal s,K/r\). Similarly, the
coefficients of the linearized operators
$
\mathcal L_{\mathcal w_{x_0}}
$
satisfy common ellipticity and regularity bounds. Therefore, the global \(W^{2,p}\)-estimates used in the proof of
Lemma \ref{correctors}, at two fixed exponents
\(1<p_0<n<p_1<\infty\), give a common inverse constant
\(c_{\rm inv}\) via real interpolation. The estimates for the nonlinear
remainder and its difference give common constants
\(c_{\rm rem}\) and \(c_{\rm diff}\). Summarizing, the constants
\(c_{\rm inv}\), \(c_{\rm rem}\), and \(c_{\rm diff}\) can be chosen
uniformly with respect to \(x_0\in\partial\Omega\), only depending on
\(n,\mathcal s,K/r,r,\Omega\), while \(c_{\rm emb}\) depends only on
\(n,r,\Omega\). Here the dependence on \(r\) and \(\Omega\) accounts
for the common quantitative geometry of the caps. With such a uniformized choice, define $\rho_{\rm uni}$ as in \rif{ww.rho}, with the present uniform meaning, and finally define
$
\eps_{\rm \partial}
:=
\rho_{\rm uni}/(2c_{\rm inv}).
$
The fixed-point argument in Lemma \ref{correctors} can then be
performed on every \(\mathcal U_{x_0}\) with the same radius
\(\rho_{\rm uni}\). It yields \rif{ww.15} with $c_{\rm uni}:=2c_{\rm inv}$ as in the statement of the Corollary. 
\end{proof}
\begin{proposition}[Strict spacelikeness at the boundary]
\label{corr}
Under assumptions \eqref{assuntidir}, 
fix \(\Lambda_0, F_0>0\). For every $r_0>0$ there exist numbers
\eqn{piccole}
$$
\varepsilon_\partial, \sigma_*\equiv \varepsilon_\partial, \sigma_*\left(n,r_0,\sigma_0,\Lambda_0,F_0,
\nr{\tx u_0}_{C^2(\overline\Omega)},\Omega\right)\in (0,1)
$$
with the following property. Let $f\in L^\infty(\Omega)$ satisfy
\(\nr f_{L^\infty(\Omega)}\leq F_0\). Let
$\mu\in L^\infty(\Omega)$, and assume 
\eqn{ww.41}
$$
[\mu]_{n,1;\Omega}\leq\Lambda_0.
$$
Assume that
\(u\in C^2(\Omega)\cap C^1(\overline\Omega)\)
is a strictly spacelike solution of
\eqn{ww.background.perturbed.problem}
$$
\begin{cases}
-\mathcal Mu=f+\mu
&
\mbox{in }\Omega,
\\[1mm]
u=\tx{u}_0
&
\mbox{on }\partial\Omega.
\end{cases}
$$
Set
\(
\Omega_{r_0}
:=
\{x\in\Omega:\dist(x,\partial\Omega)<r_0\}.
\)
If
\eqn{ww.31}
$$
[\mu]_{n,1;\Omega_{r_0}}
\leq
\varepsilon_\partial,
$$
then
\eqn{ww.boundary.gap}
$$
\nr{Du}_{C^0(\partial\Omega)}
\leq
1-\sigma_*.
$$
\end{proposition}
\begin{proof}
Again, the proof goes in two steps.

{\em Step 1: Constructing uniform caps}.
Fix a bounded \(C^2\)-extension operator to a tubular neighbourhood
\(N_{\rho_{\rm ext}}(\overline\Omega)\) of size ${\rho_{\rm ext}}$, and continue to denote the
extension of \(\tx u_0\) by \(\tx u_0\). Put
\(M_2:=
\nr{D^2\tx u_0}_{L^\infty
(N_{\rho_{\rm ext}}(\overline\Omega))}\).
The extension may be fixed so that
\(M_2
\leq
C_{\rm ext}(\Omega)
\nr{\tx u_0}_{C^2(\overline\Omega)}\). Let \(r_\Omega>0\) be the radius supplied by Corollary
\ref{correctors.uniform}, and set
\[
r_*
:=
\min\left\{
1,r_\Omega,\rho_{\rm ext},
\frac{\sigma_0}{2\max\{M_2,1\}}
\right\},
\qquad
\rho_0:=\min\{r_0,r_*\}.
\]
In particular, \(0<\rho_0\leq r_0\) and
\(\rho_0\leq r_\Omega\). Applying the construction in Corollary \ref{correctors.uniform} at
the fixed scale \(r=\rho_0\), we choose, for every
\(x_0\in\partial\Omega\), a \(C^2\)-regular boundary cap
\(\mathcal U_{x_0}\) such that
\eqn{ww.21}
$$
\Omega\cap B_{\rho_0/2}(x_0)
\subset
\mathcal U_{x_0}
\subset
\Omega\cap B_{\rho_0}(x_0),
\qquad
\partial\mathcal U_{x_0}\cap B_{\rho_0/2}(x_0)
=
\partial\Omega\cap B_{\rho_0/2}(x_0).
$$
By the same corollary, the family
\(\{\mathcal U_{x_0}\}_{x_0\in\partial\Omega}\)
has uniformly controlled \(C^2\)-geometry at the scale \(\rho_0\).
In particular, the caps are fixed solely in terms of
\(r_0,\sigma_0,
\nr{\tx u_0}_{C^2(\overline\Omega)},\Omega\),
and are independent of \(f,\mu\), and \(u\). Finally, by the definition of \(r_*\), for every
\(x_0\in\partial\Omega\) and every \(x\in B_{\rho_0}(x_0)\),
the fundamental theorem of calculus gives
\eqn{ww.24}
$$
\snr{\tx u_0(x)-\tx u_0(x_0)}
\leq
(1-\sigma_0/2)\snr{x-x_0}.
$$ Fix now \(x_0\in\partial\Omega\); we shall often abbreviate 
\(\mathcal U_0:=\mathcal U_{x_0}\).
We first establish a quantitative spacelike inequality for the
boundary values of \(u\) on \(\partial\mathcal U_0\). Fix
\(x\in\partial\mathcal U_0\cap\Omega\), and define
\(
\gamma(t):=x_0+t(x-x_0)
\)
for \(t\in[0,1]\). Since \(\gamma(1)=x\in\Omega\), let \((t_x,1]\), with
$t_x\in[0,1)$, be the connected component containing \(1\) of
$
\{t\in(0,1]:\gamma(t)\in\Omega\}.
$
Then
$
\gamma(t_x)\in\partial\Omega,
$
where possibly \(t_x=0\), in which case
\(\gamma(t_x)=x_0\).
Notice that, by \eqref{ww.21},
$\rho_0/2\le\snr{x-x_0}\le \rho_0.$
Since $\Omega$ has finite measure, the embedding
$L^\infty(\Omega)\hookrightarrow L(n,1)(\Omega)$ gives a constant
$c_{\Omega}=c(n,\Omega)$ such that
$$
[f+\mu]_{n,1;\Omega}
\leq
c_{\Omega}\nr f_{L^\infty(\Omega)}
+
[\mu]_{n,1;\Omega}
\leq
c_{\Omega}F_0+\Lambda_0
=:
\Lambda_{\rm tot}.
$$
Apply Lemma \ref{ww.lem} once, at the already fixed scale
\(\tx l=\rho_0\), with
$M=1+\nr{\tx u_0}_{C^2(\overline\Omega)}$
and global charge bound $\Lambda_{\rm tot}$. Denote its constant by
$
\sigma_{\rm ch}
=
\sigma_{\rm ch}
(
n,\rho_0,\Lambda_{\rm tot},\sigma_0,
\nr{\tx u_0}_{C^2(\overline\Omega)},\Omega
)
\in(0,1).
$
If \(t_x\le1/2\), then
\(
\rho_0/4
\le
\snr{x-\gamma(t_x)}
\le
\rho_0,
\)
and the estimate supplied by Lemma \ref{ww.lem} on
\(\overline{x\,\gamma(t_x)}\) gives
\begin{flalign*}
\snr{u(x)-u(x_0)}
&\leq
\snr{u(x)-u(\gamma(t_x))}
+
\snr{u(\gamma(t_x))-u(x_0)}
\\
&\leq
(1-\sigma_{\rm ch})
\snr{x-\gamma(t_x)}
+
\snr{\tx u_0(\gamma(t_x))-\tx u_0(x_0)}
\\
&\leq 
(1-\sigma_{\rm ch})
\snr{x-\gamma(t_x)}
+
(
1- \sigma_0/2
)
\snr{\gamma(t_x)-x_0}
\\
&\leq
(
1-\min\{\sigma_{\rm ch}, \sigma_0/2\}
)
\snr{x-x_0}.
\end{flalign*}
Here we used the fact that
\(x\), \(\gamma(t_x)\), and \(x_0\) are collinear and \eqref{ww.24} in the third line. If \(t_x>1/2\), the weak spacelike inequality on
\(\overline{x\,\gamma(t_x)}\), together with
\eqref{ww.24}, gives
\begin{flalign*}
\snr{u(x)-u(x_0)}&\le
\snr{u(x)-u(\gamma(t_x))}
+
\snr{\tx u_0(\gamma(t_x))-\tx u_0(x_0)}
\\
&\le
\snr{x-\gamma(t_x)}
+
\snr{\tx u_0(\gamma(t_x))-\tx u_0(x_0)}
\\
&\le
(1-t_x)\snr{x-x_0}
+
(
1-\sigma_0/2
)
t_x\snr{x-x_0}
\le
(
1- \sigma_0/4
)
\snr{x-x_0}.
\end{flalign*}
In any case we conclude with 
\eqn{ww.23}
$$
\snr{u(x)-u(x_0)}
\le
(1-\widetilde\sigma)\snr{x-x_0}
\qquad
\mbox{for every }
x\in\partial\mathcal U_0\cap\Omega,
$$
where
\(
\widetilde\sigma
:=
\min\{\sigma_{\rm ch},\sigma_0\}/4.
\) We now construct a suitable extension of the boundary datum
\(\left.u\right|_{\partial\mathcal U_0}\). Define
\(
R_0
:=
\overline{\mathcal U}_0
\cap
\{\rho_0/4\le\snr{x-x_0}\le \rho_0\}
\)
and note that
\(
\partial R_0
:=
(\partial\mathcal U_0\cap R_0)
\cup
(\overline{\mathcal U}_0\cap
\partial B_{\rho_0/4}(x_0)).
\)
On \(\partial R_0\), set
$$
\tau(x)
:=
\begin{cases}
\displaystyle
\frac{u(x)-\tx{u}_0(x_0)}{\snr{x-x_0}}
&
x\in\partial\mathcal U_0\cap R_0,
\\[3mm]
\displaystyle
\frac{
4(\tx u_0(x)-\tx u_0(x_0))
}{\rho_0}
&
x\in
\overline{\mathcal U}_0
\cap
\partial B_{\rho_0/4}(x_0).
\end{cases}
$$
The two definitions agree on their common domain. Indeed,
$$
\partial\mathcal U_0
\cap
\partial B_{\rho_0/4}(x_0)
\subset
\partial\Omega
\cap
B_{\rho_0/2}(x_0),
$$
and hence
$ u=\tx u_0$
there. Thus
$\tau\in C^{0}(\partial R_0).$
Moreover, \eqref{ww.23} applies on the artificial part of
\(\partial\mathcal U_0\), while \eqref{ww.24} applies on its physical
part. Hence
\(
\nr{\tau}_{C^0(\partial R_0)}
\le1-\widetilde\sigma.
\)
By the Tietze extension theorem, there exists
\(\tau_0\in C^0(\overline{R_0})\) such that
\eqn{ww.26}
$$
\tau_0=\tau
\quad\mbox{on }\partial R_0,
\qquad
\nr{\tau_0}_{C^0(\overline{R_0})}
\le
1-\widetilde\sigma.
$$
Define, for \(x\in\overline{\mathcal U}_0\),
$$
\mathfrak{u}_0(x)
:=
\begin{cases}
\tx u_0(x)
&\snr{x-x_0}\le \rho_0/4
\\[1pt]
\tx{u}_0(x_0)+\tau_0(x)\snr{x-x_0}&
\snr{x-x_0}\ge \rho_0/4.
\end{cases}
$$
By the definition of \(\tau_0\), the two expressions agree on
\(\overline{\mathcal U}_0\cap \partial B_{\rho_0/4}(x_0)\), and hence
\(
\mathfrak{u}_0
\in C^{0}(\overline{\mathcal U}_0).
\)
Furthermore,
$
\mathfrak{u}_0=u
$
on $\partial\mathcal U_0.$
Indeed, this follows from
\(\mathfrak{u}_0=\tx u_0=u\) on
\(\partial\mathcal U_0\cap B_{\rho_0/4}(x_0)\), and from
\(\tau_0=\tau\) on the remaining part of
\(\partial\mathcal U_0\). By \eqref{ww.24} and \eqref{ww.26},
$$
\snr{
\mathfrak{u}_0(x)
-
\mathfrak{u}_0(x_0)
}
\le
(1-\widetilde\sigma)\snr{x-x_0}
\qquad
\mbox{for every }
x\in\overline{\mathcal U}_0.
$$
Moreover,
\(\mathfrak{u}_0=\tx u_0\) in
\(\overline{\mathcal U}_0\cap B_{\rho_0/4}(x_0)\), and therefore
$$
\mathfrak u_0=\tx u_0
\quad\mbox{on }
\overline{\mathcal U}_0\cap B_{\rho_0/4}(x_0),
\qquad
\nr{
D^2\mathfrak{u}_0
}_{
C^0(
\overline{\mathcal U}_0
\cap
B_{\rho_0/8}(x_0)
)}
\le
M_2.
$$
We now apply the quantitative barrier construction in
\cite[proof of Proposition~3.1]{bs82}, with curvature bound
\(A=F_0\), boundary datum
\(\left.u\right|_{\partial\mathcal U_0}\), extension
\(\mathfrak u_0\), cone gap \(\widetilde\sigma\), and local
\(C^2\)-bound \(M_2\). The construction gives upper and lower
barriers
\(\mathcal w_{f,x_0}^{\pm}
\in C^2(\overline{\mathcal U}_0)\) that solve
\eqn{solvef}
$$
-\mathcal M\mathcal w_{f,x_0}^{+}=:\mathfrak f_{x_0}^{+}, \qquad -\mathcal M\mathcal w_{f,x_0}^{-}=:\mathfrak f_{x_0}^{-}
$$
and that satisfy
\eqn{ww.bs.f.sources}
$$
\mathfrak f_{x_0}^-
\leq
f
\leq
\mathfrak f_{x_0}^+
\qquad
\mbox{a.e. in }\mathcal U_0.
$$
Note that $\mathfrak f_{x_0}^{+}$ and $\mathfrak f_{x_0}^{-}$ are precisely defined via \rif{solvef}. 
Since \(\rho_0,\widetilde\sigma,M_2,F_0\), and the quantitative
\(C^2\)-geometry of the caps are uniform with respect to \(x_0\),
the parameters in the same barrier construction can be chosen
uniformly. Consequently, there exist
\(\mathcal s_f\in(0,1)\) and \(K_f<\infty\), only depending on
$
n,
\rho_0,
\widetilde\sigma,
M_2,
F_0,
\Omega,
$
such that
\eqn{ww.bs.f.bounds}
$$
\nr{D\mathcal w_{f,x_0}^{\pm}}_
{L^\infty(\mathcal U_0)}
\leq
1-\mathcal s_f,
\qquad
\rho_0
\nr{D^2\mathcal w_{f,x_0}^{\pm}}_
{L^\infty(\mathcal U_0)}
\leq
K_f.
$$
In particular,
\(\mathfrak f_{x_0}^{\pm}\in L^\infty(\mathcal U_0)\).
Moreover,
\eqn{ww.bs.f.boundary}
$$
\begin{cases}
\mathcal w_{f,x_0}^+(x)>u(x)
&
x\in\partial\mathcal U_0\setminus\{x_0\},
\\[1mm]
\mathcal w_{f,x_0}^-(x)<u(x)
&
x\in\partial\mathcal U_0\setminus\{x_0\},
\\[1mm]
\mathcal w_{f,x_0}^\pm(x_0)=u(x_0).
\end{cases}
$$

\medskip

{\em Step 2: Correctors}.
Since \(x_0\in\partial\Omega\) was arbitrary, Step~1 produces two
families
\(\{\mathcal w_{f,x_0}^+\}_{x_0\in\partial\Omega}\) and
\(\{\mathcal w_{f,x_0}^-\}_{x_0\in\partial\Omega}\), defined on the
corresponding caps \(\mathcal U_{x_0}\), and satisfying \rif{solvef} and 
\eqref{ww.bs.f.bounds} with the same constants
\(\mathcal s_f\) and \(K_f\). We are now ready to check the applicability of Corollary \ref{correctors.uniform} to such a family. To this aim we also define \(\mu_+:=\max\{\mu,0\}\) and
\(\mu_-:=\max\{-\mu,0\}\) and in the setting of Corollary \ref{correctors.uniform} we choose $\mu_{x_0}\equiv \pm\mu_{\pm}\lfloor_{\mathcal U_{x_0}}$. Let \(\eps_{\rm \partial}\) be the threshold supplied by
Corollary \ref{correctors.uniform}, applied with
\(r\equiv \rho_0\), \(\mathcal s\equiv \mathcal s_f\), and \(K\equiv K_f\) and this determines the number $\eps_{\partial}$ used here in \rif{ww.31}. 
Hence \eqref{ww.41} and \eqref{ww.31}  give
\eqn{ww.local.mu}
$$
[\pm\mu_{\pm}\lfloor_{\mathcal U_{x_0}}]_{n,1}=[\pm\mu_\pm]_{n,1;\mathcal U_{x_0}}
\leq [\mu]_{n,1;\mathcal U_{x_0}}
\leq [\mu]_{n,1;\Omega_{r_0}}
\leq \varepsilon_\partial
\quad
\mbox{for every }x_0\in\partial\Omega.
$$
In view of \eqref{ww.local.mu},
Corollary \ref{correctors.uniform} yields correctors 
$\mathcal b_{x_0}^{\pm}\in
W^{2;n,1}(\mathcal U_{x_0})\cap
W^{1;n,1}_0(\mathcal U_{x_0})$ such that 
\eqn{ww.upper.corrector}
$$
-\mathcal M(\mathcal w_{f,x_0}^{\pm}+\mathcal b_{x_0}^{\pm})
=
\mathfrak f^{\pm}_{x_0}\pm\mu_{\pm}\lfloor_{\mathcal U_{x_0}},
\qquad
\nr{D\mathcal b_{x_0}^{\pm}}_{L^\infty(\mathcal U_{x_0})}
\leq
\frac{\mathcal s_f}{4},
\qquad
\mathcal b_{x_0}^+\geq0, \ \ 
\mathcal b_{x_0}^-\leq0. 
$$
Given a generic but fixed \(x_0\), we abbreviate
\(\mathcal w^\pm:=\mathcal w_{f,x_0}^\pm\),
\(\mathfrak f^\pm:=\mathfrak f_{x_0}^\pm\), and
\(\mathcal b^\pm:=\mathcal b_{x_0}^\pm\), and we continue to denote
the restrictions of \(\mu_\pm\) to \(\mathcal U_{x_0}\) by
\(\mu_\pm\).
Define
$
\mathcal W^\pm:=\mathcal w^\pm+\mathcal b^\pm.
$
By \eqref{ww.bs.f.bounds} and \eqref{ww.upper.corrector},
\eqn{ww.corrected.gap}
$$
\nr{D\mathcal W^\pm}_{L^\infty(\mathcal U_{x_0})}
\leq
1-\mathcal s_f+\frac{\mathcal s_f}{4}=1-\frac{3\mathcal s_f}{4}<1.
$$
Since $\mathcal b^\pm\equiv0$ on $\partial\mathcal U_{x_0}$,
\eqref{ww.bs.f.boundary} remains valid with $\mathcal W^\pm$
in place of $\mathcal w^\pm$. Moreover, by \eqref{ww.bs.f.sources} and
\eqref{ww.upper.corrector}, we have 
$$
-\mathcal M\mathcal W^-=\mathfrak f^--\mu_-
\leq f+\mu \leq
\mathfrak f^++\mu_+ =
-\mathcal M\mathcal W^+.
$$
For $\delta>0$, set
$v_\delta:=(u-\mathcal W^+-\delta)_+$, so that
$\supp\,v_\delta\Subset\mathcal U_{x_0}$.
We may therefore test the difference of the equations with
$v_\delta$, obtaining
$$
0\leq
\int_{\mathcal U_{x_0}}
\langle \partial H(Du)-\partial H(D\mathcal W^+),
Dv_\delta \rangle\dx =\int_{\mathcal U_{x_0}}
(f+\mu-\mathfrak f^+-\mu_+)v_\delta\dx
\leq0.
$$
Strict monotonicity gives $Dv_\delta=0$ almost everywhere,
hence $v_\delta=0$.
Letting $\delta\downarrow0$, we obtain $u\leq\mathcal W^+$.
The analogous argument with
$(\mathcal W^--u-\delta)_+$
gives $\mathcal W^-\leq u$. Therefore,
$
\mathcal W^-\leq u\leq\mathcal W^+
$
in $\overline{\mathcal U}_{x_0}$. By the endpoint Sobolev--Lorentz embedding,
$\mathcal b^\pm\in C^1(\overline{\mathcal U}_{x_0})$.
At $x_0$, the tangential derivatives of
$\mathcal W^-$, $u$, and $\mathcal W^+$ coincide.
If $\vv_{\Omega}$ denotes the common inward unit normal to
$\partial\Omega$ and $\partial\mathcal U_{x_0}$ at $x_0$, then
$
D_{\vv_{\Omega}}\mathcal W^-(x_0) \leq D_{\vv_{\Omega}} u(x_0) \leq D_{\vv_{\Omega}}\mathcal W^+(x_0).
$
It follows from \eqref{ww.corrected.gap} that
$$
\snr{Du(x_0)} \leq \max\left\{
\snr{D\mathcal W^-(x_0)},
\snr{D\mathcal W^+(x_0)}
\right\} \leq
1-\frac{3\mathcal s_f}{4}.
$$
Since $x_0\in\partial\Omega$ was arbitrary,
\eqref{ww.boundary.gap} follows with
$
\sigma_*:=3\mathcal s_f/4.
$
\end{proof}

\begin{remark}{\em Perturbative solvability arguments near a prescribed reference graph also appear in \cite[Theorem 1]{tsu20}, in the Euclidean setting and with \(W^{2,q}\) estimates for \(q>n\). Moreover, in the setting of uniformly elliptic problems, a related barrier perturbation argument appears in \cite[Propositions~4.1--4.2]{bmw16}, where inhomogeneous barriers for Pucci operators are compared with homogeneous ones, controlling the error for sources in $L^q$, $q>n$.}
\end{remark}

\noindent
A straightforward consequence of Proposition \ref{corr} is the
following boundary-collar estimate.

\begin{corollary}\label{app.cor}
Let $\beta\in(0,1)$, let
$\Omega\subset\mathbb R^n$ be a 
$C^{2,\beta}$-regular, bounded domain, and let $\tx u_0\in C^{2,\beta}(\overline\Omega)$ be such that  \eqref{assuntidir}$_2$ is satisfied. 
Fix $\Lambda_0,F_0, r_0>0$, and let
$\varepsilon_\partial$ and $\sigma_*$ be the quantities determined in
\eqref{piccole} according to Proposition \ref{corr}. 
Assume that
$f,\mu\in C^{0,\beta}(\overline\Omega)$,
$\nr f_{L^\infty(\Omega)}\leq F_0$,  and that
\eqn{ww.corollary.data}
$$
[\mu]_{n,1;\Omega}\leq\Lambda_0,
\qquad
[\mu]_{n,1;\Omega_{r_0}}\leq \varepsilon_\partial.
$$
Then the solution $u$ to \eqref{ww.background.perturbed.problem} 
belongs to $C^{2,\beta}(\overline\Omega)$. Moreover, there exists 
$
\varepsilon_*
\equiv 
\varepsilon_*
(\sigma_*,[Du]_{C^{0,\beta}(\overline\Omega)},\beta)\in (0,1)
$
such that
\eqn{ww.50}
$$
\nr{Du}_{
L^\infty\left(
\left\{
x\in\overline\Omega:
\dist(x,\partial\Omega)\leq\varepsilon_*
\right\}
\right)}
\leq
1-\frac{\sigma_*}{2}.
$$
 Moreover, 
\eqn{sottosotto}
$$
\mathcal h_u
\le
\frac{\mathcal h_{\partial\Omega}}{2} \quad \mbox{on $\overline\Omega
\cap
\left\{
\dist(\,\cdot\,,\partial\Omega)
\le
\varepsilon_*
\right\}$},
$$
where 
\eqn{boundary.boost.threshold}
$$
2< \mathcal h_{\partial\Omega}
:=
\frac{
2}{\sqrt{1-\left(1-\frac{\sigma_*}{2}\right)^2}}
=
\frac{2}{
\sqrt{\sigma_*
\left(1-\frac{\sigma_*}{4}
\right)}}\leq \frac{4}{\sqrt{3\sigma_*}}.
$$
In particular, 
$
\supp\, 
\left(\mathcal h_u-\kk\right)_+\Subset\Omega
$
for every $
\kk\ge
\mathcal h_{\partial\Omega}/2.$ Finally, $\supp\, 
\left(
\mathcal h_u-\mathcal h_{\partial\Omega}
\right)_+
\Subset
\Omega.$
\end{corollary}

\begin{proof}
Under the assumptions of the corollary,
\cite[Theorem 3.6]{bs82} gives
$u\in C^{2,\beta}(\overline\Omega)$ and
$Du\in C^{0,\beta}(\overline\Omega)$. In view of \eqref{ww.corollary.data},
Proposition \ref{corr} applies and gives
$
\nr{Du}_{C^0(\partial\Omega)}
\leq
1-\sigma_*.
$
Set
$$
M_{Du}
:= [Du]_{C^{0,\beta}(\overline\Omega)},
\qquad
\varepsilon_*
:=\min\left\{1,\left(\frac{\sigma_*}{2\max\{M_{Du},1\}}\right)^{1/\beta}\right\}.
$$
If $\dist(x,\partial\Omega)\leq\varepsilon_*$, choose
$y\in\partial\Omega$ with
$\snr{x-y}=\dist(x,\partial\Omega)$. Then
$
\snr{Du(x)}
\leq
\snr{Du(y)}
+
M_{Du}\snr{x-y}^{\beta}
\leq
1-\sigma_*+\sigma_*/2
=
1- \sigma_*/2.
$
This proves \eqref{ww.50}. This implies 
$$
\mathcal h_u\le\left[1-\left(1-\frac{\sigma_*}{2}\right)^2\right]^{-1/2}=\frac{\mathcal h_{\partial\Omega}}{2}
$$
in the corresponding boundary neighbourhood. Hence every truncation
with
$
\kk\ge\mathcal h_{\partial\Omega}/2
$
vanishes there, and its support is contained in
$
\{
x\in\overline\Omega:
\dist(x,\partial\Omega)\ge\varepsilon_*
\}
\Subset
\Omega.
$
\end{proof}


\section{Higher integrability a priori estimates}\label{sec.4}
All estimates in this section are a priori.
Throughout, we assume that \(u\) is smooth, strictly spacelike,
and solves \eqref{bi} classically.
These assumptions are understood in all statements
below involving \(u\). For the rest of the section we recall that the number $\Ui$ has been defined in \rif{mmi}. 
\begin{lemma}\label{cacclem}
Let $u\in \mathbb{X}(\mathbb{R}^{n})\cap C^{\infty}(\mathbb{R}^{n})$ be a classical solution to \eqref{bi} with $\mu \in C^{\infty}_0(\er^n)$. Then
\begin{flalign}\label{cacc}
&(m+1)\int_{\mathbb{R}^{n}}\eta^{2}\mathcal{h}_{u}^{m-2}\snr{D\mathcal{h}_{u}}^{2}\dx+\int_{\mathbb{R}^{n}}\eta^{2}\mathcal{h}_{u}^{m+2}\snr{D^{2}u}^{2}\dx\nonumber \\
&\qquad \quad \le c\max\left\{1,\frac{1}{1+m}\right\}\int_{\mathbb{R}^{n}}\mathcal{h}_{u}^{m+2}\snr{D\eta}^{2}\dx+c\max\left\{1,1+m\right\}\int_{\mathbb{R}^{n}}\eta^{2}\mu^{2}\mathcal{h}_{u}^{m}\dx
\end{flalign}
holds whenever $m> -1$, $\eta\in W^{1,\infty}(\mathbb{R}^{n})$ is nonnegative and compactly supported,
where \(c=c_n\). 
\end{lemma}

\begin{proof}
 Since $u$ is smooth, we can then differentiate \eqref{weshall2} to achieve
\eqn{0}
$$
\int_{\mathbb{R}^{n}}\langle\partial^{2}H(Du)DD_{s}u,D\varphi\rangle\dx=-\int_{\mathbb{R}^{n}}\mu D_{s}\varphi\dx=
\int_{\mathbb{R}^{n}}D_{s} \mu \varphi\dx,
$$
for all $\varphi\in C^{1}_{0}(\mathbb{R}^{n})$ and any $s\in \{1,\cdots,n\}$. Note that, by a standard density argument, \eqref{0} remains valid for every
compactly supported $\varphi\in W^{1,2}(\mathbb R^n)$. Recalling \rif{recalla}, we shall repeatedly use the relations
\eqn{identitas}
$$
D\mathcal h_u
=
\mathcal h_u^3D^2u\,Du, \qquad \mathcal H(Du)D\mathcal h_u
=
\sum_{s=1}^n
D_su\,\partial^2H(Du)DD_su.
$$
Let $\eta\in W^{1,\infty}(\mathbb{R}^{n})$ be a non-negative, compactly supported function, $m> -1$ be a number, test \eqref{0} against $\varphi:=\eta^{2}\mathcal{h}_{u}^{m+1}D_{s}u$ and sum over $s\in \{1,\cdots,n\}$ to get, via \rif{identitas}
\begin{flalign*}
0&=2\int_{\mathbb{R}^{n}}\eta\mathcal{h}_{u}^{m+1}\langle\mathcal H(Du)D\mathcal h_{u},D\eta \rangle\dx\nonumber +(m+1) \int_{\mathbb{R}^{n}}\eta^{2}\mathcal{h}_{u}^{m}\langle\mathcal H(Du)D\mathcal h_{u},D\mathcal{h}_{u}\rangle\dx\nonumber \\
&\ \  \ +\sum_{s=1}^{n}\int_{\mathbb{R}^{n}}\eta^{2}\mathcal{h}_{u}^{m+1}\langle\partial^{2}H(Du)DD_{s}u,DD_{s}u\rangle\dx\nonumber\\
&\ \ \ +\sum_{s=1}^{n}\int_{\mathbb{R}^{n}}\mu D_{s}(\eta^{2}\mathcal{h}_{u}^{m+1}D_{s}u)\dx=:\mbox{(I)}+\mbox{(II)}+\mbox{(III)}+\mbox{(IV)}.
\end{flalign*}  
We estimate via Cauchy-Schwarz and Young inequalities,
\begin{eqnarray*}
\snr{\mbox{(I)}}
&\le&2\left(\int_{\mathbb{R}^{n}}\eta^{2}\mathcal{h}_{u}^{m}\langle\mathcal H(Du)D\mathcal{h}_{u},D\mathcal{h}_{u}\rangle\dx\right)^{1/2}\nonumber\left(\int_{\mathbb{R}^{n}}\mathcal{h}_{u}^{m+2}\langle\mathcal H(Du)D\eta,D\eta\rangle\dx\right)^{1/2}\nonumber \\
&\stackrel{\eqref{0.1}_{2}}{\le}&\varepsilon_{1}\mbox{(II)}\nonumber  +\frac{c}{\varepsilon_{1}(m+1)}\int_{\mathbb{R}^{n}}\mathcal{h}_{u}^{m+2}\snr{D\eta}^{2}\dx,
\end{eqnarray*}
for some $\varepsilon_{1}\in (0,1)$ to be fixed and $c\equiv c_{n}$. 
Finally, 
$$
\mbox{(II)} \geq (m+1)\int_{\mathbb{R}^{n}}\eta^{2}\mathcal{h}_{u}^{m-2}\snr{D\mathcal{h}_{u}}^{2}\dx, \qquad 
\mbox{(III)}\stackrel{\eqref{0.1.1}_{1}}{\ge}\int_{\mathbb{R}^{n}}\eta^{2}\mathcal{h}_{u}^{m+2}\snr{D^{2}u}^{2}\dx. 
$$
Let us now take care of the terms involving \(\mu\). Expanding
\(\mathrm{(IV)}\), we have
\begin{flalign*}
|\mathrm{(IV)}|
&\le
2\int_{\mathbb R^n}
\eta|\mu|\mathcal h_u^{m+1}
|D\eta||Du|\dx
+(m+1)\int_{\mathbb R^n}
\eta^2|\mu|\mathcal h_u^m
|\langle Du,D\mathcal h_u\rangle|\dx
\\
& \qquad 
+\int_{\mathbb R^n}
\eta^2|\mu|\mathcal h_u^{m+1}
|\Delta u|\dx.
\end{flalign*}
Using \(|Du|\le1\), the identity
\(\mathcal H(Du)Du=Du\), the Cauchy--Schwarz inequality associated
with \(\mathcal H(Du)\), and Young's inequality, we obtain
\begin{flalign*}
|\mathrm{(IV)}|
&\le
c\int_{\mathbb R^n}
\mathcal h_u^{m+2}|D\eta|^2\dx
+
c\int_{\mathbb R^n}
\eta^2\mu^2\mathcal h_u^m\dx
+\varepsilon_2\mathrm{(II)}
+
\frac{c(m+1)}{\varepsilon_2}
\int_{\mathbb R^n}
\eta^2\mu^2\mathcal h_u^m\dx
\\
& \quad 
+\varepsilon_3\mathrm{(III)}
+
\frac{c}{\varepsilon_3}
\int_{\mathbb R^n}
\eta^2\mu^2\mathcal h_u^m\dx
\\
&\le 
c\int_{\mathbb R^n}
\mathcal h_u^{m+2}|D\eta|^2\dx
+
c\left(
1+\frac{m+1}{\varepsilon_2}
+\frac1{\varepsilon_3}
\right)
\int_{\mathbb R^n}
\eta^2\mu^2\mathcal h_u^m\dx
+\varepsilon_2\mathrm{(II)}
+\varepsilon_3\mathrm{(III)},
\end{flalign*}
for $c\equiv c_{n}$, and some $\varepsilon_{2},\varepsilon_{3}\in (0,1)$. Merging the content of all previous displays after choosing $\varepsilon_{1},\varepsilon_{2},\varepsilon_{3}\in (0,1)$ sufficiently small, and recalling $\eqref{0.1}_{1}$, we obtain \eqref{cacc}. 
\end{proof}
\begin{lemma}\label{w22.lem}
Let $u\in \mathbb{X}(\mathbb{R}^{n})\cap C^{\infty}(\mathbb{R}^{n})$ be a classical solution to \eqref{bi} with $\mu \in C^{\infty}_0(\er^n)$. Let
\(B\subset\mathbb R^n\) be a ball. If \(n\ge3\) and \(q\in[2,\infty)\), then
\eqn{w22.high}
$$
\begin{aligned}
&\int_B
(|D\log\mathcal h_u|^2+|D\mathcal h_u^{-1}|^2+|D^2u|^2)\dx\\
&\qquad\le c\mathcal h_\infty(\mu)^2
|B|^{1-2/n}+c|B|^{1-2/q}
\left(1+\frac{\Ui^2}{|B|^{2/n}}
\right)\|\mu\|_{L^q(\mathbb R^n)}^2,
\end{aligned}
$$
where \(c=c(n,q)\). If \(n=2\), then
\eqn{w22.two}
$$
\int_{\mathbb R^2}
(|D\log\mathcal h_u|^2+|D\mathcal h_u^{-1}|^2+|D^2u|^2)\dx
\le
c\|\mu\|_{L^2(\mathbb R^2)}^2,
$$
where \(c>0\) is a dimensional constant. Moreover, again when $n=2$,
\eqn{w22.two.log}
$$
\int_{\mathbb R^2}
(\log\mathcal h_u)^2\dx
\le
\|\mu\|_{\mathcal D^{1,2}(\mathbb R^2)^*}^2.
$$
\end{lemma}

\begin{proof}
By Proposition \ref{connessione}, $\mu\in\DD^*$ and $u$ is the
global minimizer of $\mathcal E_\mu$. Moreover,
$\mu\in L^1_0(\mathbb R^2)$ when $n=2$.
Thus Proposition \ref{com.sup} applies.
Using that
$
|D\mathcal h_u^{-1}|^2
\le
\mathcal h_u^2|Du|^2|D^2u|^2,
$
since
$
1+\mathcal h_u^2|Du|^2=\mathcal h_u^2,
$
it follows that
\eqn{lavoro}
$$
|D\mathcal h_u^{-1}|^2+|D^2u|^2 \le \mathcal h_u^2|D^2u|^2.
$$ Assume first that \(n\ge3\).
Fix a ball \(B\subset\mathbb R^n\) and choose
\(\eta\in W^{1,\infty}_0(2B)\) such that
$
0\le\eta\le1$, $
\eta=1$ in $B$,
$|D\eta|\le c|B|^{-1/n}.$
Applying \eqref{cacc} with \(m=0\), and eventually \rif{lavoro}, we obtain
\eqn{w22.aux}
$$
\begin{aligned}
&\int_B
(|D\log\mathcal h_u|^2+|D\mathcal h_u^{-1}|^2+|D^2u|^2)\dx
\\
&\qquad\le c
\mathcal h_\infty(\mu)^2
|B|^{1-2/n}+c|B|^{-2/n}
\int_{2B}(\mathcal h_u-\mathcal h_\infty(\mu))_+^2\dx
+c\int_{2B}\mu^2\dx.
\end{aligned}
$$
By H\"older's inequality and \eqref{czcz},
\begin{flalign*}
\int_{2B}(\mathcal h_u-\mathcal h_\infty(\mu))_+^2\dx
&\le
c|B|^{1-2/q}
\|(\mathcal h_u-\mathcal h_\infty(\mu))_+\|_{L^q(\mathbb R^n)}^2\\
&\le
c|B|^{1-2/q}
\frac{\Ui^2}
{(1-\mathcal h_\infty(\mu)^{-2})^2}
\|\mu\|_{L^q(\mathbb R^n)}^2
\leq c
|B|^{1-2/q}
\Ui^2\|\mu\|_{L^q(\mathbb R^n)}^2.
\end{flalign*}
Note that we have used that
$
\mathcal h_\infty(\mu)\ge2,
$
and hence
$
1-\mathcal h_\infty(\mu)^{-2}
\ge
3/4.
$
Similarly,
$$
\|\mu\|_{L^2(2B)}^2
\le
c|B|^{1-2/q}
\|\mu\|_{L^q(\mathbb R^n)}^2.
$$
Inserting the last two estimates into
\eqref{w22.aux} gives \eqref{w22.high}. Assume now that $n=2$.
Let $\eta_T$ be the logarithmic cut-off functions defined in
\eqref{logatau}. Applying \eqref{cacc} with $m=0$ and
$\eta=\eta_T$, we obtain
\eqn{cutoffa2}
$$
\int_{\mathbb R^2}
\eta_T^2\left(|D\log\mathcal h_u|^2
+\mathcal h_u^2|D^2u|^2
\right)\dx\le c\int_{\mathbb R^2} |D\eta_T|^2
\mathcal h_u^2\dx+c\int_{\mathbb R^2}\eta_T^2\mu^2\dx.
$$
By \eqref{boost.at.infinity}, there exists $R_u>0$ such that
$
\mathcal h_u\le2
$
in $\mathbb R^2\setminus B_{R_u}(0)$.
Since $D\eta_T=0$ in $B_{T-e}(0)$, for $T>R_u+e$, using \rif{cutoffa} we find
$$
\int_{\mathbb R^2}
|D\eta_T|^2 \mathcal h_u^2\dx
\le
4\|D\eta_T\|_{L^2(\mathbb R^2)}^2
\le
\frac{c}{\log T}
\stackrel{T\to\infty}{\longrightarrow}0.
$$
As $\eta_T\to1$ pointwise and $0\le\eta_T\le1$, using the content of the last display in \rif{cutoffa2} and 
Fatou's lemma we arrive at
$$
\int_{\mathbb R^2}
\left(
|D\log\mathcal h_u|^2
+
\mathcal h_u^2|D^2u|^2
\right)\dx
\le
c\|\mu\|_{L^2(\mathbb R^2)}^2
$$
which implies \eqref{w22.two} taking \rif{lavoro} into account. Finally, by \eqref{l1l1} and $\mathcal h_u\ge1$,
$$
\|Du\|_{L^2(\mathbb R^2)}^2
\le
\int_{\mathbb R^2}\mathcal h_u|Du|^2\dx
\le
\langle\mu,u\rangle
\le
\|\mu\|_{\mathcal D^{1,2}(\mathbb R^2)^*}\|Du\|_{L^2(\mathbb R^2)}.
$$
Consequently,
$
\|Du\|_{L^2(\mathbb R^2)}
\le
\|\mu\|_{\mathcal D^{1,2}(\mathbb R^2)^*}
$
and
$
\int_{\mathbb R^2}\mathcal h_u|Du|^2\dx
\le
\|\mu\|_{\mathcal D^{1,2}(\mathbb R^2)^*}^2.
$
Using the elementary inequality
$
(\log t)^2\le t-t^{-1}
$
for $t\ge1$, together with
$
\mathcal h_u-\mathcal h_u^{-1}
=
\mathcal h_u|Du|^2,
$
we conclude that
$$
\int_{\mathbb R^2}(\log\mathcal h_u)^2\dx
\le
\int_{\mathbb R^2}\mathcal h_u|Du|^2\dx
\leq \|\mu\|_{\mathcal D^{1,2}(\mathbb R^2)^*}^2,
$$
which proves \eqref{w22.two.log}.
\end{proof}

\begin{lemma}\label{caccim} Let $u\in \mathbb{X}(\mathbb{R}^{n})\cap C^{\infty}(\mathbb{R}^{n})$ be a classical solution to \eqref{bi} with $\mu \in C^{\infty}_0(\er^n)$. Assume that $n\geq 3$. 
Consider $x_0\in \er^n$ and the condition
\eqn{smallina}
$$\nr{\mu}_{L^{n}(B_{r}(x_{0}))}\le \nr{\tx{g}}_{L^{n}(B_{2r}(x_{0}))},
$$
where $\tx{g}\in L^{n}(\mathbb{R}^{n})$ is a fixed function.  
For every $m>-1$ there exists a positive threshold 
\eqn{lasoglia}
$$r_{*}\equiv r_{*}(n,m,\tx{g})\in (0,1]$$ 
such that if \eqref{smallina} is satisfied for every $r \leq r_{*}$, then the Caccioppoli type inequality
 \eqn{cacc.c}
 $$
 \int_{\mathbb{R}^{n}}\eta^{2}\left(\mathcal{h}_{u}^{m}+\snr{D\mathcal{h}_{u}^{m/2}}^{2}\right)\dx+\int_{\mathbb{R}^{n}}\eta^{2}\mathcal{h}_{u}^{m+2}\snr{D^{2}u}^{2}\dx\le c\int_{\mathbb{R}^{n}}\mathcal{h}_{u}^{m+2}\snr{D\eta}^{2}\dx
 $$
 holds true for all non-negative $\eta\in W^{1,\infty}_{0}(B_{5r/6}(x_{0}))$ such that $r\leq r_{*}$, where $c\equiv c(n,m)$. In particular,  the choice $\tx{g}\equiv \mu$ is admissible. Consequently \eqn{ccc}
$$
\nr{\mathcal{h}_{u}^{m/2}}_{W^{1,2}(B_{3r/4}(x_{0}))}\le c(n,m)r^{-1}\nr{\mathcal{h}_{u}^{m/2+1}}_{L^{2}(B_{5r/6}(x_{0}))}. 
$$
\end{lemma}
\begin{proof}
With $\eta$ as in the statement, triangle inequality and \eqref{cacc} yield
\begin{flalign}\label{caccr.3}
&\int_{\mathbb{R}^{n}}\snr{D(\eta\mathcal{h}_{u}^{m/2})}^{2}\dx+\int_{\mathbb{R}^{n}}\eta^{2}\mathcal{h}_{u}^{m+2}\snr{D^{2}u}^{2}\dx\nonumber \\
&\qquad \qquad\quad\qquad \quad \le c\int_{\mathbb{R}^{n}}\eta^{2}\mathcal{h}_{u}^{m-2}\snr{D\mathcal{h}_{u}}^{2}\dx+c\int_{\mathbb{R}^{n}}\snr{D\eta}^{2}\mathcal{h}_{u}^{m}\dx\nonumber   +c\int_{\mathbb{R}^{n}}\eta^{2}\mathcal{h}_{u}^{m+2}\snr{D^{2}u}^{2}\dx\nonumber \\
&\qquad\qquad \qquad\quad \quad\le c\int_{\mathbb{R}^{n}}\mathcal{h}_{u}^{m+2}\snr{D\eta}^{2}\dx+c\int_{\mathbb{R}^{n}}\eta^{2}\mu^{2}\mathcal{h}_{u}^{m}\dx\nonumber \\
&\qquad \qquad\qquad\quad \quad \le c\int_{\mathbb{R}^{n}}\mathcal{h}_{u}^{m+2}\snr{D\eta}^{2}\dx+c\nr{\mu}_{L^{n}(\supp \, \eta )}^{2}\left(\int_{\mathbb{R}^{n}}\eta^{2^{*}}\mathcal{h}_{u}^{2^{*}m/2}\dx\right)^{2/2^{*}}\nonumber \\
&\qquad \qquad\qquad\quad \quad \le c\int_{\mathbb{R}^{n}}\mathcal{h}_{u}^{m+2}\snr{D\eta}^{2}\dx+c\nr{\tx{g}}_{L^{n}(B_{2r_{*}})}^{2}\left(\int_{\mathbb{R}^{n}}\eta^{2^{*}}\mathcal{h}_{u}^{2^{*}m/2}\dx\right)^{2/2^{*}},
\end{flalign}  
with $c\equiv c(n,m)$. By \eqref{cacc}, \eqref{caccr.3} and Sobolev embedding,
\begin{flalign}\label{cacc.4} 
&\left(\int_{\mathbb{R}^{n}}\eta^{2^{*}}\mathcal{h}_{u}^{2^{*}m/2}\dx\right)^{2/2^{*}}+\int_{\mathbb{R}^{n}}\eta^{2}\mathcal{h}_{u}^{m-2}\snr{D\mathcal{h}_{u}}^{2}\dx+\int_{\mathbb{R}^{n}}\eta^{2}\mathcal{h}_{u}^{m+2}\snr{D^{2}u}^{2}\dx\nonumber \\
&\qquad \qquad  \quad \le \tilde c\int_{\mathbb{R}^{n}}\mathcal{h}_{u}^{m+2}\snr{D\eta}^{2}\dx+\tilde c\nr{\tx{g}}_{L^{n}(B_{2r_{*}})}^{2}\left(\int_{\mathbb{R}^{n}}\eta^{2^{*}}\mathcal{h}_{u}^{2^{*}m/2}\dx\right)^{2/2^{*}},
\end{flalign}
where $\ti{c}\equiv \ti{c}(n,m)$. By absolute continuity of the integral, choose
\(r_*=r_*(n,m,\tx{g})\in(0,1]\) such that
\[
\widetilde c\sup_{x\in\mathbb R^n}
\|\tx{g}\|_{L^n(B_{2r_*}(x))}^{2}
\le \frac12.
\]
We can then absorb the last term on the right-hand side
of \eqref{cacc.4} and obtain \eqref{cacc.c}
by standard manipulations. As \eqref{cacc.c} holds for all non-negative $\eta\in W^{1,\infty}_{0}(B_{5r/6}(x_{0}))$, we just select one such that $\mathds{1}_{B_{3r/4}(x_{0})}\le \eta\le \mathds{1}_{B_{5r/6}(x_{0})}$ and $\snr{D\eta}\lesssim r^{-1}$, and \rif{ccc} follows directly from \eqref{cacc.c}.  
\end{proof}
\noindent To proceed, we need a preliminary, optimization lemma, which is a variant of \cite[Lemma 3]{bs20}. The proof, up to minimal changes, can be extracted from \cite[Lemma 6.1, (6.5)-(6.8)]{dkk24}, letting $\alpha=0$.
\begin{lemma}\label{bslem}
Let $n\ge3$, let $0<\tau_1<\tau_2<\infty$, and define
\eqn{gamma}
$$
\vartheta:=
\begin{cases}
\dfrac{n-3}{n-1} & n\ge4\\[2mm]
\vartheta_0\in(0,1) & n=3.
\end{cases}
$$
Let $v\ge0$ satisfy
$
v,v^\vartheta\in W^{1,2}(B_{\tau_2}).
$
Then
$$
\inf_{\substack{
\eta\in C^1_0(B_{\tau_2})\\
\mathds{1}_{B_{\tau_1}}\le\eta\le\mathds{1}_{B_{\tau_2}}
}}
\nr{vD\eta}_{L^2(B_{\tau_2})}^{2}
\le
\frac{c}
{\tau_2^{(n-1)(1-\vartheta)/\vartheta}
(\tau_2-\tau_1)^{(1+\vartheta)/\vartheta}}
\left(
\tau_2^{2/\vartheta}
\nr{Dv^\vartheta}_{L^2(B_{\tau_2})}^{2/\vartheta}
+
\nr{v^\vartheta}_{L^2(B_{\tau_2})}^{2/\vartheta}
\right),
$$
where $c=c(n)$ if $n\ge4$, and $c=c(\vartheta_0)$ if $n=3$.
\end{lemma}
Next, the core result for this part, that is, an arbitrarily high size gain for $\mathcal{h}_{u}$.
\begin{proposition}\label{p4.4} Within the same setting as Lemma \ref{caccim}, for every $p>1$ there exists a positive threshold 
\eqn{lasoglia2}
$$r_{p}\equiv r_{p}(n,p,\tx{g})\in (0,1]$$  such that if \eqref{smallina} is satisfied for every $r\leq r_p$, then   
\eqn{hiiigh}
$$
\nr{\mathcal{h}_{u}^{p/2}}_{W^{1,2}(B_{r/8}(x_0))}\le cr^{-\gamma_p}\left(\mathcal{h}_{\infty}(\mu)+\Ui\nr{\mu}_{L^{n}(\mathbb{R}^{n})}\right)^{\tx{b}_{p}}
$$
holds whenever $r \leq r_p$, where $c,\tx{b}_{p}, \gamma_p\equiv c,\tx{b}_{p},\gamma_p(n,p)\geq 1$.
\end{proposition}

\begin{proof}
Fix \(p\in[1,\infty)\). Let \(\vartheta\in(0,1)\) be defined in
\eqref{gamma}, choosing \(\vartheta_0=1/4\) when \(n=3\), and set
$
m_0:=n-2, 
$
so that 
$
m_0> 2\vartheta/(1-\vartheta).
$
Define recursively
\eqn{4.1}
$$
m_{i+1}:=\frac{m_i}{\vartheta}-2,
\qquad
i\in\mathbb N\cup\{0\}.
$$
A direct induction gives
\eqn{4}
$$
m_i=\frac{1}{\vartheta^{i}}
\left(m_0-\frac{2\vartheta}{1-\vartheta}\right)
+\frac{2\vartheta}{1-\vartheta},\qquad
i\in\mathbb N\cup\{0\},
$$
and hence
$
m_i\to\infty.
$
Choose
$$
i_p:=\min\{i\ge1:m_i\ge p\},
\qquad \mathfrak{m}_{p}:=m_{i_p},
\qquad
(0,1]\ni r_p:=\min_{0\le j\le i_p}r_*(n,m_j,\tx{g}).
$$
Here $r_*(n,m_j,\tx{g})$ is the threshold in
\eqref{lasoglia} corresponding to $m=m_j$.
This determines the number $r_p$ appearing in
\eqref{lasoglia2}.
By the proof of Lemma \ref{caccim}, estimate
\eqref{cacc.c} holds with \(m\equiv m_i\), \(0\le i\le i_p\),
for every \(0<r\le r_p\), at the fixed center \(x_0\). Fix such a ball \(B_r(x_0)\), and choose
$
 r/8\le\tau_1<\tau_2\le 3r/4.
$ In the following all the balls will be centered at $x_0$ and we shall omit to denote the center $B_r\equiv B_r(x_0)$. 
For \(i\in\mathbb N\cup\{0\}\), define
$
\varrho_i
:=
\tau_1+(\tau_2-\tau_1)/2^{i}$, 
$
B_i:=B_{\varrho_i}.
$
Then
$
\varrho_0=\tau_2$, $
\varrho_i\searrow\tau_1$, $
\varrho_i-\varrho_{i+1} = (\tau_2-\tau_1)/2^{i+1}.
$ 
 For \(i\in\{0,\ldots,i_p-1\}\), apply Lemma \ref{bslem}
with \(v:=\mathcal h_u^{(m_{i+1}+2)/2}\).
Since $\vartheta(m_{i+1}+2)=m_i$ by \eqref{4.1}
and $r/8\le\varrho_i\le r\le1$, we obtain
$$
\inf_{\substack{
\eta\in C^1_0(B_i)\\ 
\mathds{1}_{B_{i+1}}\le\eta\le\mathds{1}_{B_i}}}
\nr{\mathcal h_u^{m_{i+1}/2+1}D\eta}_{L^2(B_i)}^2
\le \frac{c}{(\varrho_i-\varrho_{i+1})^{(1+\vartheta)/\vartheta}r^{\kappa}}
\nr{\mathcal h_u^{m_i/2}}_{W^{1,2}(B_i)}^{2/\vartheta},
$$
where we are denoting 
$
\kappa:=(n-1)(1-\vartheta)/\vartheta
$. 
Applying \eqref{cacc.c} with \(m=m_{i+1}\) and taking
the infimum over the admissible cutoff functions yields
\eqn{caccr.2}
$$
\nr{\mathcal h_u^{m_{i+1}/2}}_{W^{1,2}(B_{i+1})}
\le
\frac{c(n,m_{i+1})}
{(\varrho_i-\varrho_{i+1})^{(1+\vartheta)/(2\vartheta)}r^{\kappa/2}}
\nr{\mathcal h_u^{m_i/2}}_{W^{1,2}(B_i)}^{1/\vartheta}.
$$
Set
$$
\tx{H}_i:=
\nr{\mathcal h_u^{m_i/2}}_{W^{1,2}(B_i)}^{2/m_i}.
$$
Then \eqref{caccr.2} gives
$$
\tx{H}_{i+1}
\le
\frac{c}
{(\varrho_i-\varrho_{i+1})^{(1+\vartheta)/
(\vartheta m_{i+1})}r^{\kappa/m_{i+1}}}
\tx{H}_i^{m_i/(\vartheta m_{i+1})}.
$$
Iterating this inequality, we obtain, for \(1\le i\le i_p\),
\eqn{3}
$$
\tx{H}_i
\le
\frac{c\,r^{-\alpha_i}}
{(\tau_2-\tau_1)^{\widetilde\beta_i}}
\tx{H}_0^{m_0/(\vartheta^i m_i)},
\qquad
\alpha_i:=
\frac{(n-1)(\vartheta^{-i}-1)}{m_i},
\qquad
\widetilde\beta_i:=
\frac{1+\vartheta}{m_i}
\sum_{j=1}^{i}\frac 1{\vartheta^{j}},
$$
where \(c=c(n,i)\). We now choose
$
\tau_1=r/8,
$ and $
\tau_2=3r/4.
$
By \eqref{ccc},
$$
\nr{\mathcal h_u^{m_0/2}}_{W^{1,2}(B_{3r/4})}\le
cr^{-1}\nr{\mathcal h_u^{m_0/2+1}}_{L^2(B_{5r/6})}.
$$
Since \(m_0+2=n\), \eqref{3} gives
$$
\nr{\mathcal h_u^{m_i/2}}_{W^{1,2}(B_{r/8})}
\le
Cr^{-\gamma_i}
\nr{\mathcal h_u}_{L^n(B_{5r/6})}^{\frac{n}{2\vartheta^i}},
$$
for some \(\gamma_i=\gamma_i(n,i)>0\). By \eqref{4}, the index $i_p$ is finite and depends
only on $n$ and $p$.
We apply the last estimate with $i=i_p$. Since \(\mathfrak{m}_{p}\ge p\) and
\(\mathcal h_u\ge1\), we have
$$
\snr{D\mathcal h_u^{p/2}}
=
\frac{p}{\mathfrak{m}_{p}}
\mathcal h_u^{(p-\mathfrak{m}_{p})/2}
\snr{D\mathcal h_u^{\mathfrak{m}_{p}/2}}
\le \frac{p}{\mathfrak{m}_{p}}
\snr{D\mathcal h_u^{\mathfrak{m}_{p}/2} 
}.
$$
Moreover, by H\"older's inequality,
$$
\begin{aligned}
\nr{\mathcal h_u^{p/2}}_{L^2(B_{r/8})}
&\le
|B_{r/8}|^{\frac{\mathfrak{m}_{p}-p}{2\mathfrak{m}_{p}}}
\nr{\mathcal h_u^{\mathfrak{m}_{p}/2}}_{L^2(B_{r/8})}^{p/\mathfrak{m}_{p}}\le
c\nr{\mathcal h_u^{\mathfrak{m}_{p}/2}}_{W^{1,2}(B_{r/8})}+c.
\end{aligned}
$$
Consequently,
\eqn{3.0.1}
$$
\nr{\mathcal h_u^{p/2}}_{W^{1,2}(B_{r/8})}
\le
cr^{-\gamma_p}\left(1+\nr{\mathcal h_u}_{L^n(B_{5r/6})}^{\tx{b}_p}\right),
$$
where
$
\tx{b}_p:=
n/(2\vartheta^{i_p})$ and $
\gamma_p=\gamma_p(n,p)>0$. Finally, by \eqref{czcz} with \(q=n\), recalling that
\(\mathcal h_\infty(\mu)\ge2\) and \(r\le r_p\le1\), we obtain
$$
\begin{aligned}
\|\mathcal h_u\|_{L^n(B_{5r/6})}
&\le |B_{5r/6}|^{1/n}\mathcal h_\infty(\mu)
+\|(\mathcal h_u-\mathcal h_\infty(\mu))_+\|_{L^n(\mathbb R^n)}
\\
&\le c_n\left[
\mathcal h_\infty(\mu)+\mathfrak m_\mu\|\mu\|_{L^n(\mathbb R^n)}\right].
\end{aligned}
$$
Inserting this bound into \eqref{3.0.1} and absorbing the additive
constant by using \(\mathcal h_\infty(\mu)\ge2\) yields
\eqref{hiiigh}, completing the proof.
\end{proof}

\section{A calibrated De Giorgi iteration}\label{calisec}

The next lemma provides a global, calibrated variant of
the classical De Giorgi iteration. It may be of independent interest
and useful elsewhere. Its two main features are that the iteration is
performed on a fixed domain and that the level increments are chosen
in a self-calibrating scheme. This choice enforces
a prescribed geometric decay of the truncated energy and is inspired
by nonlinear potential theoretic methods originating in
\cite{km94}; see also \cite[Section 3]{bm20}. 

\begin{lemma}[Calibrated De Giorgi-type iteration]\label{abstract.iteration}
Let $B\subset\mathbb R^n$ be a bounded domain, let
$
v:B\to[0,\infty)
$
be a measurable function, and let
$
\kappa_0\ge0.
$
For every $\kappa\ge\kappa_0$, set
$$
E_\kappa:=\{x\in B:v(x)>\kappa\}.
$$
Assume that
\eqn{abstract.zero.boundary}
$$
(v-\kappa)_+
\in
W^{1,2}_0(B)
\qquad
\mbox{for every }\kappa\ge\kappa_0,
$$
and define
$
T(\kappa)
:=
\nr{D(v-\kappa)_+}_{L^2(B)}
$
and assume that
$
S_0:=\snr{E_{\kappa_0}}>0.
$
Let
$
\omega:[0,S_0]\to[0,\infty)
$
be a nondecreasing, absolutely continuous function such that
$
\omega(0)=0,
$
and assume there exists $\mathfrak{c}_0\ge1$ such that
\eqn{abstract.energy}
$$
T(\kappa)
\le
\mathfrak{c}_0\omega(\snr{E_\kappa})
\qquad
\mbox{for every }\kappa\ge\kappa_0.
$$
It follows that
\begin{itemize}
\item
If there exist $\chi>1$ and $\mathfrak{c}_{1}\ge1$ such that
\eqn{abstract.power.level}
$$
\snr{E_{\kappa_2}}
\le
\mathfrak{c}_{1}
\left(
\frac{T(\kappa_1)}
{\kappa_2-\kappa_1}
\right)^\chi
$$
whenever
$
\kappa_0\le\kappa_1<\kappa_2,
$
then
\eqn{abstract.power.bound}
$$
\operatorname*{ess\,sup}_{B}v
\le
\kappa_0
+
c(\chi)\mathfrak{c}_0\mathfrak{c}_{1}^{1/\chi}
\int_0^{S_0}
\frac{\omega'(s)}{s^{1/\chi}}
\ds.
$$

\item
If there exist constants
$
\mathfrak{c}_{2}\ge1
$
and
$
\mathfrak{c}_{3}>0
$
such that
\eqn{abstract.exp.level}
$$
\snr{E_{\kappa_2}}
\le
\mathfrak{c}_{2}
\exp\left\{
-\frac{
\mathfrak{c}_{3}(\kappa_2-\kappa_1)^2
}{
T(\kappa_1)^2
}
\right\}
\snr{E_{\kappa_1}}
$$
whenever
$
\kappa_0\le\kappa_1<\kappa_2
$
and
$
T(\kappa_1)>0,
$
then
\eqn{abstract.exp.bound}
$$
\operatorname*{ess\,sup}_{B}v
\le
\kappa_0
+
\frac{c\mathfrak{c}_0}{\sqrt{\mathfrak{c}_{3}}}
\int_0^{S_0}
\omega'(s)
\logs\left(
\frac{e\mathfrak{c}_{2}S_0}{s}
\right)
\ds.
$$
\end{itemize}
\end{lemma}
\begin{proof}
Note that $S_0>0$ implies 
$
T(\kappa_0)>0,
$
since otherwise \rif{abstract.zero.boundary} and Poincar\'e's
inequality would give
$
(v-\kappa_0)_+=0
$
a.e. in $B$, contradicting $S_0>0$. Note also that we can always assume that $\omega$ is strictly increasing. Otherwise, for $\varepsilon>0$, we replace it  by 
$
s \mapsto \omega_\varepsilon(s):=\omega(s)+\varepsilon s.
$
Estimate \rif{abstract.energy} remains valid with
$
\omega_\varepsilon
$
in place of $\omega$, and moreover (according to the case under consideration)
$$
\varepsilon
\int_0^{S_0}s^{-1/\chi}\ds
\to0, 
\qquad 
\varepsilon
\int_0^{S_0}
\logs\left(
\frac{e\mathfrak{c}_{2}S_0}{s}
\right)
\ds
\to0
\qquad
\mbox{as }\varepsilon\to0
$$
since $\chi>1$. Applying the lemma to $\omega_{\eps}$ and eventually letting $\varepsilon\to0$ in the resulting
estimates, we therefore assume that the inverse $
\omega^{-1}
$ exists. 
\medskip

\noindent
{\em Proof of \eqref{abstract.power.bound}.}
We can assume that the right-hand side of
\rif{abstract.power.bound} is finite. Set
$
T_0:=T(\kappa_0)
$
and, for every $i\ge1$, define
\eqn{abstract.ti}
$$
t_i
:=
\frac{T_0}{\mathfrak{c}_0 2^{i}},
\qquad
\tx{b}_i
:=
\omega^{-1}(t_i)>0.
$$
Notice that
$t_1=T_0/(2\mathfrak{c}_0)\le\omega(S_0)/2,$
by \eqref{abstract.energy}, and hence each $\tx{b}_i$ is well-defined.
Define recursively
\eqn{abstract.di.power}
$$
d_i
:=
\frac{
\mathfrak{c}_{1}^{1/\chi}T(\kappa_{i-1})
}{
\tx{b}_i^{1/\chi}
},
\qquad
\kappa_i
:=
\kappa_{i-1}+d_i,
\qquad
i\ge1.
$$
This definition makes sense for every $i\ge1$. Thus, if
$
T(\kappa_j)=0
$
for some $j\ge0$, then
$
\kappa_i=\kappa_j
$
for every $i\ge j$, while
$
d_i=0
$
for every $i\ge j+1$. We claim that
\eqn{abstract.decay.power}
$$
T(\kappa_i)
\le
\frac{T_0}{2^{i}}
\qquad
\mbox{for every }i\ge0.
$$
The assertion is trivial for $i=0$. Fix $i\ge1$. If
$
T(\kappa_{i-1})=0,
$
then
$
d_i=0,
$
$
\kappa_i=\kappa_{i-1},
$
and therefore
$
T(\kappa_i)=0,
$
so the assertion follows. Suppose instead that
$
T(\kappa_{i-1})>0.
$
Then $d_i>0$ and
$
\kappa_{i-1}<\kappa_i.
$
By \eqref{abstract.power.level} and \eqref{abstract.di.power},
$$
\snr{E_{\kappa_i}}
\le
\mathfrak{c}_{1}
\left[
\frac{T(\kappa_{i-1})}{d_i}
\right]^\chi
=
\tx{b}_i.
$$
Therefore, by \eqref{abstract.energy},
\eqn{ray0}
$$
T(\kappa_i)
\le
\mathfrak{c}_0\omega(\snr{E_{\kappa_i}})
\le
\mathfrak{c}_0\omega(\tx{b}_i)
=
\mathfrak{c}_0t_i
=
\frac{T_0}{2^{i}}.
$$
This proves \eqref{abstract.decay.power}. We next estimate the sum of the increments. From
\eqref{abstract.decay.power},
$$
d_i
\le
\mathfrak{c}_{1}^{1/\chi}
\frac{2^{-i+1}T_0}{\tx{b}_i^{1/\chi}}
=
2\mathfrak{c}_0\mathfrak{c}_{1}^{1/\chi}
\frac{t_i}{\omega^{-1}(t_i)^{1/\chi}}
$$
for every $i\ge1$. If
$
t\in[t_i/2,t_i],
$
then
$
\omega^{-1}(t)
\le
\omega^{-1}(t_i),
$
and hence
$$
\int_{t_i/2}^{t_i}
\frac{\dt}{\omega^{-1}(t)^{1/\chi}}
\ge
\frac{t_i}{2\omega^{-1}(t_i)^{1/\chi}}.
$$
Since
$
t_i/2=t_{i+1},
$
the intervals
$
(t_i/2,t_i]
$
are pairwise disjoint and cover $(0,t_1]$. Consequently,
\eqn{walmart}
$$
\sum_{i=1}^{\infty}d_i
\leq
4\mathfrak{c}_0\mathfrak{c}_{1}^{1/\chi}
\int_{0}^{t_1}
\frac{\dt}{\omega^{-1}(t)^{1/\chi}}
=
4\mathfrak{c}_0\mathfrak{c}_{1}^{1/\chi}
\int_0^{\tx{b}_1}
\frac{\omega'(s)}{s^{1/\chi}}
\ds
\leq
4\mathfrak{c}_0\mathfrak{c}_{1}^{1/\chi}
\int_0^{S_0}
\frac{\omega'(s)}{s^{1/\chi}}
\ds.
$$
In the equality above we used the change of variables
$t=\omega(s)$. Set
\eqn{kkkk}
$$
\kappa_\infty
:=\kappa_0+\sum_{i=1}^{\infty}d_i
\stackrel{\eqref{walmart}}{\leq}
\kappa_0+
4\mathfrak{c}_0\mathfrak{c}_{1}^{1/\chi}
\int_0^{S_0}
\frac{\omega'(s)}{s^{1/\chi}}
\ds.
$$
Hence $\kappa_\infty<\infty$. Moreover, by
\eqref{abstract.decay.power},
$
T(\kappa_i)\to0.
$
Since
$
(v-\kappa_i)_+\in W^{1,2}_0(B),
$
Poincar\'e's inequality gives
$
\nr{(v-\kappa_i)_+}_{L^2(B)}
\le
c(B)T(\kappa_i)
\to0.
$
Moreover,
$
\kappa_i\uparrow\kappa_\infty
$
and
$
(v-\kappa_i)_+ \to (v-\kappa_\infty)_+
$
pointwise a.e. in $B$. Hence, by Fatou's lemma,
$$
\nr{(v-\kappa_\infty)_+}_{L^2(B)}^2
\le
\liminf_{i\to\infty}
\nr{(v-\kappa_i)_+}_{L^2(B)}^2
=
0.
$$
We conclude that
$
(v-\kappa_\infty)_+=0
$
a.e. in $B$. Estimate \eqref{abstract.power.bound} now follows from
\eqref{kkkk}. 

\medskip

\noindent
{\em Proof of \eqref{abstract.exp.bound}.}
As before, we assume that the right-hand side of \rif{abstract.exp.bound}
is finite; the proof is a variant of the power case,
and we shall partially use the same notation. 
Again set
$
T_0:=T(\kappa_0)
$
and define $t_i$ and $\tx{b}_i$ as in \rif{abstract.ti}. We use the
convention
$$
[\log r]_+
:=
\begin{cases}
0
&\mbox{if $0\le r\le1$}
\\ \log r&\mbox{if $r>1$}.
\end{cases}
$$
For every $i\ge1$, define recursively
\eqn{abstract.di.exp}
$$
d_i
:=
\frac{T(\kappa_{i-1})}{\sqrt{\mathfrak{c}_{3}}}
\left[
\log\left(
\frac{
\mathfrak{c}_{2}\snr{E_{\kappa_{i-1}}}
}{
\tx{b}_{i}
}
\right)
\right]_+^{1/2},
\qquad
\kappa_i
:=
\kappa_{i-1}+d_i.
$$
This is well-defined also when
$
T(\kappa_{i-1})=0,
$
because in that case
$
\snr{E_{\kappa_{i-1}}}=0
$
by \rif{abstract.zero.boundary} and Poincar\'e's inequality, and our
convention gives $d_i=0$. We again claim
\rif{abstract.decay.power}. 
The assertion is trivial for $i=0$. Fix $i\ge1$. If
$
T(\kappa_{i-1})=0,
$
then
$
d_i=0,
$
$
\kappa_i=\kappa_{i-1},
$
and
$
T(\kappa_i)=0.
$
Suppose therefore that
$
T(\kappa_{i-1})>0.
$
If
$
\mathfrak{c}_{2}\snr{E_{\kappa_{i-1}}}
\le
\tx{b}_{i},
$
then $d_i=0$, so that $\kappa_i=\kappa_{i-1}$, and
$
\snr{E_{\kappa_i}}
=
\snr{E_{\kappa_{i-1}}}
\le
\mathfrak{c}_{2}^{-1}\tx{b}_{i}
\le
\tx{b}_{i}.
$
Otherwise, $d_i>0$ and
$
\kappa_{i-1}<\kappa_i.
$
Hence \eqref{abstract.exp.level} and \eqref{abstract.di.exp} give
$$
\snr{E_{\kappa_i}}
\le
\mathfrak{c}_{2}
\exp\left\{-\log\left(\frac{\mathfrak{c}_{2}\snr{E_{\kappa_{i-1}}}}{\tx{b}_{i}}\right)\right\}\snr{E_{\kappa_{i-1}}}
=\tx{b}_{i}.
$$
Thus, in every case,
$
\snr{E_{\kappa_i}}\le\tx{b}_{i}.
$
Using \eqref{abstract.energy}, we obtain \rif{ray0}, which 
proves \eqref{abstract.decay.power}. Since
$
\snr{E_{\kappa_{i-1}}}\le S_0,
$
estimate \eqref{abstract.decay.power} gives
$$
d_i
\leq 
\frac{2^{-i+1}T_0}{\sqrt{\mathfrak{c}_{3}}}
\logs\left(\frac{\mathfrak{c}_{2}S_0}{\omega^{-1}(t_i)}
\right)
=
\frac{2\mathfrak{c}_0}{\sqrt{\mathfrak{c}_{3}}}
t_i
\logs\left(\frac{\mathfrak{c}_{2}S_0}{\omega^{-1}(t_i)}\right).
$$
The function
$$
t
\longmapsto
\logs\left(
\frac{
\mathfrak{c}_{2}S_0
}{
\omega^{-1}(t)
}
\right)
$$
is nonincreasing on $(0,t_1]$. Arguing as in the previous case,
we obtain
\begin{flalign*}
\sum_{i=1}^{\infty}d_i
&\leq
\frac{4\mathfrak{c}_0}{\sqrt{\mathfrak{c}_{3}}}
\int_0^{t_1}
\logs\left(
\frac{
\mathfrak{c}_{2}S_0
}{
\omega^{-1}(t)
}
\right)
\dt
\\
&=
\frac{4\mathfrak{c}_0}{\sqrt{\mathfrak{c}_{3}}}
\int_0^{\tx{b}_1}
\omega'(s)
\logs\left(
\frac{
\mathfrak{c}_{2}S_0
}{
s
}
\right)
\ds\leq \frac{4\mathfrak{c}_0}{\sqrt{\mathfrak{c}_{3}}}
\int_0^{S_0}\omega'(s)
\logs\left(\frac{e\mathfrak{c}_{2}S_0}{s}\right)
\ds.
\end{flalign*}
Upon setting this time
$$
\kappa_\infty
:=
\kappa_0+\sum_{i=1}^{\infty}d_i\le
\kappa_0
+
\frac{c\mathfrak{c}_0}{\sqrt{\mathfrak{c}_{3}}}
\int_0^{S_0}
\omega'(s)
\logs\left(
\frac{
e\mathfrak{c}_{2}S_0
}{
s
}
\right)
\ds,
$$
the rest of the proof follows as in the proof of \eqref{abstract.power.bound}. 
\end{proof}


\section{Boost function $L^\infty$-bounds via calibrated iterations}\label{caliapp} 
\noindent 
Here we use the content of the previous section to show the following borderline a priori estimates:
\begin{proposition}\label{boostdopo}
Let
$u\in\mathbb X(\mathbb R^n)\cap C^\infty(\mathbb R^n)$
be the minimizer of \eqref{bi.en} with charge
$\mu\in C^\infty_0(\mathbb R^n)$.
If $n=2$, assume in addition that $\mu\in L^1_0(\er^2)$. 
Let $\mathcal h_\infty(\mu)$ be as in
Proposition \ref{com.sup}.
\begin{itemize}
\item If $n\ge3$, then
\eqn{boost.bound}
$$
\nr{\mathcal h_u}_{L^\infty(\mathbb R^n)}
\le
\exp\left\{c_n[\mu]_{n,1}\right\}.
$$
\item If $n=2$, then, for every $\alpha>2$,
\eqn{boost.2d}
$$
\nr{\mathcal h_u}_{L^\infty(\mathbb R^2)}
\le\mathcal h_\infty(\mu)\exp\left\{c_\alpha\nr{\mu}_{L^2(\log L)^\alpha(\mathbb R^2)}\right\}.
$$
\end{itemize}
Here $c_n$ depends only on $n$, and $c_\alpha$
only on $\alpha$.
\end{proposition}

\begin{proof} By Proposition \ref{exex}, $u$ is a classical solution
to \eqref{bi}. Therefore, we can use all the results already proved for these ones. The proof splits in two steps. 

{\em Step 1: Level sets estimate for $\log \mathcal h_u$}. If $n\ge3$, fix an arbitrary $\tau>1$.
If $n=2$, set $\tau:=\mathcal h_\infty(\mu)$.
By \eqref{boost.at.infinity}, there exists a ball
$B$ centered at the origin such that
$
\supp\,(\mathcal h_u-\tau)_+\Subset B.
$
When $n=2$, we take
$B:=B_{\tx r_\mu}(0)$ as in Proposition \ref{com.sup}.
For $\kk\ge\tau$, set
$
\tx A_\kk:=\{x\in B:\mathcal h_u(x)>\kk\}.
$
Then
\eqn{supporto}
$$
\supp\,(\mathcal h_u-\kk)_+\Subset B
\quad
\mbox{for every $\kk\ge\tau$}.
$$We test \eqref{0} against $\varphi:=(\mathcal{h}_{u}-\kk)_{+}D_{s}u$ and sum over $s\in \{1,\cdots,n\}$ to get
\begin{flalign*}
0&=\sum_{s=1}^{n}\int_{\mathbb{R}^{n}}(\mathcal{h}_{u}-\kk)_{+}\langle\partial^{2}H(Du)DD_{s}u,DD_{s}u\rangle\dx\nonumber  +\int_{\mathbb{R}^{n}}\langle\mathcal H(Du)D\mathcal h_u,D(\mathcal{h}_{u}-\kk)_{+}\rangle\dx\nonumber \\
&\quad +\int_{\mathbb{R}^{n}}\mu(\mathcal{h}_{u}-\kk)_{+}\Delta u\dx+\int_{\mathbb{R}^{n}}\mu\langle D(\mathcal{h}_{u}-\kk)_{+}, Du\rangle\dx=:\sum_{i=1}^{4}\mbox{(I)}_{i}.
\end{flalign*}
Here we used \rif{identitas} and the definition of $\mathcal H$ in
\eqref{recalla}.
Via \eqref{0.1.1}-\eqref{0.1}  we bound from below as follows
$$
\mbox{(I)}_{1}+\mbox{(I)}_{2}\ge \int_{\mathbb{R}^{n}}(\mathcal{h}_{u}-\kk)_{+}\mathcal{h}_{u}\snr{D^{2}u}^{2}\dx\nonumber +\int_{\mathbb{R}^{n}}\langle\mathcal{H}(Du)D(\mathcal{h}_{u}-\kk)_{+},D(\mathcal{h}_{u}-\kk)_{+}\rangle\dx.
$$
Next, recalling \rif{supporto}, via Young's inequality we find
\begin{flalign*}
\snr{\mbox{(I)}_3}
&\le
c_n
\int_{\mathbb R^n}
\snr{\mu}\snr{D^2u}
(\mathcal h_u-\kk)_+
\dx
\\
&\le
\frac14
\int_{\mathbb R^n}
(\mathcal h_u-\kk)_+
\mathcal h_u\snr{D^2u}^2
\dx
+
c_n
\int_{\mathbb R^n}
\frac{(\mathcal h_u-\kk)_+}{\mathcal h_u}
\mu^2\dx
\\
&\le
\frac14
\int_{\mathbb R^n}
(\mathcal h_u-\kk)_+
\mathcal h_u\snr{D^2u}^2
\dx
+
c_n
\int_{\tx A_\kk}\mu^2\dx.
\end{flalign*}
Similarly, this time by the Cauchy-Schwarz inequality relative to $\mathcal H(Du)$ and using 
$
\mathcal H(Du)Du=Du
$, we obtain
\begin{flalign*}
\snr{\mbox{(I)}_{4}}&=\left|\int_{\mathbb{R}^{n}}\mu\langle \mathcal H(Du) D(\mathcal{h}_{u}-\kk)_{+}, Du\rangle\dx\right|
\\
&\le
\frac14
\int_{\mathbb R^n}
\left\langle
\mathcal H(Du)D(\mathcal h_u-\kk)_+,
D(\mathcal h_u-\kk)_+
\right\rangle\dx+
4\int_{\tx A_\kk}\mu^2\snr{Du}^2\dx
\\
&\le
\frac14\int_{\mathbb R^n}\left\langle
\mathcal H(Du)D(\mathcal h_u-\kk)_+,
D(\mathcal h_u-\kk)_+
\right\rangle\dx+4\int_{\tx A_\kk}\mu^2\dx.
\end{flalign*}
Merging the content of all previous displays and reabsorbing terms, we obtain
\eqn{matrix}
$$
\int_{\tx A_\kk}\left\langle\mathcal H(Du)D\mathcal h_u,D\mathcal h_u\right\rangle\dx=\int_B
\left\langle
\mathcal H(Du)D(\mathcal h_u-\kk)_+,
D(\mathcal h_u-\kk)_+\right\rangle\dx
\le c_{n} \int_{\tx A_\kk}\mu^2\dx. 
$$
Recalling that
$
\left\langle
\mathcal H(Du)\xi,\xi
\right\rangle
\ge
\mathcal h_u^{-2}\snr{\xi}^2,
$
\eqref{matrix} implies
\begin{flalign}
\nonumber  \nr{
D(\log\mathcal h_u-\log\kk)_+
}_{L^2(B)}^2
 & =
\int_{\tx A_\kk}
\mathcal h_u^{-2}\snr{D\mathcal h_u}^2\dx\\
& \le
\int_{\tx A_\kk}
\left\langle
\mathcal H(Du)D\mathcal h_u,
D\mathcal h_u
\right\rangle\dx
\le
c
\int_{\tx A_\kk}\mu^2\dx \label{whit}. 
\end{flalign}
Define
\eqn{ray10}
$$
\omega_\mu(s):=
\left(
\int_0^s
\left(
(\snr\mu\mathds1_B)^*(r)
\right)^2
\drr
\right)^{1/2} , \qquad  0 \leq s \leq |B|. 
$$
By the classical Hardy--Littlewood rearrangement inequality,
$$
\int_{\tx A_\kk}\mu^2\dx
=
\int_{\tx A_\kk}
(\snr{\mu}\mathds1_B)^2\dx
\le
\int_0^{\snr{\tx A_\kk}}\left(
(\snr\mu\mathds1_B)^*(r)
\right)^2\drr
=
\omega_\mu(\snr{\tx A_\kk})^2,
$$
so that \rif{whit} finally becomes
\eqn{og.cacc}
$$
\nr{D(\log\mathcal h_u-\log\kk)_+}_{L^2(B)}
\le c_{n}\omega_\mu(\snr{\tx A_\kk}).
$$
The constant $c_{n}$ appearing in the display is independent of $\tau$, $\kk\geq \tau$, and $B$.
In particular, the preceding estimates use only
$0\le(\mathcal h_u-\kk)_+/\mathcal h_u\le1$
and are valid independently of any lower bound on $\tau-1$.
Note that the function $\omega_\mu$ is nondecreasing and absolutely continuous
on $[0,\snr B]$, and
$
\omega_\mu(0)=0.
$
Before approaching the next step we record an estimate for $\omega_\mu'$. Since $s \mapsto (\snr{\mu}\mathds 1_B)^*(s)$ is
nonincreasing, we have  
$
\omega_\mu(s)^2
\ge
s(\snr{\mu}\mathds 1_B)^*(s)^2.
$
It follows that, for a.e. $s\in(0,|B|)$ such that
$
\omega_\mu(s)>0,
$
\eqn{boost.omega.derivative}
$$
\omega_\mu'(s)
=
\frac{(\snr{\mu}\mathds 1_B)^*(s)^2}{2\omega_\mu(s)}
\leq 
\frac{(\snr{\mu}\mathds 1_B)^*(s)}{2\sqrt{s}}.
$$
\indent
{\em Step 2: Iterations according to Lemma \ref{abstract.iteration}.}
The occurrence of \rif{og.cacc} leads us to define the function $v$
$$
v:=\left(
\log\frac{\mathcal h_u}{\tau}
\right)_+
\Longrightarrow
(v-\kappa)_+
=
\left(
\log\frac{\mathcal h_u}{\tau e^\kappa}
\right)_+,
\quad \kappa\ge0,
$$
to which we apply Lemma \ref{abstract.iteration} with initial level
$
\kappa_0:=0.
$
Note that $v$ is non-negative, continuous and belongs to
$W^{1,2}_0(B)$.
For every $\kappa\ge0$, set
$$
E_\kappa
:=
\{x\in B:v(x)>\kappa\}
=
\left\{
x\in B:
\mathcal h_u(x)>\tau e^\kappa
\right\}.
$$
In particular,
$
\tau e^\kappa\ge\tau
$
and
$
(v-\kappa)_+\in W^{1,2}_0(B).
$
Following the notation of Lemma \ref{abstract.iteration}, define
$
T(\kappa)
:=
\nr{D(v-\kappa)_+}_{L^2(B)}
$
and set
\eqn{ss00}
$$
S_0:=\snr{E_{\kappa_0}}
=\snr{E_0}
=
\snr{\left\{x\in B:\mathcal h_u(x)>\tau\right\}}.
$$
If $S_0=0$, then $v=0$ in $B$ and the estimates for $v$
obtained below are immediate.
To derive these estimates, we may therefore assume that
$
S_0>0.
$
Moreover,
$
E_\kappa\subset E_0
$
for every $\kappa\ge0$, and hence
$
\snr{E_\kappa}\le S_0.
$
Replacing $\kk$ by $\tau e^\kappa$ in \rif{og.cacc}, we obtain
$
T(\kappa)
\le
c_n\omega_\mu(\snr{E_\kappa})
$
for every $\kappa\ge0$.
Thus condition \eqref{abstract.energy} holds with
$\kappa_0=0$, $\mathfrak c_0=c_n$, and
$\omega=\omega_\mu$ on $[0,S_0]$,
where $\omega_\mu$ is defined in \eqref{ray10}.

\medskip

\medskip

\noindent
{\em Case $n\ge3$.} We verify \rif{abstract.power.level}. 
Take the usual Sobolev embedding exponent 
$
2^*
:=
2n/(n-2)
$
and fix
$
0\le\kappa_1<\kappa_2.
$
We have 
$
(v-\kappa_1)_+\in W^{1,2}_0(B),
$
while
$
(v-\kappa_1)_+
\ge
\kappa_2-\kappa_1$
a.e. on $E_{\kappa_2}$. This last fact, Chebyshev and Sobolev inequalities give 
$$
\snr{E_{\kappa_2}}
\leq \frac{\nr{(v-\kappa_1)_+}_{L^{2^*}(B)}^{2^*}}{(\kappa_2-\kappa_1)^{2^*}}
\leq 
c_{n}
\frac{
\nr{D(v-\kappa_1)_+}_{L^2(B)}^{2^*}}{(\kappa_2-\kappa_1)^{2^*}}
=
c_{n}
\frac{T(\kappa_1)^{2^*}}{(\kappa_2-\kappa_1)^{2^*}}.
$$
Thus \eqref{abstract.power.level} holds with
$
\chi=2^*$ and 
$\mathfrak c_1=c_{n}\ge1.
$
All the assumptions in the first part of Lemma
\ref{abstract.iteration} are therefore satisfied. Since the initial
level for $v$ is
$
\kappa_0=0,
$
estimate \eqref{abstract.power.bound} gives
\eqn{boost.abstract.high}
$$
\nr v_{L^\infty(B)}
\le
c_{n}
\int_0^{S_0}
\frac{
\omega_\mu'(s)
}{
s^{1/2^*}
}
\ds.
$$
Using \eqref{boost.omega.derivative} and the identity
$
1/2+1/2^*=1-1/n$
we find
$$
\int_0^{S_0}
\frac{
\omega_\mu'(s)
}{
s^{1/2^*}
}
\ds
\leq 
\frac12
\int_0^{S_0}
s^{1/n}(\snr{\mu}\mathds 1_B)^*(s)\frac{\ds}{s}
\leq 
c[\mu]_{n,1;B}
\leq 
c[\mu]_{n,1}.
$$
Combining this estimate with \eqref{boost.abstract.high}, we obtain
$$
\nr v_{L^\infty(B)}
\le
c_{n}[\mu]_{n,1}.
$$
The constant $c_n$ in the above display is independent of $\tau$ and $B$. 
By the definition of $v$, we obtain
$
\mathcal h_u
\le
\tau\exp\left\{c_n[\mu]_{n,1}\right\}
$
in $B$.
Moreover, the choice of $B$ gives
$
\mathcal h_u\le\tau
$
in $\mathbb R^n\setminus B$.
Therefore,
$$
\nr{\mathcal h_u}_{L^\infty(\mathbb R^n)}
\le
\tau\exp\left\{c_n[\mu]_{n,1}\right\}.
$$
Since $\tau>1$ is arbitrary, letting $\tau\downarrow1$
proves \rif{boost.bound}.

\medskip

\noindent
{\em Case $n=2$.} Again, we have to check \rif{abstract.exp.level}. 
Fix
$
0\le\kappa_1<\kappa_2
$
such that 
$
T(\kappa_1)>0.
$
Since $\mathcal h_u$ is smooth, $v$ is continuous and
$
E_{\kappa_1}
$
is open, so that 
$
(v-\kappa_1)_+
\in
W^{1,2}_0(E_{\kappa_1}).
$
Lemma \ref{trudinger.lemma} yields absolute constants
$
\mathfrak c_2\ge1
$
and
$
\mathfrak c_3>0
$
such that
$$
\int_{E_{\kappa_1}}
\exp\left\{
\frac{\mathfrak c_3(v-\kappa_1)_+^2}{T(\kappa_1)^2}
\right\}\dx
\le
\mathfrak c_2
\snr{E_{\kappa_1}}.
$$
On $E_{\kappa_2}$ we have
$
(v-\kappa_1)_+
\ge
\kappa_2-\kappa_1.
$
It follows that
\begin{flalign*}
\snr{E_{\kappa_2}}
\exp\left\{
\frac{
\mathfrak c_3(\kappa_2-\kappa_1)^2
}{
T(\kappa_1)^2
}
\right\}
&\le
\int_{E_{\kappa_1}}
\exp\left\{
\frac{
\mathfrak c_3(v-\kappa_1)_+^2
}{
T(\kappa_1)^2
}
\right\}
\dx
\leq 
\mathfrak c_2
\snr{E_{\kappa_1}}.
\end{flalign*}
Therefore we obtain the following inequality, which shows that \eqref{abstract.exp.level} is also satisfied
$$
\snr{E_{\kappa_2}}
\le\mathfrak c_2\exp\left\{-\frac{
\mathfrak c_3(\kappa_2-\kappa_1)^2}{T(\kappa_1)^2}\right\}
\snr{E_{\kappa_1}}.
$$
All the assumptions in
the second part of Lemma \ref{abstract.iteration} hold with
$
\kappa_0=0,
$
and \eqref{abstract.exp.bound} gives
\eqn{boost.abstract.two}
$$
\nr v_{L^\infty(B)}
\le
c
\int_0^{S_0}
\omega_\mu'(s)
\logs\left(
\frac{e\mathfrak c_2S_0
}{s}\right)
\ds.
$$ 
We now estimate the right-hand side of the above inequality as follows (recall $S_0$ in \eqref{ss00})
\begin{flalign*}
&\int_0^{S_0}
\omega_\mu'(s)
\logs\left(
\frac{e\mathfrak c_2S_0}{s}
\right)
\ds
\\
&\qquad\stackleq{boost.omega.derivative}
c
\int_0^{S_0}
(\snr{\mu}\mathds 1_B)^*(s)
\logs\left(
e+
\frac{e\mathfrak c_2S_0}{s}
\right)
\frac{\ds}{\sqrt{s}}
\\
&\qquad\stackleq{stimamis}
c \int_0^{\mathfrak{s}_n}
(\snr{\mu}\mathds 1_B)^*(s)
\logs\left(e+\frac1s\right)
\frac{\ds}{\sqrt{s}}
\\
&\qquad\ \ \le
c
\left(
\int_0^{\mathfrak{s}_n}
(\snr{\mu}\mathds 1_B)^*(s)^2
\log^\alpha\left(
e+\frac1s
\right)\ds\right)^{1/2}
\left(\int_0^{\mathfrak{s}_n} \log^{1-\alpha}\left(e+\frac1s \right)\frac{\ds}{s}\right)^{1/2}
\\
&\qquad \stackleq{orlicz.rearr}
c
\nr{\mu}_{L^2(\log L)^\alpha(\mathbb R^n)}
\left( \int_0^{\mathfrak{s}_n}
\log^{1-\alpha}\left( e+\frac1s\right) \frac{\ds}{s} \right)^{1/2}
 \le
c \nr{\mu}_{L^2(\log L)^\alpha(\mathbb R^n)},
\end{flalign*}
where $c\equiv c(n,\alpha)$. We used the convention that
$
(\snr{\mu}\mathds1_B)^*
$
is extended by zero outside $[0,\snr B]$. Moreover,
 note that the last integral appearing in the preceding display is finite precisely when $\alpha >2$. 
Combining this last estimate with \rif{boost.abstract.two}, we obtain
$
\nr v_{L^\infty(B)}\le c(n,\alpha)\nr{\mu}_{L^2(\log L)^\alpha}
$
and, since this time $\tau=\mathcal h_\infty(\mu)$,
the definition of $v$ and the bound $\mathcal h_u\le\tau$
outside $B$ yield \eqref{boost.2d}.
\end{proof}
\begin{remark}\label{controre}
\emph{ When $n=2$ the factor in \eqref{boost.2d} cannot be dropped. For this, we largely use the material in Proposition \ref{connessione}.
Indeed, let $\eta_T$ be the cut-off functions
defined in \eqref{logatau}.
For $a\in(0,1)$ and $T$ sufficiently large, set 
$
u_{a,T}(x):=a x_1\eta_T(x)$ and $
\mu_{a,T}:=
-\diver(\mathcal h_{u_{a,T}}Du_{a,T}).
$ 
Direct computations give 
$
\nr{Du_{a,T}}_{L^\infty(\mathbb R^2)}
\le
a\left(1+C/\log T\right)
$, so that, choosing $T\equiv T(a)>1$ large enough, we have 
$
\nr{Du_{a,T}}_{L^\infty(\mathbb R^2)}< 1$, i.e., 
$u_{a,T}$ is strictly spacelike and belongs to $\mathbb X(\mathbb R^2)$.
The function $\mu_{a,T}$ is smooth and compactly supported. Moreover, $\mu_{a,T} \in L^1_0(\er^2)$
by Proposition \ref{connessione} and 
$u_{a,T}$ is the associated global minimizer, again by Proposition \ref{connessione}.
Using computations similar to those in Proposition \ref{connessione}, and the definition of 
Luxemburg norm given in \rif{luxi}, we find that for every fixed $\alpha>0$,
\eqn{contra}
$$
\nr{\mu_{a,T}}_{L^2(\log L)^\alpha(\mathbb R^2)}
\lesssim_{a, \alpha}\frac{1}{\sqrt{\log T}}.
$$
On the other hand, $\eta_T=1$ near the origin, so
\eqn{contra2}
$$
\mathcal h_{u_{a,T}}(0)=\frac1{\sqrt{1-a^2}}.
$$
Note that this identity is independent of $T$. 
Now assume that 
$\nr{\mathcal h_u}_{L^\infty(\mathbb R^2)}
\le \exp\{c_\alpha\nr{\mu}_{L^2(\log L)^\alpha(\mathbb R^2)}\}$ holds for some constant $c_{\alpha}$ and for all admissible $\mu$. 
For every integer $m\geq 2$, by \rif{contra}-\rif{contra2}, we can find $a<1$ and, consequently, $T\equiv T(a)>1$, such that $
\nr{Du_{a,T}}_{L^\infty(\mathbb R^2)}< 1$ and  $$\mathcal h_{u_{a,T}}(0) > m > 1+1/m > \exp\left\{c_\alpha\nr{\mu_{a,T}}_{L^2(\log L)^\alpha(\mathbb R^2)}\right\},$$ which is a contradiction. 
}\end{remark}


\section{Approximation schemes for Theorems \ref{t1}, \ref{t2} and \ref{t222}}\label{apsec}

\subsection{General set-up for the approximation}

\noindent
The estimates obtained in Propositions \ref{p4.4}, and
\ref{boostdopo} are a priori estimates for smooth, compactly supported
charges and smooth strictly spacelike solutions. In this section we construct suitable regularizations of the datum. Throughout this section, $\mu$ is a function satisfying the
assumptions of the theorem under consideration. In particular,
$\mu\in\mathbb Y(\mathbb R^n)$. In particular, when $n=2$, we recall that 
$\mu\in L^1_0(\mathbb R^2)$. In the following we denote by  $v_\mu\in\DD$ the Riesz representative of $\mu$, in the sense of \rif{dual.riesz} (see also \rif{dual.ide}). We recommend the reader to keep in mind the content of Remark \ref{concreto}.

\subsubsection{Truncation for $n\ge3$}\label{addi1}

Let $R\ge1$, and choose a cut-off function
$\eta_R\in C^\infty_0(\mathbb R^n)$ satisfying
$
0\le\eta_R\le1$, $
\eta_R=1$ in $B_R(0)$, 
$\eta_R=0$ in $\mathbb R^n\setminus B_{2R}(0)$, and $|D\eta_R|\le c/R.$
In particular,
$\|D\eta_R\|_{L^n(\mathbb R^n)}\le c(n)$.
We set
$
\nu_R:=\eta_R\mu
$  and  check that
$
\nu_R\to\mu
$
in $\DD^*$ as $R\to\infty$. Take 
$\varphi\in C^\infty_0(\mathbb R^n)$ with $\|D\varphi\|_{L^2(\er^n)}\leq 1$. Since
$(1-\eta_R)\varphi\in C^\infty_0(\mathbb R^n)\subset\DD$, 
\eqref{dual.riesz} gives
$$
\begin{aligned}
\langle\mu-\nu_R,\varphi\rangle 
&=
\int_{\mathbb R^n}
(1-\eta_R)
\langle Dv_\mu,D\varphi\rangle\dx
-
\int_{\mathbb R^n}
\varphi
\langle Dv_\mu,D\eta_R\rangle\dx.
\end{aligned}
$$
The first term in the right-hand side satisfies
$$
\left|
\int_{\er^n}
(1-\eta_R)
\langle Dv_\mu,D\varphi\rangle\dx
\right|
\le
\|Dv_\mu\|_{L^2(\mathbb R^n\setminus B_R)}.
$$
For the second term, H\"older's inequality and the Sobolev embedding
give ($2^*=2n/(n-2)$)
$$
\begin{aligned}
\left|
\int_{\mathbb R^n}
\varphi
\langle Dv_\mu,D\eta_R\rangle\dx
\right|
&\le
\|Dv_\mu\|_{L^2(B_{2R}\setminus B_R)}
\|D\eta_R\|_{L^n(\er^n)}
\|\varphi\|_{L^{2^*}(\er^n)}
\\
&\le
c
\|Dv_\mu\|_{L^2(B_{2R}\setminus B_R)}
\|D\varphi\|_{L^2}\le
c
\|Dv_\mu\|_{L^2(B_{2R}\setminus B_R)}
.
\end{aligned}
$$
Combining the content of the last three displays, and then taking the supremum over all
$\varphi\in C^\infty_0(\mathbb R^n)$ with 
$\|D\varphi\|_{L^2}\le1$, we obtain
\eqn{app.cutoff.dual.high}
$$
\|\mu-\nu_R\|_{\DD^*}
\le c \|Dv_\mu\|_{L^2(\mathbb R^n\setminus B_R)}
\stackrel{R\to \infty}{\longrightarrow}0.
$$

\subsubsection{Truncation for $n=2$}\label{addi2}
Fix a nonincreasing
function $\vartheta\in C^\infty(\mathbb R)$ such that
$0\le\vartheta\le1$, $\vartheta=1$ on $(-\infty,1]$, and
$\vartheta=0$ on $[2,\infty)$. For $R>e^2$, we define $\eta_R$ as in \rif{logatau}, with $T=R$, and use
the support and gradient estimates established there. In particular, $\eta_R=1$ in $B_{R-e}(0)$, $\eta_R=0$ in
$\mathbb R^2\setminus B_{R^2-e}(0)$, and
\eqn{app.cutoff.two.gradient}
$$
|D\eta_R(x)|
\le
\frac{c}
{(e+|x|)\log R}
\mathds1_{\{R-e<|x|<R^2-e\}}(x), \quad \|D\eta_R\|_{L^2(\mathbb R^2)}
\le
\frac{c}{\sqrt{\log R}}
\stackrel{R\to \infty}{\longrightarrow}0.
$$
The function $\eta_R\mu$ does not necessarily have zero integral and therefore needs a correction. 
We fix
$\psi\in C^\infty_0(B_1(0))$ such that
$\int_{\mathbb R^2}\psi\dx=1$. This time we set 
\eqn{defninu}
$$
\nu_R:=\eta_R\mu-a_R\psi, \quad \mbox{where }\quad 
a_R
:=
\int_{\mathbb R^2}\eta_R\mu\dx.  
$$
It follows that
$$
\int_{\mathbb R^2}\nu_R\dx=0.
$$
Moreover, since $\eta_R\in C^\infty_0(\mathbb R^2)$ and
$\mu\in\DDDD^*$, the definition of the dual norm and
\eqref{app.cutoff.two.gradient} give
\eqn{app.mean.small}
$$
|a_R|=
\left|\int_{\mathbb R^2}\mu\eta_R\dx\right|
\le\|\mu\|_{\DD^*}
\|D\eta_R\|_{L^2(\mathbb R^2)}
\le \frac{c\|\mu\|_{\DDDD^*}}{\sqrt{\log R}}
\stackrel{R\to \infty}{\longrightarrow}0.
$$
We next prove that $\nu_R\to\mu$ in $\DDDD^*$. Let
$\varphi\in C^\infty_0(\mathbb R^2)$ with $\|D\varphi\|_{L^2(\er^2)}\leq 1$, and set
$
P\varphi
:=
\varphi-(\varphi)_{B_1(0)}
$, and, as in Remark \ref{concreto}, note that 
$P\varphi\in\DDDD$,
$(P\varphi)_{B_1(0)}=0$ and
$D(P\varphi)=D\varphi$. Since $\eta_R=1$ on $B_1(0)$, the function
$(1-\eta_R)P\varphi$ also belongs to $\DDDD$. Indeed, its gradient is
$
(1-\eta_R)D\varphi-P\varphi D\eta_R,
$
and the logarithmic Hardy inequality gives
\eqn{app.cutoff.two.multiplier}
$$
\|P\varphi D\eta_R\|_{L^2(\mathbb R^2)}\stackleq{app.cutoff.two.gradient}
c
\left\|
\frac{P\varphi}
{(1+|\cdot|)\log(e+|\cdot|)}
\right\|_{L^2(\mathbb R^2)}
\le
c
\|D\varphi\|_{L^2(\mathbb R^2)}\leq c .
$$
In the last line we have again used Poincaré inequality. 
Using \eqref{defninu} and \eqref{dual.riesz},
we therefore obtain
$$
\begin{aligned}
\langle\mu-\nu_R,\varphi\rangle 
=\langle\mu-\nu_R,P\varphi\rangle 
&=\left\langle
\mu,(1-\eta_R)P\varphi
\right\rangle+a_R\int_{\mathbb R^2}\psi P\varphi\dx
\\
&=\int_{\mathbb R^2}
(1-\eta_R)
\langle Dv_\mu,D\varphi\rangle\dx
-\int_{\mathbb R^2}P\varphi\langle Dv_\mu,D\eta_R\rangle\dx+a_R\int_{\mathbb R^2}\psi P\varphi\dx.
\end{aligned}
$$
The first two integrals are supported outside $B_{R-e}(0)$.
Moreover, by Poincaré's inequality on $B_1(0)$,
$$
\left|
\int_{\er^2}\psi P\varphi\dx
\right|
\le
c_\psi\|D\varphi\|_{L^2(\er^2)} \leq c .
$$
Thus, also using \eqref{app.mean.small} and \eqref{app.cutoff.two.multiplier}, we obtain
\eqn{app.cutoff.dual.two}
$$
\|\mu-\nu_R\|_{\DDDD^*}
\le
c
\|Dv_\mu\|_{L^2(\mathbb R^2\setminus B_{R-e})}
+
c |a_R|
\stackrel{R\to \infty}{\longrightarrow}0.
$$
\subsubsection{Convergence in the additional data spaces}\label{addi3}

We now consider the additional function-space assumptions separately.
\begin{itemize}
\item If $\mu\in L^q(\mathbb R^n)$, with $1\le q<\infty$, then
$\eta_R\mu\to\mu$ in $L^q(\mathbb R^n)$ by dominated convergence. When $n=2$, we also have
$
\|a_R\psi\|_{L^q(\mathbb R^2)}
=
|a_R|\|\psi\|_{L^q(\mathbb R^2)}
\to0,
$
and therefore $\nu_R\to\mu$ in $L^q(\mathbb R^2)$.

\item If $n\ge3$ and $\mu\in L(n,1)(\mathbb R^n)$, then
$\eta_R\mu\to\mu$ in $L(n,1)(\mathbb R^n)$ by the absolute
continuity of the Lorentz norm.

\item When $n=2$, dominated convergence gives
$\eta_R\mu\to\mu$ in $L^1(\mathbb R^2)$. Since $a_R\to0$ and $\psi$ is fixed and
smooth, the correction $a_R\psi$ converges to zero.
Consequently, $\nu_R\to\mu$ in
$L^1(\mathbb R^2)$.
\item If, in addition,
$\mu\in L^2(\log L)^\alpha(\mathbb R^2)$, then
$\eta_R\mu\to\mu$ in the corresponding Luxemburg norm. For completeness we give a short proof of this. To see this, recall that
$A_\alpha(t)=t^2\log^\alpha(e+t)$.
Since $A_\alpha$ satisfies the global $\Delta_2$ condition,
membership in $L^{A_\alpha}(\mathbb R^2)$ implies\footnote{The global $\Delta_2$ condition means that
$A_\alpha(2t)\le c_\alpha A_\alpha(t)$ for every $t\ge0$.
Iteration of this and monotonicity easily imply that for every $a>0$ there exists
$c(\alpha,a)$ such that
$A_\alpha(at)\le c(\alpha,a)A_\alpha(t)$ for all $t\ge0$.
Consequently, finiteness of
$\int_{\mathbb R^2}A_\alpha(|\mu|/\lambda)\dx$
for some $\lambda>0$, guaranteed by membership in
$L^{A_\alpha}(\mathbb R^2)$, implies its finiteness
for every $\lambda>0$.}
$$
\int_{\mathbb R^2}
A_\alpha\left(\frac{|\mu|}{\eps}\right)\dx
<\infty
\qquad\text{for every }\eps>0.
$$
Fix $\eps>0$. Since $0\le\eta_R\le1$ and
$\eta_R\to1$ pointwise, dominated convergence gives
$$
\int_{\mathbb R^2}
A_\alpha\left(
\frac{|(1-\eta_R)\mu|}{\eps}
\right)\dx
\stackrel{R\to \infty}{\longrightarrow} 0.
$$
For all sufficiently large $R$, this integral is at most $1$.
By the definition of the Luxemburg norm, it follows that
$
\|(1-\eta_R)\mu\|_{L^2(\log L)^\alpha(\mathbb R^2)}
\le\varepsilon.
$
As $\eps>0$ is arbitrary, this proves the claimed
norm convergence. Moreover,
$
\|a_R\psi\|_{L^2(\log L)^\alpha(\mathbb R^2)}
=
|a_R|\,\|\psi\|_{L^2(\log L)^\alpha(\mathbb R^2)}
\to 0
$ as $R\to \infty$. 
We deduce
$\nu_R\to\mu$ in $L^2(\log L)^\alpha(\mathbb R^2)$.
\end{itemize}


\subsubsection{Mollification}\label{mollisi}

Let
$
\rho\in C^\infty_0(B_1(0))
$
be a nonnegative, radially symmetric function such that
$
\int_{\mathbb R^n}\rho\dx=1,
$
and set
$
\rho_\varepsilon(x):=\varepsilon^{-n}\rho(x/\varepsilon).
$
We  define
$
\tilde \mu_i
:=
\rho_{1/i}*\nu_{i}, 
$ for $i\geq e^2$ (in order to use the arguments of Section \ref{addi2}, where we took $R\geq e^2$). 
Since $\nu_{i}$ has compact support, it follows that
$
\tilde \mu_i\in C^\infty_0(\mathbb R^n)
$
and
$
\supp\, \tilde \mu_i
\subset
\supp \, \nu_{i} +\overline{B_{1/i}(0)}.
$
Moreover, when $n=2$, the zero-mean property of $\nu_{i}$ is obviously preserved by translations/convolutions:
\eqn{app.mollification.mean}
$$
\int_{\mathbb R^2}\tilde \mu_i\dx
=
\int_{\mathbb R^2}\rho_{1/i}\dx
\int_{\mathbb R^2}\nu_{i}\dx
=
0.
$$
Let us prove that 
\eqn{app.mollification.dual.limit}
$$
\tilde \mu_i \to \mu\qquad\mbox{in }\DD^*. 
$$
 Let
$
f\in\DD^* \cap L^1_{\loc}(\er^n)
$, with $f \in L^1_0(\er^2)$ when $n=2$, and let 
$v_f\in\DD$
be its Riesz representative in the sense of  \rif{dual.riesz}. 
Define
$
v_{f,\varepsilon}:=\rho_\varepsilon*v_f
$
if $n\ge3$. When $n=2$, we as usual correct it defining $
v_{f,\varepsilon}
:=
\rho_\varepsilon*v_f
-
(\rho_\varepsilon*v_f)_{B_1(0)}.
$
It is not difficult to show that  $v_{f,\varepsilon}$ is the Riesz representative of
$\rho_\varepsilon*f$ in the sense of \rif{dual.riesz}. Using this fact and Young's inequality for convolution we then obtain
$$
\|\rho_\varepsilon*f\|_{\DD^*}
\stackrel{\eqref{dual.riesz}}{=}
\|Dv_{f,\varepsilon}\|_{L^2(\mathbb R^n)}
=
\|\rho_\varepsilon*Dv_f\|_{L^2(\mathbb R^n)}
\le 
\|Dv_f\|_{L^2(\mathbb R^n)}
=
\|f\|_{\DD^*}.
$$
We now write
$
\tilde \mu_i-\mu
=
\rho_{1/i}*(\nu_{i}-\mu)
+
(\rho_{1/i}*\mu-\mu).
$
Using the content of the above display with $f\equiv \nu_{i}-\mu$, we obtain
$$
\begin{aligned}
\|\tilde \mu_i-\mu\|_{\DD^*}
&\le
\|\rho_{1/i}*(\nu_{i}-\mu)\|_{\DD^*}
+
\|\rho_{1/i}*\mu-\mu\|_{\DD^*}
\\
&\le
\|\nu_{i}-\mu\|_{\DD^*}
+
\|\rho_{1/i}*Dv_\mu-Dv_\mu\|_{L^2(\mathbb R^n)}.
\end{aligned}
$$
The first term tends to zero by \rif{app.cutoff.dual.high} and \rif{app.cutoff.dual.two}. The second term tends to zero because
$
Dv_\mu\in L^2(\mathbb R^n;\mathbb R^n). 
$
Therefore
\rif{app.mollification.dual.limit} follows. We next consider the additional function spaces as done in Section \ref{addi3}. We denote by  $Z$ any of the additional spaces involved in
the theorem under consideration, namely
$
L^q(\mathbb R^n),
$
$
L(n,1)(\mathbb R^n),
$
or
$
L^2(\log L)^\alpha(\mathbb R^2)
$, and we assume that $\mu\in Z$, as required by the corresponding statement.
Standard properties of convolutions guarantee that  $\|\rho_{1/i}*\mu-\mu\|_Z\to 0$. 
Then, we write 
$$
\begin{aligned}
\|\tilde \mu_i-\mu\|_Z
&\le
\|\rho_{1/i}*(\nu_{i}-\mu)\|_Z+\|\rho_{1/i}*\mu-\mu\|_Z
\\
&\le c_Z\|\nu_{i}-\mu\|_Z +\|\rho_{1/i}*\mu-\mu\|_Z \to 0.
\end{aligned}
$$
Here we use the convergence properties established in Section \ref{addi3}. The same argument
applies to $L^1(\mathbb R^2)$.
Finally, after a harmless scalar rescaling
$
\mu_i:=\theta_i\tilde \mu_i,
$
with
$
0<\theta_i\le1
$
and
$
\theta_i\to1,
$
we may further assume that all the corresponding norms of $\mu_i$ do
not exceed those of $\mu$. Summarizing, the sequence
$
\{\mu_i\}_{i\in\mathbb N}
\subset C^\infty_0(\mathbb R^n)
$
therefore converges to $\mu$ in all the spaces appearing in the
theorems under consideration and satisfies
\eqn{app.2}
$$
\begin{aligned}
&\|\mu_i\|_{\DD^*}
\le
\|\mu\|_{\DD^*},
\\
&\|\mu_i\|_{L^q(\mathbb R^n)}
\le
\|\mu\|_{L^q(\mathbb R^n)}
\qquad
\mbox{under the assumptions of Theorem \ref{t1}},
\\
&[\mu_i]_{n,1;\mathbb R^n}
\le
[\mu]_{n,1;\mathbb R^n}
\ 
\mbox{under the assumptions of Theorems
\ref{t2} and \ref{t3}, with $n\geq 3$},
\\
&
\|\mu_i\|_{L^1(\mathbb R^2)}
\le
\|\mu\|_{L^1(\mathbb R^2)}
\ 
\mbox{if }n=2,
\\
&\|\mu_i\|_{L^2(\log L)^\alpha(\mathbb R^2)}
\le
\|\mu\|_{L^2(\log L)^\alpha(\mathbb R^2)}
\ 
\mbox{under the assumptions of Theorems \ref{t222} and \ref{t3}, with $n=2$}.
\end{aligned}
$$
Moreover,
\eqn{dualc}
$$
\mu_i\to \mu
\qquad
\mbox{in }\mathbb Y(\mathbb R^n),
$$
and, when $n=2$,
\eqn{app.approximation.mean}
$$
\int_{\mathbb R^2}\mu_i\dx=0
\qquad
\mbox{for every }i\in\mathbb N.
$$

\subsection{Approximate minimizers and compactness}\label{sezio}

For every $i\in\mathbb N$, let
\eqn{app.energy}
$$
\mathcal E_{\mu_i}(w)
:=
\int_{\mathbb R^n}H(Dw)\dx
-
\langle\mu_i,w\rangle=
\int_{\mathbb R^n}H(Dw)\dx
-
\int_{\mathbb R^n}\mu_iw\dx, 
\qquad
w\in\mathbb X(\mathbb R^n).
$$
Note that the identity above is justified by recalling that $\mu_i$ is smooth and compactly supported, and has zero mean
when $n=2$, so that the duality term in \eqref{app.energy} agrees with the integral. By Proposition \ref{exex}, there exists a unique minimizer
$u_i\in\mathbb X(\mathbb R^n)$, which is smooth, strictly spacelike, and a classical solution to
\eqref{bi} with datum $\mu_i$. Using Proposition \ref{boun.p},  \rif{coerciva} and \rif{app.2} imply 
\eqn{app.coercivity}
$$
\|u_i\|_{\DD}\le
2\|\mu_i\|_{\DD^*}
\le
2\|\mu\|_{\DD^*}.
$$
Thus $\{u_i\}$ is bounded in $\DD$ and, up to a not relabelled subsequence, we can assume that 
there exists
$\tilde u\in\DD$ such that
$
u_i\rightharpoonup\tilde u
$
weakly in $\DD$. Since $\mathbb X(\mathbb R^n)$ is convex and
strongly closed in $\DD$, it is weakly closed. Hence
$\tilde u\in\mathbb X(\mathbb R^n)$. The definition of
$\mathfrak m_\mu$ in \rif{mmi} implies
\eqn{app.3}
$$
\begin{cases}
\displaystyle
\|u_i\|_{L^\infty(\mathbb R^n)}
\stackleq{linf} 
c\mathfrak m_{\mu_i}\stackleq{app.2}
c\mathfrak m_\mu
& \mbox{if} \ \ n\ge3
\\[3mm]
\displaystyle
\left\|
\frac{u_i}{\log(e+|\cdot|)}
\right\|_{L^\infty(\mathbb R^2)}
\stackleq{linf} 
c\mathfrak m_{\mu_i}\stackleq{app.2} c\mathfrak m_\mu
&\mbox{if} \ \  n=2.
\end{cases}
$$

Together with $\|Du_i\|_{L^\infty}\le1$, these estimates show that
$\{u_i\}$ is locally uniformly bounded and equi-Lipschitz.
Arzel\`a--Ascoli's theorem therefore gives, after passing to a further
subsequence,
$
u_i\to\tilde u
$
locally uniformly in $\mathbb R^n.$
We now identify $\tilde u$ with the minimizer $u$ corresponding
to the original datum. First,
\eqn{app.pairing.limit}
$$
\left|
\langle\mu_i,u_i\rangle
-
\langle\mu,\tilde u\rangle
\right|
\le
\|\mu_i-\mu\|_{\DD^*}
\|u_i\|_{\DD}+
\left|
\langle\mu,u_i-\tilde u\rangle
\right|
\to 0.
$$
The first term tends to zero by \eqref{dualc} and
\eqref{app.coercivity}, while the second one tends to zero by 
weak convergence. Let $w\in\mathbb X(\mathbb R^n)$. The minimality of $u_i$ gives
$
\mathcal E_{\mu_i}(u_i)\le\mathcal E_{\mu_i}(w).
$
The convexity and lower semicontinuity of $H(\cdot)$ imply
$$
\int H(D\tilde u)\dx \le \liminf_{i \to \infty}\int H(Du_i)\dx.
$$
Moreover,
$\langle\mu_i,w\rangle\to\langle\mu,w\rangle$.
Together with \eqref{app.pairing.limit}, this yields
$$
\mathcal E_\mu(\tilde u)
\le
\liminf_{i\to\infty}
\mathcal E_{\mu_i}(u_i)
\le
\limsup_{i\to\infty}
\mathcal E_{\mu_i}(w)
=
\mathcal E_\mu(w).
$$
Thus $\tilde u$ is a minimizer of $\mathcal E_\mu$ in
$\mathbb X(\mathbb R^n)$. By uniqueness,
$\tilde u=u$. In particular, we have 
\eqn{convconv}
$$
u_i\to u
\quad
\mbox{locally uniformly in }\mathbb R^n.
$$


\section{Theorems \ref{t1}, \ref{sm.t2}, \ref{t2}, \ref{t222} and Corollary \ref{sm.c}}\label{apsec2}

We retain the notation and the approximation scheme of the
preceding section. As established in Section \ref{sezio},
each $u_i\in\mathbb X(\mathbb R^n)$ is the minimizer of
$\mathcal E_{\mu_i}$ and, as a consequence of Proposition \ref{exex} is a smooth, strictly spacelike,
classical solution to \eqref{bi}, with
$\mu_i\in C^\infty_0(\mathbb R^n)$.
Thus the a priori estimates obtained above apply to $u_i$. Let $\mathcal h_\infty(\mu)$ be defined as in
Proposition \ref{com.sup}. Since
$\|\mu_i\|_{\DD^*}\le\|\mu\|_{\DD^*}$, one has
$\mathcal h_\infty(\mu_i)\le\mathcal h_\infty(\mu)$.
Consequently, \eqref{external.compact.support} implies
\eqn{app.4}
$$
\supp\, 
\left(\mathcal h_{u_i}-\mathcal h_\infty(\mu)\right)_+\Subset\mathbb R^n
\qquad
\mbox{for every }i\in\mathbb N.
$$
In view of \eqref{app.4}, when $n\ge3$, estimate
\eqref{czcz} applies uniformly to $u_i$
with $\mathcal h_\infty=\mathcal h_\infty(\mu)$.
Throughout the proof, we write
$\mathcal h_i:=\mathcal h_{u_i}$.

\subsection{Proof of Theorem \ref{t1}}

\subsubsection*{Common regularity and weak solvability}

Assume first that $n\ge3$. Since
$
\mathcal h_i
\le \mathcal h_\infty(\mu)
+(\mathcal h_i-\mathcal h_\infty(\mu))_+,
$
estimate \eqref{czcz} gives, for every ball
$B\subset\mathbb R^n$,
\eqn{cz.0}
$$
\int_{2B}\mathcal h_i^q\dx
\le c\mathcal h_\infty(\mu)^q|B|
+c\mathfrak m_\mu^q
\|\mu\|_{L^q(\mathbb R^n)}^q.
$$
Here and below we use \eqref{app.2},
$\mathfrak m_{\mu_i}\le\mathfrak m_\mu$, and
$\mathcal h_\infty(\mu)\ge2$.
Moreover, for every $2\le q\le n$,
\eqref{w22.high} gives
\eqn{app.55}
$$
\begin{aligned}
&
\|D\log\mathcal h_i\|_{L^2(B)}
+
\|D\mathcal h_i^{-1}\|_{L^2(B)}
+
\|D^2u_i\|_{L^2(B)}
\\
&\qquad\le
c\mathcal h_\infty(\mu)
|B|^{\frac{n-2}{2n}}
+
c|B|^{\frac12-\frac1q}
\left(
1+\frac{\mathfrak m_\mu}{|B|^{1/n}}
\right)
\|\mu\|_{L^q(\mathbb R^n)},
\end{aligned}
$$
where $c=c(n,q)$. Assume now that $n=2$, so that $q=2$.
By \eqref{w22.two} and \eqref{app.2},
\eqn{app.555}
$$
\|D\log\mathcal h_i\|_{L^2(\mathbb R^2)}+
\|D\mathcal h_i^{-1}\|_{L^2(\mathbb R^2)}+
\|D^2u_i\|_{L^2(\mathbb R^2)}
\le c\|\mu\|_{L^2(\mathbb R^2)}.
$$
Moreover, \eqref{w22.two.log} gives
$$
\|\log\mathcal h_i\|_{L^2(\mathbb R^2)}
\le
\|\mu_i\|_{\mathcal D^{1,2}_{\#}(\mathbb R^2)^*}
\le
\|\mu\|_{\mathcal D^{1,2}_{\#}(\mathbb R^2)^*}.
$$
Thus $\{\log\mathcal h_i\}$ is bounded in
$W^{1,2}(\mathbb R^2)$.
We now use the Moser--Trudinger inequality to obtain
uniform local integrability of $\mathcal h_i$ to every finite
exponent.

Fix a ball $B\subset\mathbb R^2$ and choose
$\eta\in C^\infty_0(4B)$ such that
$
0\le\eta\le1$, $
\eta=1$ on $2B$, and
$|D\eta|\le c|B|^{-1/2}$.
Set
$
w_i:=\eta\log\mathcal h_i\in W^{1,2}_0(4B).
$
The preceding estimates imply
$$
\|Dw_i\|_{L^2(4B)}^2
\le
c\left(
\|\mu\|_{L^2(\mathbb R^2)}^2
+
|B|^{-1}\|\mu\|_{\mathcal D^{1,2}_{\#}(\mathbb R^2)^*}^2
\right).
$$
If $\|Dw_i\|_{L^2(4B)}=0$, then $w_i=0$ almost everywhere
and $\mathcal h_i=1$ on $2B$.
Otherwise, for every $1\le p<\infty$, Young's inequality gives
$$
pw_i
\le
\frac{\mathfrak c_3w_i^2}{\|Dw_i\|_{L^2(4B)}^2}
+
\frac{p^2\|Dw_i\|_{L^2(4B)}^2}{4\mathfrak c_3},
$$
where $\mathfrak c_3$ is the constant in
\eqref{trudinger.scale}.
Since $w_i=\log\mathcal h_i$ on $2B$, applying
\eqref{trudinger.scale} on $4B$ yields
$$
\begin{aligned}
\int_{2B}\mathcal h_i^p\dx
&\le
\exp\left\{
\frac{p^2\|Dw_i\|_{L^2(4B)}^2}{4\mathfrak c_3}
\right\}
\int_{4B}
\exp\left\{
\frac{\mathfrak c_3w_i^2}{\|Dw_i\|_{L^2(4B)}^2}
\right\}\dx\\
&\le
c|B|
\exp\left\{
\frac{p^2\|Dw_i\|_{L^2(4B)}^2}{4\mathfrak c_3}
\right\}.
\end{aligned}
$$
In conclusion, in both cases,
\eqn{cz.1}
$$
\int_{2B}\mathcal h_i^p\dx
\le
c|B|
\exp\left\{
c_n p^2\left(
\|\mu\|_{L^2(\mathbb R^2)}^2
+
|B|^{-1}\|\mu\|_{\DD^*}^2
\right)
\right\},
\qquad 1\le p<\infty.
$$
In particular, taking $p=2$ gives a uniform bound in
$L^2(2B)$.
Together with \eqref{app.3} and
$\|Du_i\|_{L^\infty(\mathbb R^n)}\le1$,
these estimates show that $\{u_i\}$ is bounded in
$W^{2,2}(B)$ for every ball $B$.
By local compactness, after passing to a (not relabelled) subsequence
and using \eqref{convconv} to identify the limit,
we obtain
\eqn{conv.grad}
$$
\begin{cases}
u_i\rightharpoonup u
\quad\mbox{weakly in }W^{2,2}_{\loc}(\mathbb R^n),
\\[1mm]
Du_i\to  Du
\quad\mbox{strongly in }L^2_{\loc}(\mathbb R^n),
\\[1mm]
Du_i(x)\to  Du(x)
\quad\mbox{for almost every }x\in\mathbb R^n.
\end{cases}
$$
In particular,
$u\in W^{2,2}_{\loc}(\mathbb R^n)$.
Recall that
$
\mathcal h_u(x)=
1/\sqrt{1-|Du(x)|^2},
$
with the convention $\mathcal h_u(x)=+\infty$
when $|Du(x)|=1$.
Then $\mathcal h_i\to\mathcal h_u$ almost everywhere.
Fatou's lemma and \eqref{cz.0}--\eqref{cz.1} imply
$\mathcal h_u\in L^q_{\loc}(\mathbb R^n)$, which is the first assertion in \rif{unaw}. 
Consequently,
$
|\{x\in\mathbb R^n:|Du(x)|=1\}|=0,
$
and, after passing to a further subsequence,
\eqn{conv.0}
$$
\mathcal h_i\rightharpoonup\mathcal h_u
\quad
\mbox{weakly in }L^q_{\loc}(\mathbb R^n).
$$
Since $q\ge2$ and
$0\le\log\mathcal h_i\le\mathcal h_i$,
estimates \eqref{cz.0}, \eqref{cz.1},
\eqref{app.55} and \eqref{app.555}, show that
$\{\log\mathcal h_i\}$ is bounded in
$W^{1,2}_{\loc}(\mathbb R^n)$.
Similarly, $0<\mathcal h_i^{-1}\le1$ and the same
gradient estimates show that
$\{\mathcal h_i^{-1}\}$ is bounded in
$W^{1,2}_{\loc}(\mathbb R^n)$.
Their almost everywhere limits are
$\log\mathcal h_u$ and $\mathcal h_u^{-1}$, respectively.
Weak compactness therefore gives
$
\mathcal h_u^{-1},\log\mathcal h_u
\in W^{1,2}_{\loc}(\mathbb R^n),
$
finishing the proof of \eqref{unaw}.
The bounds in \eqref{duaw} follow from
Proposition \ref{boun.p}. We now pass to the limit in the Euler--Lagrange equation.
For every $\varphi\in C^\infty_0(\mathbb R^n)$,
\eqn{app.weak.approx}
$$
\int_{\mathbb R^n}
\langle\mathcal h_iDu_i,D\varphi\rangle\dx
=
\int_{\mathbb R^n}\mu_i\varphi\dx.
$$
Set $K:=\supp\,\varphi$ and $q':=q/(q-1)\le2$.
By \eqref{conv.grad},
$Du_i\to Du$ strongly in $L^{q'}(K)$.
Together with \eqref{conv.0} and the strong convergence
$\mu_i\to\mu$ in $L^q(\mathbb R^n)$, this allows us to
pass to the limit in \eqref{app.weak.approx}, obtaining
$$
\int_{\mathbb R^n}
\langle\mathcal h_uDu,D\varphi\rangle\dx
=
\int_{\mathbb R^n}\mu\varphi\dx.
$$
Since
$\mathcal h_uDu\in L^q_{\loc}(\mathbb R^n;\mathbb R^n)$,
the minimizer $u$ weakly solves \eqref{bi}.

\subsubsection*{Proof of \textnormal{(I)}} Assume that $2<q<n$ and set $m:=q-2$.
We apply \eqref{cacc} to $(u_i,\mu_i)$ with exponent $m$
and a cutoff $\eta\in C^\infty_0(2B)$ satisfying
$
0\le\eta\le1$, 
$
\eta=1\ \mbox{on }B$, $
|D\eta|\le c|B|^{-1/n}.
$ 
Since $\mathcal h_i\ge1$, the left-hand side controls
$\|D\mathcal h_i^{(q-2)/2}\|_{L^2(B)}^2$
and $\|D^2u_i\|_{L^2(B)}^2$.
On the right-hand side, Young's inequality and
\eqref{app.2} give
$$
\int_{2B}\eta^2\mu_i^2\mathcal h_i^{q-2}\dx
\le
\|\mathcal h_i\|_{L^q(2B)}^q
+
\|\mu\|_{L^q(\mathbb R^n)}^q.
$$
Combining this with \eqref{cz.0}, we obtain
\eqn{app.5}
$$
\begin{aligned}
&
\|D\mathcal h_i^{(q-2)/2}\|_{L^2(B)}
+
\|D^2u_i\|_{L^2(B)}
\\
&\qquad\le
c\bigl(|B|^{-1/n}+1\bigr)
\left(
|B|^{1/2}\mathcal h_\infty(\mu)^{q/2}
+
[\mathfrak m_\mu^{q/2}+1]
\|\mu\|_{L^q(\mathbb R^n)}^{q/2}
\right),
\end{aligned}
$$
where $c=c(n,q)$.
Moreover,
$\mathcal h_i^{q-2}\le\mathcal h_i^q$.
Thus \eqref{cz.0} and \eqref{app.5} imply that
$\{\mathcal h_i^{(q-2)/2}\}$ is bounded in
$W^{1,2}_{\loc}(\mathbb R^n)$.
Its almost everywhere convergence and weak compactness yield
$
\mathcal h_u^{(q-2)/2}
\in W^{1,2}_{\loc}(\mathbb R^n).
$

\subsubsection{Proof of \textnormal{(II)}}\label{equalize}

Assume that $q=n\ge3$.
After passing to a further subsequence, we may assume that
$
\sum_{i=1}^\infty\|\mu_i-\mu\|_{L^n(\mathbb R^n)}<\infty.
$
Set $
g:=|\mu|+
\sum_{i=1}^\infty|\mu_i-\mu|$. 
Then $g\in L^n(\mathbb R^n)$ and
$|\mu_i|\le g$ almost everywhere.
In particular, for every $x_0\in\mathbb R^n$ and $r>0$,
$
\|\mu_i\|_{L^n(B_r(x_0))}
\le
\|g\|_{L^n(B_{2r}(x_0))}.
$
It follows that \eqref{smallina} holds uniformly with respect to both $i$ and $x_0$.
Of course, the absolute continuity of the integral of $g^n$
allows the threshold radii in Lemma \ref{caccim}
and Proposition \ref{p4.4} to be chosen independently
of $i$ and of the center $x_0$. Fix $p>1$.
Proposition \ref{p4.4}, \eqref{app.2}, and a finite
covering argument give
$
\sup_{i\in\mathbb N}\|\mathcal h_i^{p/2}\|_{W^{1,2}(B)}<\infty
$
for every ball $B\subset\mathbb R^n$.
Passing to the limit by weak compactness and almost
everywhere convergence yields
$
\mathcal h_u^{p/2}
\in W^{1,2}_{\loc}(\mathbb R^n).
$
The case $p=1$ follows by applying the usual Sobolev chain rule
to $t\mapsto t^{1/p_0}$ on $[1,\infty)$, with any fixed
$p_0>1$, and to the function $\mathcal h_u^{p_0/2}$.

\subsubsection{Proof of \textnormal{(III)}}

Assume that $n=2$.
For every ball $B\subset\mathbb R^2$ and every
$1\le p<\infty$, estimate \eqref{cz.1} gives
$
\sup_i\int_{2B}\mathcal h_i^p\dx<\infty.
$
Since $\mathcal h_i\to\mathcal h_u$ almost everywhere,
Fatou's lemma yields
$
\mathcal h_u\in L^p_{\loc}(\mathbb R^2)
$
for every $1\le p<\infty$, completing the proof.

\subsection{Proof of Theorems \ref{t2} and \ref{t222}}

Assume first that $n\ge3$ and that the assumptions of Theorem
\ref{t2} hold. Since
$
L(n,1)(\mathbb R^n)\hookrightarrow L^n(\mathbb R^n),
$
Theorem \ref{t1}, with $q=n$, already implies that $u$ is a
distributional solution and that
$u\in W^{2,2}_{\loc}(\mathbb R^n)$. When $n=2$ and the assumptions of Theorem \ref{t222} hold, the
embedding
$
L^2(\log L)^\alpha(\mathbb R^2)\hookrightarrow L^2(\mathbb R^2)
$
allows us to apply Theorem \ref{t1} with $q=2$. Thus the same
conclusions hold in dimension two. Apply now Proposition \ref{boostdopo} to each smooth solution $u_i$.
Using \eqref{app.2}, we obtain
\eqn{app.boost.uniform}
$$
\|\mathcal h_i\|_{L^\infty(\mathbb R^n)}
\le
\begin{cases}
\displaystyle
\exp\left\{c_n[\mu]_{n,1}\right\}.
& \mbox{if} \ \ n\ge3
\\[3mm]
\displaystyle
\mathcal h_\infty(\mu)
\exp\left\{
c_\alpha
\|\mu\|_{L^2(\log L)^\alpha(\mathbb R^2)}
\right\}
& \mbox{if} \ \ n=2.
\end{cases}
$$
Recall that $\mathcal h_\infty(\mu)$ is as in
Proposition \ref{com.sup}. 
The compactness argument in the proof of Theorem \ref{t1} gives
$Du_i\to Du$ almost everywhere. Hence
$\mathcal h_i\to\mathcal h_u$ almost everywhere, and letting $i\to \infty$ in 
\eqref{app.boost.uniform} yields the bounds stated in Theorems \ref{t2} and \ref{t222}. It remains to prove the continuity of $Du$. We differentiate
directly the equation satisfied by $u$. By Theorem \ref{t1},
$
u\in W^{2,2}_{\loc}(\mathbb R^n)
$
and $u$ is a weak solution to \eqref{bi}. Moreover, the
boost bound already obtained gives a constant $M<\infty$ such that
$
\mathcal h_u\le M
$
almost everywhere. Consequently,
$
|Du|\le\theta:=\sqrt{1-M^{-2}}<1$
holds a.e. in $\er^n$. 
Set
$
A(x):=\partial^2H(Du(x)).
$
Since
$
A=\mathcal h_u I+\mathcal h_u^3Du\otimes Du,
$
the matrix $A(\cdot)$ is measurable, bounded, and uniformly elliptic.
More precisely,
\eqn{app.uniform.ellipticity}
$$
|\xi|^2
\le\langle A(x)\xi,\xi\rangle\le M^3|\xi|^2
\qquad
\mbox{for every }\xi\in\mathbb R^n
\mbox{ and almost every }x\in\mathbb R^n.
$$
Fix $s\in\{1,\ldots,n\}$; a standard difference quotient argument leads to 
\eqn{app.linearized}
$$
\int_{\mathbb R^n}
\langle ADD_su,D\varphi\rangle\dx
=
-\int_{\mathbb R^n}\mu D_s\varphi\dx
\qquad
\mbox{for every }\varphi\in C^\infty_0(\mathbb R^n).
$$
Set $v:=D_su$ and $f:=-\mu e_s$. Then
\eqref{app.linearized} reads
$
-\diver(ADv)=-\diver f. 
$
If $n\ge3$, then $f\in L_{\loc}(n,1)$, and the
endpoint continuity theorem for scalar uniformly elliptic
equations with bounded measurable coefficients gives
$v\in C_{\loc}^0(\mathbb R^n)$; see for instance \cite[Theorem 5.8]{sa21}. 
If $n=2$, \rif{immergilog} gives
$\mu\in L_{\loc}(2,1)(\mathbb R^2)$.
Meyers' estimates and real interpolation then yield
$Dv\in L_{\loc}(2,1)$, and the endpoint Sobolev embedding
implies that $v$ is continuous.
Since $s$ is arbitrary, $Du$ is continuous in
$\mathbb R^n$.
\subsection{Proof of Theorem \ref{sm.t2} and Corollary \ref{sm.c}}
By Theorem \ref{t1}, $u$ weakly solves \eqref{bi} and
$$
Du,\ \log\mathcal h_u
\in W^{1,2}_{\loc}(\mathbb R^n).
$$
Applying Theorem \ref{cap.t.1} with $p=2$ gives the
capacitary and Hausdorff dimension bounds in
Theorem \ref{sm.t2}.
Under the additional assumption on the diameters of
light rays, the same theorem yields Corollary \ref{sm.c}.

\section{Theorem \ref{dir.borderline}}\label{diriproof}
Let $\tilde\varepsilon_\partial$ and $\sigma_*$
be given by Proposition \ref{corr}, and choose
$
\varepsilon_\partial:=\tilde\varepsilon_\partial/2$ in Theorem \ref{dir.borderline}. 
Define
$\mathcal h_{\partial\Omega}$ by
\eqref{boundary.boost.threshold}.
We follow the approximation scheme of Section \ref{apsec},
describing only the modifications needed for the Dirichlet
problem.
Set 
$Z:=L(n,1)(\Omega)$ when $ n\ge3 $ and 
$Z:= L^2(\log L)^\alpha(\Omega)$ when $n=2$.
Extend $f$ and $\mu_*$ by zero outside $\Omega$, and let
$f_i,\nu_i\in C^\infty(\overline\Omega)$ be the restrictions
to $\Omega$ of their convolutions with the mollifiers used in Section \ref{mollisi}. 
As in Section \ref{apsec}, we have
\eqn{per}
$$
\|f_i-f\|_Z+\|\nu_i-\mu_*\|_Z\to 0, \quad \|f_i\|_{L^\infty(\Omega)}
\le \|f\|_{L^\infty(\Omega)}
\le F_0, \quad [\nu_i]_{n,1;\Omega}
\le[\mu_*]_{n,1;\Omega}
\le\Lambda_0.
$$
Moreover, strong convergence in $L(n,1)(\Omega)$ and
\eqref{data2} give
$$
[\nu_i]_{n,1;\Omega_{r_0}}
\le
[\mu_*]_{n,1;\Omega_{r_0}}
+
[\nu_i-\mu_*]_{n,1;\Omega}
\le
\frac{\tilde\varepsilon_\partial}{2}+\tx{o}(1).
$$
Thus, after discarding finitely many terms,
$[\nu_i]_{n,1;\Omega_{r_0}}\le\tilde\varepsilon_\partial$.
We can therefore set
$
\mu_{*,i}:=\nu_i$ and $
\mu_i:=f_i+\nu_i$
and apply Corollary \ref{app.cor} with the threshold
$\tilde\varepsilon_\partial$. Let $u_i$ be the Dirichlet minimizer with charge $\mu_i$
and the fixed boundary datum $u_0$.
By \cite[Theorem 3.6]{bs82} and convexity,
$u_i\in C^{2,\beta}(\overline\Omega)\cap C^\infty(\Omega)$
is a strictly spacelike classical solution.
Set $\mathcal h_i:=\mathcal h_{u_i}$.
Corollary \ref{app.cor}, applied with
$f_i$ and $\mu_{*,i}$, gives
$\supp(\mathcal h_i-\kappa)_+\Subset\Omega
$ for every $\kappa\ge\mathcal h_{\partial\Omega}$, where $\mathcal h_{\partial\Omega}$ has been defined in \rif{boundary.boost.threshold}. 
The threshold $\mathcal h_{\partial\Omega}$ is independent
of $i$; the width of the corresponding boundary collar
may depend on $i$. Note that this is similar to \rif{supporto}. At this point we can therefore repeat the testing argument in the proof
of Proposition \ref{boostdopo}, using compactly supported
truncations in $\Omega$. It gives
\[
\int_\Omega
\left|D(\log\mathcal h_i-\log\kappa)_+\right|^2\dx
\le
c_n\int_{\{\mathcal h_i>\kappa\}}|\mu_i|^2\dx,
\qquad
\kappa\ge\mathcal h_{\partial\Omega},
\]
which is indeed the analogue of \eqref{whit}. 
Applying the same iteration to
\[
v_i:=
\left(\log\frac{\mathcal h_i}
{\mathcal h_{\partial\Omega}}\right)_+
\in W^{1,2}_0(\Omega),
\]
we obtain
\begin{equation}\label{dir.approx.boost}
\|\mathcal h_i\|_{L^\infty(\Omega)}
\le
\begin{cases}
\displaystyle
\mathcal h_{\partial\Omega}
\exp\!\left\{c_n[\mu_i]_{n,1;\Omega}\right\}
& \mbox{if} \ \ n\ge3\\[2mm]
\displaystyle
\mathcal h_{\partial\Omega}
\exp\!\left\{
c_{\alpha,|\Omega|}
\|\mu_i\|_{L^2(\log L)^\alpha(\Omega)}
\right\}
& \mbox{if} \ \ n=2.
\end{cases}
\end{equation}
Indeed, the Sobolev and Moser--Trudinger inequalities used in Proposition \ref{boostdopo} 
apply to the zero extensions of these truncations.
In dimension two, we use
$|\{\mathcal h_i>\mathcal h_{\partial\Omega}\}|\le|\Omega|$
in place of \eqref{stimamis}; this creates in the constants a dependence on $|\Omega|$ (appearing indeed in the final statement). By strong convergence of the charges, \eqref{dir.approx.boost}
provides a constant $M\ge2$, independent of $i$, such that
$
\mathcal h_i\le M$ and $
|Du_i|\le\theta:=\sqrt{1-M^{-2}}<1$ in 
$\Omega$. The convergence argument of
Section \ref{apsec}, now with fixed boundary datum, yields,
after passing to a subsequence,
$
u_i\to u$ uniformly on $\overline\Omega$,
 and $
\mathcal E_{\mu_i}(u_i;\Omega)
\to\mathcal E_\mu(u;\Omega).
$
Moreover, up to a further subsequence,
$u_i\rightharpoonup^*u$ in $W^{1,\infty}(\Omega)$.
By weak-* lower semicontinuity,
$
\nr{Du}_{L^\infty(\Omega)}
\le
\liminf_{i\to\infty}\nr{Du_i}_{L^\infty(\Omega)}
\le\theta<1.
$
Moreover, the uniform, strict convexity of $H$ and the minimality
of $u_i$, applied to $(u_i+u)/2$, give, via standard convexity arguments
\[
\frac14\int_\Omega|Du_i-Du|^2\dx
\le
\mathcal E_{\mu_i}(u;\Omega)
-
\mathcal E_{\mu_i}(u_i;\Omega)
\to 0.
\]
Thus $Du_i\to Du$ strongly in $L^2(\Omega)$ and, after a
further subsequence, almost everywhere.
In particular, $\mathcal h_i\to\mathcal h_u$ almost everywhere.
Passing to the limit in \eqref{dir.approx.boost}, using
\rif{per}, proves
\eqref{boost.bound.dir} and \eqref{boost.2d.dir}.
The uniform spacelikeness and the strong convergence
of the gradients allow us to pass to the limit in the
equations, so $u$ weakly solves \eqref{bi} in $\Omega$.
Standard local difference quotient estimates, using
$\mu\in L^2(\Omega)$, give
$u\in W^{2,2}_{\loc}(\Omega)$.
The argument used in the proof of Theorems \ref{t2}
and \ref{t222} then yields $u\in C^1(\Omega)$; see \rif{app.uniform.ellipticity}-\rif{app.linearized}.
\begin{remark}\label{prefactor}
\emph{The factor $\mathcal h_{\partial\Omega}$ in
\eqref{boost.bound.dir} cannot be dropped.
To see this, we go back to Remark
\ref{controre}. 
For $n\ge3$, define $\eta_T$ on $\mathbb R^n$
by the same formula as in \eqref{logatau}, and set
$
u_{a,T}(x):=a x_1\eta_T(x),
$
where $a\in(0,1)$ and $T$ is sufficiently large. On the fixed domain $\Omega=B_1(0)$, we have
$
u_{a,T}(x)=a x_1.
$
Therefore, the corresponding charge vanishes in $\Omega$, and $u_{a,T}|_\Omega$ is the Dirichlet minimizer with
boundary datum $a x_1$, by convexity of $H$. All the assumptions of Theorem \ref{dir.borderline}
are satisfied with $\tx u_0(x)=a x_1$,
$\sigma_0=1-a$, and $f=\mu_*=0$.
However,
$$
\nr{\mathcal h_{u_{a,T}}}_{L^\infty(\Omega)}=\frac1{\sqrt{1-a^2}}>1
=
\exp\left\{c_n[\mu]_{n,1;\Omega}\right\}.
$$}
\end{remark}

\section{Theorem \ref{t3}}\label{passa}

We pass to the limit in the a priori estimate \eqref{0.10},
using the approximation constructed in Section \ref{apsec}.
Recall that the corresponding minimizers $u_i$ are smooth
classical solutions with charges $\mu_i\in C^\infty_0(\mathbb R^n)$, and recall also 
\eqref{convconv} and \eqref{conv.grad}. Set $\mathcal h_i:=\mathcal h_{u_i}$ and let $\gamma$
be the exponent in Proposition \ref{p31}. 
Theorems \ref{t2} and \ref{t222} give $u\in C^1(\mathbb R^n)$.
Moreover, by \eqref{app.boost.uniform}, there exists $M\ge2$
such that
\eqn{uniunni}
$$
\mathcal h_i,\mathcal h_u\le M\Longrightarrow |Du_i|,|Du|\le\theta:=\sqrt{1-M^{-2}}<1 \quad \mbox{in $\er^n$}. 
$$
The approximation also gives $\mu_i\to\mu$ strongly in $Z$, where 
$Z:=L(n,1)(\mathbb R^n)$ when $n\ge3$ and 
$Z:=L^2(\log L)^\alpha(\mathbb R^2)$ when $n=2$. Proceeding as in Section \ref{equalize}, we choose a
(not relabelled) subsequence such that
$\sum_i\|\mu_i-\mu\|_Z<\infty$. Then
\eqn{lagg}
$$
|\mu_i|,|\mu|\le
g:=|\mu|+\sum_{i=1}^\infty|\mu_i-\mu|\in Z,
$$
and $\mu_i\to\mu$ almost everywhere in $\mathbb R^n$.
Write
$K_s^i(x_0):=K_s^{u_i}(x_0)$ and $K_s(x_0):=K_s^u(x_0).$
Lemma \ref{balls} \textnormal{(h$_1$)} and \rif{uniunni} give
\eqn{00.1}
$$
B_s(x_0)\subset K_s^i(x_0)\cap K_s(x_0),
\qquad
K_s^i(x_0)\cup K_s(x_0)\subset B_{Ms}(x_0),
$$
for every $x_0\in\mathbb R^n$ and $s>0$.
For a fixed center $x_0$, consider
$
F_{x_0}(x):=|x-x_0|^2-|u(x)-u(x_0)|^2.
$
Since $u\in C^1(\mathbb R^n)$ and $|Du|\le\theta$,
\[
\begin{aligned}
\langle D_xF_{x_0}(x),x-x_0\rangle
&=
2|x-x_0|^2
-2 (u(x)-u(x_0) )
\langle Du(x),x-x_0\rangle
\\
&\ge
2(1-\theta^2)|x-x_0|^2>0
\qquad\text{if }x\ne x_0.
\end{aligned}
\]
Thus every positive level set of $F_{x_0}$ is a
$C^1$ hypersurface, and
$|\partial K_s(x_0)|=0$ for every $s>0$.
The local uniform convergence of $u_i$ therefore implies
\eqn{passa.sets}
$$
\mathds1_{K_s^i(x_0)}
\to 
\mathds1_{K_s(x_0)}
\quad\text{almost everywhere in }\mathbb R^n,
$$
for every fixed $x_0$ and $s>0$.
Together with \eqref{00.1}, this gives
$|K_s^i(x_0)|\to|K_s(x_0)|$.
Since $\mathcal h_i^{-(\gamma+1)}\to
\mathcal h_u^{-(\gamma+1)}$ almost everywhere by \rif{conv.grad}, dominated convergence yields
\eqn{passa.mean}
$$
\mint_{K_r^i(x_0)}
\mathcal h_i^{-(\gamma+1)}\dx
\to 
\mint_{K_r(x_0)}
\mathcal h_u^{-(\gamma+1)}\dx
$$
for every $x_0\in\mathbb R^n$ and $r>0$. To deal with the potential terms, for $m\in\{1,2\}$, define
$$
q_m(s):=s^{m-n-1}\int_0^{\omega_nM^ns^n} (g^*(t))^m\,\dd t,\qquad s>0, 
$$
where $g^*$ denotes the nonincreasing rearrangement of $g$,
as defined in \eqref{rearrangia}.
By \eqref{00.1} and \rif{lagg} and the classical Hardy--Littlewood rearrangement
inequality, we have 
\eqn{sopra}
$$
s^{m-n-1}
\int_{K_s^i(x_0)}|\mu_i|^m\dx\le q_m(s), \quad  s^{m-n-1}
\int_{K_s(x_0)}|\mu|^m\dx
\le q_m(s).
$$
If $n\ge3$, Fubini's theorem gives
$$
\int_0^\infty q_m(s)\,\dd s
=\frac{(\omega_nM^n)^{1-m/n}}{n-m}\int_0^\infty\bigl(t^{1/n}g^*(t)\bigr)^m\frac{\dd t}{t}
<\infty
$$
and here we have used 
$L(n,1)\hookrightarrow L(n,m)$ for $m=1,2$ and the very definition in \rif{definorma} to establish the finiteness stated above. 
If $n=2$, set
$
\mathcal L(s):=\log(e+1/[\omega_2M^2s^2]).
$
By \eqref{orlicz.rearr},
\[
\int_0^{\omega_2M^2s^2}
(g^*(t) )^2\,\dd t
\le
\frac{c_\alpha \|g\|_{L^2(\log L)^\alpha(\mathbb R^2)}^2}{\mathcal L(s)^\alpha}.
\]
Hence, also using Cauchy--Schwarz,
\[
q_2(s)\le
\frac{c_\alpha }{s\mathcal L(s)^\alpha} \|g\|_{L^2(\log L)^\alpha(\mathbb R^2)}^2,
\qquad
q_1(s)\le
\frac{c_\alpha\sqrt{\omega_nM^n}}{s\mathcal L(s)^{\alpha/2}}\|g\|_{L^2(\log L)^\alpha(\mathbb R^2)}.
\]
Since $\alpha>2$, both the functions in the right-hand sides of the above display are integrable near zero
and on every compact subinterval of $(0,\infty)$.
Consequently, $q_1,q_2\in L^1(0,r)$ for every $r>0$. For every fixed $x_0$ and $s>0$, \eqref{passa.sets},
the almost everywhere convergence of $\mu_i$, and the
local integrability of $g^m$ give
\[
\int_{K_s^i(x_0)}|\mu_i|^m\dx
\to 
\int_{K_s(x_0)}|\mu|^m\dx.
\]
Together with \eqref{sopra}, this allows us to apply the dominated convergence theorem, yielding
\eqn{passa.potentials}
$$
\mathbf{PL}_{m,u_i}^{1}
(|\mu_i|^m;x_0,r)
\to
\mathbf{PL}_{m,u}^{1}
(|\mu|^m;x_0,r),
\qquad m=1,2,
$$ for every $x_0\in\mathbb R^n$ and $r>0$. 
Apply \eqref{0.10} to $u_i$ and discard the nonnegative
Hessian term. Using \eqref{passa.mean} and
\eqref{passa.potentials}, we can pass to the limit and
obtain \eqref{int.pot} for almost every $x_0$ for every $r>0$. Finally, fix $r>0$. We show that all terms in the limiting
inequality are continuous with respect to $x_0$. For this we already know, from Theorems \ref{t2} and \ref{t222}, that $Du$, and therefore $\mathcal h_u^{-1}$, are continuous. 
If $x_j\to x_0$, the continuity of $u$ and the null measure
of every positive level set of $F_{x_0}$ imply that
$
\mathds1_{K_s(x_j)}
\to 
\mathds1_{K_s(x_0)}$
a.e., and for every $s>0$.
The sets $K_s(x_j)$, $0<s\le r$, lie in a common bounded
ball. Dominated convergence and absolute continuity of the integral therefore gives continuity
of the mean in \eqref{int.pot}. 
It also gives continuity of the integrals defining the
potentials at every fixed scale. The 
inequalities  \rif{sopra} then yield continuity of
the potentials.
Therefore \eqref{int.pot} extends to every
$x_0\in\mathbb R^n$. As $r>0$ was arbitrary, the proof
is complete.

\section{Theorems \ref{light.t} and \ref{b.th}} \label{s7.7}

\subsection{Theorem \ref{light.t}}

By Remark \ref{globale.locale}, $u$ minimizes the
corresponding Dirichlet functional in every ball and,
in particular, is a local minimizer in the sense of
Definition \ref{minimilocali}. Assume by contradiction that
$\mathcal S_u\neq\varnothing$.
By the definition of $\mathcal S_u$, there exists a
nondegenerate light segment
$\overline{xy}\subset\mathbb R^n$.
Up to exchanging $x$ and $y$, we may assume that
$
u(y)-u(x)=|y-x|.
$
Set
$
x_t:=x+t(y-x)
$
for $t\in\mathbb R$.
By weak spacelikeness,
$
u(x_t)=u(x)+t|y-x|
$
for every $t\in[0,1]$.
Theorem \ref{al.t}, applied with $x_0=x$ and $x_1=y$,
therefore gives
$
u(x_t)=u(x)+t|y-x|$
for every $t\in\mathbb R$.
This contradicts the global boundedness of $u$ given
by \eqref{infsup} when $n\ge3$, and the sublinear
growth property \eqref{weightsmall} when $n=2$.
Hence $\mathcal S_u=\varnothing$, and $u$ has no
light segments. Assume now, in addition, that
$\mu\in L^2_{\loc}(\mathbb R^n)$, and fix a ball
$B\Subset\mathbb R^n$.
Since $u$ has no light segments, its restriction to $B$
is a spacelike extension of the boundary datum
$u|_{\partial B}$.
Moreover, $\mu|_B\in L^1(B)\cap L^2(B)$.
We may therefore apply \cite[Theorem 1.14(i)]{bimm24}
when $n\ge3$, and \cite[Theorem 1.11(i)]{bimm24}
when $n=2$, with zero singular part of the charge.
These results give
$
\mathcal h_u\in L^1_{\loc}(B)
$
and
$$
\int_B
\langle\mathcal h_uDu,D\varphi\rangle\dx
=
\int_B\mu\varphi\dx
\qquad
\mbox{for every }\varphi\in C^\infty_0(B).
$$
Since $B$ is arbitrary, it follows that
$
\mathcal h_uDu\in L^1_{\loc}(\mathbb R^n;\mathbb R^n)
$
and $u$ weakly solves \eqref{bi} in $\mathbb R^n$.
This proves Theorem \ref{light.t}.

\subsection{Theorem \ref{b.th}}
Suppose by contradiction that $u$ has a light segment
$
\overline{xy}\Subset\Omega, 
$ 
and assume that
$
u(y)-u(x)=|y-x|.$
Set
$
x_t:=x+t(y-x),$
and let $(t_-,t_+)$ be the connected component of
$
\{t\in\mathbb R:x_t\in\Omega\}$
containing $[0,1]$. Since $\Omega$ is bounded,
$
-\infty<t_-<0<1<t_+<\infty,$
and the points $
x_-:=x_{t_-}$ , $
x_+:=x_{t_+}
$
belong to $\partial\Omega$. Moreover,
$
\overline{x_-x_+}\subset\overline\Omega.
$ By Theorem \ref{al.t},
$
u(x_t)=u(x)+t|y-x|$
for every $t\in(t_-,t_+)$.
Passing to the endpoints along this segment and using the boundary
condition in the sense of Definition \ref{defidatobordo}, we obtain
$
u_0(x_+)-u_0(x_-)=(t_+-t_-)|y-x|=|x_+-x_-|.
$
This contradicts the strict chord condition in
\eqref{u0u0.bd00}. Therefore $u$ has no light segments in $\Omega$. 
Finally, when 
$
\mu\in L^2_{\loc}(\Omega),
$
we can conclude exactly as in the proof of Theorem \ref{light.t}.

\section{Appearance of light segments and Theorem \ref{ex.t}}\label{shine} 

Our ultimate goal is to construct a weakly spacelike function $u\in \mathcal{D}^{1,2}(\mathbb{R}^{n})$ with a light segment and whose Lorentzian mean curvature belongs to $L(n-1,\infty)(\mathbb{R}^{n})\setminus L(n-1,s)(\mathbb{R}^{n})$ for all $0<s<\infty$. The construction can be split in two steps: first, we design a weakly spacelike function $v\in W^{1,\infty}_{\loc}(\mathbb{R}^{n})$ such that $\diver(\mathcal{h}_{v}Dv)\in L_{\loc}(n-1,\infty)(\mathbb{R}^{n})$ and then cut it off so that the newly defined function belongs to the right energy space and it is a minimizer. We warn the reader that the rest of the proof will be full of elementary calculus computations, that we are going to report in some detail in order to facilitate digestion. 

{\em Step 1: Main construction}. 
Let $\tx{l}>0$, take $x_{1}:=-\tx{l}\tx{e}_{n}$, where, as usual, $\tx{e}_{n}:=(0,\cdots,0,1)$ and $x_{0}=0$, recall that $\overline{x_{1}x_{0}}=\{-t\tx{e}_{n}, \ t\in [0,\tx{l}]\}$ and introduce the function $$v(x):=\frac{\snr{x-x_{1}}-\snr{x-x_{0}}-\tx{l} }{2}=\frac{\snr{x+\tx{l}\tx{e}_{n}}-\snr{x}-\tx{l} }{2}.$$ 
Note that, for every
$x\in\mathbb R^n\setminus\{x_0,x_1\}$,
\eqn{formulac}
$$
Dv(x)=\frac{1}{2}\left(\frac{x+\tx{l}\tx{e}_{n}}{\snr{x+\tx{l}\tx{e}_{n}}}-\frac{x}{\snr{x}}\right)
$$
and, therefore
\eqn{formulac2}
$$
\snr{Dv(x)}^2
=\frac14\left|
\frac{x+\tx{l}\tx{e}_n}{\snr{x+\tx{l}\tx{e}_n}}
-\frac{x}{\snr{x}}
\right|^2=
\frac12\left(
1-\frac{\langle x+\tx{l}\tx{e}_n,x\rangle}
{\snr{x+\tx{l}\tx{e}_n}\snr{x}}\right)=
\frac{
\tx{l}^2-(\snr{x}-\snr{x+\tx{l}\tx{e}_n})^2}{4\snr{x}\snr{x+\tx{l}\tx{e}_n}},
$$
so that $v\in W^{1,\infty}_{\loc}(\mathbb{R}^{n})$ and $\nr{Dv}_{L^{\infty}(\mathbb{R}^{n})}\le 1$.
Specifically, for every $x\in\mathbb R^n\setminus\{x_0,x_1\}$,
\eqn{aa.0}
$$
\snr{Dv(x)}=1 \mbox{ and } \mathcal{h}_{v}^{-1}(x)=0
\mbox{ iff } x\in (x_{1}x_{0}).
$$
In fact, for every such $x$, we have
\eqn{aa.1}
$$
\begin{cases}
\displaystyle
\ \snr{Dv(x)}^{2}=1 \ \Longleftrightarrow \ \frac{\langle x+\tx{l}\tx{e}_{n},x\rangle}{\snr{x+\tx{l}\tx{e}_{n}}\snr{x}}=-1\vspace{2mm}\\
\displaystyle
\ \mathcal{h}_{v}^{-2}(x)=\frac{1}{4}\left(\frac{(\snr{x}+\snr{x+\tx{l}\tx{e}_{n}}-\tx{l})(\snr{x}+\snr{x+\tx{l}\tx{e}_{n}}+\tx{l})}{\snr{x}\snr{x+\tx{l}\tx{e}_{n}}}\right)=0\ \Longleftrightarrow \ \snr{x}+\snr{x+\tx{l}\tx{e}_{n}}=\tx{l}.
\end{cases}
$$
Moreover,
\eqn{aa.6}
$$
\begin{cases}
\displaystyle
\ v(-t_{2}\tx{e}_{n})-v(-t_{1}\tx{e}_{n})=-\snr{t_{2}-t_{1}}\qquad 0\le t_{1}\le t_{2}\le \tx{l}\vspace{1.5mm}\\
\displaystyle
\ v(-t\tx{e}_{n})=-\tx{l}\ \ \mbox{if}\ \ t\ge \tx{l}\quad \mbox{and}\quad v(-t\tx{e}_{n})=0\ \ \mbox{if}\ \ t\le 0,
\end{cases}
$$
so $\overline{x_{1}x_{0}}$ is a compact light ray for $v$. Set
\eqn{primadiv}
$$
\mu_v(x):=
\begin{cases}
-\diver(\mathcal h_vDv)(x) 
& x\notin\overline{x_1x_0} \\[2pt]
0 
& x\in\overline{x_1x_0}\,.
\end{cases}
$$
For $x\notin\overline{x_1x_0}$, a lengthy but straightforward computation yields
\begin{flalign}
\displaystyle \mu_{v}(x)&=-\langle D\mathcal{h}_{v}(x),Dv\rangle-\mathcal{h}_{v}\Delta v\nonumber \\
\displaystyle &=-\frac{\mathcal{h}_{v}}{4}\left(\frac{1}{\snr{x+\tx{l}\tx{e}_{n}}}-\frac{1}{\snr{x}}\right)\left(1-\frac{\langle x+\tx{l}\tx{e}_{n},x\rangle}{\snr{x}\snr{x+\tx{l}\tx{e}_{n}}}\right)\nonumber \\
\displaystyle &\quad -\mathcal{h}_{v}\left(\frac{n-1}{2}\right)\left(\frac{1}{\snr{x+\tx{l}\tx{e}_{n}}}-\frac{1}{\snr{x}}\right)\nonumber \\
&=-\frac{\mathcal{h}_{v}}{2}\left(\frac{1}{\snr{x+\tx{l}\tx{e}_{n}}}-\frac{1}{\snr{x}}\right)\left(\frac{1}{2}\left(1-\frac{\langle x+\tx{l}\tx{e}_{n},x\rangle}{\snr{x}\snr{x+\tx{l}\tx{e}_{n}}}\right)+n-1\right),
\label{espressionemu}
\end{flalign}
where we also used that 
\eqn{dvdv.00}
$$
\begin{cases}
\displaystyle
\ \frac{\snr{x}^{2}+\snr{x+\tx{l}\tx{e}_{n}}^{2}-\tx{l}^{2}}{2\snr{x}\snr{x+\tx{l}\tx{e}_{n}}}=\frac{\langle x,x+\tx{l}\tx{e}_{n}\rangle}{\snr{x}\snr{x+\tx{l}\tx{e}_{n}}}\vspace{1.5mm}\\
\displaystyle
\ D^{2}v(x)=\frac{1}{2}\left(\frac{1}{\snr{x+\tx{l}\tx{e}_{n}}}-\frac{1}{\snr{x}}\right)\mathbb{I}_{n}-\frac{1}{2}\left(\frac{(x+\tx{l}\tx{e}_{n})\otimes (x+\tx{l}\tx{e}_{n})}{\snr{x+\tx{l}\tx{e}_{n}}^{3}}-\frac{x\otimes x}{\snr{x}^{3}}\right)
\end{cases}
$$
and that
\eqn{andthat}
$$
\langle D^{2}v\,Dv,Dv\rangle=
\frac{1}{2}\left(
\frac{1}{\snr{x+\tx{l}\tx{e}_{n}}}-\frac{1}{\snr{x}}\right)\snr{Dv}^{2}\left(1-\snr{Dv}^{2}\right).
$$
Indeed, for this last identity, for $x\notin\{x_0,x_1\}$, formulas
\eqref{formulac}--\eqref{formulac2} give
\[
\begin{aligned}
\left\langle
\frac{x+\tx{l}\tx{e}_n}{\snr{x+\tx{l}\tx{e}_n}},
Dv
\right\rangle
&=
\frac12\left(
1-\frac{\langle x+\tx{l}\tx{e}_n,x\rangle}
{\snr{x+\tx{l}\tx{e}_n}\snr{x}}
\right)
=\snr{Dv}^2,
\\
\left\langle
\frac{x}{\snr{x}},Dv
\right\rangle
&=
\frac12\left(
\frac{\langle x+\tx{l}\tx{e}_n,x\rangle}
{\snr{x+\tx{l}\tx{e}_n}\snr{x}}-1
\right)
=-\snr{Dv}^2.
\end{aligned}
\]
Contracting the Hessian in \eqref{dvdv.00}$_2$ with $Dv$
and using these identities, we obtain
\[
\begin{aligned}
\langle D^2v\,Dv,Dv\rangle
&=
\frac12\left(
\frac1{\snr{x+\tx{l}\tx{e}_n}}-\frac1{\snr{x}}
\right)\snr{Dv}^2
-\frac12\left(
\frac{\langle x+\tx{l}\tx{e}_n,Dv\rangle^2}
{\snr{x+\tx{l}\tx{e}_n}^3}
-\frac{\langle x,Dv\rangle^2}{\snr{x}^3}
\right)
\\
&=
\frac12\left(
\frac1{\snr{x+\tx{l}\tx{e}_n}}-\frac1{\snr{x}}
\right)\snr{Dv}^2
-\frac12\left(
\frac1{\snr{x+\tx{l}\tx{e}_n}}-\frac1{\snr{x}}
\right)\snr{Dv}^4
\\
&=
\frac12\left(
\frac1{\snr{x+\tx{l}\tx{e}_n}}-\frac1{\snr{x}}
\right)
\snr{Dv}^2\left(1-\snr{Dv}^2\right)
\end{aligned}
\]
so that \rif{andthat} is proved. Fix $x\notin\overline{x_1x_0}$. The ``test''  point
$$
z:=-\frac{\tx{l}\snr{x}}
{\snr{x}+\snr{x+\tx{l}\tx{e}_{n}}}\tx{e}_{n}
$$
belongs to $\overline{x_1x_0}$, and a direct computation gives
$$
\dist(x,\overline{x_1x_0})^2
\le\snr{x-z}^2
=
\frac{
\snr{x}\snr{x+\tx{l}\tx{e}_{n}}
\left((\snr{x}+\snr{x+\tx{l}\tx{e}_{n}})^2-\tx{l}^2\right)}
{(\snr{x}+\snr{x+\tx{l}\tx{e}_{n}})^2}.
$$
By the triangle inequality,
$
\snr{x}+\snr{x+\tx{l}\tx{e}_{n}}
\ge
\snr{(x+\tx{l}\tx{e}_{n})-x}
=\tx{l},
$
with equality if and only if $x\in\overline{x_1x_0}$.
Thus
$
(\snr{x}+\snr{x+\tx{l}\tx{e}_{n}})^2-\tx{l}^2>0.
$
Therefore,
\begin{eqnarray}
\mathcal h_v(x)
&\stackrel{\eqref{aa.1}}{=}&
\frac{2\sqrt{\snr{x}\snr{x+\tx{l}\tx{e}_{n}}}}
{\sqrt{(\snr{x}+\snr{x+\tx{l}\tx{e}_{n}})^2-\tx{l}^2}} \nonumber \\
&\le&
\frac{2\snr{x}\snr{x+\tx{l}\tx{e}_{n}}}
{(\snr{x}+\snr{x+\tx{l}\tx{e}_{n}})
\dist(x,\overline{x_1x_0})}
\le
\frac{\snr{x}+\snr{x+\tx{l}\tx{e}_{n}}}
{2\dist(x,\overline{x_1x_0})}.\label{boost.segment.bound}
\end{eqnarray}
By \eqref{formulac2}, the expression \eqref{espressionemu} reads
$$
\mu_v(x)
=-\frac{\mathcal h_v}{2}
\left(
\frac1{\snr{x+\tx{l}\tx{e}_{n}}}-\frac1{\snr{x}}
\right)\left(n-1+\snr{Dv}^2\right).
$$
Since $\mathcal h_v>0$ and $n-1+\snr{Dv}^2\ge0$,
taking absolute values and estimating, we obtain
$$
\begin{aligned}
\snr{\mu_v(x)}
&=
\frac{\mathcal h_v}{2}
\left|\frac1{\snr{x+\tx{l}\tx{e}_{n}}}-\frac1{\snr{x}}\right|
\left(n-1+\snr{Dv}^2\right)\\
&\le
\frac{n\mathcal h_v}{2}
\frac{\snr{\snr{x}-\snr{x+\tx{l}\tx{e}_{n}}}}
{\snr{x}\snr{x+\tx{l}\tx{e}_{n}}}\\
&\le
\frac n2
\frac{2\snr{x}\snr{x+\tx{l}\tx{e}_{n}}}
{(\snr{x}+\snr{x+\tx{l}\tx{e}_{n}})
\dist(x,\overline{x_1x_0})}
\frac{\snr{\snr{x}-\snr{x+\tx{l}\tx{e}_{n}}}}
{\snr{x}\snr{x+\tx{l}\tx{e}_{n}}}\\
&=
\frac{n\snr{\snr{x}-\snr{x+\tx{l}\tx{e}_{n}}}}
{(\snr{x}+\snr{x+\tx{l}\tx{e}_{n}})
\dist(x,\overline{x_1x_0})}
\le
\frac{n}{\dist(x,\overline{x_1x_0})}.
\end{aligned}
$$
Here we used $\snr{Dv}\le1$, \eqref{boost.segment.bound}, and
$
\snr{\snr{x}-\snr{x+\tx{l}\tx{e}_{n}}}
\le\snr{x}+\snr{x+\tx{l}\tx{e}_{n}}.
$
Since $\dist(x,\overline{x_1x_0})\ge\snr{x}-\tx{l}\ge\tx{l}$ outside
$B_{2\tx{l}}(0)$, this proves
\eqn{aa.2}
$$
\mu_{v}\in L^{\infty}(\mathbb{R}^{n}\setminus B_{2\tx{l}}(0))\qquad \mbox{and}\qquad \snr{\mu_{v}(x)}\le\frac{n}{\dist(x,\overline{x_{1}x_{0}})} \ \ \mbox{in} \ \ B_{2\tx{l}}(0)\setminus\overline{x_1x_0}.
$$
Now notice that $$\left\{x\in B_{2\tx{l}}(0)\colon \snr{\mu_{v}(x)}>\lambda \right\}\stackrel{\eqref{aa.2}}{\subseteq} \left\{x\in B_{2\tx{l}}(0)\colon \frac{n}{\dist(x,\overline{x_{1}x_{0}})}>\lambda \right\}\qquad \mbox{for all} \ \ \lambda>0,$$
thus
\begin{flalign*}
 \sup_{\lambda>0}\lambda^{n-1}\left|\left\{x\in B_{2\tx{l}}(0)\colon \snr{\mu_{v}(x)}>\lambda\right\}\right |&\le  \sup_{\lambda>0}\lambda^{n-1}\left|\left\{x\in B_{2\tx{l}}(0)\colon \frac{n}{\dist(x,\overline{x_{1}x_{0}})}>\lambda\right\}\right|\nonumber \\
 &\le  \sup_{\lambda>0}\lambda^{n-1}\left|\left\{x\in B_{2\tx{l}}(0)\colon \frac{n}{\lambda}\ge\dist(x,\overline{x_{1}x_{0}})\right\}\right|\le c(n,\tx{l}),
\end{flalign*}
given that
$$
\left|\left\{x\in B_{2\tx{l}}(0)\colon \frac{n}{\lambda}\ge\dist(x,\overline{x_{1}x_{0}})\right\}\right|\le \frac{c(n,\tx{l})}{\lambda^{n-1}}
$$
and
$$
\mu_{v}\in L(n-1,\infty)(B_{2\tx{l}}(0))\cap L^{\infty}(\mathbb{R}^{n}\setminus B_{2\tx{l}}(0)), 
$$ 
and, therefore
$$
\mu_{v}\in L^{p}(B_{2\tx{l}}(0))\cap L^{\infty}(\mathbb{R}^{n}\setminus B_{2\tx{l}}(0)), \qquad \mbox{for all $1\le p<n-1$}.
$$ 

{\em Step 2: Cut-off}.
Let $\tx{L}>\max\{2\tx{l},1\}$ be a number to be fixed, and let
$\eta_{\tx{L}}\in C^{2}_{0}(\mathbb{R}^{n})$ satisfy
\eqn{may}
$$
\mathds{1}_{B_{\tx{L}}(0)}
\le \eta_{\tx{L}}
\le \mathds{1}_{B_{2\tx{L}}(0)},
\qquad
\snr{D\eta_{\tx{L}}}\lesssim_n \frac1{\tx{L}},
\qquad
\snr{D^{2}\eta_{\tx{L}}}\lesssim_n \frac{1}{\tx{L}^{2}}.
$$
Set $u:=\eta_{\tx{L}}v$. Since $v$ is globally Lipschitz and
$\eta_{\tx{L}}$ has compact support, we have
$u\in\mathcal{D}^{1,2}(\mathbb{R}^{n})$. Moreover, recall that $\nr{v}_{L^{\infty}(\mathbb{R}^{n})}\le\tx{l}$. For $x\in B_{2\tx{L}}(0)\setminus B_{\tx{L}}(0)$, and recalling \rif{formulac} and \rif{formulac2}, we find
$$
\snr{Dv(x)}
\le
\frac12\left|
\frac{x+\tx{l}\tx{e}_{n}}{\snr{x+\tx{l}\tx{e}_{n}}}
-\frac{x}{\snr{x+\tx{l}\tx{e}_{n}}}
\right|+\frac12\left|
\frac{x}{\snr{x+\tx{l}\tx{e}_{n}}}
-\frac{x}{\snr{x}}
\right|\le
\frac{\tx{l}}{\snr{x+\tx{l}\tx{e}_{n}}}
\le\frac{2\tx{l}}{\tx{L}},
$$
where we used
$\snr{x+\tx{l}\tx{e}_{n}}\ge\snr{x}-\tx{l}\ge\tx{L}/2$ and, by  \rif{may}
\eqn{a.42}
$$
\snr{Du(x)} \le
\eta_{\tx{L}}(x)\snr{Dv(x)}
+\snr{v(x)}\snr{D\eta_{\tx{L}}(x)}
\leq\frac{c_*(n)\tx{l}}{\tx{L}}
$$
for some constant $c_*(n)\geq 3$. 
We now fix $
\tx{L}:=\max\left\{2,\;2c_*(n)\tx{l}\right\}$, which is  
a choice ensuring that $\tx{L}>\max\{2\tx{l},1\}$, and that
the right-hand side of \eqref{a.42} is at most $1/2$.
Since $u=v$ on $B_{\tx{L}}(0)$ and $u=0$ outside
$B_{2\tx{L}}(0)$, we obtain
\eqn{a.43}
$$
\begin{cases}
\displaystyle
\nr{Du}_{L^{\infty}(B_{\tx{L}}(0))}
=
\nr{Dv}_{L^{\infty}(B_{\tx{L}}(0))}
\le1
\\[6pt]
\displaystyle
\nr{Du}_{L^{\infty}(\mathbb{R}^{n}\setminus B_{2\tx{L}}(0))}
=0
\\[6pt]
\displaystyle
\nr{Du}_{L^{\infty}(B_{2\tx{L}}(0)\setminus B_{\tx{L}}(0))}
\le1/2.
\end{cases}
$$
Thus $u\in\mathbb{X}(\mathbb{R}^{n})$.
Moreover, $u=v$ in a neighbourhood of
$\overline{x_{1}x_{0}}$, so \eqref{aa.6} shows that this
segment is a light segment for $u$ and cannot be extended
past either endpoint. We deduce that $\overline{x_{1}x_{0}}$ is a
compact light ray also for $u$. According to \rif{primadiv}, we define
$$
\mu_u(x):=
\begin{cases}
-\diver(\mathcal h_uDu)(x) 
& x\notin\overline{x_1x_0} \\[2pt]
0 
& x\in\overline{x_1x_0}.
\end{cases}
$$
Since $\eta_{\tx{L}}=1$ on
$B_{\tx{L}}(0)\supset B_{2\tx{l}}(0)$, we have
$\mu_u=\mu_v$ on
$B_{\tx{L}}(0)\setminus\overline{x_{1}x_{0}}$.
Also, $\mu_u=0$ outside $B_{2\tx{L}}(0)$.
It remains to estimate $\mu_u$ in
$B_{2\tx{L}}(0)\setminus B_{\tx{L}}(0)$.
On this annulus, \eqref{a.43} gives
$\mathcal h_u\le2/\sqrt3$, while \eqref{dvdv.00} and
$
\snr{x}\ge\tx{L}$, $\snr{x+\tx{l}\tx{e}_{n}}\ge\tx{L}/2$
give $\snr{D^2v(x)}\le c(n)/\tx{L}$. Since
$v\in C^\infty(\mathbb R^n\setminus\{x_0,x_1\})$
and $\eta_{\tx{L}}\in C^2_0(\mathbb R^n)$, we have
$u\in C^2(\mathbb R^n\setminus\{x_0,x_1\})$.
In particular, on
$B_{2\tx{L}}(0)\setminus B_{\tx{L}}(0)$,
the bound $\snr{Du}\le1/2$ ensures that
$\mathcal h_uDu$ is continuously differentiable.
Thus all derivatives, and the following computation, hold pointwise in the classical sense:
$$
\begin{aligned}
\snr{\mu_u}
&\le
\mathcal h_u^3\snr{Du}^2\snr{D^2u}
+\mathcal h_u\snr{\Delta u} \le c(n)\snr{D^2u}\\
&\le c(n)\left(
\snr{D^2v}
+2\snr{D\eta_{\tx{L}}}\snr{Dv}
+\snr{v}\snr{D^2\eta_{\tx{L}}}
\right)\lesssim_n
\frac1{\tx{L}}+\frac{\tx{l}}{\tx{L}^2}
\lesssim_n\frac{1}{\tx{L}}.
\end{aligned}
$$
Together with \eqref{aa.2}, this proves that
$$
\mu_u\in L(n-1,\infty)(\mathbb R^n),
\qquad \supp\, \mu_u\subset\overline{B_{2\tx{L}}(0)}.
$$
Since $n\ge3$ and $2n/(n+2)<n-1$, 
and being $\mu_u$ compactly supported, we have 
$
\mu_u\in L^{2n/(n+2)}(\mathbb R^n)
\subset\mathcal D^{1,2}(\mathbb R^n)^*.
$
We next verify that
$-\diver(\mathcal h_uDu)=\mu_u$
weakly holds in $\mathbb R^n$.
Note that the inequality in  \eqref{boost.segment.bound}, which holds for $x\notin\overline{x_1x_0}$,  involves a right-hand side which is locally integrable when $n\geq 3$. Since $u\equiv v$ on $B_{\tx{L}}(0)$ and
$\snr{Du}\le1/2$ outside $B_{\tx{L}}(0)$, we conclude that
$$
\mathcal h_uDu\in L^1(\mathbb R^n;\mathbb R^n),\qquad \supp\, \mathcal h_uDu \subset\overline{B_{2\tx{L}}(0)}.
$$
We consider the tube 
$
T_\varepsilon
:=
\left\{
x\in\mathbb R^n:
\dist(x,\overline{x_1x_0})<\varepsilon
\right\},
$ for $\varepsilon>0$ small enough to have 
$\overline{T_\varepsilon}\subset B_{\tx{L}}(0),$
and denote by $\nu_\varepsilon$ its outward unit normal.
The boundary $\partial T_\varepsilon$ consists of a
cylindrical lateral surface and two hemispherical caps
centred at $x_0$ and $x_1$.
We call the cylindrical part of $\partial T_\varepsilon$ the following set 
$
\left\{
x=(\bar x,x_n)\in\mathbb R^n:
\snr{\bar x}=\varepsilon,\quad
-\tx{l}<x_n<0
\right\}
$
and the outward unit normal is
$
\nu_\varepsilon(x)
=
\left( \bar x/\varepsilon,0\right).
$
Since $u=v$ in a neighbourhood of $\overline{T_\varepsilon}$,
the expression of $Dv$ in \rif{formulac} gives
$$
\mathcal h_u\langle Du, \nu_\varepsilon\rangle=
\frac{\varepsilon\mathcal h_v}{2}\left(
\frac{1}{\snr{x+\tx{l}\tx{e}_n}}-\frac{1}{\snr{x}}\right).
$$
On the cylindrical part we have
$\dist(x,\overline{x_1x_0})=\snr{\bar x}=\varepsilon$.
Using the preceding identity and \eqref{boost.segment.bound}, we obtain
\eqn{cyly}
$$
\begin{aligned}
\snr{\mathcal h_u\langle Du,\nu_\varepsilon\rangle}
&=
\frac{\varepsilon\mathcal h_v
\snr{\snr{x}-\snr{x+\tx{l}\tx{e}_{n}}}}
{2\snr{x}\snr{x+\tx{l}\tx{e}_{n}}}\\
&\le
\frac{\varepsilon
\snr{\snr{x}-\snr{x+\tx{l}\tx{e}_{n}}}}
{2\snr{x}\snr{x+\tx{l}\tx{e}_{n}}}
\frac{2\snr{x}\snr{x+\tx{l}\tx{e}_{n}}}
{(\snr{x}+\snr{x+\tx{l}\tx{e}_{n}})\varepsilon} =
\frac{\snr{\snr{x}-\snr{x+\tx{l}\tx{e}_{n}}}}
{\snr{x}+\snr{x+\tx{l}\tx{e}_{n}}}
\le1.
\end{aligned}
$$
On the two caps
$$
\left\{
x\in\mathbb R^n:
\snr{x}=\varepsilon,\quad x_n\ge0
\right\}
\quad\mbox{and}\quad
\left\{
x\in\mathbb R^n:
\snr{x+\tx{l}\tx{e}_n}=\varepsilon,\quad
x_n\le-\tx{l}
\right\},$$
we have $
\langle x+\tx{l}\tx{e}_n,x\rangle
=
\snr{\bar x}^2+x_n(x_n+\tx{l})
\ge0.
$
By \eqref{formulac2}, this gives $\snr{Dv}^2\le1/2$ and hence
$\snr{\mathcal h_uDu}=\snr{\mathcal h_vDv}\le1$.
Consequently, also recalling \rif{cyly}, for every $\varphi\in C^\infty_0(\mathbb R^n)$,
$$
\left|
\int_{\partial T_\varepsilon}
\varphi\,\mathcal h_u\langle Du,\nu_\varepsilon\rangle 
\,d\mathcal H^{n-1}\right|\le
c(n)\nr{\varphi}_{L^\infty}
\left(\tx{l}\varepsilon^{n-2}+\varepsilon^{n-1}\right)
\stackrel{\eps\to 0}{\longrightarrow}0.
$$
Note that here we are assuming $n\geq 3$. Integrating by parts outside $T_\varepsilon$ and letting
$\varepsilon\downarrow0$, we obtain
$$
\int_{\mathbb R^n}
\mathcal h_u\langle Du, D\varphi\rangle\dx
=
\int_{\mathbb R^n}\mu_u\varphi\dx\qquad \mbox{for every }\varphi\in C^\infty_0(\mathbb R^n).
$$
Thus $u$ is a global weak solution with datum $\mu_u$.
Since $\mathcal h_uDu$ and $\mu_u$ are integrable and
compactly supported, a standard cutoff and mollification
argument allows us to test the weak formulation with
$w-u$ for every $w\in\mathbb X(\mathbb R^n)$.
Convexity of the integrand then shows that $u$
is a global minimizer of $\mathcal E_{\mu_u}$,
and, as usual, strict convexity yields uniqueness. Since the graph of $u$ admits light segments, by Theorem \ref{light.t} we can exclude that 
$\mu_u\in L(n-1,s)(\mathbb R^n)$ for any $0<s<\infty$.
The proof is complete.


\section{A fat $\mathcal K_u$ and Corollary \ref{brutto} }\label{esempione}

\subsection{A one point construction}\label{onepoint} Starting from the example in Theorem \ref{ex.t}, we show that
$\mathcal S_u$ is in general not necessarily closed and hence may be
strictly contained in $\mathcal K_u$.
This phenomenon can only occur outside the regime covered by
Theorem \ref{al.t}.
Recall that
$\mathcal K_u=\mathcal S_u\dot\cup\mathcal L_u$
by Proposition \ref{lem:structure-Kmu}. Let $n\ge3$, and let
$u\in\mathcal D^{1,2}(\mathbb R^n)$ and
$\mu_u\in L(n-1,\infty)$ be the function and charge constructed in the proof
of Theorem \ref{ex.t}. Recall that
$\overline{x_1x_0}$ is the unique compact light ray of $u$,
see  \eqref{aa.1}, \eqref{aa.6} and \eqref{a.43}, and that $u$, $\mu_u$, and
$\partial H(Du)$ are compactly supported. Fix $\tx{r}>0$ so large that $\supp\, u, \, \supp\, \mu_{u},\,\supp\, \partial H(Du)\Subset B_{\tx{r}}(0)$, and, for $i\in \mathbb{N}\setminus \{0\}$, set $y_{i}:=2^{-i}\tx{e}_{1}$, $\delta_{i}:=2^{-3i}/(10\tx{r})$, $\rr_{i}:=\tx{r}\delta_{i}$, and observe that
\eqn{ww.60}
$$
B_{\rr_{i}}(y_{i})\cap B_{\rr_{j}}(y_{j})=\varnothing \ \Longleftrightarrow \ i\not =j.
$$
A direct computation indeed yields that $\snr{y_{i}-y_{j}}>2\max\{\rr_{i},\rr_{j}\}$ whenever $i\not =j$ - specifically, $\rr_{i}<\snr{y_{i}}$, so $0\not \in B_{\rr_{i}}(y_{i})$ for all $i\in \mathbb{N}\setminus \{0\}$. We further define the scaled functions
\eqn{riscalate}
$$
\mathbb{R}^{n}\ni x\mapsto u_{i}(x):=\delta_{i}u\left(\frac{x-y_{i}}{\delta_{i}}\right).
$$
By construction we have that $\supp\, u_{i}\Subset B_{\rr_{i}}(y_{i})$ and \eqref{ww.60} assures that $\supp\, u_{i} \cap \supp\, u_{j} =\varnothing$ for all $i\not =j$, so letting
$$
\mathbb{R}^{n}\ni x\mapsto \mathcal{u}_{k}(x):=\sum_{i=1}^{k}u_{i}(x),
$$
we deduce that $\nr{D\mathcal{u}_{k}}_{L^{\infty}(\mathbb{R}^{n})}\le 1$ as any $x\in \mathbb{R}^{n}$ belongs to at most one of the balls $B_{\rr_{i}}(y_{i})$. Moreover, $\nr{u_{i}}_{L^{\infty}(\mathbb{R}^{n})}=\delta_{i}\nr{u}_{L^{\infty}(B_{\tx{r}}(0))}$ and $\sum_{i=1}^{\infty}\delta_{i}<\infty$, thus 
\eqn{mima}
$$\mathcal{u}_{k}\to \mathcal{u}:=\sum_{i=1}^{\infty}u_{i}$$ uniformly in $\mathbb{R}^{n}$ and $\nr{D\mathcal{u}}_{L^{\infty}(\mathbb{R}^{n})}\le 1$. In particular,
$$
\supp\,\mathcal u=\{0\}\cup\bigcup_{i=1}^{\infty}\supp\,u_i=\{0\}\cup\bigcup_{i=1}^{\infty}\left(y_i+\delta_i\supp\,u\right)
$$
and, as $\nr{Du_{i}}_{L^{2}(\mathbb{R}^{n})}=\delta_{i}^{n/2}\nr{Du}_{L^{2}(\mathbb{R}^{n})}$ we obtain
$$
\nr{D\mathcal{u}}_{L^{2}(\mathbb{R}^{n})}=\left(\sum_{i=1}^{\infty}\nr{Du_{i}}_{L^{2}(\mathbb{R}^{n})}^{2}\right)^{1/2}=\left(\sum_{i=1}^{\infty}\delta_{i}^{n}\right)^{1/2}\nr{Du}_{L^{2}(\mathbb{R}^{n})}<\infty \ \Longrightarrow \ \mathcal{u}\in \mathbb{X}(\mathbb{R}^{n}).
$$
Recall that the original function $u$ has a unique compact light ray
$
\ell=\overline{x_1x_0}.
$
For every $i\ge1$, set
$
y_{0;i}:=y_i+\delta_i x_0=y_i$ and $
y_{1;i}:=y_i+\delta_i x_1
=y_i-\tx l\delta_i\tx e_n$. 
By scaling, the unique compact light ray of $u_i$ is
$
\ell_i:=y_i+\delta_i\ell
=\overline{y_{1;i}y_{0;i}}.
$
By contradiction, assume that $j>i$, and that $\overline{z_{j}z_{i}}$ is a segment with $z_{j}\in \supp\, u_{j}$, $z_{i}\in \supp\, u_{i}$, and, keeping \eqref{ww.60} in mind, $\snr{\mathcal{u}(z_{j})-\mathcal{u}(z_{i})}=\snr{u_{j}(z_{j})-u_{i}(z_{i})}=\snr{z_{j}-z_{i}}$ holds. Name $z_{j;\partial}\in \supp\, u_{j}$ as the last point in $\supp\, u_{j}$ encountered by $\overline{z_{j}z_{i}}$ when moving from $z_{j}$ to $z_{i}$, and $z_{i;\partial}\in\supp\, u_{i}$ as the first point in $\supp\, u_{i}$ subsequently encountered, so that $\snr{z_{j;\partial}-z_{i;\partial}}\ge \dist(\supp\, u_{j},\supp\, u_{i})>0$ and $z_{j},\ z_{j;\partial}, \ z_{i;\partial}, \ z_{i}$ are collinear. Since $\left.u_{k}\right|_{\partial \supp\, u_{k}}=0$, $k\in \{i,j\}$, we have that $\snr{u_{k}(x)}\le \dist(x,\partial\supp\, u_{k})$, $k\in \{i,j\}$, therefore
\begin{flalign*}
\snr{z_{j}-z_{i}}&= \snr{\mathcal{u}(z_{j})-\mathcal{u}(z_{i})}\stackrel{\eqref{ww.60}}{=}\snr{u_{j}(z_{j})-u_{i}(z_{i})}\nonumber \\
&\le  \snr{u_{j}(z_{j})}+\snr{u_{i}(z_{i})}\le \dist(z_{j},\partial\supp\, u_{j})+\dist(z_{i},\partial \supp\, u_{i})\nonumber \\
&\le \snr{z_{j}-z_{j;\partial}}+\snr{z_{i}-z_{i;\partial}} \le \snr{z_{i}-z_{j}}-\dist(\supp\, u_{i},\supp\, u_{j})\stackrel{\eqref{ww.60}}{<}\snr{z_{i}-z_{j}},
\end{flalign*}
a contradiction. Therefore no light segment of $\mathcal u$ can meet
two distinct supports $\supp\,u_i$ and $\supp\,u_j$.
Every light segment must meet at least one such support,
since $\mathcal u$ vanishes outside their union.
Hence $\mathcal u=u_i$ along that segment for some $i$,
and the segment is therefore contained in $\ell_i$.
Conversely, $\mathcal u=u_i$ on $\ell_i$, so the light
rays of $\mathcal u$ are precisely the segments $\ell_i$.
Consequently,
$$
\mathcal S_{\mathcal u}=
\bigcup_{i=1}^{\infty}\ell_i=
\bigcup_{i=1}^{\infty}\overline{y_{1;i}y_{0;i}}.
$$
In particular, $\snr{\mathcal S_{\mathcal u}}=0$. Moreover, by disjointness of the supports and the
corresponding property of $u$, we have
$|D\mathcal u|<1$ almost everywhere in $\mathbb R^n$. In this way the functions
$$
\mu_{i}(x):=-\diver\, \partial H(Du_{i})(x)=\frac{1}{\delta_{i}}\mu_{u}\left(\frac{x-y_{i}}{\delta_{i}}\right),\qquad \quad \mu_{\mathcal{u}}(x):=\sum_{i=1}^{\infty}\mu_{i}(x)
$$
are defined almost everywhere. Let us show that $-\diver\, \partial H(D\mathcal{u})=\mu_{\mathcal{u}}$. In this respect, observe that
\eqn{ww.62}
$$
\nr{\partial H(Du_{i})}_{L^{1}(\mathbb{R}^{n})}=\delta_{i}^{n}\nr{\partial H(Du)}_{L^{1}(\mathbb{R}^{n})},\qquad  \nr{\mu_{i}}_{L^{1}(\mathbb{R}^{n})}=\delta_{i}^{n-1}\nr{\mu_{u}}_{L^{1}(\mathbb{R}^{n})},
$$
and, for any fixed $k\in \mathbb{N}\setminus \{0\}$, it holds that
\eqn{egoego}
$$
-\sum_{i=1}^{k}\diver\, \partial H(Du_{i})=\sum_{i=1}^{k}\mu_{i} \qquad\text{in }\mathcal D'(\mathbb R^n).
$$
To justify \eqref{egoego} in the sense of distributions,
fix $\varphi\in C_0^\infty(\mathbb R^n)$ and set
$\psi_i(z):=\varphi(y_i+\delta_i z)$.
Then $\psi_i\in C_0^\infty(\mathbb R^n)$ and
$D\psi_i(z)=\delta_i D\varphi(y_i+\delta_i z)$.
Since
$Du_i(x)=Du((x-y_i)/\delta_i)$ almost everywhere,
the change of variables $x=y_i+\delta_i z$ and the weak
equation satisfied by $u$ give
\begin{align*}
\int_{\mathbb R^n}
\langle\partial H(Du_i),D\varphi\rangle\,dx
&=\delta_i^n\int_{\mathbb R^n}
\langle\partial H(Du(z)),
D\varphi(y_i+\delta_i z)\rangle\,dz =\delta_i^{n-1}\int_{\mathbb R^n}
\langle\partial H(Du),D\psi_i\rangle\,dz\\
&=\delta_i^{n-1}\int_{\mathbb R^n}
\mu_u(z)\psi_i(z)\,dz =\int_{\mathbb R^n}
\frac{1}{\delta_i}
\mu_u\left(\frac{x-y_i}{\delta_i}\right)
\varphi(x)\,dx =\int_{\mathbb R^n}\mu_i\varphi\,dx.
\end{align*}
Summing these identities for $i=1,\ldots,k$, we obtain
\begin{equation*}
\sum_{i=1}^k\int_{\mathbb R^n}
\langle\partial H(Du_i),D\varphi\rangle\,dx
=
\int_{\mathbb R^n}
\big(\sum_{i=1}^k\mu_i\big)\varphi\,dx.
\end{equation*}
Since $\varphi$ is arbitrary, this proves \eqref{egoego}
in $\mathcal D'(\mathbb R^n)$. Since summing in \eqref{ww.62} leads to finite quantities, we deduce that 
\eqn{pino0}
$$
\sum_{i=1}^{k}\mu_{i}\stackrel{k\to \infty}{\longrightarrow} \mu_{\mathcal{u}}, \quad \mbox{and} \quad \sum_{i=1}^{k}\partial H(Du_{i})\stackrel{k\to \infty}{\longrightarrow} \partial H(D\mathcal{u}), \qquad \mbox{in $L^{1}(\mathbb{R}^{n})$}.
$$ 
This allows to pass to the limit in the weak formulations of \rif{egoego} thereby obtaining
\eqn{pino}
$$
-\diver\,\partial H(D\mathcal u)=\mu_{\mathcal u}
\qquad\text{in }\mathcal D'(\mathbb R^n).
$$
Indeed, for every $\varphi\in C_0^\infty(\mathbb R^n)$, the 
strong $L^1$ convergences in \rif{pino0} and \eqref{egoego} yield
\begin{equation*}
\begin{aligned}
\int_{\mathbb R^n}
\langle\partial H(D\mathcal u),D\varphi\rangle\,dx
&=
\lim_{k\to\infty}\sum_{i=1}^k
\int_{\mathbb R^n}
\langle\partial H(Du_i),D\varphi\rangle\,dx\\
&=
\lim_{k\to\infty}
\int_{\mathbb R^n}
\big(\sum_{i=1}^k\mu_i\big)\varphi\,dx =
\int_{\mathbb R^n}\mu_{\mathcal u}\varphi\,dx.
\end{aligned}
\end{equation*}
Since $\partial H(D\mathcal u)$ and $\mu_{\mathcal u}$
are integrable and compactly supported, the same
convexity argument used
in the proof of Theorem \ref{ex.t} shows that
$\mathcal u$ is a global minimizer of \eqref{bi.en}
with charge $\mu_{\mathcal u}$. 
Uniqueness follows from the strict convexity of $H$,
see \eqref{0.1.1}.
Notice also that 
\begin{eqnarray*}
\lambda^{n-1}\snr{\{x\in \mathbb{R}^{n}\colon \snr{\mu_{\mathcal{u}}(x)}>\lambda \}}&\stackrel{\eqref{ww.60}}{=}&\lambda^{n-1}\sum_{i=1}^{\infty}\snr{\{x\in \mathbb{R}^{n}\colon \snr{\mu_{i}(x)}>\lambda\}}\nonumber \\
&=&\sum_{i=1}^{\infty}(\delta_{i}\lambda )^{n-1}\snr{\{x\in \mathbb{R}^{n}\colon \snr{\mu_{u}(x)}>\delta_{i}\lambda\}}\delta_{i}\nonumber \\
&\le&
[\mu_u]_{n-1,\infty;\mathbb R^n}^{n-1}
\sum_{i=1}^{\infty}\delta_i\\
&= &
\frac{1}{70\tx{r}}
[\mu_u]_{n-1,\infty;\mathbb R^n}^{n-1}.
\end{eqnarray*}
We conclude that  $\mu_{\mathcal{u}}\in L(n-1,\infty)$, but $\mu_{\mathcal{u}}\not \in L_{\loc}(n-1,s)$ for any $0<s<\infty$, as via \eqref{ww.60} and scaling, it would be $\mu_{u}\in L_{\loc}(n-1,s)$ in contradiction with Theorem \ref{ex.t}. We have therefore reproduced the same situation as in Theorem
\ref{ex.t}, but now
$$
\mathcal S_{\mathcal u}=
\bigcup_{i=1}^{\infty}\ell_i=
\bigcup_{i=1}^{\infty}\overline{y_{1;i}y_{0;i}}.
$$
Moreover, for every $z\in\ell_i$,
$
|z|\le|y_{1;i}|+|y_{0;i}|\le2|y_i|+\tx l\delta_i\to 0$ as 
$i\to\infty$.
Hence
$
\overline{\mathcal S_{\mathcal u}} = \mathcal S_{\mathcal u}\cup\{0\}$ and $\mathcal K_{\mathcal u} =\mathcal S_{\mathcal u}\cup\{0\}.
$
On the other hand,
$
\dist(0,\ell_i)
\le |y_{0;i}| = |y_i| =1/2^{i}
\to 0$ 
and 
$
\diam\, (\ell_i)=\tx l\,\delta_i \to 0.
$
Therefore, by the definition in \rif{language3},
$
0\in\mathcal L_{\mathcal u}.
$
Since
$
\mathcal K_{\mathcal u}\setminus\mathcal S_{\mathcal u}=\{0\},
$
Proposition \ref{lem:structure-Kmu} yields
$
\mathcal L_{\mathcal u}=\{0\}.
$
Let us finally point out that, as all maps from Theorem \ref{ex.t} have compact support, all the procedure is essentially local, and the preceding construction can be translated and scaled in such a way it remains valid on bounded domains.

\begin{figure}[ht]
\centering
\begin{tikzpicture}[
x=1.5cm,
y=1.5cm,
font=\small,
light ray/.style={
draw=red!75!black,
line width=1.2pt,
line cap=round
},
support ball/.style={
draw=blue!60!black,
fill=blue!5,
line width=.5pt
}
]

\draw[->,gray!65]
(-.25,0)--(7.8,0)
node[right] {$\tx e_1$};

\foreach \xx/\rr in {
6.4/.85,
3.2/.42,
1.6/.20,
.8/.092,
.4/.041,
.2/.017,
.1/.0065
}{
\draw[support ball]
(\xx,0) circle[radius=\rr];

\draw[light ray]
(\xx,0)--(\xx,{-0.75*\rr});

\pgfmathsetmacro{\pointsize}{min(.018,.18*\rr)}
\fill
(\xx,0) circle[radius=\pointsize];
}

\fill (0,0) circle[radius=.025];
\node[below left] at (0,0) {$0$};

\node[above right] at (6.4,0) {$y_1$};
\node[above right] at (3.2,0) {$y_2$};
\node[above] at (1.6,.20) {$y_3$};
\node[above] at (.8,.092) {$y_4$};

\node[right,text=red!75!black]
at (6.44,-.36) {$\ell_1$};

\node[right,text=red!75!black]
at (3.24,-.21) {$\ell_2$};

\node[below,text=red!75!black]
at (1.6,-.21) {$\ell_3$};

\node[above,text=blue!60!black]
at (6.4,.90) {$B_{\tx r\delta_1}(y_1)$};

\node[above,text=blue!60!black]
at (3.2,.47) {$B_{\tx r\delta_2}(y_2)$};

\draw[->,gray!75]
(2,-.85)--(.35,-.85)
node[midway,below,text=black]
{$y_i\to0,\qquad \diam\,(\ell_i)\to0$};

\end{tikzpicture}
\caption{
The light rays $\ell_i=y_i+\delta_i\ell$ lie inside the
pairwise disjoint balls $B_{\tx r\delta_i}(y_i)$ and accumulate
only at the origin. Thus
$\mathcal S_{\mathcal u}=\bigcup_i\ell_i$,
$\mathcal K_{\mathcal u}=\mathcal S_{\mathcal u}\,\dot\cup\,\{0\}$, and
$\mathcal L_{\mathcal u}=\{0\}$.}
\label{fig:onepoint}
\end{figure}
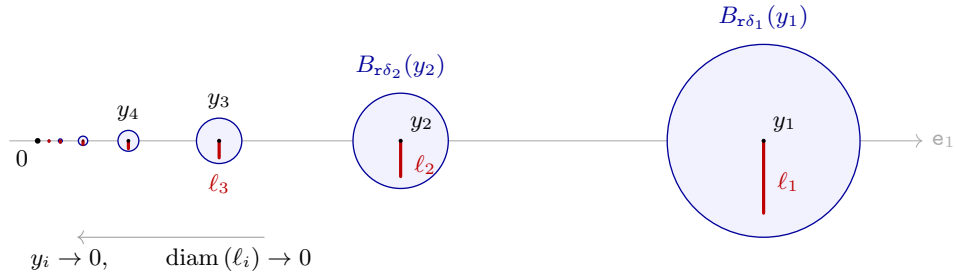

\medskip
\subsection{Extension to compact sets and proof of Corollary \ref{brutto}}
The preceding construction extends to any non-empty compact set
$C\Subset B_1(0)$ with empty interior. The procedure is a rather standard one in Real Analysis (see Remark \ref{standre}). 
Indeed, take a Whitney decomposition of
$\mathbb R^n\setminus C$ and retain the cubes $Q_i$
whose centres $y_i$ lie in $B_1(0)$; see
\cite[Chapter~I, Section~3, Theorem~3]{ste70}.
The Whitney construction made there implies that the set of accumulation
points of these centres is exactly $C$.
Choose $0<\delta_i\le2^{-i}$ sufficiently small that
$
\overline{B_{\tx r\delta_i}(y_i)}
\Subset
\operatorname{int}Q_i\cap B_2(0).
$
These balls are pairwise disjoint since the Whitney cubes $Q_i$ have pairwise disjoint interiors.
Mimicking the construction in \rif{riscalate} and \rif{mima}, this time set
$$
\mathcal u_C(x)
:=
\sum_{i=1}^{\infty}
\delta_i u\left(\frac{x-y_i}{\delta_i}\right).
$$
By disjointness of the supports and $\sum_i\delta_i^n<\infty$,
the partial sums converge strongly in
$W^{1,2}(\mathbb R^n)$. In particular,
$D\mathcal u_C=0$ almost everywhere on $C$.
The $L^1$-limit of the corresponding fluxes is therefore
$\partial H(D\mathcal u_C)$. 
All the  estimates and arguments in Section \ref{onepoint} apply unchanged,
since they depend only on the disjointness of supports of the rescaled functions appearing in the sum in the above display; moreover $\delta_i\le2^{-i}$ guarantees all the
required summability properties. 
Thus $\mathcal u_C$ is a compactly supported global
minimizer and weak solution, with compactly supported
charge
$
\mu_{\mathcal u_C}\in L(n-1,\infty)(\mathbb R^n).
$
Its light rays are precisely $\ell_i=y_i+\delta_i\ell$.
As in Section \ref{onepoint}, since $y_i\in\ell_i$ and
$\diam\,(\ell_i)\to0$, sequences of points on distinct rays
have precisely the same limits as the corresponding centres,
whose accumulation set is $C$.
Since each $\ell_i$ is closed and disjoint from $C$, we obtain
$$
\mathcal S_{\mathcal u_C}=\bigcup_{i=1}^{\infty}\ell_i,
\qquad
\mathcal K_{\mathcal u_C}
=
\mathcal S_{\mathcal u_C}\,\dot\cup\,C,
\qquad
\mathcal L_{\mathcal u_C}=C.
$$
Choosing $C$ with $|C|>0$ gives
$
|\mathcal K_{\mathcal u_C}|>0
$
and
$
\dim_{\mathcal H}\mathcal K_{\mathcal u_C}=n.
$
Moreover, compact support of the charge implies
$
\mu_{\mathcal u_C}\in L^q(\mathbb R^n)
$
for every $1\le q<n-1$.
Restricting to $B_3(0)$ yields a Dirichlet minimizer with
zero boundary datum.
Consequently, Conjecture 4 in \cite{bimm24} is false
for $n\ge4$ and $2\le q<n-1$, and Question 5 has a
negative answer for every $n\ge3$.

\begin{remark}\label{standre}{\em
The one above is a typical example of replication of singularities. Once identified the basic, one point phenomenon, one can replicate it using standard tools from Real Analysis, distributing the singularities on suitable bad sets and their accumulation points. See for instance \cite[Section~6]{jms89},
\cite[Section~3]{fmm04}, and \cite{gwz17} for various types of constructions, also involving fractals. In this situation the extension is rather straightforward as the initial construction is already based on the replication and the scaling of the same profile on balls with disjoint supports.  
}
\end{remark}

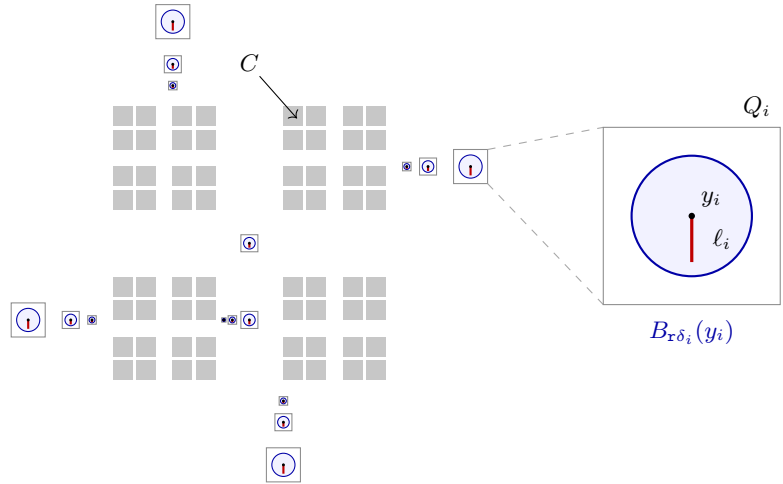
\begin{figure}[ht]
\centering
\begin{tikzpicture}[
x=0.9cm,y=0.9cm,
font=\small,
whitney/.style={draw=black!45,thin},
ball/.style={draw=blue!65!black,fill=blue!5,thin},
ray/.style={draw=red!75!black,line width=0.8pt}
]

\def\cantorintervals{
0/0.28125,
0.34375/0.625,
0.875/1.15625,
1.21875/1.5,
2.5/2.78125,
2.84375/3.125,
3.375/3.65625,
3.71875/4
}

\foreach \a/\b in \cantorintervals {
  \foreach \c/\d in \cantorintervals {
    \fill[black!22] (\a,\c) rectangle (\b,\d);
  }
}

\node at (2,4.65) {$C$};
\draw[->,thin] (2.15,4.45) -- (2.7,3.83);

\foreach \xx/\yy/\hh in {
0.875/5.25/0.25,
0.875/4.625/0.125,
0.875/4.3125/0.0625,
5.25/3.125/0.25,
4.625/3.125/0.125,
4.3125/3.125/0.0625,
-1.25/0.875/0.25,
-0.625/0.875/0.125,
-0.3125/0.875/0.0625,
2.5/-1.25/0.25,
2.5/-0.625/0.125,
2.5/-0.3125/0.0625,
2/0.875/0.125,
1.75/0.875/0.0625,
1.625/0.875/0.03125,
2/2/0.125
}{
  \begin{scope}[shift={(\xx,\yy)}]
    \draw[whitney]
      (-\hh,-\hh) rectangle (\hh,\hh);
    \draw[ball]
      (0,0) circle[radius={0.68*\hh}];
    \draw[ray]
      (0,0) -- (0,{-0.52*\hh});
    \fill (0,0) circle[radius=0.55pt];
  \end{scope}
}

\draw[black!35,dashed,thin]
  (5.5,3.375) -- (7.2,3.7);
\draw[black!35,dashed,thin]
  (5.5,2.875) -- (7.2,1.1);

\begin{scope}[shift={(8.5,2.4)}]
  \draw[whitney]
    (-1.3,-1.3) rectangle (1.3,1.3);
  \node[above left] at (1.3,1.3) {$Q_i$};

  \draw[ball,line width=0.8pt]
    (0,0) circle[radius=0.884];

  \draw[ray,line width=1.2pt]
    (0,0) -- (0,-0.676)
    node[midway,right=4pt] {$\ell_i$};

  \fill (0,0) circle[radius=1.2pt];
  \node[above right] at (0,0) {$y_i$};

  \node[text=blue!65!black] at (0,-1.7)
    {$B_{\tx r\delta_i}(y_i)$};
\end{scope}

\end{tikzpicture}
\caption{Schematic representation of the construction. The grey set refers to 
a finite approximation of $C$; only selected Whitney cubes
of $\mathbb R^n\setminus C$ are shown.
Each light ray $\ell_i$, with endpoint $y_i$, lies inside
$B_{\tx r\delta_i}(y_i)$, whose closure is contained in
$\operatorname{int}Q_i$.}\label{fig.1}
\end{figure}

\section{Appendix: revisiting and extending basic results}\label{appe}
Here we revisit and complement several results from the existing literature that have been used in the preceding sections. We also slightly extend some higher-dimensional results and formulate them in a form suited to the present treatment. We begin by addressing the existence of minimizers of \eqref{bi.en} for locally bounded charges. The case $n\geq 3$ was treated in \cite{bdp16}, from which we borrow most of the argument of the following Proposition \ref{exex}, extending it here to dimension two.

\begin{proposition}[Existence and Euler--Lagrange equation]\label{exex}
Let $\mu\in \DD^*\cap L^\infty_{\loc}(\mathbb R^n)$ and, if $n=2$, assume in addition that 
\eqn{addition}
$$
\int_{\mathbb R^2}\mu\dx=0.
$$
There exists a unique minimizer $u\in\mathbb X(\mathbb R^n)$ of the functional
$\mathcal E_\mu$ such that
\eqn{l1l1}
$$
\int_{\mathbb R^n}
\frac{|Du|^2}{\sqrt{1-|Du|^2}}\dx
\le
\langle\mu,u\rangle.
$$
Moreover,
$
u\in C^1(\mathbb R^n)\cap W^{2,2}_{\loc}(\mathbb R^n)
$
and $u$ is a strictly spacelike strong solution to
\eqref{bi}. In particular,
\eqn{mm.9}
$$
\int_{\mathbb R^n}
\langle\partial H(Du),Dw\rangle\dx
=
\int_{\mathbb R^n}\mu w\dx
$$
for every compactly supported
$w\in W^{1,\infty}(\mathbb R^n)$.
If $\mu\neq0$ in $\mathcal D^{1,2}(\mathbb R^n)^*$,
then $u\not\equiv0$.
If, in addition, $\mu\in C^k(\mathbb R^n)$
for an integer $k\ge1$, then
$u\in C^{k+1}(\mathbb R^n)$ and \eqref{bi}
holds classically.
\end{proposition}
\begin{proof} 
Using \rif{elema}, for every \(w\in\mathbb X(\mathbb R^n)\),
\begin{align}
\mathcal E_\mu(w)
&=
\int_{\mathbb R^n}H(Dw)\dx-\langle\mu,w\rangle
\ge
\frac12\|Dw\|_{L^2(\mathbb R^n)}^2
-
\|\mu\|_{\DDD}\|w\|_{\DD}
\nonumber\\
&\ge
\frac14\|w\|_{\DD}^2
-
\|\mu\|_{\DDD}^2.
\label{mm.1}
\end{align}
Since \(0\in\mathbb X(\mathbb R^n)\) and
\(\mathcal E_\mu(0)=0\), it follows that
$
-\infty
<
m_\mu
:=
\inf_{w\in\mathbb X(\mathbb R^n)}
\mathcal E_\mu(w)
\le0.
$
Let
$
\{u_j\}_{j\in\mathbb N}
\subset\mathbb X(\mathbb R^n)
$
be a minimizing sequence so that \eqref{mm.1} shows that
\(\{u_j\}\) is bounded in \(\DD\).
Therefore, after passing to a subsequence, there exists
\(u\in\DD\) such that
$
u_j\rightharpoonup u$
weakly in $\DD.$
The set \(\mathbb X(\mathbb R^n)\) is convex and strongly closed in
\(\DD\), and, consequently, it is also weakly closed so that 
$
u\in\mathbb X(\mathbb R^n).
$
Since \(H\) is convex and lower semicontinuous, the integral functional is weakly lower semicontinuous, while the linear term is weakly continuous. Therefore 
$
\mathcal E_\mu(u)
\le
\liminf_{j\to\infty}\mathcal E_\mu(u_j)
=
m_\mu,
$
and hence \(u\) is a minimizer. As for uniqueness, let $v \in \DD$ be another minimizer.  Then \((u+v)/2\in\mathbb X(\mathbb R^n)\) and the
strict convexity of \(H\) implies 
$
Du=Dv$ a.e. in $\mathbb R^n$. Since \(w\mapsto Dw\) is injective on \(\DD\), we conclude
that \(u=v\). Notice that, in dimension two, this last conclusion
follows from the condition 
$(u-v)_{B_1(0)}=0. 
$
The validity of \eqref{l1l1} for $n=2$ can be retrieved as in \cite[Proposition 2.7]{bdp16}, and we briefly report the argument for completeness. For every \(t\in[0,1]\), one has
$
tu\in\mathbb X(\mathbb R^n).
$
If 
$
\psi(t):=\mathcal E_\mu(tu)
$, then the minimality of $u$ implies 
$
\psi(1)\le\psi(t)
$
for every $t\in[0,1].$ 
The function \(\psi\) is convex and, for \(t\in(0,1)\),
\[
\psi'(t)
=
\int_{\mathbb R^n}
\frac{t|Du|^2}{\sqrt{1-t^2|Du|^2}}\dx
-
\langle\mu,u\rangle.
\]
By convexity,
\[
\psi'(t)
\le
\frac{\psi(1)-\psi(t)}{1-t}
\le0 \Longrightarrow 
\int_{\mathbb R^n}
\frac{t|Du|^2}{\sqrt{1-t^2|Du|^2}}\dx
\le
\langle\mu,u\rangle.
\]
Letting \(t\uparrow1\) and using the monotone convergence theorem, we
obtain \rif{l1l1}. 
Fix $R>1$ and set $B_R:=B_R(0)$.
By Remark \ref{globale.locale}, $u$ minimizes the
corresponding Dirichlet functional in $B_R$, namely
$
\mathcal E_\mu(u;B_R)\le\mathcal E_\mu(w;B_R)
$
 for every $w\in\mathfrak X(u;B_R)$.
Denote 
$$
\mathcal S_{u,\partial}(B_R)
:=
\bigcup
\left\{\overline{xy}:
x,y\in\partial B_R,\quad
x\neq y,\quad
|u(x)-u(y)|=|x-y|
\right\}.
$$
It is then easy to see that \(\mathcal S_{u,\partial}(B_R)\) is relatively closed in \(B_R\). Since \(u|_{\partial B_R}\) is bounded and admits a weakly spacelike
extension to \(B_R\), namely \(u\) itself, we can apply
\cite[Corollary 4.2]{bs82}. It follows that \(u\) is strictly
spacelike and solves \eqref{bi} in
$
B_R\setminus\mathcal S_{u,\partial}(B_R).
$ 
Actually, we do have that 
\eqn{vuotino}
$$
\mathcal S_{u,\partial}(B_R)=\varnothing
\qquad
\text{for every }R>1.
$$
Assume by contradiction that this is not the case. Then there exist
\(x,y\in\partial B_R\), \(x\neq y\), such that
$
|u(x)-u(y)|=|x-y|
$
(as usual we may assume that
\(u(x)>u(y)\)). Since \(u\) is \(1\)-Lipschitz we have
\eqn{mm.4}
$$
u(y+t(x-y))
=
u(y)+t|x-y|
\qquad
\text{for every }t\in[0,1].
$$
Now fix \(R_{1}>R\). Since \(u\) minimizes the corresponding
Dirichlet problem in \(B_{R_{1}}\), and the segment
\(\overline{xy}\Subset B_{R_{1}}\) satisfies \eqref{mm.4},
Theorem \ref{al.t} implies that
$
u(y+t(x-y))
=
u(y)+t|x-y|
$
for every \(t\in\mathbb R\) such that
$
y+t(x-y)\in B_{R_{1}}.
$ Here we recall that we are assuming that $\mu$ is locally bounded. 
Since \(R_1>R\) is arbitrary, we conclude that
$
u(y+t(x-y))
=
u(y)+t|x-y|$
for every $t\ge0. 
$
This means that $u$ has linear growth at infinity which obviously contradicts the fact that $u(x)\to 0$ at infinity when $n\geq 3$ and \eqref{weightsmall} when $n=2$. This establishes \rif{vuotino} and therefore we infer that \(u\) is strictly spacelike in
\(\mathbb R^n\). By \cite[Corollary 4.2]{bs82} we have that 
$
u\in
C^1_{\loc}(\mathbb R^n)
\cap
W^{2,2}_{\loc}(\mathbb R^n)
$
and \(u\) solves \eqref{bi} in every ball. Consequently,
\rif{mm.9}
holds for every smooth $w$ with compact support. 
We finally extend the validity of \rif{mm.9} to compactly supported
\(W^{1,\infty}\)-functions $w$. Choose \(R_w>1\) such that
$
\supp\,   w \Subset B_{R_w}.
$
By standard mollification, there exists
\(\{\varphi_j\}\subset C_{0}^\infty(\mathbb R^n)\) such that
$
\supp\,   \varphi_j\Subset B_{2R_w}$ and 
$\varphi_j\to w$ uniformly, 
and
$
D\varphi_j\rightharpoonup^* Dw$ 
in $
L^\infty(B_{2R_w};\mathbb R^n)$. Obviously 
\eqn{mm.99}
$$
\int_{\mathbb{R}^{n}}\langle\partial H(Du),D\varphi_j\rangle\dx=\int_{\mathbb{R}^{n}}\mu \varphi_j\dx\,.
$$
Since \(u\in C^1_{\loc}(\mathbb R^n)\) is strictly spacelike,
compactness gives
$
\|Du\|_{L^\infty(B_{2R_w})}\le\theta_w<1.
$
Hence
$\partial H(Du)=Du/\sqrt{1-|Du|^2}\in L^\infty(B_{2R_w};\mathbb R^n).$
We can therefore pass to the limit in \eqref{mm.99} and obtain
\rif{mm.9}. If $u\equiv0$, then \eqref{mm.9} yields
$\int_{\mathbb R^n}\mu\varphi\,dx=0$
for every $\varphi\in C^\infty_0(\mathbb R^n)$.
By Remark \ref{concreto} and density, this implies
$\mu=0$ in $\DD^*$, proving the nontriviality assertion. Moreover, since $
u\in W^{2,2}_{\loc}(\mathbb R^n)
$
and \(u\) is strictly spacelike, we conclude that 
$
\partial H(Du)\in W^{1,2}_{\loc}(\mathbb R^n;\mathbb R^n),
$
so \eqref{bi} holds almost everywhere. Thus \(u\) is a so-called strong
solution. Finally, if \(\mu\in C^k(\mathbb R^n)\), the equation is uniformly
elliptic on every compact subset of \(\mathbb R^n\). Standard
elliptic bootstrap therefore gives
$
u\in C^{k+1}(\mathbb R^n);
$
see also \cite[Remark, pag.~147]{bs82}. Hence, if $k\ge1$, $u$ is a classical solution
of \eqref{bi}. 
\end{proof}

\begin{proposition}[Calder\'on--Zygmund type estimate]\label{cz.p}
Let $n\ge3$, let $\mu\in C^\infty_0(\mathbb R^n)$ and let
$u\in\mathbb X(\mathbb R^n)$ be the minimizer
of $\mathcal E_\mu$.
Assume that there exists $\mathcal h_\infty>1$
such that $\supp\, (\mathcal h_u-\mathcal h_\infty)_+$
is compact. Then, for every $2\le q<\infty$, 
\eqn{czcz}
$$
\int_{\mathbb R^n}
(\mathcal h_u-\mathcal h_\infty)_+^q\dx
\le
\frac{c\Ui^{q}}{(1-\mathcal{h}_{\infty}^{-2})^{q}}
\int_{\mathbb R^n}|\mu|^q\dx,
$$
holds with $c=c(n,q)$, where $\Ui$ has been defined in \eqref{mmi}. 
Finally, $u$ is smooth, strictly
spacelike, and solves \eqref{bi} classically. In fact, the same estimates hold for any smooth, strictly
spacelike classical solution $u\in\mathbb X(\mathbb R^n)$
to \eqref{bi}, with the same assumptions on $\mu$
and on the support of
$(\mathcal h_u-\mathcal h_\infty)_+$.\end{proposition}
\begin{proof}
The proof closely follows those in \cite[Theorem 3.5]{bs82},
\cite[Theorem 3.6]{haa20} and \cite[Proposition 4.2]{bi23},
where such a statement is available in a slightly different
form. Although the corresponding results in
\cite[Theorem 3.6]{haa20} and
\cite[Proposition 4.2]{bi23} are stated for $q>n$,
the argument below applies to every $q\ge2$.
By Proposition \ref{exex}, $u$ is smooth, strictly
spacelike, and solves \eqref{bi} classically. Conversely, for a smooth strictly spacelike solution
satisfying the stated hypotheses, the support condition
gives $\mathcal h_u\in L^\infty(\mathbb R^n)$, hence
$\partial H(Du)\in L^2(\mathbb R^n;\mathbb R^n)$. 
By density and convexity, such a solution minimizes
$\mathcal E_\mu$ as well. The proof is based on the structural fact that the boost
function $\bof$ satisfies an inhomogeneous subsolution
inequality for a linear elliptic operator with coefficient
matrix $\mathcal H(Du)$ - see \rif{cz.aux} below.
When $\mu\equiv0$, this reduces to a genuine subsolution
property. Throughout the proof, we set
$
\eta:=(\mathcal h_u-\mathcal{h}_{\infty})_+
$
and
$
\delta:=1-\mathcal{h}_{\infty}^{-2}\in(0,1).
$
When $q=2$, every occurrence of $\eta^{q-2}$
is understood as $\mathds{1}_{\{\eta>0\}}$, consistently
with $D\eta=\mathds{1}_{\{\eta>0\}}D\mathcal h_u$
almost everywhere.
We use the notation in Section \ref{ner}. Testing \rif{0} with $ \varphi D_su$ instead of $\varphi$, where
$\varphi\ge0$ is compactly supported, summing over
$s\in\{1,\ldots,n\}$, recalling \rif{identitas},
and dropping the nonnegative term containing
$\partial^2H(Du)D_sDu$, we obtain
\eqn{cz.aux}
$$
\int_{\mathbb R^n}
\langle\mathcal H(Du)D\mathcal h_u,D\varphi\rangle\dx
\le
\int_{\mathbb R^n}
\varphi\langle Du,D\mu\rangle\dx.
$$
We now essentially reproduce the arguments of \cite[Theorem 3.6]{haa20} and \cite[Proposition 4.2]{bi23}. 
Let \(q\in[2,\infty)\). Since \(\eta\) has compact support, we may
use
$
\varphi
=
u\mathcal h_u^{-1}\eta^q
$
as a test function in \eqref{mm.9} to get (recall \rif{recalla})
$$
\begin{aligned}
\int_{\mathbb R^n}\eta^q|Du|^2\dx
={}&
\int_{\mathbb R^n}
u\mathcal h_u^{-1}\eta^q
\langle Du,D\mathcal h_u\rangle\dx
-
q\int_{\mathbb R^n}
u\eta^{q-1}
\langle Du,D\mathcal h_u\rangle\dx
+
\int_{\mathbb R^n}
u\mathcal h_u^{-1}\eta^q\mu\dx.
\end{aligned}
$$
Since
$
0\le \eta/\mathcal h_u\le1
$
and, using \(\mathcal H(Du)Du=Du\) with the Cauchy--Schwarz inequality 
\eqn{cscs}
$$
|\langle Du,D\mathcal h_u\rangle|
=|\langle
\mathcal H(Du)Du,D\mathcal h_u
\rangle|\le
|Du|\langle\mathcal H(Du)D\mathcal h_u,D\mathcal h_u\rangle^{1/2},
$$
we infer that
\eqn{cz.111}
$$
\int_{\mathbb R^n}\eta^q|Du|^2\dx
\le
(q+1)
\int_{\mathbb R^n}
|u|\eta^{q-1}|Du|
\langle
\mathcal H(Du)D\mathcal h_u,D\mathcal h_u
\rangle^{1/2}\dx
+
\int_{\mathbb R^n}
|u|\eta^{q-1}|\mu|\dx.
$$
On the support of \(\eta\), we have
\(\mathcal h_u\ge\mathcal{h}_{\infty}\), and hence
\eqn{cz.delta}
$$
|Du|^2=1-\mathcal h_u^{-2}\ge1-\mathcal{h}_{\infty}^{-2}=\delta.
$$
By Young's inequality, \eqref{cz.111} yields
\eqn{cz.2}
$$
\begin{aligned}
\int_{\mathbb R^n}\eta^q|Du|^2\dx
\le{}&
c\|u\|_{L^\infty(\mathbb R^n)}^2
\int_{\mathbb R^n}
\eta^{q-2}\langle
\mathcal H(Du)D\mathcal h_u,D\mathcal h_u
\rangle\dx+\frac{c}{\delta}\|u\|_{L^\infty(\mathbb R^n)}^2\int_{\mathbb R^n}\eta^{q-2}\mu^2\dx.
\end{aligned}
$$
We next control the first integral on the right-hand side.
Taking
$
\varphi=\eta^{q-1}
$
in \eqref{cz.aux}, we obtain
$$
(q-1)
\int_{\mathbb R^n}\eta^{q-2}
\langle\mathcal H(Du)D\mathcal h_u,D\mathcal h_u\rangle\dx
\le\int_{\mathbb R^n}
\eta^{q-1}\langle Du,D\mu\rangle\dx.
$$
Integrating by parts and using
$
-\diver\, (\mathcal h_uDu)=\mu,
$
we find
$$
\int_{\mathbb R^n}
\eta^{q-1}\langle Du,D\mu\rangle\dx
=
\int_{\mathbb R^n}
\mathcal h_u^{-1}\eta^{q-1}\mu^2\dx
+
\int_{\mathbb R^n}
\mu\eta^{q-2}
\left(
\mathcal h_u^{-1}\eta-(q-1)
\right)
\langle Du,D\mathcal h_u\rangle\dx.
$$
Again using \rif{cscs} and 
Young's inequality we obtain
\eqn{cz.3}
$$
\int_{\mathbb R^n}
\eta^{q-2}
\langle
\mathcal H(Du)D\mathcal h_u,D\mathcal h_u
\rangle\dx
\le
c
\int_{\mathbb R^n}
\eta^{q-2}\mu^2\dx.
$$
Combining \eqref{cz.2} and \eqref{cz.3}, and recalling that
\(\delta<1\), gives
$$
\int_{\mathbb R^n}
\eta^q|Du|^2\dx
\le
\frac{c}{\delta}
\|u\|_{L^\infty(\mathbb R^n)}^2
\int_{\mathbb R^n}\eta^{q-2}\mu^2\dx
$$
with $c\equiv c(q)$. 
By \eqref{cz.delta},
$$
\int_{\mathbb R^n}\eta^q\dx
\le
\frac{c}{\delta^2}
\|u\|_{L^\infty(\mathbb R^n)}^2
\int_{\mathbb R^n}\eta^{q-2}\mu^2\dx.
$$
H\"older's inequality yields (not needed when $q=2$)
$$
\int_{\mathbb R^n}\eta^{q-2}\mu^2\dx
\le
\left(
\int_{\mathbb R^n}\eta^q\dx
\right)^{\frac{q-2}{q}}
\left(
\int_{\mathbb R^n}|\mu|^q\dx
\right)^{\frac2q}.
$$
Consequently,
$$
\int_{\mathbb R^n}\eta^q\dx
\le
\frac{c(n,q)}
{\delta^q}
\|u\|_{L^\infty(\mathbb R^n)}^q
\int_{\mathbb R^n}|\mu|^q\dx
$$
for $c\equiv c(n,q)$. 
By Proposition \ref{boun.p},
$
\|u\|_{L^\infty(\mathbb R^n)}
\le
c\,\Ui,
$
and hence \eqref{czcz} follows. 
 The proof is complete.
\end{proof}

\end{document}